\documentclass[10pt,oneside]{report}
\title{Double $L$-theory}
\author{P. H. Orson}
\date{2015}
\usepackage[phd,hyperref,colour,fancyhdr]{edmaths}

\usepackage{amsmath,amsxtra,amsthm,amssymb,makeidx,color,graphics,graphicx,enumerate,setspace,pinlabel,caption,subcaption,pinlabel,sseq, mathtools}
\usepackage[all]{xy}
\xyoption{import,dvips,rotate}

\DeclareMathOperator{\Tot}{Tot}
\DeclareMathOperator{\im}{im}

\DeclareMathOperator{\coker}{coker}

\DeclareMathOperator{\Hom}{Hom}

\DeclareMathOperator{\Ext}{Ext}
\DeclareMathOperator{\Tor}{Tor}

\DeclareMathOperator{\ann}{ann}
\DeclareMathOperator{\Aut}{Aut}

\DeclareMathOperator{\Ch}{Ch}
\DeclareMathOperator{\ev}{ev}

\DeclareMathOperator{\cl}{cl}

\renewcommand{\emptyset}{\varnothing}
\def\A{\mathbb{A}}
\def\B{\mathbb{B}}
\def\C{\mathbb{C}}
\def\D{\mathbb{D}}
\def\FF{\mathbb{F}}
\def\H{\mathbb{H}}
\def\R{\mathbb{R}}
\def\Z{\mathbb{Z}}
\def\N{\mathbb{N}}
\def\Q{\mathbb{Q}}

\def\NN{{\mathfrak{N}}}

\def\p{{\mathfrak{p}}}

\def\lmat{\left(\begin{smallmatrix}}
\def\rmat{\end{smallmatrix}\right)}

\def\sm{\setminus}

\def\id{\operatorname{id}}

\def\HH{\mathcal{H}}

\theoremstyle{definition}
\newtheorem{theorem}{Theorem}[section]
\newtheorem{definition/proposition}[theorem]{Definition/Proposition}
\newtheorem{proposition}[theorem]{Proposition}
\newtheorem{lemma}[theorem]{Lemma}
\newtheorem{corollary}[theorem]{Corollary}
\newtheorem{definition}[theorem]{Definition}

\newtheorem*{remark}{Remark}
\newtheorem{example}[theorem]{Example}
\newtheorem{claim}[theorem]{Claim}
\newtheorem{question}[theorem]{Question}
\newcommand{\eps}{\varepsilon}

\begin{document}
\flushbottom
\pagenumbering{roman}
\maketitle


\begin{abstract}

This thesis is an investigation of the difference between metabolic and hyperbolic objects in a variety of settings and how they interact with cobordism and `double cobordism', both in the setting of algebraic $L$-theory and in the context of knot theory.

Let $A$ be a commutative Noetherian ring with involution and $S$ be a multiplicative subset. The Witt group of linking forms $W(A,S)$ is defined by setting metabolic linking forms to be 0. This group is well-known for many localisations $(A,S)$ and it is a classical fact that it forms part of a localisation exact sequence, essential to many Witt group calculations. However, much of the deeper `signature' information of a linking form is invisible in the Witt group. The beginning of the thesis comprises the first general definition and careful investigation of the \textit{double Witt group of linking forms} $DW(A,S)$, given by the finer equivalence relation of setting hyperbolic linking forms to be 0. The treatment will include invariants, structure theorems and localisation exact sequences for various types of rings and localisations. We also make clear the relationship between the double Witt groups of linking forms over a Laurent polynomial ring and the double Witt group of those forms over the ground ring that are equipped with an automorphism. In particular we prove the isomorphism between the double Witt group of Blanchfield forms and the double Witt group of Seifert forms.

In the main innovation of the thesis, we next define chain complex generalisations of the double Witt groups which we call the double $L$-groups $DL^n(A,S)$. In double $L$-theory, the underlying objects are the symmetric chain complexes of algebraic $L$-theory but the equivalence relation is now the finer relation of algebraic double-cobordism. In the main technical result of the thesis we solve an outstanding problem in this area by deriving a double $L$-theory localisation exact sequence. This sequence relates the $DL$-groups of a localisation to both the free $L$-groups of $A$ and a new group analogous to a `double' algebraic homology surgery obstruction group of chain complexes over the localisation. We investigate the periodicity of the double $L$-groups via skew-suspension and surgery `above and below the middle dimension'. We then reconcile the double $L$-groups with the double Witt groups, so that we also prove a double Witt group localisation exact sequence.

Finally, in a topological application of double Witt and double $L$-groups, we apply our results to the study of doubly-slice knots. A doubly-slice knot is a knot that is the intersection of an unknotted sphere and a plane. We show that the double knot-cobordism group has a well-defined map to the $DL$-group of Blanchfield complexes and easily reprove some classical results in this area using our new methods.
\end{abstract}

\tableofcontents
\newpage

\pagenumbering{arabic}

\chapter{Introduction}\label{chap:introduction}

Poincar\'{e} duality is the fundamental symmetry present in the structure of a closed, oriented topological manifold. The various algebraic expressions and consequences of this symmetry are basic to the programme of understanding the structure of manifolds via algebraic topology. For an even-dimensional manifold, the simplest algebraic invariant expressing Poincar\'{e} duality is the \textit{intersection form}, defined on a finitely generated free module. For an odd-dimensional manifold it is the \textit{linking form} on a finitely generated torsion module. In this thesis we shall be particularly concerned with the odd-dimensional case, studying linking forms and their chain complex generalisations, \textit{symmetric Poincar\'{e} complexes with torsion homology}.

In several interrelated settings within both algebra and topology there has been observed a subtle difference between what we call \textit{metabolic} and \textit{hyperbolic} objects. Roughly speaking, an object is metabolic if it is nullcobordant and it is hyperbolic if it is \textit{doubly nullcobordant} - that is it admits two nullcobordisms that glue together to make a trivial object\footnote{Various precise definitions will be given in the sequel!}.

\begin{figure}[h]\[\def\picnullcob{\resizebox{0.35\textwidth}{!}{ \includegraphics{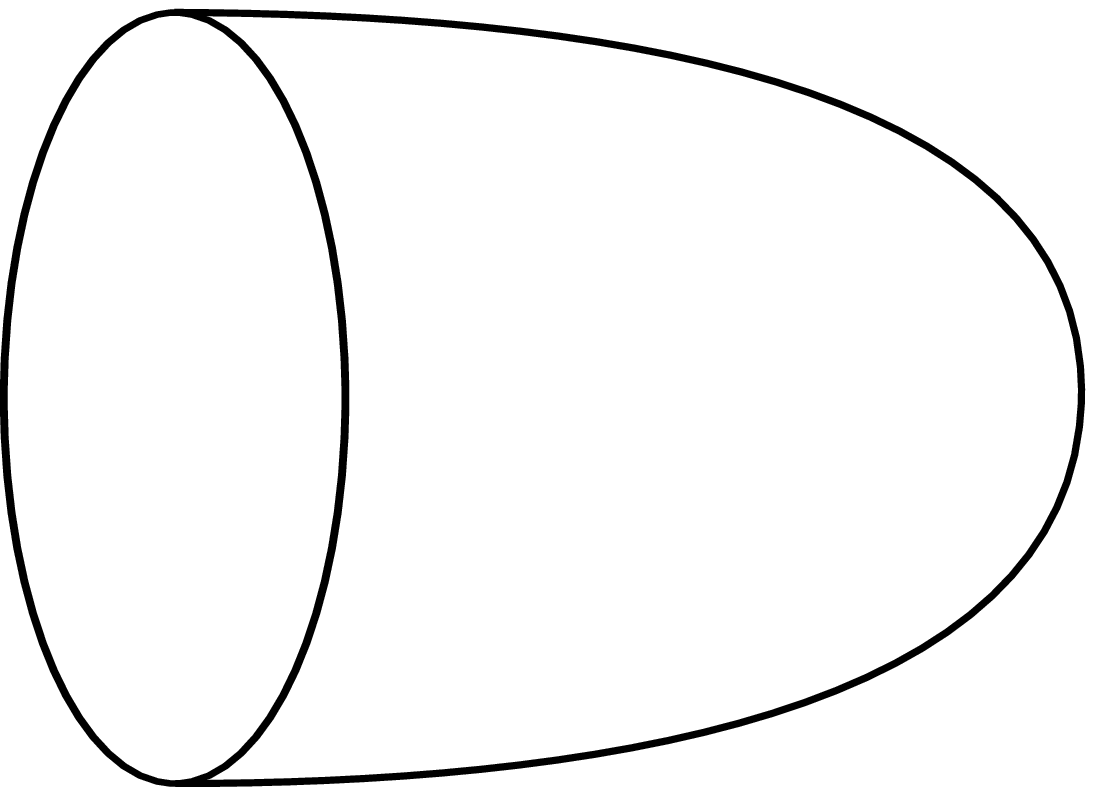}}}
\begin{xy} \xyimport(300,300){\picnullcob}
,!+<6pc,-1.5pc>*+!\txt{A nullcobordism of $N$}
,(35,148)*!L{N}
,(165,148)*!L{M}
\end{xy}
\qquad\qquad
\def\picsphere{\resizebox{0.3\textwidth}{!}{ \includegraphics{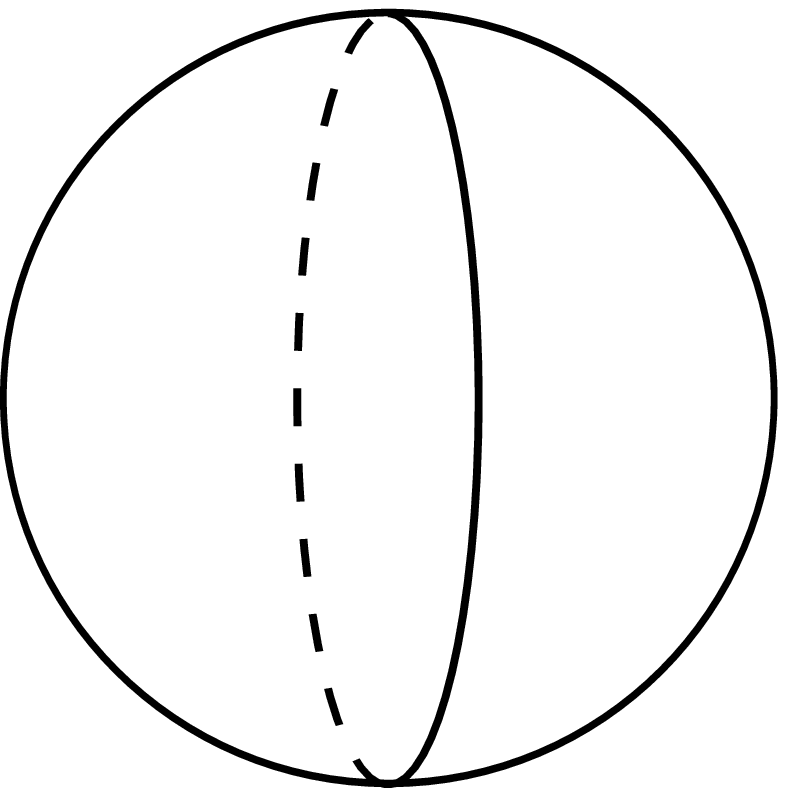}}}
\begin{xy} \xyimport(259,139){\picsphere}
,!CD+<0.3pc,-1pc>*+!CU\txt{A double-nullcobordism of $N$.}
,(30,70)*!L{M_+}
,(275,70)*!L{\cong 0}
,(119,72)*!L{N}
,(190,70)*!L{M_-}
\end{xy}\]\caption{$M$, $N$, $M_\pm$ can be either manifolds or chain complexes with duality.}
\end{figure}

The aim of this thesis is to make precise the metabolic/hyperbolic distinction and to define and develop invariants to measure it. This will be developed firstly, and most extensively, in the algebraic setting of linking forms and then of symmetric $L$-theory. The results will then be applied to construct new invariants in the motivating topology of double knot-cobordism.

The main topological intuition and motivation for the metabolic/hyperbolic distinction comes from knot theory. An $n$-knot $K:S^n\hookrightarrow S^{n+2}$ is called \textit{slice} if it is the boundary of a disc embedded in $D^{n+3}$, and is \textit{doubly slice} if it is the boundary of two discs in $D^{n+3}$ such that when you glue the discs along the knot, the resulting $(n+1)$-sphere is unknotted in $S^{n+3}$. The problem of detecting whether a knot is doubly knot-nullcobordant, called the \textit{doubly slice problem}, is a long-standing and difficult problem in knot theory that is yet to be effectively tackled.

We first investigate the metabolic/hyperbolic distinction in the case of linking forms. We make the first general definitions and calculations of what we call the \textit{double Witt groups} of a localisation of a ring with involution. The double Witt groups refine the classical Witt groups of linking forms and, working over a Dedekind domain, we calculate the kernel of the forgetful map to show precisely the `deeper' signature information that the new groups capture.

Inspired by the concept of a double-nullcobordism, we next develop the central innovation of this thesis -- a theory of chain complex double-cobordism, which we call \textit{double $L$-theory}. Similarly to the case of linking forms, our new groups refine the classical $L$-groups of a localisation of a ring with involution. The double $L$-groups are studied in some detail and in particular we prove our main technical result, (Theorem \ref{thm:DLLES} in the thesis):

\begin{theorem}[Double $L$-theory localisation exact sequence]Suppose $(A,S)$ defines a localisation of  a commutative Noetherian ring with involution $A$ containing a half-unit. For $n\geq0$, the sequence of group homomorphisms\[0\to L^{n}(A,\eps)\xrightarrow{Di} D\Gamma^{n}(A\to S^{-1}A,\eps)\xrightarrow{D\partial} DL^{n}(A,S,\eps)\]is exact.\end{theorem}

This sequence is in the vein of the classical localisation exact sequences of Bass, Karoubi, Quillen, Ranicki, Vogel et al, and probes the structure of double $L$-theory by relating the torsion double $L$-groups to the classical projective $L$-groups via an algebraic `double homology surgery obstruction' group. As a corollary we obtain a localisation exact sequence for the double Witt groups, thus solving an outstanding problem in this area:

\begin{theorem}[Double Witt group localisation exact sequence]Suppose $(A,S)$ defines a localisation of  a commutative Noetherian ring with involution $A$ containing a half-unit. The sequence of group homomorphisms\[0\to W^\eps(A)\xrightarrow{Di} \widetilde{D\Gamma}^\eps(A\to S^{-1}A)\xrightarrow{D\partial} DW^\eps(A,S),\]is exact.\end{theorem}

On the topological side, we apply the double Witt groups and the double $L$-groups to obtain invariants of the double knot-cobordism class of a knot. Our invariants are also robust enough to recover the main classical doubly slice invariants (including Blanchfield form, Seifert form, Farber-Levine pairing) in a single framework. Using our chain complex methods we recover some classical results in this area.

\section{A motivating example}

As much of the work of this thesis is algebraic in nature, we would like to start here with an easy and well-known geometric example to illustrate the main themes we intend to develop.

\begin{example}\label{ex:1}Suppose we work in a fixed category $Diff$, $PL$ or $TOP$ and that $N= S^3_\Q$ is an oriented manifold that is a rational homology 3-sphere ($\Q$HS). For any finite abelian group $T$ there is an isomorphism $\Ext^1_\Z(T,\Z)\cong \Hom_\Z(T,\Q/\Z)$. The groups $H_1(N)$ and $H^2(N)$ are finite, so \[\begin{array}{rcll}\qquad H_1(N)&\cong& \Ext^1_\Z(H^2(N),\Z)&\qquad\text{(Universal Coefficient Theorem)}\\&\cong&\Hom_\Z(H^2(N),\Q/\Z).\end{array}\] Hence the Poincar\'{e} duality isomorphism $D:H^2(N)\xrightarrow{\cong} H_1(N)$ can be expressed adjointly as a non-singular, skew-symmetric bilinear pairing \[\lambda:H^2(N)\times H^2(N)\to \Q/\Z\]which is called the \textit{linking form for $N$}. Geometrically the pairing between cocycles $x,y\in C^2(N)$ is described as follows. Because $H_1(N)$ is a torsion abelian group, there exists a positive integer $s\in\N$ such that $sD(y)= d a$ for some choice of $a\in C_2(N)$. It is possible to choose generically transversely intersecting submanifolds $X^1$ and $Z^2$ of $N$ such that $[X]=D[x]\in H_1(N)$ and $[Z]=[a]\in H_2(N)$. The linking form is then defined to be\[\lambda(x,y)=\frac{1}{s}\sum_{p\in X\cap Z}\eps_p\qquad\text{(where $\eps_p=\pm1$}),\] the signed count of the points of intersection of $X$ and $Z$, divided by the number $s$. It is a straightforward check that this is well-defined in $\Q/\Z$ regardless of the choices we have made. For rational homology 3-spheres the linking form encodes all non-trivial homological information.

Now suppose that $N=\partial M$ for a compatibly oriented manifold $M = D^4_\Q$ which is a rational homology 4-ball. By combining the naturality of the Poincar\'{e} duality maps (drawn vertically) with the homology and cohomology long exact sequences we obtain a diagram\[\xymatrix{0\ar[r]&H^2(M,N)\ar[r]\ar[d]^-{\cong}&H^2(M)\ar[r]^{i^*}\ar[d]^-{\cong}&H^2(N)\ar[r]\ar[d]^-{\cong}&H^3(M,N)\ar[r]\ar[d]^-{\cong}&0\\
0\ar[r]&H_2(M)\ar[r]&H_2(M,N)\ar[r]&H_1(N)\ar[r]&H_1(M)\ar[r]&0}\]showing that the submodule $i^*(H^2(M))\subset H^2(N)$ is maximally self-annihilating with respect to the linking form for $N$. We call a maximally self-annihilating submodule a \textit{lagrangian} and say  a linking form is \textit{metabolic} if it admits a lagrangian.

In fact, the set of skew-symmetric linking forms modulo the metabolic forms makes a group $W^{-}(\Z,\Z\sm\{0\})$ called the \textit{Witt group of linking forms}. The linking form of a $\Q$HS is metabolic if and only if it vanishes in the Witt group of linking forms (this is an algebraic statement related to $\Z$ being a principal ideal domain and is discussed in Chapter \ref{chap:DLtheory}) - hence calculating this Witt group would identify an obstruction to the existence of the proposed nullcobordism $M$ of $N$. Indeed, the set of $\Q$HS's modulo nullcobordism via rational homology balls also forms a group, denoted $\Theta^3_\Q$, and we have described a homomorphism of groups\[\Theta^3_\Q\to W^-(\Z,\Z\sm\{0\}).\](Actually this is a surjective morphism \cite{MR594531}.)
\end{example}

To motivate the next step in the example, we must now mention the wider vista of pairings on modules. It is very common to consider non-singular, symmetric (or skew-symmetric) bilinear pairings\[\alpha:K\times K\to A\]where $A$ is a commutative Noetherian ring and $K$ is a finitely generated projective $A$-module. Similarly to linking forms, these pairings arise as the basic algebraic invariant expressing Poincar\'{e} duality, but now for even-dimensional manifolds. Suppose for such an algebraic pairing that there exists a lagrangian $L\subset K$. Then $L$ is automatically a direct summand as $K$ is projective, and there exists a choice of splitting of $K$ with respect to which $\alpha$ is adjointly given by a \textit{metabolic matrix}\[\left(\begin{matrix}0&\ast\\ \ast& \ast\end{matrix}\right):K\to \Hom_A(K,A).\]If 2 is invertible in the ring $A$ then it is possible to improve the choice of splitting of $K$ so $\alpha$ is given by a \textit{hyperbolic matrix} \[\left(\begin{matrix}0&\ast\\ \ast&0\end{matrix}\right):K\to \Hom_A(K,A).\]Whenever we can find a complementary submodule to $L$, that is also a lagrangian, then we call the form itself \textit{hyperbolic}.

But in contrast, it is easy to think of examples of linking forms which admit lagrangians that are not direct summands. For instance the Lens space $L(4,1)$ is a $\Q$HS with linking pairing\[\Z/4\Z\times\Z/4\Z\to\Q/\Z;\qquad(x,y)\mapsto xy/4.\]The only possible lagrangian submodule is given by $2:\Z/2\Z\hookrightarrow \Z/4\Z$ but this is not a direct summand.

Combining everything so far motivates the following questions:
\begin{itemize}
\item When might a hyperbolic linking form example arise from geometry?
\item When might the distinction between metabolic and hyperbolic linking forms be geometrically significant?
\item Is the metabolic/hyperbolic distinction a shadow of a higher level phenomenon in algebra and topology?
\end{itemize}

\begin{example}\label{ex:2}We offer one partial answer to the first question now, by continuing with Example \ref{ex:1}. Extending the interpretation of a lagrangian as a nullcobordism, suppose there exist oriented manifolds $M_+$ and $M_-$ with common boundary $\partial M_\pm\cong N= S^3_\Q$ and consider glueing them together along $N$. Call the nullcobordisms $(M_\pm,N)$ \textit{complementary} if $M_+\cup_N M_-= S^4_\Z$, an integral homology 4-sphere.

\[\begin{xy} 
\def\picsphere{\resizebox{0.3\textwidth}{!}{ \includegraphics{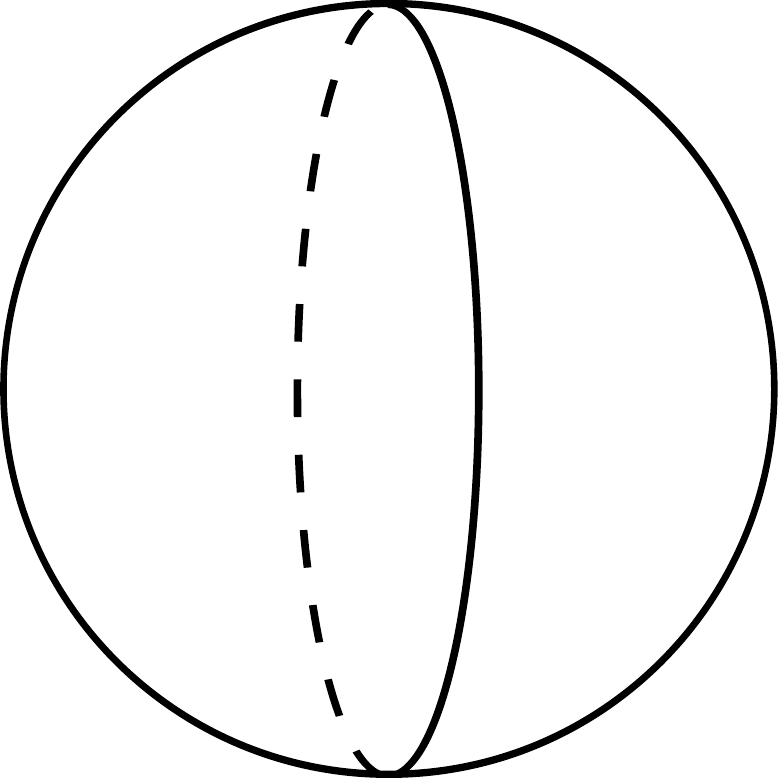}}}
\xyimport(259,139){\picsphere}
,(30,70)*!L{M_+}
,(275,70)*!L{= S^4_\Z}
,(279,80)*!L{?}
,(119,71)*!L{N}
,(190,70)*!L{M_-}
\end{xy}\]

If $(M_\pm,N)$ are complementary then an analysis of the Meier-Vietoris sequence shows that this actually forces $M_\pm$ to be rational homology 4-balls and that the map \[(i_+^*\,\,\,i_-^*):H^2(M_+)\oplus H^2(M_-)\xrightarrow{\cong} H^2(N)\] must be an isomorphism. Hence the linking form on $N$ is hyperbolic. In fact we claim another easy Meier-Vietoris argument shows the converse. So that:

\begin{claim}Suppose $\partial M_\pm=N$. Then $M_+\cup_N M_-=S^4_\Z$ if and only if $M_\pm= D_\Q^4$ and $M_\pm$ induce complementary lagrangian submodules for the linking form of $N$.
\end{claim}

\end{example}

In one sense Example \ref{ex:2} is neater than Example \ref{ex:1}. Not only does the hyperbolic property give us a `Witt-type' obstruction to the geometric situation but we have an `if and only if' statement.

However, in several ways, Example \ref{ex:2} is a more delicate situation. Firstly, until this thesis, there was no general algebraic theory of linking forms modulo hyperbolic forms (although the theory has been developed in special cases). Secondly, Example \ref{ex:2} appears to be concerned with something we might call a `double-nullcobordism'. But there is not a general concept of `double-cobordism' or of a `double-cobordism group', so that at first sight there is no group homomorphism analogous to $\Theta^3_\Q\to W(\Z,\Z\sm\{0\})$ from Example \ref{ex:1}. Let's delve a little deeper into this second concern now.

\section{Geometric `double-cobordism'?}

One common definition of a \textit{cobordism category} $(\mathcal{C},\partial, i)$ does not require a great deal of structure: \begin{itemize}
\item $\mathcal{C}$ is a small category with a finite coproduct `$+$', and initial object, written `$\emptyset$',
\item $\partial:\mathcal{C}\to \mathcal{C}$ is an additive functor (with respect to $+$),
\item $i:\partial\to \text{Id}_\mathcal{C}$ is a natural transformation such that $\partial \partial M=\emptyset$ for all objects $M$ of $\mathcal{C}$.
\end{itemize}In such a setting, two objects $N$ and $N'$ are \textit{cobordism equivalent} if there exist $M$ and $M'$ such that $N+\partial M\cong N'+\partial M'$. The set of isomorphism classes of objects, modulo this equivalence relation forms a commutative monoid. But in many topological examples there is a natural notion of cobordism inverse (usually given by some notion of a cylinder) so that we have, moreover, a \textit{cobordism group}.

When might a concept of `double-cobordism' make sense and be useful? Desirable properties of a good notion of double-cobordism would be\begin{enumerate}
\item It is a refinement of the cobordism relation.
\item Doubly-nullcobordant is to nullcobordant as hyperbolic is to metabolic.
\item The set of objects modulo the double-cobordism relation forms an interesting and tractable algebraic object e.g.\ a group.
\end{enumerate}

Geometrically, desirable property 3 is a big ask and is the major obstruction to developing this concept in generality. Assuming we already have property 1, to see a group structure we would need to take a geometric object $M$ and its standard cobordism inverse $-M$, then show that $M+(-M)$ is doubly-nullcobordant. There is simply no general geometric reason this should be the case - it requires real geometric input.

Remarkably, this property \textit{does} hold in the theory of knot-cobordism. If $K:S^n\hookrightarrow S^{n+2}$ is an $n$-dimensional knot then $K\#-K$ is a doubly-slice knot. (This follows from a geometric result of Zeeman on `twist-spinning', for which we give a quick and elementary reproof in Appendix \ref{chap:twistspin}.) This observation led Stoltzfus (\cite{MR521738}) to make the first and (to this date) only definition of a geometric double-cobordism group, the \textit{$n$-dimensional double knot-cobordism group}  $\mathcal{DC}_n$. This group is still very poorly understood. This is even the case in `high dimensions' $n\geq 3$, where the powerful tools of classical surgery theory (\cite{MR0189052}, \cite{MR0179803}, \cite{MR0246314}, \cite{MR0467764}, \cite{MR620795}) have provided a complete description of the \textit{$n$-dimensional knot-cobordism group}. $\mathcal{C}_n=0$ when $n$ is even, and when $n$ is odd\begin{equation}\label{eq:concordance}\mathcal{C}_n\xrightarrow{\cong}L^{n+1}(\Z[z,z^{-1}],P)\cong\bigoplus_\infty\Z\oplus\bigoplus_\infty(\Z/2\Z)\oplus\bigoplus_\infty(\Z/4\Z).\end{equation}Indeed, there are very few examples of high-dimensional cobordism problems that have not been effectively tackled via surgery theory. We can even go so far as to say the classical surgery programme \textit{cannot} tackle the doubly-slice problem as it has been shown by Ruberman (\cite{MR709569}, \cite{MR933307}) that the technique of `surgery below the middle dimension' cannot be applied to a double knot-cobordism class $[K]\in\mathcal{DC}_n$ for any $n$.

One further complication, and perhaps the central issue with the double knot-cobordism group, is that it isn't known to be `transitive', in the sense that knot $K$ represents 0 in $\mathcal{DC}_n$ if and only if there exist doubly slice knots $J$ and $J'$ such that $K\#J\simeq J'$, but this does not necessarily mean that $K$ itself is doubly slice.

\section{Localisation and algebraic double-cobordism}

On the algebraic side, the story of double-cobordism is very rich. Suppose $A$ is a ring equipped with a unit and an involution. If $S\subset A$ is a multiplicatively closed subset containing the unit and closed under the involution then we may localise $A$ at $S$ to form the ring $S^{-1}A$ of fractions $a/s$ where $a\in A$, $s\in S$. Under some algebraic assumptions on our underlying ring $A$ (stated below), we will show how to take the classical algebraic cobordism theories of the \textit{Witt group of linking forms} $W(A)$ and of \textit{symmetric algebraic $L$-theory} and refine them to double-cobordism theories so that desirable properties 1, 2, and 3 are all satisfied. Our refined versions of these theories are the \textit{double Witt groups of linking forms} $DW(A,S)$ and the \textit{torsion double $L$-groups $DL^n(A,S)$} of $n$-dimensional symmetric Poincar\'{e} complexes over $A$ that have $S$-torsion homology.

For the assumptions on our underlying ring, first we will assume $A$ is Noetherian and commutative. In all geometric applications we have in mind, these conditions are satisfied. The conditions are mainly for convenience and the algebraic theory we develop will not really need them. All the constructions in the double Witt groups and double $L$-theory can be developed without these assumptions. However, we will further assume that there exists a \textit{half-unit} in the ring, that is an element $s\in A$ such that $s+\overline{s}=1$. This assumption is essential to ensure our theory has desirable property 3: a group structure. It is also intimately related to the geometry of knot exterior via Alexander duality. Given a knot $K$, the half-unit corresponds to the fact that if $z$ generates the group of deck transformations of the infinite cyclic cover of the knot complement $X_K$, then $1-z$ is an automorphism of the homology $H_*(X_K;\Z[z,z^{-1}])$. The standard involution on $\Z[z,z^{-1}]$, extended from $z\mapsto z^{-1}$, then makes $s=(1-z)^{-1}$ a half-unit.

Both linking forms and torsion symmetric $L$-theory are the historical tools directly used in the computation of equation \ref{eq:concordance}, so that certainly one end-result we have in mind for our double-cobordism groups is an application to the doubly-slice problem. However these double-cobordism theories are interesting in their own right and turn out to fit neatly into the general scheme in algebraic $K$- and $L$-theory of understanding the localisations of rings with involution.

Define categories\[\begin{array}{rcc}\A(A)&=&\{\text{finitely generated, projective $A$-modules}\},\\
\H(A,S)&=&\{\text{finitely generated, homological dimension 1, $S$-torsion $A$-modules}\}.\end{array}\]What we will call a \textit{localisation exact sequence} is a way of measuring the difference between the categories $\A(A)$ and of $\H(A,S)$.

The first example of such a sequence is in $K$-theory. The categories $\A(A)$ and $\H(A,S)$ are abelian and hence carry a natural \textit{exact structure} in the sense of Quillen so that one can define all higher algebraic $K$-groups $K_n(A)=K_n(\A(A))$ and $K_n(A,S)=K_{n-1}(\H(A,S))$. There is then a long exact sequence for $n>0$ in these $K$-groups \cite{MR0338129}:\[\xymatrix{\dots\ar[r]&K_n(A)\ar[r]&K_n(S^{-1}A)\ar[r]&K_n(A,S)\ar[r]&K_{n-1}(A)\ar[r]&\dots}.\]This specialises in low dimensions to the exact sequence of Bass \cite[XII]{MR0249491}:\[\xymatrix{K_1(A)\ar[r]&K_1(S^{-1}A)\ar[r]&K_1(A,S)\ar[r]&K_{0}(A)\ar[r]&K_0(S^{-1}A)\ar[r]&0}.\] Correspondingly in algebraic $L$-theory there is the Ranicki-Vogel long exact sequence in Wall's quadratic surgery obstruction $L$-groups (\cite{MR1687388}, \cite{MR620795}, \cite{MR585677}):\[\xymatrix{\dots\ar[r]&L_n(A)\ar[r]&L_n(S^{-1}A)\ar[r]&L_n(A,S)\ar[r]&L_{n-1}(A)\ar[r]&\dots}.\]Again, this specialises in low dimensions to the more classical localisation exact sequence of Milnor-Husemoller \cite{MR0506372} (for $A$ a Dedekind domain and $S=A\sm\{0\}$ so that $S^{-1}A$ is the fraction field), now relating the associated Witt groups $W(A)$ of symmetric sesquilinear forms on finitely generated projective modules over $A$ and $W(A,S)$ of symmetric sesquilinear linking forms on finitely generated torsion modules over $A$: \[\xymatrix{0\ar[r]&W(A)\ar[r]&W(S^{-1}A)\ar[r]&W(A,S)\ar[r]&0}.\]

Our double $L$-groups fit into a similar sequence. This is our main technical result of double $L$-theory, Theorem \ref{thm:DLLES}, the \textit{double $L$-theory localisation exact sequence}. It compares the symmetric $L$-groups of $A$ to the torsion double $L$-groups of a localisation $(A,S)$ via the comparison term $D\Gamma^n(A\to S^{-1}A)$. The objects in $D\Gamma^n(A\to S^{-1}A)$ agree with the objects in the groups $\Gamma^n(A\to S^{-1}A)\cong L^n(S^{-1}A)$, the $n$-dimensional algebraic homology surgery obstruction groups (see \cite[2.4]{MR620795}, also cf.\ \cite{MR0339216}), however a more delicate double-cobordism relation called \textit{$\partial$-double-cobordism} is required in our definition. We show that for $n\geq 0$, there is an exact sequence\[\xymatrix{0\ar[r]&L^n(A)\ar[r]&D\Gamma^n(A\to S^{-1}A)\ar[r]&DL^n(A,S)}.\]In the case that $n=2$ and for $A$ a Dedekind domain, we prove that this specialises to an exact sequence in our double Witt group setting:\[\xymatrix{0\ar[r]&W(A)\ar[r]&\widetilde{D\Gamma}(A\to S^{-1}A)\ar[r]&DW(A,S)},\]where the $\widetilde{D\Gamma}$-group in this sequence is a double Witt group for (singular) symmetric forms on an $A$-lattice inside a finitely generated projective $S^{-1}A$-module.

\section{Organisation}

In Chapter \ref{chap:linking} we develop enough classical language to precisely describe the metabolic/hyperbolic distinction in the context of linking forms. Some time will be spent carefully putting together well-known results in this area so that we can prove the \textit{Main Decomposition Theorem}, Theorem \ref{thm:levinePID}. Although this result is new, it essentially follows from several results already in the literature. We then make the first general definitions of the \textit{split Witt} and \textit{double Witt} groups. Next we use the Main Decomposition Theorem to determine the structure of these groups and the kernel of the forgetful map $DW(A,S)\to W(A,S)$ in many cases. Finally we develop the language required to state and understand our \textit{double Witt group localisation exact sequence}, although the proof is postponed for later.

In Chapter \ref{chap:laurent} we review the topological motivation for the study of linking forms, focussing on the case of odd-dimensional manifolds admitting an infinite cyclic cover. We review some results in the literature, including the passage between $\Z$-equivariant linking forms on the infinite cyclic cover and \textit{autometric forms} on a fundamental domain for the cover. We set the algebraic theory of Blanchfield and Seifert forms in this framework. In each case we analyse the associated double Witt groups and we prove isomorphisms between the double Witt group of Blanchfield forms and the double Witt group of Seifert forms.

In Chapter \ref{chap:algLtheory} we prepare for the development of the double $L$-groups by reviewing algebraic $L$-theory in some detail. This account, while not original material, is presented in an original way that is well-suited to our later applications. This chapter is intended as a self-contained introduction to $L$-theory that will prepare the reader for the techniques and constructions used in Chapter \ref{chap:DLtheory}.

In Chapter \ref{chap:DLtheory} we develop the main new tool of this thesis: double $L$-theory. We define the $DL$-groups and confirm basic properties. We then develop the necessary tools to write down and prove the double $L$-theory localisation exact sequence. After this we analyse the potential periodicity in double $L$-theory by investigating the possibility of algebraic surgery `above and below the middle dimension', although periodicity is not proven in general. Finally, we compare the $DL$-groups to the double Witt groups, proving isomorphisms in low dimensions and establishing the double Witt group localisation exact sequence.

In Chapter \ref{chap:blanchfield} we define the knot invariant we will be using, the \textit{Blanchfield complex of an $n$-knot}. This is a chain complex generalisation, originally defined by Ranicki \cite{MR620795} of the Blanchfield form. We establish basic properties of this invariant and recap the knot-cobordism classification of high-dimensional knots in the language of algebraic $L$-theory.

In Chapter \ref{chap:knots} we prove that the class of the Blanchfield complex inside a double $L$-group is an invariant of the double knot-cobordism class of the knot. We recover several classical results in this area using the techniques of double $L$-theory and finally outline some further areas of study that our research suggests.

\chapter{Linking forms}\label{chap:linking}

The difference between the concepts of hyperbolic and metabolic objects is a central theme of this thesis. In Chapter \ref{chap:linking} we set down enough basic, classical algebra to precisely define and explore the concepts and their difference in the context of linking forms. The set of linking forms modulo metabolic equivalence will form the classical Witt group of linking forms. Taking the set of linking forms modulo hyperbolic equivalence, we make the first general definition of the \textit{double Witt group of linking forms}. Using techniques of algebraic number theory we will show how to decompose linking forms and compute the Witt group and double Witt group for various rings and localisations.

We postpone the geometric interpretation of linking forms and the discussion of their significance to algebraic topology until Chapter \ref{chap:laurent}. Chapter \ref{chap:linking} is limited to their careful algebraic description. Many of the details included in this chapter can be found elsewhere and references will be given. However, this is the first time these details, ideas and proofs have been collected in one place and this description has been designed to best fit in with the extensions of the ideas in later chapters and our general methods for understanding ring localisations.

In the final section of Chapter \ref{chap:linking}, we discuss the idea of a localisation exact sequence, a sequence relating linking forms to forms on projective modules, in the double Witt setting. The existence of such a sequence in the classical setting is famously proved in Milnor and Husemoller's textbook \cite{MR0506372}. The existence of such a sequence in our double Witt group setting will be stated as a theorem and will be proved in Chapter \ref{chap:DLtheory} when we have developed more powerful techniques with which to tackle this problem.

\section{Some algebraic conventions}

In the following, $A$ will always be a commutative, Noetherian ring with a unit and involution. The involution is denoted\[\overline{\phantom{A}}:A\to A;\qquad a\mapsto \overline{a}.\]

As $A$ is commutative, we may switch between left and right $A$-module structures trivially. However, we wish to use the involution to define a more sophisticated way of switching, that extends to non-commutative rings with involution and permits an efficient way of describing sesquilinear pairings between left $A$-modules. 

A left $A$-module $P$ may be regarded as a right $A$-module $P^t$ by the action \[P^t\times A\to P^t;\qquad (x,a)\mapsto \overline{a}x.\]Similarly, a right $A$-module $P$ may be regarded as a left $A$-module $P^t$. Unless otherwise specified, henceforth the term `$A$-module' will refer to a left $A$-module. Given two $A$-modules $P,Q$, the tensor product is an abelian group denoted $P^t\otimes_A Q$. We will sometimes write simply $P\otimes Q$ to ease notation, but the right $A$-module structure $P^t$ is implicit, so that for example $x\otimes ay=\overline{a}x\otimes y$.

In the following, $S\subset A$ will always be a \textit{multiplicative} subset, that is

\begin{singlespace*}
\begin{enumerate}[i.]
\item $st\in S$ for all $s,t\in S$.
\item $sa=0\in A$ for some $s\in S$ and $a\in A$ only if $a=0\in A$.
\item $\overline{s}\in S$ for all $s\in S$.
\item $1\in S$.
\end{enumerate}
\end{singlespace*}

The \textit{localisation of $A$ away from $S$} is $S^{-1}A$, the $A$-module of equivalence classes of pairs $(a,s)\in A\times S$ under the relation $(a,s)\sim (b,t)$ if and only if there exists $c\neq0\in A$ such that $c(at-bs)=0\in A$. We say the pair $(A,S)$ \textit{defines a localisation} and denote the equivalence class of $(a,s)$ by $a/s\in S^{-1}A$.  If $P$ is an $A$-module denote $S^{-1}P:=S^{-1}A\otimes_AP$ and write the equivalence class of $(a/s)\otimes x$ as $ax/s$. Similarly, if $f:P\to Q$ is a morphism of $A$-modules then there is induced a morphism of $S^{-1}A$-modules $S^{-1}f=1\otimes f:S^{-1}P\to S^{-1}Q$. If $S^{-1}P\cong0$ then the $A$-module $P$ is \textit{$S$-torsion}. Using condition (ii) in the assumptions for $S$, the natural morphism $i:A\to S^{-1}A$ sending $a\mapsto a/1$ induces a short exact sequence of $A$-modules \[0\to A\to S^{-1}A\to S^{-1}A/A\to 0.\]$i:A\to S^{-1}A$ is moreover an injective ring morphism, but in general $S^{-1}A/A$ is not a ring. For an $A$-module $P$, the induced morphism $i:P\to S^{-1}P$ is not generally injective. This is measured by the exact sequence\[0\to \Tor^A_1(S^{-1}A,P)\to \Tor_1^A(S^{-1}A/A,P)\to P\to S^{-1} P\to (S^{-1}A/A)\otimes_A P.\]From this, we see that $i:P\to S^{-1}P$ is injective if and only if $\Tor^A_1(S^{-1}A/A,P)$ vanishes. This happens, for instance, when $P$ is a projective module. Define the \textit{$S$-torsion of $P$} to be\[TP:=\ker(P\to S^{-1}P).\]

\begin{example}$\,$
\begin{itemize}
\item If $A$ is an integral domain then $S=A\sm\{0\}$ is multiplicatively closed and the pair $(A,S)=(A,A\backslash \{0\})$ defines the \textit{quotient field of $A$}.

\item If $\p$ is a prime ideal of $A$ and $\overline{\p}=\p$ then $S=A\backslash\p$ is a multiplicative subset and we can define the \textit{localisation of $A$ at $\p$} by $A_\p:=S^{-1}A$. Recall a \textit{local ring} is a ring with a unique maximal ideal. $A_\p$ is a local ring, with unique maximal ideal corresponding to the image of $\p$ under the injective $A$-module morphism\[i:A\mapsto A_\p;\qquad a\mapsto a/1.\]
\end{itemize}
\end{example}

Here are some handy lemmas when working with localisations.

\begin{lemma}[Localisation is exact]\label{lem:locisexact}If \[C:\qquad\dots\to C_n\xrightarrow{d} C_{n-1}\to \dots\]is a chain complex of $A$-modules and $(A,S)$ defines a localisation, then for all $r\in\Z$\[S^{-1}H_r(C)\cong H_r(S^{-1}A\otimes_A C)\]
\end{lemma}

\begin{proof}The functor $S^{-1}A\otimes_A-$ is right exact because it is a tensor product functor. We must show this functor is moreover left-exact. Suppose $f:P\to Q$ is an injective $A$-module morphism. If $x\in P$, $s\in S$ and $(S^{-1}f)(x/s):=f(x)/s=0$, then there exists $t\in S$ such that $0=tf(x)=f(tx)$. But $f$ is injective, meaning $tx=0$ so that $x/s=0$. Hence $S^{-1}f:S^{-1}P\to S^{-1}Q$ is injective.
\end{proof}

\begin{lemma}\label{lem:injective}Suppose the inclusion $i:P\to S^{-1}P$ is an injective morphism. If $f:P\to Q$ is an $A$-module morphism such that $S^{-1}f:S^{-1}P\to S^{-1}Q$ is an $S^{-1}A$-module isomorphism then $f$ is injective.
\end{lemma}

\begin{proof}Suppose $x\in P$ and $f(x)=0$. Then $(S^{-1}f)(x/1)=0$ and hence $x/1=0\in S^{-1}P$. But then $x=0$ as $P\to S^{-1}P$ is injective.
\end{proof}

\subsection*{Modules and localisation}

We now introduce categories for comparing projective modules over $A$ to projective modules over $S^{-1}A$. This comparison will be made precise later in the form of the various localisation exact sequences derived in \ref{sec:localisation}, \ref{subsec:localisation} and \ref{sec:localisation2}.

Define a category\[\A(A)=\{\text{finitely generated (\text{f.g.}), projective $A$-modules}\},\] with $A$-module morphisms.

\begin{definition}An $A$-module $Q$ has \textit{homological dimension $m$} if it admits a resolution of length $m$ by f.g.\ projective $A$-modules, i.e.\ there is an exact sequence\[0\to P_m\to P_{m-1}\to\dots\to P_0\to Q\to 0,\]with $P_i$ in $\A(A)$. If this condition is satisfied by all $A$-modules $Q$ we say $A$ is of homological dimension $m$.
\end{definition}

If $(A,S)$ defines a localisation, define a category \[\H(A,S)=\{ \text{f.g.\ $S$-torsion $A$-modules of homological dimension 1}\}\]with $A$-module morphisms.

\begin{definition}\label{def:cartmorph}For $s\in A$ denote $(s)=sA\unlhd A$ the ideal generated by $s$. If $(A,S)$ and $(B,T)$ define localisations then a morphism of rings with involution $f:A\to B$ restricting to a bijection $f|_S:S\to T$ is called \textit{cartesian} if, for each $s\in S$\[[f]:A/(s)\to B/(f(s))\]is an isomorphism of rings.
\end{definition}

For any multiplicative subset $S\subset A$ there is an isomorphism of $A$-modules\[\varinjlim_{s\in S}(A/(s))\xrightarrow{\cong} S^{-1}A\] and so a cartesian morphism $(A,S)\to (B,T)$ determines a cartesian square (hence the name) \[\xymatrix{A\ar[r]\ar[d]&S^{-1}A\ar[d]\\B\ar[r]&T^{-1}B}\]

\begin{theorem}\label{thm:cartmorph}{\cite[App.\ 5]{MR0384894}} A cartesian morphism $(A,S)\to (B,T)$ induces an equivalence between the categories \[\H(A,S)\longleftrightarrow\H(B,T).\]
\end{theorem}

The main purpose of introducing the categories denoted by $\A$ and $\H$ is the heuristic that \[``\textit{$\H(A,S)$ is the difference between $\A(A)$ and $\A(S^{-1}A)$.''}\]

\begin{remark}One precise form that this heuristic takes is the classical algebraic $K$-theory localisation exact sequence of Bass. We briefly mention this for interest and for comparison with the $L$-theoretic versions which will be considered later. According to {\cite[IX.\ Theorem 6.3]{MR0249491}}, there is an exact sequence of abelian groups\[K_1(\A(A))\to K_1(\A(S^{-1}A))\to K_0(\H(A,S))\to K_0(\A(A))\to K_0(\A(S^{-1}A)).\]Where $K_0$ is the functor that returns the Grothendieck group of an abelian category ({\cite[VII.\ 1.2]{MR0249491}}). $K_1$ is the functor that returns the abelianisation (called `Whitehead group' in \cite{MR0249491}) of the category of automorphisms $(P,\alpha:P\to P)$ in an abelian category (\cite[VII.\ 1.4]{MR0249491}).
\end{remark}

\subsection*{Modules with duals}

Finitely generated projective modules have a good notion of duality, coming from the $\Hom$ functor. One consequence of the heuristic that $\H(A,S)$ is the difference between $\A(A)$ and $\A(S^{-1}A)$ is that $\H(A,S)$ has a correspondingly good notion of duality in a derived sense, as we now explain.

Given $A$-modules $P$, $Q$, we denote the additive abelian group of $A$-module homomorphisms $f:P\to Q$ by $\Hom_A(P,Q)$. The \textit{dual} of an $A$-module $P$ is the $A$-module \[P^*:=\Hom_A(P,A)\]where the action of $A$ is $(a,f)\mapsto (x\mapsto f(x)\overline{a})$. If $P$ is in $\A(A)$, then there is a natural isomorphism\[\sm-:P^t\otimes Q\xrightarrow{\cong}\Hom_A(P^*,Q);\qquad x\otimes y\mapsto (f\mapsto \overline{f(x)}y).\]In particular, using the natural $A$-module isomorphism $P\cong P^t\otimes A$, there is a natural isomorphism \[P\xrightarrow{\cong} P^{**};\qquad x\mapsto (f\mapsto \overline{f(x)}).\]Using this, for any $A$-module $Q$ in $\A(A)$ and $f\in\Hom_A(Q,P^*)$ there is a \textit{dual morphism}\[f^*:P\to Q^*;\qquad x\mapsto (y\mapsto \overline{f(y)(x)}).\]

\begin{lemma}\label{lem:ext1}If $T$ is a f.g.\ $A$-module with homological dimension 1 and $T^*=0$ then there is a natural isomorphism of $A$-modules $T\cong \Ext^1_A(\Ext^1_A(T,A),A)$.
\end{lemma}

\begin{proof}Taking the dual of a projective resolution $P_1\xrightarrow{d} P_0\to T$ we obtain a projective resolution for $\Ext^1_A(T,A)$\[T^*=0\to P_0^*\xrightarrow{d^*} P_1^*\to \Ext^1_A(T,A)\to 0.\]Taking the dual again we obtain an exact sequence\[0\to (\Ext^1_A(T,A))^*\to P_1^{**}\xrightarrow{d^{**}} P_0^{**}\to \Ext^1_A(\Ext^1_A(T,A),A).\]The natural isomorphisms $P_i\cong P_i^{**}$ for $i=1, 2$ in particular imply that $(\Ext^1_A(T,A))^*=\ker{d^{**}}\cong\ker{d}=0$. A natural isomorphism of projective resolutions induces a natural isomorphism $T\cong \Ext^1_A(\Ext^1_A(T,A),A)$ (independent of the resolution chosen).
\end{proof}

\begin{lemma}\label{lem:ext2}If $(A,S)$ defines a localisation and $T$ is a f.g.\ $A$-module that is $S$-torsion, then $T^*=0$ and there is a natural isomorphism\[\Ext^1_A(T,A) \cong\Hom_A(T,S^{-1}A/A).\]
\end{lemma}

\begin{proof}As $T$ is f.g.\ there exists $s\in S$ such that $sT=0$, so clearly $T^*=0$. Similarly $\Hom_A(T,S^{-1}A)=0$. Using the short exact sequence \[0\to A\to S^{-1}A\to S^{-1}A/A\to 0,\] we obtain the long exact sequence for $\Ext$\[0\to \Hom_A(T,S^{-1}A/A)\to \Ext_A^1(T,A)\to \Ext_A^1(T,S^{-1}A)\]and as localisation is exact we have $\Ext^1_A(T,S^{-1}A)\cong S^{-1}\Ext^1_A(T,A)=0$ as $T$ is $S$-torsion.
\end{proof}

Lemmas \ref{lem:ext1} and \ref{lem:ext2} show that, proceeding as in the category $\A(A)$ we may define a duality and the same results will follow. Precisely, the \textit{torsion dual} of a module $T$ in $\H(A,S)$ is the module\[T^\wedge:=\Hom_A(T,S^{-1}A/A)\]in $\H(A,S)$ with the action of $A$ given by $(a,f)\mapsto (x\mapsto f(x)\overline{a})$. There is a natural isomorphism\[T\xrightarrow{\cong} T^{\wedge\wedge};\qquad x\mapsto (f\mapsto \overline{f(x)}),\]and for $R,T$ in $\H(A,S)$, $f\in\Hom_A(R,T^\wedge)$ there is a \textit{torsion dual morphism}\[f^\wedge:T\to R^\wedge;\qquad x\mapsto(y\mapsto\overline{f(y)(x)}).\]

\begin{remark}We will only need the `dual' and the `torsion dual' in the sequel but it is interesting to note that these results exist in much greater generality, provided one sets the correct conditions on the ring and modules. For instance, for a commutative Noetherian ring $A$ and a finitely generated $A$-module $M$ with $\Ext^i_A(M,A)=0$ for $i=0, 1,\dots,n-1$, there is a natural homomorphism $M\to \Ext^n_A(\Ext^n_A(M,A),A)$ \cite{MR0263806}. If we assume $\Ext^i_A(M,A)=0$ for all $i\neq n$ and that $M$ is a Cohen-Macaulay module of height $n$ (cf.\ \cite{MR750687}) then there is a natural isomorphism \cite[4.35]{MR0269685}\[M\cong \Ext^n_A(\Ext^n_A(M,A),A).\]
\end{remark}

\subsection*{Forms and linking forms}

From now on, let $\eps\in A$ be a unit such that $\eps\overline{\eps}=1$. For example, $\eps=\pm1$ is always a possible choice.

Standard references for the algebra of forms and classical Witt groups is Milnor-Husemoller \cite{MR0506372}. For a more expanded classical treatment, the canonical texts are Scharlau \cite{MR770063} and O'Meara \cite{MR1754311}. The language we use for the more general setting of rings with involution is based on that found in \cite[1.6, 3.4]{MR620795} although we caution that our terminology, particularly later on regarding lagrangian submodules, differs slightly. Also, our use of the word `split' in reference to forms and linking forms is entirely different to Ranicki's use.

\begin{definition}An \textit{$\eps$-symmetric form over $A$} is a pair $(P,\theta)$ consisting of an object $P$ of $\A(A)$ and an injective $A$-module morphism $\theta:P\hookrightarrow P^*$ such that $\theta(x)(y)=\eps\overline{\theta(y)(x)}$ for all $x,y\in P$ (equivalently $\theta=\eps\theta^*$). A form $(P,\theta)$ is \textit{non-singular} if $\theta$ is an isomorphism. A form induces a sesquilinear pairing also called $\theta$\[\theta:P\times P\to A;\qquad (x,y)\mapsto \theta(x,y):=\theta(x)(y).\]A morphism of $\eps$-symmetric forms $(P,\theta)\to (P',\theta')$ is an $A$-module morphism $f:P\to P'$ such that $\theta(x)(y)=\theta'(f(x))(f(y))$ (equivalently $\theta=f^*\theta' f$), it is an isomorphism when $f$ is an $A$-module isomorphism. The set of isomorphism classes of $\eps$-symmetric forms over $A$, equipped with the addition $(P,\theta)+(P',\theta')=(P\oplus P',\theta\oplus \theta')$ forms a commutative monoid \[\NN^\eps(A)=\{\text{$\eps$-symmetric forms over $A$}\}.\]
\end{definition}

\begin{definition}Suppose $(A,S)$ defines a localisation. An \textit{$\eps$-symmetric linking form over $(A,S)$} is a pair $(T,\lambda)$ consisting of an object $T$ of $\H(A,S)$ and an injective $A$-module morphism $\lambda:T\hookrightarrow T^\wedge$ such that $\lambda(x)(y)=\eps\overline{\lambda(y)(x)}$ for all $x,y\in T$ (equivalently $\lambda=\eps\lambda^\wedge$). A linking form $(T,\lambda)$ is \textit{non-singular} if $\lambda$ is an isomorphism. A linking form induces a sesquilinear pairing also called $\lambda$\[\lambda:T\times T\to S^{-1}A/A;\qquad (x,y)\mapsto \lambda(x,y):=\lambda(x)(y).\]A morphism of $\eps$-symmetric linking forms $(T,\lambda)\to (T',\lambda')$ is an $A$-module morphism $f:T\to T'$ such that $\lambda(x)(y)=\lambda'(f(x))(f(y))$ (equivalently $\lambda=f^\wedge\lambda' f$). $f$ is an isomorphism of forms when $f$ is an $A$-module isomorphism. The set of isomorphism classes of $\eps$-symmetric linking forms over $A$, equipped with the addition $(T,\lambda)+(T',\lambda')=(T\oplus T',\lambda\oplus \lambda')$ forms a commutative monoid \[\NN^\eps(A,S)=\{\text{$\eps$-symmetric linking forms over $(A,S)$}\}.\]
\end{definition}

\begin{definition}[Terminology]When $\eps=1$ we will omit $\eps$ from the terminology e.g.\ \textit{symmetric form} indicates $(+1)$-symmetric form.
\end{definition}

\section{Decomposition of linking forms}

In Section \ref{sec:witt} we will define and compute various groups of linking forms, called Witt groups, split Witt groups and (most importantly) double Witt groups. In order to compute these groups, it will be important to be able to reduce a linking form into its `elementary building blocks' and derive invariants from each block. This process is most effective when the $A$-module $T$ of the linking form is `primary', hence the success of our decomposition programme depends greatly on the underlying ring $A$ and how far it permits generic primary decomposition of modules.

\subsection{Decomposition of torsion modules}

Before we begin decomposing the entire linking form data $(T,\lambda)$, we discuss the structure of primary modules over unique factorisation domains, and then primary decomposition of torsion modules in Dedekind domains (all of which is well-known and classical).

\begin{definition}\label{def:rings}A \textit{unique factorisation domain (UFD)} is an integral domain in which every non-zero element can be written as a product of a unit and prime elements of $R$.\end{definition}

Suppose for now that $A$ is an integral domain, that $\p\subset A$ is a prime ideal and that $T$ is in $\H(A,A\sm\{0\})$. $T$ is called \textit{$\p$-primary} if $\p^dT=0$ for some $d>0$. Now suppose further that $p\in A$ is a prime such that $(p)=\p$. If we assume $\p=\overline{\p}$, then the subset\[\p^\infty:=\{up^k\,|\,\text{$u$ is a unit, }k>0\}\]is a multiplicative subset, and $T$ is $\p$-primary if and only if it is $\p^\infty$-torsion.

We say a general $A$-module $H$ has $\p$-only torsion if the obvious $A$-module morphism $H\to A_\p\otimes_A H$ is injective. Clearly, if our torsion $A$-module $T$ has $\p$-only torsion then it is $\p$-primary. Interestingly, the converse is not true over a general integral domain, or even a UFD. A module can be simultaneously primary in more than one way, as the next example shows, so that the definition of $\p$-only torsion is a necessary additional concept in general.

\begin{example}\label{ex:pionly}For the UFD $A=\Z[z,z^{-1}]$, the $A$-module $T=(\Z/2\Z)\otimes_A(A/(1-z)A)$ is both $2A$-primary and $(1-z)A$-primary.
\end{example}

Fix a choice of generator $(p)=\p$. For each $m\geq 0$, define an $A$-module\[K_m(T)=\ker(p^m:T\to T),\]easily seen to be independent of the choice of $p\in A$. Then there is a filtration\[\{0\}=K_0\subseteq K_1\subseteq K_2\subseteq\dots\subseteq K_{d-1}\subseteq K_d=K_{d+1}=\dots\]which terminates as $A$ was assumed Noetherian. Each $K_m$ is $\p$-primary, and $T$ is $\p$-primary if and only if $K_d=T$. The subquotients $J_{l}:=K_{l}/K_{l-1}$ are $A/\p$-modules and are used to define the \textit{$l$th auxiliary module} over $A/\p$\begin{equation}\label{eq:auxiliary}\Delta_l(T):=\coker(p:J_{l+1}\to J_{l})\cong\frac{K_l}{K_{l-1}+pK_{l+1}}\cong \frac{p^lK_{l+1}}{p^{l+1}K_{l+2}},\end{equation}which is clearly natural in $T$ and will form part of an important linking form invariant later.

\begin{example}\label{ex:auxil}Consider $T=(\Z/2\Z)\oplus(\Z/4\Z)\oplus (\Z/32\Z)$ in $\H(\Z,(2)^\infty)$. We have \[\begin{array}{rccrrcl}
K_1\qquad&(1&2&16):&\Z/2\Z\oplus \Z/2\Z\oplus \Z/2\Z&\hookrightarrow& T,\quad\\
K_2\qquad&(1&1&8):&\Z/2\Z\oplus \Z/4\Z\oplus \Z/4\Z&\hookrightarrow& T,\quad\\
K_3\qquad&(1&1&4):&\Z/2\Z\oplus \Z/4\Z\oplus \Z/8\Z&\hookrightarrow& T,\quad\\
K_4\qquad&(1&1&2):&\Z/2\Z\oplus \Z/4\Z\oplus \Z/16\Z&\hookrightarrow& T,\quad
\end{array}\]and $K_5=T$. The auxiliary modules are the vector spaces $\Delta_1\cong\Delta_2\cong\Delta_5\cong\Z/2\Z$, and $\Delta_l=0$ otherwise.
\end{example}

\subsection*{Dedekind domains}

Once we are already in a $\p$-primary module $T$, we will use our auxiliary modules $\Delta_l(T)$ as decomposition invariants. To ensure that a general torsion module first has a good decomposition with respect to its $\p$-primary submodules, we will eventually need to work over a Dedekind domain.

\begin{definition}A \textit{Dedekind domain} is an integral domain $R$ for which every $R$-module has homological dimension 1. An equivalent definition is that a Dedekind domain is an integral domain in which every non-zero proper ideal factors uniquely as a product of prime ideals. A \textit{principal ideal domain (PID)} is an integral domain in which every ideal is principal.
\end{definition}

Assume for the rest of this subsection that $A$ is a Dedekind domain. In particular, this means the homological dimension 1 assumption on the objects of $\H(A,S)$ (for any $S$) is redundant. As Dedekind domains are a common tool in commutative algebra we will use many details without justification. We refer the reader to \cite{MR0349811} for a detailed discussion of the properties of a Dedekind domain.

A Dedekind domain is a UFD if and only if its `ideal class group' (see \cite[p.9]{MR0349811}) is trivial and in this case it is also a PID. A PID is always a Dedekind domain but the converse is not true generally.

Over a Dedekind domain, a torsion $A$-module $T$ has annihilator $\ann(T ) = \p_1^{l_1} \p_2^{l_2} ...\p_m^{l_m}$ for some prime ideals $\p_i$, so that there is a natural isomorphism of $A$-modules\[T\cong\bigoplus_{j=1}^m T_{\p_j};\qquad T_{\p_j}:=A_{\p_j}\otimes_A T\cong \p_1^{l_1}... \p_{j-1}^{l_{j-1}}\p_{j+1}^{l_{j+1}} ...\p_m^{l_m}T.\]Hence $T$ is $\p$-primary if and only if $\ann(T)=\p^k$ for some $k>0$. Any element $p\in\p\sm\p^2$ is called a \textit{uniformiser} and if we assume $\p=\overline{\p}$ then $T$ is $\p$-primary if and only if it is torsion with respect to the multiplicative subset\[\p^\infty:=\{up^k\,|\,\text{$u$ is a unit, }k>0\}.\]If $\p$ is principal, we may choose the uniformiser $p$ to be the generator and we observe that this definition of $\p^\infty$ agrees with our previous one.

Over a Dedekind domain, if $T$ is $\p$-primary then it has $\p$-only torsion. In fact now the condition of being $\p$-primary is equivalent to the statement that there is an $A$-module isomorphism $T\cong T_{\p}:=A_{\p}\otimes_A T$. Assume $\p=\overline{\p}$ to ensure that $\p^\infty$ is a multiplicative subset, then the morphism $(A,\p^\infty)\to (A_{\p},\p^\infty)$ is cartesian so, by Theorem \ref{thm:cartmorph}, the isomorphisms $T\cong T_{\p}$ define an equivalence of categories\[\H(A,\p^\infty)\longleftrightarrow\H(A_{\p},\p^\infty),\](where we have used the fact that we can regard $T_\p$ as an $A_\p$-module). The cartesian morphism sends a uniformiser $p\in A$ of $\p$ to a generator of the unique maximal ideal (which we are also denoting $\p$) of the local ring $A_{\p}$. As $A$ is a Dedekind domain, $A_\p$ is also a PID. (A ring that is both a local ring and a PID, but is not a field, is often called a \textit{discrete valuation ring} but we won't be needing this terminology.)

For $T$ in $\H(A_\p,\p^\infty)$ and $\p=(p)$, we may choose a length 1, f.g.\ projective $A_\p$-module resolution and, as f.g.\ projective modules over local rings are free, we may take a resolution\[0\to \bigoplus_1^NA_\p\xrightarrow{d} \bigoplus_1^NA_\p \to T \to0.\]Choosing bases, we may diagonalise $d$ and assume without loss of generality that\[d=\left(\begin{array}{cccc}p^{i_1}&\,&\,&\,\\ \,&p^{i_2}&\,&\,\\ \,&\,&\ddots&\,\\ \, &\, &\, & p^{i_N}\end{array}\right).\]This allows us to identify summands $T_l$ of $T$ \[T_l\cong\bigoplus_{i_j=l}A_\p/(p^{i_j}A_\p),\qquad\text{and that}\qquad T\cong\bigoplus_{l=1}^\infty T_l.\]Note that $\ann(T_l)=\p^l$, $pT_l=\ker(p^{l-1}:T_l\to T_l)$ and $T_l$ can be regarded as a free $A_\p/(p^{l}A_\p)$-module. These $T_l$ will be our `elementary building blocks' of torsion modules. We now compare with our previous terminology. The reader can easily confirm that we may use the $T_l$ to express the $K_m$ filtration, the subquotients $J_l$ and our auxiliary modules $\Delta_l(T)$ as:\[K_m=\left(\bigoplus_{l=1}^{m} T_l\right)\oplus\left(\bigoplus_{l=m+1}^\infty p^{l-m}T_l\right),\qquad J_m=\bigoplus_{l=m}^\infty \frac{p^{l-m}T_l}{p^{l-m+1}T_l}, \qquad \Delta_l(T)\cong T_l/pT_l.\]

\begin{example}Recall Example \ref{ex:auxil}, $T=(\Z/2\Z)\oplus(\Z/4\Z)\oplus (\Z/32\Z)$ in $\H(\Z,(2)^\infty)$. Now $T_1=\Z/2\Z$, $T_2=\Z/4\Z$, $T_5=\Z/32\Z$ and $T_l=0$ otherwise.
\end{example}

\begin{example}[Principal ideal domains] If $A$ is moreover a PID, then given a prime ideal $\p\subset A$ we can and will choose our uniformiser $p$ to be a prime element that is a generator $(p)=\p$. If $T$ is a $\p$-primary $A$-module then there is a decomposition on the level of $A$-modules $T\cong\oplus_{l=1}^\infty T_l$ and $T_l$ is now a free $(A/p^lA)$-module for each $l>0$. \textit{We do not need to pass to the localisation $A_\p$ to obtain this decomposition.}
\end{example}

\subsection{Decomposition of linking forms}

How do linking forms interact with prime ideals? If $\p\subset A$ is a prime ideal and $(T,\lambda)$ is a non-singular $\eps$-symmetric linking form over $(A,A\sm\{0\})$ then, by considering our algebraic convention for the $A$-module structure of a torsion dual, we see the restriction of $\lambda_\p:=\lambda|_{T_\p}$ defines an isomorphism\[\lambda_{\p}:T_\p\xrightarrow{\cong} T_{\overline{\p}}^\wedge.\]So if $\p\neq \overline{\p}$ then there is a non-singular linking form\[\left(T_\p\oplus T_{\overline{\p}},\lambda|_{T_\p\oplus T_{\overline{\p}}}\right).\]And if $\p=\overline{\p}$ there is a non-singular linking form $(T_\p,\lambda_\p)$. The following is readily seen for Dedekind domains:

\begin{proposition}\label{prop:dedekinddecomp}If $A$ is a Dedekind domain and $(T,\lambda)$ is a non-singular $\eps$-symmetric linking form over $(A,A\sm\{0\})$, then the primary decomposition of $T$ induces a natural decomposition into non-singular linking forms\[(T,\lambda)\cong\left(\bigoplus_{\p=\overline{\p}} (T_\p,\lambda_\p)\right)\oplus\left(\bigoplus_{\p\neq\overline{\p}}\left(T_\p\oplus T_{\overline{\p}},\lambda|_{T_\p\oplus T_{\overline{\p}}}\right)\right)\]If $\p=\overline{\p}$ for all prime $\p\subset A$ there is induced a natural isomorphism of commutative monoids\[\NN^\eps(A,A\sm\{0\})\cong\bigoplus_{\p}\NN^\eps(A,\p^\infty).\]
\end{proposition}

We now proceed to outline and carry out the programme for decomposing linking forms using prime ideals. Such a programme has already been achieved in the case $A=\Z$. The full analysis of the commutative monoid $\NN(\Z,\Z\sm\{0\})$ was begun by Wall in \cite{MR0156890} where all generators were computed and the relations were given for non-2-torsion generators. It was completed by Kawauchi-Kojima in \cite{MR594531} where missing relations were calculated and a complete set of invariants was given.

We include this example now in some detail, as an indication of a successful approach that we will follow to some extent. The example also illustrates the headaches that occur at the prime 2. We will avoid these added complications later.

\begin{example}\label{ex:KK}We wish to compute $\NN(\Z,\Z\sm\{0\})$. Using Proposition \ref{prop:dedekinddecomp}, it is sufficient to consider the commutative monoids $\NN(\Z,(p)^\infty)$ for $p\in\Z$ a prime.

For $p$ a prime, let $\NN_{p,l}\subset \NN(\Z,(p)^\infty)$ be the sub-monoid consisting of those forms that have annihilator $p^l$. Wall gives explicit generators for each $\NN_{p,l}$ and shows that for each odd $p$ we have \[\NN(\Z,(p)^\infty) \cong \bigoplus_{l=1}^\infty \NN_{p,l}.\]At the prime 2, we still have a sum of monoids \[\NN(\Z,(2)^\infty)\cong\sum_{l=1}^\infty\NN_{2,l},\]but this is not a direct sum as there are relations between the generators of the different $\NN_{2,l}$.

\begin{proof}[Proof (sketch)]Any linking form with underlying module a finite abelian group $G$ decomposes as a direct sum of linking forms into its $p$-primary summands, hence we analyse each prime in turn.

Fix a prime $p$. An element $(G,\lambda)\in\NN(\Z,(p)^\infty)$ can always be decomposed as $(G,\lambda)=\bigoplus_1^\infty(G_l,\lambda_l)$ where $G_l$ is a free $\Z/{p^l\Z}$-module (this is a nontrivial extension of the structure theorem for finite abelian groups). Use the notation $(T,a/p^N)$ for the standard linking form with pairing $(x,y)\mapsto axy/p^N$.

If $p$ is odd then Wall shows the linking form $(G_l, \lambda_l)$ decomposes linearly, into a sum of standard linking forms $(\Z/{p^l}\Z , a/p^l)$ for some non-zero $a\in\Z/{p^l}\Z$. Suppose $l=1$, then for fixed non-zero $a\in\Z/p\Z$, any automorphism of the linking form $(\Z/p\Z,a/p)$ is described by multiplication by a non-zero $c\in \Z/p\Z$, with resulting linking form $(\Z/p\Z,c^2a/p)$. As quadratic residues act transitively on quadratic non-residues, the forms $(\Z/p\Z,a/p)$ are thus classified up to isomorphism by whether or not $a$ is a quadratic residue modulo $p$. As every form over $\Z/p\Z$ decomposes as a direct sum of forms spanned by a single $\Z/p\Z$-summand $\NN_{p,1}$ is generated by two forms $A_p$ and $B_p$, corresponding to a choice of quadratic residue and a choice of quadratic non-residue respectively.

Recall that a consequence of Hensel's Lemma is that $a\in\Z$ is a quadratic residue mod $p$ if and only if $a$ is a quadratic residue mod $p^l$. This is used by Wall to lift our classification of $p$-torsion forms to a classification of forms $(\Z/p^l\Z,a/p^l)$ for all $l>0$. Again, for fixed $l>0$, there are only two possibilities for generators - when $a$ is a quadratic residue or when it is not and we call these forms $A_{p^l}$ and $B_{p^l}$ respectively. Wall shows they have the relation $2A_{p^l} = 2B_{p^l}$ for each $l>0$ but that this is the only relation among the generators $A_{p^l}$, $B_{p^l}$ for $p$ odd.

If $p = 2$, things get more complicated and there are more generators and relations because Hensel's Lemma does not give such neat results. $\NN(\Z,(2)^\infty)$ is generated by\begin{itemize}
\item $A^l(k):=(\Z/{2^l}\Z,k/2^l)$, for $l\geq1$, where $k=1(l=1), \pm1(l=2), \pm1\text{ or }\pm5(l>3)$,
\item $E_0^l :=\left(\Z/{2^l}\Z \oplus \Z/{2^l}\Z,\left(\begin{array}{cc}0&2^{-l}\\2^{-l}&0\end{array}\right)\right)$ for $l\geq1$,
\item $E_1^l :=\left(\Z/{2^l}\Z \oplus \Z/{2^l}\Z,\left(\begin{array}{cc}2^{1-l}&2^{-l}\\2^{-l}&2^{1-l}\end{array}\right)\right)$ for $l\geq1$.
\end{itemize}
For a full list of relations, see \cite{MR594531}. But, for example, for $l\geq 1$\[A^l(k_1)\oplus A^{l+1}(k_2)=A^l(k_1+2k_2)\oplus A^{l+1}(k_2+2k_1).\] So, significantly, \textit{there are relations between elements with different values of $l$}. 
\end{proof}

\end{example}

How can this proof be used as a model for decomposing linking forms in general? Suppose $A$ is a PID and a local ring, with involution invariant maximal ideal $(p)=\p=\overline{\p}$. Here is the programme we will now follow for decomposing a non-singular linking form $(T,\lambda)$ over $(A,\p^\infty)$.

\begin{itemize}
\item Firstly, avoid 2-torsion. In other words define the most general notion of `2 is invertible' in the residue field $A/\p$ and assume this. This will be the concept of a \textit{half-unit}.
\item Define the $l$th auxiliary form $(\Delta_l(T),b_l(\lambda))$, a way of isolating the $p^l$-torsion part of $(T,\lambda)$ \textit{in a natural way}.
\item Show that the set of auxiliary forms is a complete invariant of the isomorphism class of the linking form. This will be Theorem \ref{thm:levinePID}, the Main Decomposition Theorem.
\item Use the Main Decomposition Theorem to lift generators of $\NN^\eps(A/\p)$ to $\NN^\eps_{p,l}$ for each $l>0$ and show that this gives a direct sum decomposition of monoids\[\NN^\eps(A,\p^\infty)\cong\bigoplus_{l=1}^\infty\NN^\eps_{\p,l}.\]
\end{itemize}

\subsection*{Main Decomposition Theorem}

We assume for this subsection that $(T,\lambda)$ is an $\eps$-symmetric linking form over $(A,A\sm\{0\})$ and that $T$ is $\p$-primary, where $(p)=\p=\overline{\p}$ for some prime $p\in A$. 

\medskip

We now begin the decomposition programme outlined above. The main objective of this section is to prove the \textit{Main Decomposition Theorem} (Theorem \ref{thm:levinePID}). The build-up to this will lay out the necessary terminology and prove some useful lemmas in this direction. Much of the work of this section has been covered by other authors, particularly Levine \cite{MR564114}. However it is difficult to find the details required so we have included them here, with citations.

\medskip

The following definition generalises the idea of 2 being invertible in a ring and will be extremely important, not just in this section, but throughout this thesis.

\begin{definition}A \textit{half-unit} $s$ in a ring with involution $A$ is a unit such that\[s+\overline{s}=1.\]A basic example is the ring $A=\Z[\frac{1}{2}]$ with trivial involution and $s=1/2$, or generally any ring with trivial involution, in which 2 is invertible.
\end{definition}

\begin{claim}\label{clm:halfunit}If $A$ is a commutative ring with involution and $p\in A$ is any element such that there is a half unit $s\in A/(p)$ then for any $d>0$ there is a half unit in $A/(p)^d$.
\end{claim}
\begin{proof}By assumption, there exists $x\in A$ such that $s+\overline{s}=1+px$, and note that this equation implies $px=\overline{px}$. Set $s'=s(1-px)$. Then \[s'+\overline{s'}=(s+\overline{s})(1-px)=(1+px)(1-px)\equiv 1\,\,\text{mod $p^2$}.\]The claim follows by induction.
\end{proof}

For now, we will assume $A$ is a UFD and we will continue our decomposition programme for a $\p$-primary $A$-module $T$ by extending the auxiliary modules $\Delta_l(T)$ to auxiliary forms.

\begin{definition}Suppose $(p)=\p$ is a principal ideal, then given an $\eps$-symmetric linking form $(T,\lambda)$ over $(A,\p^\infty)$, define for each $l>0$ the \textit{$l$th auxiliary form}, $(\Delta_l(T),b_l(\lambda))$ over $A/\p$ given by \[b_l([x],[y]):=p^{l-1}\lambda(x,y)\in A/\p\]where $x,y\in K_l$ are choices of representative of $[x],[y]\in J_l/pJ_{l+1}$ and $A/\p\subset S^{-1}A/A$ via $x\mapsto x/p$.
\end{definition}

\begin{lemma}\label{lem:invinv}Suppose $(p)=\p$ is a prime ideal and $T$ has $\p$-only torsion. Then if $\lambda:T\hookrightarrow T^\wedge$ is an injective morphism we have $p=u_p\overline{p}$ for some unit $u_p\in A$ with $\overline{u_p}=u_p^{-1}$ (and hence $\p=\overline{\p}$).
\end{lemma}

\begin{proof}Suppose $p^dT=0$, then $\lambda(\overline{p}^dx,y)=\lambda(x,p^dy)$=0. But then $\overline{p}^dx=0$ for all $x\in T$ as $\lambda$ is an injective morphism. $T$ has $\p$-only torsion and so $p$ divides $\overline{p}^d$. As $\overline{p}$ is a prime, we have $p=u\overline{p}$ for $u$ a unit.
\end{proof}

\begin{claim}\label{clm:auxwelldef}If $A$ is a PID with prime ideal $(p)=\p=\overline{\p}$ and $(T,\lambda)$ is a (non-singular) $\eps$-symmetric linking form over $(A,\p^\infty)$, then for each $l>0$ the $l$th auxiliary form $(\Delta_l(T),b_l(\lambda))$ is a well-defined (non-singular) $(u_p^l\eps)$-symmetric form over $A/\p$.
\end{claim}

We take this opportunity to examine the kind of example we will be thinking about in Chapter \ref{chap:laurent}.

\begin{example}\label{lem:andrewripoff}Let $\FF$ be a field with involution and our principal ideal domain is the Laurent polynomial ring $A=\FF[z,z^{-1}]$ with involution extended linearly from $\overline{z}=z^{-1}$. Monic polynomials $p(z)\in\FF[z,z^{-1}]$ are of the type $p=u_p\overline{p}$ if and only if $p(z)=(z-a)$ where $\overline{a}a=1$ and $u_p=-az$. For any such $a$ and corresponding $p$, there is a linking form $(A/p^lA,\lambda)$ with \[\lambda(x,y)= x\overline{y}/(z-a)^l\in \FF(z)/\FF[z,z^{-1}].\]This is easily checked to be $(u_p)^l$-symmetric. The corresponding auxilliary form $(\Delta_l(A/p^lA),b_l(\lambda))$ is the standard inner product on the field $A/(z-a)A$ given by the symmetric pairing $(x,y)\mapsto x\overline{y}\in A/pA$.
\end{example}

In order to prove Claim \ref{clm:auxwelldef} we will need to some more terminology and a few little lemmas.

\begin{definition}Suppose $j:L\hookrightarrow T$ is the inclusion of a submodule, then the submodule \[L^\perp:=\ker(j^\wedge\lambda:T\to L^\wedge)\subset T\]is called the \textit{orthogonal of $j(L)$} with respect to $(T,\lambda)$.
\end{definition}

\begin{lemma}Suppose $(T,\lambda)$ is a non-singular $\eps$-symmetric linking form and $j:L\hookrightarrow T$ is the inclusion of a submodule, $i:L^\perp\hookrightarrow T$ the inclusion of the orthogonal. If the restricted form $(L,\lambda|_{L})$ is a non-singular linking form then $(L^\perp,\lambda|_{L^\perp})$ is a non-singular linking form and there is an isomorphism of linking forms\[(j\,\,i):(L,\lambda|_{L})\oplus(L^\perp,\lambda|_{L^\perp})\xrightarrow{\cong}(T,\lambda).\]\end{lemma}

\begin{proof}Easy verification.\end{proof}

\begin{lemma}\label{lem:extend}Suppose $A$ is a PID, $T$ is in $\H(A,\p^\infty)$ and $j:L\hookrightarrow T$ is the inclusion of a submodule. Then $j^\wedge:T^\wedge \to L^\wedge$ is surjective.
\end{lemma}

\begin{proof}(Cf.\ \cite[Lemma 19.2]{MR564114}) Let $f\in L^\wedge$ and set $L(0)=L$, $L(m)=L(m-1)+K_m(T)$ for $m>1$. We will show how to lift $f$ to $T^\wedge$ by induction on $m$. Suppose $f\in L(m)^\wedge$. Observe that as $K_{m+1}/K_m$ is a vector space over $A/\p$, so is $L(m+1)/L(m)$. Hence there is some $A$-module $J$ such that there is a vector space isomorphism $g:J/pJ\cong L(m+1)/L(m)$. We have a lot of choice in $J$, and we may take $J$ to be a free $(A/p^{m+1}A)$-module with the right number of summands. Therefore we have $L(m+1)=L(m)+J$ and $J\cap L(m)=\ker(p^m:J\to J)=pJ$. But we can always extend a homomorphism $f|_{pJ}\in (pJ)^\wedge$ to a homomorphism on the whole of $J$. Hence we can lift $f\in L(m)^\wedge$ to $\tilde{f}\in L(m+1)^\wedge$.
\end{proof}

\begin{lemma}\label{lem:perp}If $A$ is a PID and $(T,\lambda)$ is a non-singular $\eps$-symmetric linking form over $(A,\p^\infty)$ then for each $m>0$:\begin{enumerate}[(i)]
\item $K_m(T)^\perp=p^mT$,
\item $(p^mT)^\perp=K_m(T)$,
\item $(p^mK_{m+1}(T))^\perp=K_m(T)+pT$.
\end{enumerate}
\end{lemma}

\begin{proof}For (i), it is clear that $p^mT\subseteq K_m^\perp$. For the reverse inclusion, consider that by definition $K_m=\ker(p^m:T\to T)$ so we have $p^m:T/K_m\xrightarrow{\cong}p^mT$. Let $x\in K_m^\perp$. From this there is a well-defined morphism $\mu_x:p^mT\to (\p^\infty)^{-1}A/A$ defined by $\lambda(x,y)=\mu_x(p^my)$. But by Lemma \ref{lem:extend}, $\mu_x$ lifts to $\tilde{\mu}_x\in T^\wedge$. But if $z:=\lambda^{-1}(\tilde{\mu}_x)$, then for any $y\in T$\[\lambda(x,y)=\mu_x(p^my)=\lambda(z,p^my)=\lambda(\overline{p}^mz,y).\]But $\lambda$ is injective, hence $x=\overline{p}^mz\in p^mT$ by Lemma \ref{lem:invinv}.

For (ii), it is clear that $K_m\subseteq (p^mT)^\perp$. For the reverse inclusion, suppose $x\in (p^mT)^\perp$. Then $0=\lambda(x,p^my)=\lambda(\overline{p}^mx,y)$ for all $y\in T$. But $\lambda$ is injective, so $x\in\ker(\overline{p}^m:T\to T)=K_m$ by Lemma \ref{lem:invinv}.

For (iii), it is a standard result (or one may easily confirm) that for general submodules $K,L\hookrightarrow T$, we have $(K\cap L)^\perp=K^\perp+ L^\perp$. But as $p^mK_{m+1}=p^mT\cap K_1$ the result follows from i. and ii.
\end{proof}

We can finally prove Claim \ref{clm:auxwelldef}.

\begin{proof}[Proof (of Claim \ref{clm:auxwelldef})]Recall the isomorphism $\Delta_l(T)\cong p^lK_{l+1}/p^{l+1}K_{l+2}$ of equation \ref{eq:auxiliary}. Using this isomorphism, the pairing $b_l$ is given for $x,y\in K_{l+1}$ by $b_l(p^lx,p^ly)=p^l\lambda(x,y)=\lambda(p^lx,y)$. In other words $b_l:\Delta_l\to \Delta_l^*=\Hom_{A/\p}(\Delta_l,A/\p)$ is injective if and only if the kernel of\[\lambda:p^lK_{l+1}\to K_{l+1}^\wedge\]is a submodule of $p^{l+1}K_{l+2}$. But if $x\in p^lK_{l+1}$ and $\lambda(x)=0\in K_{l+1}^\wedge$, then $x\in K_{l+1}^\perp=p^{l+1}T$ by Lemma \ref{lem:perp}.

For the symmetry claim, simply observe that for $x,y\in K_{l+1}$ \[\begin{array}{rcl}b_l(p^lx,p^ly)&=
&p^l\lambda(x,y)\\&=&p^{l}\eps\overline{\lambda(y,x)}\\&=&u_p^l\eps\overline{p^l\lambda(y,x)}\\&=&u_p^l\eps\overline{b_l(p^ly,p^lx)}.\end{array}\]

For non-singularity we must show $b_l$ is an isomorphism whenever $\lambda$ is. Without loss of generality we may assume $T=T_l$. But then there is a natural $A$-module isomorphism $\Hom_A(T,S^{-1}A/A)\cong\Hom_{A/p^lA}(T,A/p^lA)=:T^*$ so we consider $\lambda$ to be the isomorphism of $A/p^lA$-modules $T\cong T^*$. This induces an isomorphism of $A/pA$-modules $T/pT\cong T^*/pT^*$. Moreover, there is a natural isomorphism $p^{l-1}:T^*/pT^*\cong(T/pT)^*$. As $b_l([x]):=([y]\mapsto p^{l-1}\lambda(x)(y))$, the result follows.
\end{proof}

The main aim of this subsection is to prove the following theorem:

\begin{theorem}[Main Decomposition Theorem]\label{thm:levinePID}Suppose $A$ is a PID, $\p=(p)$, and $A/\p$ contains a half-unit $s$. If $(T,\lambda), (T',\lambda')$ are non-singular $\eps$-symmetric linking forms over $(A,\p^\infty)$, then there is an isomorphism of linking forms $(T,\lambda)\cong(T',\lambda')$ if and only if there is an isomorphism of $A$-modules $T\cong T'$ that induces an isomorphism of $(u_p^l\eps)$-symmetric forms $(\Delta_l(T),b_l(\lambda))\cong(\Delta_l(T'),b_l(\lambda'))$ over $A/\p$ for each $l$.
\end{theorem}

This is actually a corollary of the much stronger theorem of Levine:

\begin{theorem}[{\cite[Theorem 20.1]{MR564114}}]\label{thm:levine} Suppose $A$ is a UFD, $\p=(p)$, $p=\overline{p}$ and $A/\p$ is a Dedekind domain that contains a half-unit $s$. If $(T,\lambda)$, $(T',\lambda')$ are non-singular $\eps$-symmetric linking forms over $(A,\p^\infty)$, then there is an isomorphism of linking forms $(T,\lambda)\cong(T',\lambda')$ if and only if there is an isomorphism of $A$-modules $T\cong T'$ that induces an isomorphism of forms $(\Delta_l(T),b_l(\lambda))\cong(\Delta_l(T'),b_l(\lambda'))$ for each $l$.
\end{theorem}

\begin{remark}In fact, Levine shows how to prove Theorem \ref{thm:levine} in the additional case that $A=\Z[z,z^{-1}]$ and $p\in A$ is any prime such that $p(1)=\pm 1$.\end{remark}

The general proof for Theorem \ref{thm:levine} given in \cite{MR564114} is very long and there are many small details to consider - especially relating to the concepts required to handle the distinction between $\p$-primary and $\p$-only torsion modules. We will only use the PID version in this thesis so we have made simplifications to Levine's proof where possible. Even our proof is not short! We begin by rephrasing the hypothesis.

\begin{claim}\label{clm:rephrase}An isomorphism $f:T\cong T'$ induces an isomorphism of the auxiliary $l$-form if and only if\[(f^\wedge\lambda' f-\lambda)=0:p^lK_{l+1}\to (K_{l+1})^\wedge.\]
\end{claim}

\begin{proof}Recall the isomorphism $\Delta_l(T)\cong p^lK_{l+1}/p^{l+1}K_{l+2}$ of Equation \ref{eq:auxiliary}. Using this isomorphism, the pairing $b_l$ is given for $x,y\in K_{l+1}$ by $b_l(p^lx,p^ly)=p^l\lambda(x,y)=\lambda(p^lx,y)$. The result clearly follows.
\end{proof}

We are now in a position to prove Theorem \ref{thm:levinePID}.

\begin{proof}[Proof (of Theorem \ref{thm:levinePID})]One direction of the statement is clear as the auxiliary forms are naturally defined.

For the reverse statement, we will rephrase the hypothesis using Claim \ref{clm:rephrase}. Our hypothesis (denoted H$(l)$) is now that there is an isomorphism $f:T\cong T'$ such that for each $l>0$ we have \[(f^\wedge\lambda'f-\lambda)=0:p^lK_{l+1}\to (K_{l+1})^\wedge \tag{$\text{H}(l)$}\label{hypothesis}.\]As $T$ is $\p$-torsion, there is a least $d>0$ such that $K_d=T$, so that H$(d-1)$ gives the morphism\[(f^\wedge\lambda'f-\lambda)=0:p^{d-1}K_{d}\to (K_{d})^\wedge=T^\wedge\]Define $\mu=f^\wedge\lambda'f-\lambda:T\to T^\wedge$ and note that $\mu=\eps\mu^\wedge$. We will improve $f$ by 2 separate inductions that successively increase the submodule on which $\mu$ vanishes until the theorem is proved: \[\begin{array}{rl}\left.
{\begin{array}{rcl}
p^{d-1}K_d&\to&T^\wedge,\\
p^{d-2}K_{d-1}&\to&T^\wedge,\\
&\vdots&\\
p^0K_1&\to&T^\wedge,
\end{array}}
\right\}&\text{Induction (1)}\\

\left.
\begin{array}{rcl}
K_1&\to&T^\wedge,\\
K_2&\to&T^\wedge,\\
&\vdots&\\
K_d=T&\to&T^\wedge.
\end{array}\right\}&\text{Induction (2)}
\end{array}\]

As induction (2) is easier and contains the main idea, we begin here. Suppose we have the base case that $\mu$ restricted to $K_1$ vanishes. Suppose, for the induction, that $\mu$ vanishes on $K_m$ for some $m<d$. We use Claim \ref{clm:halfunit} to improve our given half-unit to an element $s\in A$ such that the class of $s$ is a half-unit in $A/\p^d$. Define a morphism\[h=s\lambda^{-1}\mu:T\to T.\]Now $h(T)\subseteq (K_m)^\perp$ (orthogonal with respect to $\lambda$) by construction. But $(K_m)^\perp=p^mT$ by Lemma \ref{lem:perp}, and by inspection $h^j(T)\subseteq p^{jm}T$ so $h$ is nilpotent. We use the isomorphism $(1+h):T\to T$ to improve $f$ to $f':=f(1+h)^{-1}$. This improves $\mu$ to $\mu':=(f')^\wedge\lambda' f'-\lambda$. We must check that $\mu'$ vanishes on $K_{m+1}$. First calculate that \[\begin{array}{rcl}
(1+h)^\wedge\lambda(1+h)&=&\lambda+h^\wedge\lambda+\lambda h+h^\wedge\lambda h\\
&=&\lambda+\eps(\lambda^\wedge h)^\wedge+s\mu+h^\wedge s\mu\\
&=&\lambda+\eps(s\mu)^\wedge+s\mu+sh^\wedge\mu\\
&=&\lambda+\mu+s\mu+sh^\wedge\mu\\
&=&f^\wedge\lambda' f+sh^\wedge\mu.
\end{array}\]So we need only check that $h^\wedge\mu$ vanishes on $K_{m+1}$. Writing the inclusion as $j_m:K_m\hookrightarrow T$, we can dualise $h^\wedge\mu j_{m+1}$ and alternatively check the statement that $j_{m+1}^\wedge\mu h$ vanishes on $T$. $h(T)\subset p^mT$, so let $y\in T$ and $j^\wedge_{m+1}\mu(p^my)=(j_{m+1}^\wedge p)\mu(p^{m-1}y)$. So we consider the vanishing of $j_{m+1}^\wedge p\mu$, or dually $\mu\overline{p}j_{m+1}$. But, using that $p=u\overline{p}$, we note that $pj_{m+1}(K_{m+1})\subseteq j_m(K_m)$, and we know that $\mu$ vanishes on $K_{m}$ already. So we are done.

\medskip

For induction (1) we will modify the argument from induction (2). We proceed as before, noting along the way that our submodule for induction is now $p^mK_{m+1}$ so our analogous $h$ has $h(T)\subseteq (p^mK_{m+1})^\perp=K_m+pT$ by Lemma \ref{lem:perp}. We will now need to use hypothesis H$(l)$ at various stages, which means we need to carry it with us at each stage of the induction. We will assume for now that H$(l)$ holds for our $f:T\to T'$ and prove at the end that H$(l)$ also holds after we improve $f$.

The first thing to check is that $h$ is nilpotent. Firstly, for each $i\geq0$ we have $h(K_{i+1})\subseteq (p^iK_{i+1})^\perp=K_i+pT$. This follows from the definitions and application of hypothesis H$(i)$, or rather its dual. This tells us that repeated application of $h$ gives\[h^{d-m}(K_m)\subseteq K_{0}+pT=pT.\]Hence $h^{d(d-m)}=0$ and so that $(1+h)$ is an isomorphism. Define $f'=f(1+h)^{-1}$. 

Next we must check that $f'$ has the desired property that the corresponding improved $\mu$ vanishes on $p^{m-1}K_{m}$. This check is reduced (by a similar calculation as in induction (2)) to checking the statement that $\mu h=0$ on $p^{m-1}K_m$. It is sufficient to check that $\mu(x)(y)=0$ for $x\in K_m+pT$, $y\in p^{m-1}K_m$. If $x\in pT$ then this is true by sesquilinearlity and if $x\in K_m$ then the result follows from use our hypothesis H$(m)$.

Finally we must check that $f'$ has property H$(l)$ whenever $f$ does. In other words we need that for each $l>0$ \begin{eqnarray*}0&=&\lambda((1+h)x,(1+h)y)-\lambda'(f(x),f(y))\\
&=&\lambda(hx,hy)+\lambda(hx,y)+\lambda(x,hy)\end{eqnarray*}whenever $x\in p^lK_{l+1}$, $y\in K_{l+1}$. But the latter two summands vanish as it was already noted that $h(K_{l+1})\subseteq (p^iK_{l+1})^\perp$. We must show that $\mu(x)(h(y))=0$. Write $x=p^lw$ with $w\in K_{l+1}$. But again, either $h(y)\in K_{l}$, in which case $\mu(p^lz)(h(y))=0$, or $h(y)=pz$ for some $z\in T$, in which case $\mu(p^lw)(pz)=\mu(p^{l+1}w)(z)=0$.
\end{proof}

\subsection*{Decomposition in a principal ideal domain}

Assume for the rest of this section that $A$ is a PID.

\medskip

We now consider our previous decomposition of a $\p$-primary $T$ into its $T_l$-summands, but with the extra structure of a non-singular linking form $(T,\lambda)$. Unfortunately the decomposition \textit{does not} extend naturally to a decomposition of the form. But here is a way to do it non-naturally:

\begin{proposition}\label{prop:nonnatural}Let $A$ be a local ring with maximal ideal $\p=\overline{\p}$. Suppose $(T,\lambda)$ is an $\eps$-symmetric linking form over $(A,\p^\infty)$, then there exist $\eps$-symmetric linking forms $(T_l,\lambda_l)$ over $(A,\p^\infty)$ such that there is a (non-natural) isomorphism\[(T,\lambda)\cong\bigoplus_{l=1}^\infty(T_l,\lambda_l).\]
\end{proposition}

\begin{proof} Choose a generator $p\in A$ of $\p$ and let $d$ be the largest integer such that $p^{d-1} T\neq 0$. Write $j:T_d\hookrightarrow T$ for the inclusion. We claim that the restriction $\lambda_d:=j^\wedge\lambda j$ of $\lambda$ to $T_d$ defines a non-singular linking form $(T_d,\lambda_{d})$ and that there is a (non-natural) choice of isomorphism $(T, \lambda) \cong(T_d, \lambda_d)\oplus(T',\lambda_{T'})$ where $T':=T/T_d$. Note that if $d'$ is the largest integer such that $p^{d'-1}T'=0$ then $d'<d$ so the Lemma is proved by iteration.

There is an $A$-module isomorphism $T_d^\wedge\cong T^*_d:=\Hom_{A/p^dA}(T_d,A/p^dA)$ so that both $T_d$ and $T_d^\wedge\cong T^*_d$ have the structure of free $(A/p^dA)$-modules. $\lambda_d\in\Hom_A(T_d,T_d^\wedge)\cong\Hom_{A/p^dA}(T_d,T_d^*)$ is an $A$-module isomorphism if and only if it is an $(A/p^dA)$-module isomorphism. An element $a \in A/p^dA$ is a unit if and only if $[a] \in A/pA$ is a unit, hence the $(A/p^dA)$-module morphism $\lambda_d$ is an isomorphism if and only if the $(A/pA)$-module morphism \[[\lambda_d]:T_d/pT_d \to T_d^*/pT_d^*\] is an isomorphism. But as $d$ is the largest integer for which $p^{d-1}T \neq 0$, we have $T_d/pT_d = T/pT$, so in fact \[[\lambda_d] = [\lambda] : T/pT \to T^\wedge/pT^\wedge.\] As $\lambda:T\to T^\wedge$ is an $A$-module isomorphism, so $[\lambda]$ is an $(A/pA)$-module isomorphism. Therefore, tracing the line of reasoning backwards, we see that $(T_d,\lambda_d)$ is non-singular.

As $j$ was originally a split injection we make a choice of splitting so that $T'=T/T_d$ and $T\cong T_d\oplus T'$. Restricting $\lambda$ to $T'$ via this isomorphism defines a non-singular linking form $(T',\lambda')$ as claimed.
\end{proof}

The module $T_l$ in the decomposition $(T,\lambda)\cong\bigoplus_{l=1}^\infty(T_l,\lambda_l)$ is uniquely determined up to isomorphism by the isomorphism class of $T$. However a \textit{choice of splitting of the form} was made when we wrote $(T, \lambda) \cong(T_d, \lambda_d)\oplus(T', \lambda')$. 

In general, it is not even true that the monoid $\NN^\epsilon(A,\p^\infty)$ decomposes accordingly - the forms $(T_l,\lambda_l)$ are not uniquely determined up to isomorphism. For instance within $\NN^\epsilon(\Z,\Z\sm\{0\})$, the localisation away from the prime 2 has forms which interrelate among the different values of $l$.

What can be salvaged? As the Main Decomposition Theorem \ref{thm:levinePID} suggests, the only problems occur when there is no half-unit.

\begin{theorem}\label{thm:MDT2}Let $A$ be a local ring with maximal ideal $(p)=\p=\overline{\p}$ such that $A/\p$ contains a half-unit. Define $\NN^\eps_{\p,l}\subset \NN^\eps(A,\p)$ to be the submonoid consisting of those linking forms $(T,\lambda)$ with $T=T_l$. Then for each $l>0$ there is an isomorphism of commutative monoids\[\NN^\eps_{\p,l}\cong \NN^{v_p}(A/\p)\qquad v_p:=\left\{\begin{array}{lll}\eps&&\text{$l$ even,}\\u_p\eps&&\text{$l$ odd.}\end{array}\right.\]Furthermore, there is a natural isomorphism of commutative monoids\[\NN^\eps(A,\p^\infty)\cong\bigoplus_{l=1}^\infty\NN^{v_p}(A/\p).\]\end{theorem}

\begin{proof}Suppose $T=T_l$, $T'=T'_l$ for some $l>0$. Then non-singular linking forms $(T,\lambda)$, $(T',\lambda')$ over $(A,\p^\infty)$ have naturally defined auxiliary forms $(\Delta_l,b_l)$, $(\Delta'_l,b'_l)$ (respectively) over $A/\p$. Choose a basis $e_1,\dots,e_n$ of $\Delta_l$ and $e_1',\dots,e_m'$ of $\Delta_l'$, so that a morphism $f:\Delta_l\to\Delta_l'$ is a matrix $(f_{ij})$ with respect to these bases. We may lift these to bases $[e_i]$ and $[e_i']$ of $T$ and $T'$ respectively (regarded as free $A/p^lA$-modules). This allows us to define an $A/p^lA$-module morphism $[f]:T\to T'$ by $[f]([e_i])=\sum_{j}f_{ij}[e_j]$. If $m=n$, then $f$ is an isomorphism if and only if the determinant of $f$ is a unit in $A/\p$. But if $u\in A$, then $u\in A/\p$ is a unit if and only if $u\in A/\p^k$ is a unit. Hence if there is an isomorphism $f:(\Delta_l,b_l)\cong(\Delta_l',b_l')$ then it is induced by an isomorphism $T\cong T'$. But this argument, taken together with the Main Decomposition Theorem, shows there is then a well-defined injective morphism of monoids \begin{equation}\label{eq:monoidmorph}\NN^\eps_{\p,l}\to \NN^{u_p^l\eps}(A/\p);\qquad(T,\lambda)\mapsto (\Delta_l(T),b_l(\lambda)).\end{equation}To show it is surjective, take any non-singular $\eps$-symmetric form $(\Delta,b)$ over $A/\p$ with a choice of basis $e_1,\dots, e_n$ of $\Delta$. Using this basis, we define a free $A/\p^l$-module $T$ and an $A/\p^l$-module morphism $\lambda:T\to T^*=\Hom_{A/\p^l}(T,A/\p^l)$ by $\lambda(x,y):=p^{l-1}b(x,y)$. Then $(\Delta_l(T),b_l(\lambda))=(\Delta,b)$ as required. Therefore the morphism (\ref{eq:monoidmorph}) is an isomorphism of monoids.

The Main Decomposition Theorem gives a well-defined, natural morphism of commutative monoids\[\NN^\eps(A,\p^\infty)\to\bigoplus_{l=1}^\infty\NN^{u_p^l\eps}(A/\p);\qquad (T,\lambda)\mapsto \bigoplus_{l=1}^\infty (\Delta_l(T),b_l(\lambda)).\]By Proposition \ref{prop:nonnatural}, for any non-singular $\eps$-symmetric linking form $(T,\lambda)$ over $(A,\p^\infty)$ we can make a choice of decomposition $(T,\lambda)\cong\bigoplus_{l=1}^\infty(T_l,\lambda_l)$, with $(T_l,\lambda_l)$ in $\NN^\eps_{\p,l}$. By the Main Decomposition Theorem, this isomorphism induces an isomorphism $(\Delta_l(T),b_l(\lambda))\cong(\Delta_l(T_l),b_l(\lambda_l))$ for each $l>0$. We showed earlier how to make a choice of lift of elements in $\NN^{u_p\eps}(A/\p)$ to elements in $\NN^\eps_{\p,l}\subset \NN^\eps(A,\p^\infty)$ and our choice of decomposition $(T,\lambda)\cong\oplus_{l=1}^\infty(T_l,\lambda_l)$ shows that $\NN^\eps(A,\p^\infty)$ is generated by such lifted elements.

We only need to show that there are no interrelations between the lifted generators for different values of $l$. But suppose we have a finite sum \[\sum_{l=1}^\infty\sum_k F_{l,k}=\sum_{l=1}^\infty\sum_k F'_{l,k}\in \NN^\eps(A,\p^\infty),\]where $F_{l,k},F'_{l,k}\in\NN^\eps_{\p,l}$ are non-singular linking forms for each $k$. Then writing the auxiliary forms of $F_{l,k}$ and $F'_{l,k}$ as $\Delta_{l,k}$ and $\Delta'_{l,k}$ respectively we must have for each $l>0$\[\sum_k\Delta_{l,k}=\sum_k\Delta'_{l,k}\in \NN^{u_p^l\eps}(A/\p).\]So any relation in $\NN^\eps(A,\p^\infty)$ results in a finite set of relations, indexed by $l$, each relation within a respective copy of $\NN^{u_p^l\eps}(A/\p)$. Conversely, any such finite set of relations on $\NN^\eps(A/\p)$, indexed by $l$, will recover a relation in $\NN^\eps(A,\p^\infty)$ by the Main Decomposition Theorem. Hence there is a relation in $\NN^\eps(A,\p^\infty)$ if and only if it is generated by a finite set of relations within respective $\NN^{u_p\eps}(A/\p)$, independent for each $l>0$. 

We finally have to modify the symmetries to the $v_p$ in the statement of the theorem, but we do this by a standard trick. If $u=\overline{u}^{-1}$ is a unit in a ring $R$ and $(K,\alpha)$ is a $(u^2\eps)$-symmetric form over $R$ then the pairing $\alpha_u(x,y):=\alpha(\overline{u}x,y)=\overline{u}\alpha(x,y)=u\eps\overline{\alpha(y,x)}=\eps\overline{\alpha(\overline{u}y,x)}=\eps\overline{\alpha_u(y,x)}$ is $\eps$-symmetric. The assignment $(K,\alpha)\mapsto (K,\alpha_u)$ defines an isomorphism $\NN^{u^2\eps}(R)\cong \NN^\eps(R)$. Hence we may reduce our symmetries to the $v_p$ claimed. This completes the proof.\end{proof}

Our proof says nothing about what any of the generators of the monoids \textit{are}. We have not computed the monoids involved, hence the conspicuous lack of something like Hensel's Lemma. Here is a silly corollary suggesting that we have bypassed Hensel's Lemma and can actually recover it (to some extent) from Theorem \ref{thm:MDT2}.

\begin{corollary}Let $A=\Z$ and $p\in\Z$ be an odd prime. Fix $l>0$ and suppose we know that modulo $p^k$, multiplication by quadratic residues is transitive on quadratic non-residues for $k=1,l$. Then $a\in\Z$ is a quadratic residue modulo $p$ if and only if it is a quadratic residue modulo $p^l$.
\end{corollary}

\begin{proof}Any non-singular form $(\Delta,b)$ over $\Z/p\Z$ decomposes into a direct sum of linear forms $\Delta(a):=(\Z/p\Z,a/p)$ for some $a=1,2,\dots,p-1$. There is an isomorphism $\Delta(a)\cong\Delta(a')$ if and only if $a\equiv c^2a' \text{ mod }p$ for some $c=1,2,\dots p-1$. Using the assumption that quadratic residues act transitively on quadratic non-residues modulo $p$, there are exactly two isomorphism classes of linear forms, given by $\Delta(1)$ and $\Delta(n)$ where $n$ is a choice of fixed quadratic non-residue modulo $p$.

Our isomorphism $\NN(\Z/\p)\cong\NN_{\p,l}$ tells us that $\NN_{\p,l}$ is generated by choices of lift of $\Delta(1)$ and $\Delta(n)$. Moreover, our analysis of isomorphism classes of linear forms mod $p$ holds mod $p^l$ verbatim. We may choose $\Delta(1)$ to lift to $(\Z/p^l\Z,1/p^l)$, and as $1$ is a quadratic residue mod $p^l$ this forces any choice of lift of $n$ to be a quadratic non-residue mod $p^l$. Using our assumption that quadratic residues act transitively on quadratic non-residues modulo $p^l$, the result follows.
\end{proof}

\section{Split and double Witt groups}\label{sec:witt}

We will now define several ways in which an $\eps$-symmetric form or linking form can be considered trivial, that all involve the idea a lagrangian submodule.

\subsection{Lagrangians}

\begin{definition}A \textit{lagrangian} for a non-singular $\eps$-symmetric form $(P,\theta)$ over $A$ is a submodule $j : L \hookrightarrow P$ in $\A(A)$ such that the sequence\[0\to L\xrightarrow{j}P\xrightarrow{j^*\theta} L^*\to0\]is exact. As modules in the category $\A(A)$ are projective, all surjective morphisms split, and a lagrangian is always a direct summand. If $(P,\theta)$ admits a lagrangian it is called \textit{metabolic}. If $(P,\theta)$ admits two lagrangians $j_\pm:L_\pm\hookrightarrow P$ (labelled ``$+$'' and ``$-$'') such that they are complementary as submodules\[\left(\begin{matrix}j_+\\j_-\end{matrix}\right):L_+\oplus L_-\xrightarrow{\cong} P,\]then the form is called \textit{hyperbolic}.
\end{definition}

\begin{remark}If $j:L\hookrightarrow P$ is a submodule, then $L$ is a lagrangian if and only if $L=L^\perp$. In this case the form $\theta$ vanishes on $L$ and $L$ is referred to as a `maximally self-annihilating submodule'. The reader might prefer to define a lagrangian in these terms.\end{remark}

\begin{lemma}\label{lem:splitishyp1}If $j:L\hookrightarrow P$ is a lagrangian for a non-singular $\eps$-symmetric form $(P,\theta)$ over $A$ then for any splitting $k:L^*\to P$ of $j^*\theta$, there is an isomorphism of forms\[(j\,\,k):\left(L\oplus L^*,\left(\begin{matrix}0&1\\\eps &k^*\theta k\end{matrix}\right)\right)\xrightarrow{\cong}(P,\theta).\] If $A$ contains a half-unit then we may choose a splitting $k'$ such that there is an isomorphism of forms\[(j\,\, k'):\left(L\oplus L^*,\left(\begin{matrix}0&1\\\eps&0\end{matrix}\right)\right)\xrightarrow{\cong}(P,\theta),\] in particular all metabolic $\eps$-symmetric forms $(P,\theta)$ over $A$ are hyperbolic in this case.
\end{lemma}

\begin{proof}\cite[2.2]{MR560997} The first part is easily checked. Now assume there is a half-unit $s\in A$. To modify the splitting to $k'$, first define $\nu=s\theta$ so that $\theta=\nu+\eps\nu^*$. This allows us to define the $A$-module isomorphism\[g=(j\,\,(k-\overline{\eps}jk^*\nu k)):L\oplus L^*\to P,\]which results in the required isomorphism of forms.
\end{proof}

\begin{definition}Suppose $(A,S)$ defines a localisation. A \textit{(split) lagrangian} for a non-singular $\eps$-symmetric linking form $(T,\lambda)$ over $(A,S)$ is a submodule $j : L \hookrightarrow T$ in $\H(A,S)$ such that the sequence\[0\to L\xrightarrow{j}T\xrightarrow{j^\wedge\lambda} L^\wedge\to0\]is (split) exact. If $(T,\lambda)$ admits a (split) lagrangian it is called \textit{(split) metabolic}. If $(T,\lambda)$ admits two lagrangians $j_\pm:L_\pm\hookrightarrow T$ such that they are complementary as submodules\[\left(\begin{matrix}j_+\\j_-\end{matrix}\right):L_+\oplus L_-\xrightarrow{\cong} T,\]then the form is called \textit{hyperbolic}.
\end{definition}

\begin{lemma}\label{lem:splitishyp2}If $A$ contains a half-unit then all split metabolic $\eps$-symmetric linking forms $(T,\lambda)$ over $(A,S)$ are hyperbolic.
\end{lemma}

\begin{proof}The proof of Lemma \ref{lem:splitishyp1} carries through exactly the same, using torsion duals.
\end{proof}

\begin{remark}The terms `lagrangian', `hyperbolic' etc. are often used to mean slightly different things in the literature. What we have called a lagrangian of a linking form is often known as a `metaboliser' and the term `lagrangian' is sometimes reserved for what we have called a split lagrangian. A split lagrangian is also classically called a `maximal isotropic subspace'. The literature does not usually focus on linking forms and our terminology was selected to better tackle this more subtle situation.

In general for a \textbf{linking form} there is a hierarchy:\[\text{hyperbolic}\,\,\subsetneq\,\,\text{split metabolic}\,\,\subsetneq\,\,\text{metabolic}.\]And for a \textbf{form}\[\text{hyperbolic}\,\,\subsetneq\,\,\text{split metabolic}\,\,=\,\,\text{metabolic}.\]Moreover, Lemmas \ref{lem:splitishyp1} and \ref{lem:splitishyp2} show that the presence of a half-unit destroys the distinction between split metabolic and hyperbolic for both forms and linking forms. More generally, if the form or linking form admits a `quadratic extension' \cite{MR620795}, this distinction is destroyed.
\end{remark}

\begin{example}

Some examples to illustrate the differences between the types. \begin{enumerate}[i.]
\item A symmetric form over $\Z$ that is metabolic but not hyperbolic: \[(P,\theta)=\left(\Z\oplus\Z,\left(\begin{array}{cc}0&1\\1&1\end{array}\right)\right)\] with lagrangian $\Z\oplus 0$. The property of taking only even values is preserved under isomorphisms of forms over $\Z$ (cf. `Type I' inner product spaces \cite[4.2]{MR0506372}), so as $\theta((x,y),(x,y))=2xy+y^2$, we can deduce this is not isomorphic to the only possible hyperbolic form here, which is $\left(\Z\oplus \Z,\lmat0&a\\b&0\rmat\right)$ for $a,b\in\{\pm1\}$. 
\item A symmetric linking form over $(\Z,\Z\sm\{0\})$ that is metabolic but not split metabolic: $(T,\lambda)=(\Z/4\Z,(a,b)\mapsto ab/4)$ with lagrangian $\Z/2\Z$. This is the only possible lagrangian and it is not a direct summand.
\item To see a symmetric linking form over $(\Z,\Z\sm\{0\})$ that is split metabolic but not hyperbolic we refer the reader to Banagl-Ranicki \cite{MR2189218}.
\end{enumerate}
\end{example}

\subsection{Witt groups: single or double?}

We will make much use of the following basic construction.

\begin{definition}[Monoid construction]\label{def:monoid} Let $(M, +)$ be an commutative monoid and let $N$ be a submonoid of $M$ (i.e. $0 \in N$ and $N + N \subset N$). Consider the equivalence relation: for $m_1, m_2 \in M$, define $m_1 \sim m_2$ if there exists $n_1, n_2 \in N$ such that $m_1 + n_1 = m_2 + n_2$. Then the set of equivalence classes $M/\sim$ inherits a structure of abelian monoid via $[m] + [m'] := [m + m']$. It is denoted by $M/N$. Assume that for any element $m \in M$ there is an element $m' \in M$ such that $m + m' \in N$, then $M/N$ is an abelian group with $-[m] = [m']$. \textit{It is then canonically isomorphic to the quotient of the Grothendieck group of $M$ by the subgroup generated by $N$.}
\end{definition}

\begin{remark}We wish to emphasise that even if the monoid construction $M/N$ is a group, the group does not necessarily have transitivity, in the sense that $[m]=0\in M/N$ does not imply $m\in N$ in general. We will often prove this as an additional property, but it does not come for free!
\end{remark}

\begin{lemma}If $(P,\theta)$ is an $\eps$-symmetric form over $A$ then the form \[(P,\theta)+(P,-\theta)=(P\oplus P,\theta\oplus-\theta)\]is split metabolic.
\end{lemma}
\begin{proof}The diagonal $\lmat 1\\1\rmat:P\to P\oplus P$ is a lagrangian with splitting $(1\,\,0) :P\oplus P\to P$.
\end{proof}

\begin{lemma}\label{lem:halfunithyp}If $s\in A$ is a half-unit and $(P,\theta)$ is an $\eps$-symmetric form over $A$ then \[(P,\theta)+(P,-\theta)=(P\oplus P,\theta\oplus-\theta)\]is hyperbolic.
\end{lemma}

\begin{proof}$(P\oplus P, \theta\oplus -\theta)$ has lagrangians \[\begin{array}{lccrcl}(&1&1&)&:&(P,0)\to (P\oplus P,\theta\oplus -\theta),\\(&\bar{s}&-s&)&:&(P,0)\to (P\oplus P,\theta\oplus -\theta).\end{array}\]They are complementary as \[\left(\begin{matrix}s & 1\\\bar{s} & -1\end{matrix}\right)\left(\begin{matrix}1&1\\\bar{s}& -s\end{matrix}\right)=\left(\begin{matrix}1&1\\\bar{s}& -s\end{matrix}\right)\left(\begin{matrix}s & 1\\\bar{s} & -1\end{matrix}\right)=I.\]\end{proof}

The preceding lemmas hold completely analogously for linking forms. 

\medskip

This justifies the following definitions:

\begin{definition}Suppose $(A,S)$ defines a localisation. The monoid constructions\[\begin{array}{rcl}
W^\eps(A)&=&\NN^\eps(A)/\{\text{metabolic forms}\}\\
W^\eps(A,S)&=&\NN^\eps(A,S)/\{\text{metabolic linking forms}\}
\end{array}\]are abelian groups called the \textit{$\eps$-symmetric Witt group} of $A$ and of $(A,S)$ respectively.

The monoid construction\[\begin{array}{rcl}
SW^\eps(A,S)&=&\NN^\eps(A,S)/\{\text{split metabolic linking forms}\}
\end{array}\]is an abelian group called the \textit{$\eps$-symmetric split Witt group} of $(A,S)$.

If $A$ contains a half-unit then the monoid constructions\[\begin{array}{rcl}
DW^\eps(A)&=&\NN^\eps(A)/\{\text{hyperbolic forms}\}\\
DW^\eps(A,S)&=&\NN^\eps(A,S)/\{\text{hyperbolic linking forms}\}
\end{array}\] are abelian groups called the \textit{$\eps$-symmetric double Witt group} of $A$ and of $(A,S)$ respectively.
\end{definition}

\begin{remark}In the case that $A$ is a local ring with maximal ideal $\p$, the double Witt group $DW^\eps(A,\p^\infty)$ is still a group if we replace the hypothesis that there is a half-unit in $A$ with the hypothesis that there is a half-unit in the residue class field $A/\p$.
\end{remark}

We would, of course, like to know that these similar looking Witt groups are different to one-another!  To that end, here are some obvious consequences of the definitions:

\begin{claim}\label{clm:DWvsW}There is a surjective forgetful functor\begin{eqnarray*}SW^\eps(A,S)&\twoheadrightarrow& W^\eps(A,S).\end{eqnarray*}If $A$ contains a half-unit, then by \ref{lem:splitishyp1} and \ref{lem:splitishyp2}, the forgetful functors are isomorphisms \begin{eqnarray*}DW^\eps(A)&\xrightarrow{\cong}& W^\eps(A),\\DW^\eps(A,S)&\xrightarrow{\cong}& SW^\eps(A,S).\end{eqnarray*}
\end{claim}

\begin{remark}The first appearance of a split Witt group of a localisation $(A,S)$ is in \cite[3.3]{MR521738}. Stoltzfus defines the `hyperbolic algebraic concordance groups of $\eps$-symmetric isometric structures' over a Dedekind domain $A$, which he denotes $CH^\eps(A)$ when the underlying modules are projective and denotes $CH^\eps(K/A)$ (for $K$ the fraction field of $A$) when the underlying modules are torsion. These groups are defined using a split metabolic definition but Stoltzfus shows \cite[3.2]{MR521738} that the resulting groups are double Witt groups in our sense. Indeed, for the projective case, we show in Chapter \ref{chap:laurent} that Stoltzfus' group $CH^\eps(A)$ is isomorphic to the double Witt group $DW^\eps(A[z,z^{-1}],P)$ where $P=\{p(z)\in A[z,z^{-1}]\,|\,\text{$p(1)\in A$ is a unit}\}$. The essential tool used is the presence of the half-unit $(1-z)^{-1}$ in the localisation $(A[z,z^{-1},(1-z)^{-1}],P)$ (cf.\ Chapter \ref{chap:laurent}) which is a geometric expression of the fact that a knot exterior (cf.\ Chapter \ref{chap:blanchfield}) has the homology of a circle.
\end{remark}

For later applications, especially those in $L$-theory and in knot theory, we will mainly be interested in rings with a half-unit. Hence, assuming $A$ contains a half-unit, we now come to the main question of Chapter \ref{chap:linking}:
\begin{question}\label{q:DWkernel}What is the kernel of the surjective forgetful morphism\[DW^\eps(A,S)\twoheadrightarrow W^\eps(A,S)\text{?}\]\end{question}

\subsection*{Calculating Witt groups of linking forms}

Assume for the rest of this section that $A$ is a Dedekind domain. 

\medskip

An obvious consequence of Proposition \ref{prop:dedekinddecomp} is that when $A$ is a Dedekind domain and $\p\neq\overline{\p}$ are prime ideals, then if $(T,\lambda)$ is a non-singular linking form over $(A,A\sm\{0\})$, the non-singular linking form $(T_\p\oplus T_{\overline{\p}},\lambda|_{T_\p\oplus T_{\overline{\p}}})$ is hyperbolic. The decomposition in Proposition \ref{prop:dedekinddecomp} then becomes a decomposition of Witt groups.

\begin{proposition}\label{prop:dedekinddecomp2}For $A$ a Dedekind domain, there are natural isomorphisms of abelian groups\begin{eqnarray*}W^\eps(A,A\sm\{0\})&\cong&\bigoplus_{\p=\overline{\p}}W^\eps(A,\p^\infty),\\
SW^\eps(A,A\sm\{0\})&\cong&\bigoplus_{\p=\overline{\p}}SW^\eps(A,\p^\infty).\end{eqnarray*}
When we also assume $A$ contains a half-unit, there is a natural isomorphism of abelian groups\[
DW^\eps(A,A\sm\{0\})\cong\bigoplus_{\p=\overline{\p}}DW^\eps(A,\p^\infty).\]
\end{proposition}

In light of this, we will now specialise to $\p$-primary Witt groups for $\p=\overline{\p}$. First we mention a standard technique for manipulating Witt groups of linking forms over $(A,\p^\infty)$ called \textit{devissage}.

\begin{proposition}[Devissage]\label{prop:devissage}Let $A$ be a local ring with involution invariant maximal ideal $\p$. Suppose $(T,\lambda)$ is a non-singular $\eps$-symmetric linking form over $(A,\p^\infty)$ and that $T=T_l$ for some $l>0$, then
\[(T,\lambda)\sim\left\{\begin{array}{lcl}0&&\text{$l$ even}\\ (T',\lambda')&&\text{$l$ odd}\end{array}\right. \quad\in W^\eps(A,\p^\infty),\]where $(T',\lambda')$ is such that $p T'=0$.
\end{proposition}

\begin{proof}Choose a generator $p=\overline{p}$ of $\p$. For $2k\geq l$, define a submodule $L=p^{k} T\subset T$, then $L^\perp = K_k(T)$ by Lemma \ref{lem:perp}. Then we have\[L\subseteq L^\perp\subseteq p^{l-k}T\] and it is standard for linking forms (or easily checked) that $\lambda$ restricts to a non-singular $\eps$-symmetric linking form on the quotient $(L^\perp/L,\lambda)$ when $L\subseteq L^\perp$. There is a lagrangian\[L^\perp\hookrightarrow (T\oplus (L^\perp/L),\lambda\oplus -\lambda);\qquad x\mapsto (x,[x]),\]so that $(T,\lambda)\sim (L/L^\perp,\lambda)\in W^\eps(A,\p^\infty)$. In particular, when $l=2k$ we have $L^\perp/L$=0, and when $l=2k-1$ we have $p(L^\perp/L)\subset p(p^{k-1}T/L)=0$ so the result follows.
\end{proof}

\begin{proposition}\label{prop:residuefield}Let $A$ be a local ring with involution invariant maximal ideal $\p=(p)$ such that $u_p\overline{p}=p$ for some unit $u_p=\overline{u_p}^{-1}$. Then there is an isomorphism of abelian groups\[W^{u_p\eps}(A/\p)\cong W^\eps(A,\p).\]
\end{proposition}

\begin{proof}There is an equivalence of categories between the category of finitely generated $A$-modules $T$ such that $pT=0$ and the category of finite dimensional vector spaces over $A/\p$. Under which, for any f.g.\ $\p$-torsion $A$-module $T$, we observe that there is an isomorphism of $A$-modules \[\Hom_{A/\p}(T,A/\p)\xrightarrow{\cong}\Hom_A(T,\p^{-1}A/A);\qquad f\mapsto (x\mapsto f(x)/p).\]But consider that this interacts with symmetry in the following way. If $(T,\theta)$ is a $(u_p\eps)$-symmetric form over $A/\p$, then the corresponding $(T,\lambda)$ has \[\begin{array}{rcl}\eps\lambda^\wedge(x)(y)&=&\eps\overline{\lambda(y,x)}\\&=&\eps\overline{\theta(y,x)}/\overline{p}\\&=&u_p\eps\theta^*(x,y)/p=\lambda(x)(y).\end{array}\]This induces an isomorphism of commutative monoids $\NN^{u_p\eps}(A/\p)\cong \NN^\eps(A,\p)$. As the submonoids of metabolic forms are preserved under the isomorphism $\NN^{u_p\eps}(A/\p)\cong \NN^\eps(A,\p)$, there is an isomorphism of Witt groups after the respective monoid constructions.
\end{proof}

Combining Propositions \ref{prop:devissage} and \ref{prop:residuefield} we obtain the following theorem first proved in the case that $1/2\in A$ by Karoubi {\cite{MR0384894} and first proved generally by Ranicki \cite[4.2.1]{MR620795}.

\begin{theorem}\label{thm:karoubi}Let $A$ be a Dedekind domain. Then there is an isomorphism of abelian groups\[\sigma^W:W^\eps(A,A\sm\{0\})\xrightarrow{\cong} \bigoplus_{\p=\overline{\p}}W^{u_p\eps}(A/\p).\]The image $\sigma^W(T,\lambda)$ of a non-singular $\eps$-symmetric linking form is called the \textit{Witt multisignature}.
\end{theorem}

\begin{example}There is a decomposition of the Witt group $W^\eps(\Z,\Z\sm\{0\})\cong\bigoplus_{\p=\overline{\p}} W^\eps(\Z,\p^\infty)$ and for each individual summand $W^\eps(\Z,\p^\infty)\cong W^\eps(\Z_\p,\p^\infty)\cong W^\eps(\Z_\p/\p)\cong W^\eps(\Z/\p)$. The computation of these Witt groups is well known (e.g.\ \cite[IV.2]{MR0506372}); $W^-(\Z/p\Z)=0$ for all primes $p\in\Z$ and \[W(\Z/p\Z)\cong\left\{\begin{array}{lcl}\Z/2\Z\oplus\Z/2\Z&&p\equiv 1\,\,\,(\text{mod 4}),\\ \Z/4\Z&&p\equiv 3 \,\,\,(\text{mod 4}),\\
\Z/2\Z&&p=2.\end{array}\right.\]
\end{example}

What about split Witt groups of linkings? The following examples discusses the only computation ever made of these groups prior to this thesis (to the best of our knowledge).

\begin{example}Using the decomposition of $\NN^\eps(\Z,\Z\sm\{0\})$ obtained in Example \ref{ex:KK}, Kawauchi-Kojima \cite[Proposition 5.2]{MR594531} compute the corresponding split Witt group (which they refer to as the `Witt group of linkings').

For $p$ odd, Kawauchi-Kojima use a system of invariants to compute that $SW(\Z,(p)^\infty)=\bigoplus_{l=1}^\infty SW_p^l(\Z)$, where the individual groups $SW_p^l(\Z)$ are the monoids $\NN_p^l$ modulo the split metabolic relation. The calculation of these reduces to the calculation of the single Witt groups as there is no difference between split metabolic and metabolic forms for vector spaces $SW_p^l(\Z)\cong W(\Z/p\Z)$.

For $p = 2$, by imposing the split metabolic relation on the generators of $\NN(\Z,(2)^\infty)$ Kawauchi-Kojima obtain that the group $SW(\Z,(2)^\infty)$ is isomorphic to a direct sum of infinite copies of $\Z/2\Z$, generated by the elements $A^l(1)$ for $l \geq 1$. However, as the generators are interrelated for different $l$ (see Example \ref{ex:KK}), this direct sum is subject to relations. It is computed that\[ SW_p^l(\Z):=\left\{\begin{array}{lll}
\Z/2\Z\oplus\Z/2\Z&\,&p\equiv 1\,\,\text{(mod 4)},\\
\Z/4\Z&\,&p\equiv 3\,\,\text{(mod 4)},\\
\Z/2\Z&\,&p=2, \,l=1,\\
\Z/8\Z&\,&p=2,\, l=2,\\
\Z/8\Z\oplus\Z/2\Z&\,&p=2, \,l\geq 3.\end{array}\right.\]\begin{eqnarray*}
SW(\Z,(p)^\infty)&=&\bigoplus_{l=1}^\infty SW_p^l(\Z),\quad\text{for $p$ an odd prime,}\\
SW(\Z,(2)^\infty)&=&\bigoplus_1^\infty\Z/2\Z,\quad\text{but subject to relations.}
\end{eqnarray*}
\end{example}

The previous example shows that there can be no such thing as devissage for the split Witt groups in general. In fact we show now that there is \textit{never} devissage in the split or double Witt groups over Dedekind domains. In other words, the split and double Witt relations will respect the decomposition of Theorem \ref{thm:MDT2}.

\begin{definition}If $A$ is a Dedekind domain and $(T,\lambda)$ is a non-singular $\eps$-symmetric linking form over $(A,A\sm\{0\})$ then for each involution invariant prime ideal $\p\subset A$, with choice of uniformiser $p\in \p\sm\p^2$, and $l>0$ define the $(\p,l)$-signature of $(T,\lambda)$ to be the Witt class \[\sigma_{\p,l}(T,\lambda):=[(\Delta_l(T_\p),b_l(\lambda_\p))]\in W^{v_p}(A/\p).\]
\end{definition}

\begin{theorem}\label{thm:multisignature}Let $A$ be a Dedekind domain such that for each involution invariant prime ideal $\p\subset A$, there exists a half-unit $s(\p)\in A/\p$. Then there are isomorphisms of abelian groups\[\begin{array}{rcl}SW^\eps(A,A\sm\{0\})&\xrightarrow{\cong}& \bigoplus_{\p=\overline{\p}} SW^\eps(A,\p^\infty)\\
&\xrightarrow{\cong}& \bigoplus_{\p=\overline{\p}}\bigoplus_{l=1}^\infty W^{v_p}(A/\p).\end{array}\]The composite isomorphism is given by the collection of $(\p,l)$-signatures $\sigma_{\p,l}(T,\lambda)$ and is called the \textit{split Witt multisignature} \[\sigma^{SW}:SW^\eps(A,A\sm\{0\})\xrightarrow{\cong}\bigoplus_{\p=\overline{\p}}\bigoplus_{l=1}^\infty W^{v_p}(A/\p).\]
\end{theorem}

\begin{proof}The proof essentially falls out of our decompositions so far, the Main Decomposition Theorem and Theorem \ref{thm:MDT2}. The final isomorphism in the statement of the theorem uses the fact that $DW^{v_p}(A/\p)\cong W^{v_p}(A/\p)$ as $A/\p$ is a field.

The only thing we must still check is that a non-singular linking form $(T,\lambda)$ over $(A,\p^\infty)$ is split metabolic if and only if $(\Delta_l(T),b_l(\lambda))$ is split metabolic for each $l>0$. The proof of this is essentially given in \cite[1.5]{MR1004605}. The proof is as follows.

Clearly if $T\cong L\oplus T/L$ for a lagrangian $L$ then $\Delta_l(T)\cong\Delta_l(L)\oplus\Delta_l(T/L)$ and $\Delta_l(L)$ is a lagrangian for $(\Delta_l(T),b_l(\lambda))$ for all $l>0$.

For the converse, suppose $(T,\lambda)$ is a non-singular linking form over $(A,\p^\infty)$ and $(\Delta_l(T),b_l(\lambda))$ is split metabolic for each $l>0$. Fix $l>0$ and make a choice of lift of $(\Delta_l(T),b_l(\lambda))$ to some $(T_l,\lambda_l)\in \NN^\eps_{\p,l}$ as in the proof of Theorem \ref{thm:MDT2}. As in the proof of Proposition \ref{prop:nonnatural} we prefer to think of this not as a linking form, but equivalently as a form over the ring $A/(p)^l$. We may lift a split lagrangian to a split injection $j:L\hookrightarrow T$ likewise, so that there is an isomorphism of $A/(p)^l$-modules $L\oplus T_l/L\cong T_l$, under which identification we may write the form as \[\left(\begin{matrix}f&g\\\eps g^*&h\end{matrix}\right):L\oplus (T/L)\to L^*\oplus (T/L)^*\cong (L\oplus (T/L))^*\](with $-^*=\Hom_{A/(p)^l}(-,A/(p)^l)$), where $h=\eps h^*$ and $f=\eps f^*$, $f$ is divisible by $p$, and $g$ must be an isomorphism (by non-singularity of $(T_l,\lambda_l)$). We assume without loss of generality that $T/L=L^*$ and that $g=1$ as we may modify the submodule $i:T/L\hookrightarrow T$ to $ig^{-1}:L^*\hookrightarrow T$. 

We use the fact that $f$ is divisible by $p$ as the base case of an induction on $l>k\geq 1$. Assume that $f=p^k\Phi$ for some $A$-module morphism $\Phi:L\to L^*$ and pick $s\in A$ such that $s+\overline{s}=1\in A/(p)^l$ as in Claim \ref{clm:halfunit}. We use this to modify the inclusion of the submodule $j:L\hookrightarrow T$ to the split injection \[\left(j\,\,\, -i(\eps sp^k\Phi)^*\right):L\oplus L\hookrightarrow T.\]This new lagrangian submodule is still complementary to $i:L^*\hookrightarrow T$ and the modification does not affect $h$ in our matrix. However, $f$ is modified to
\begin{eqnarray*}
&&f-(p^ks\Phi+\eps(p^ks\Phi)^*)+(p^{k}s\Phi)h(p^ks\Phi)^*\\
&=&f-(sf+\eps \overline{s}f^*)+p^{k}\overline{p}^ks\overline{s}\Phi h\Phi^*\\
&=&u_p^kp^{2k}s\overline{s}\Phi h\Phi^*,
\end{eqnarray*}
where we have used $\overline{p}^l=u_pp^k$. So the modification of $f$ is divisible by $p^{2k}$. By induction we may assume $f$ is divisible by $p^l$ and $j:L\hookrightarrow T$ is a split lagrangian.
\end{proof}

\begin{corollary}Let $A$ be a Dedekind domain containing a half-unit. Then there are isomorphisms of abelian groups \[\begin{array}{rcl}DW^\eps(A,A\sm\{0\})&\xrightarrow{\cong}& \bigoplus_{\p=\overline{\p}} DW^\eps(A,\p^\infty)\\
&\xrightarrow{\cong}& \bigoplus_{\p=\overline{\p}}\bigoplus_{l=1}^\infty DW^{v_p}(A/\p)\\
&\xrightarrow{\cong}& \bigoplus_{\p=\overline{\p}}\bigoplus_{l=1}^\infty W^{v_p}(A/\p)\end{array}\]The composite isomorphism is given by the collection of $(\p,l)$-signatures $\sigma_{\p,l}(T,\lambda)$ and is called the \textit{double Witt multisignature} \[\sigma^{DW}:DW^\eps(A,A\sm\{0\})\xrightarrow{\cong}\bigoplus_{\p=\overline{\p}}\bigoplus_{l=1}^\infty W^{v_p}(A/\p).\]
\end{corollary}

\begin{proof}As in Theorem \ref{thm:multisignature}, the only thing we must still check is that a non-singular linking form $(T,\lambda)$ over $(A,\p^\infty)$ is hyperbolic if and only if $(\Delta_l(T),b_l(\lambda))$ is hyperbolic for each $l>0$.

Clearly if $T\cong L_+\oplus L_-$ for lagrangians $L_\pm$ then $\Delta_l(T)\cong\Delta_l(L_+)\oplus\Delta_l(L_-)$ and $\Delta_l(L_\pm)$ are lagrangians for $(\Delta_l(T),b_l(\lambda))$ for all $l>0$.

For the converse, use the same method as the proof of Theorem \ref{thm:multisignature} but now lift both split lagrangians to complementary split submodules of $T$. Each split submodule can now be improved to a split lagrangian using the method of \ref{thm:multisignature}. To see that we can modify one submodule without affecting the other, we need only recall that the morphism $h$ in the proof of \ref{thm:multisignature} was unaffected by the modifications made to the submodule $L$.
\end{proof}

\begin{corollary}\label{cor:forget}If $A$ is a Dedekind domain, such that for each involution invariant prime ideal $\p$ there is a half unit $s(\p)\in A/\p$, then the forgetful functor is given by \[SW^\eps(A,A\sm\{0\})\to W^\eps(A,A\sm\{0\});\qquad(T,\lambda)\mapsto \bigoplus_{\p=\overline{\p}}\bigoplus_{l \text{ odd}}\sigma_{\p,l}(T,\lambda).\]If $A$ is a Dedekind domain containing a half unit then the forgetful functor is given by \[DW^\eps(A,A\sm\{0\})\to W^\eps(A,A\sm\{0\});\qquad(T,\lambda)\mapsto \bigoplus_{\p=\overline{\p}}\bigoplus_{l \text{ odd}}\sigma_{\p,l}(T,\lambda).\]
\end{corollary}

\subsection*{Examples}

We now offer some examples of rings and localisations to which we can apply Corollary \ref{cor:forget}.

\begin{example}We start with the example of Wall-Kawauchi-Kojima away from the prime 2. Set $(A,S)=(\Z[\frac{1}{2}],A\sm\{0\})$, $\eps=1$. Then $(T,\lambda)\in\ker(DW(A,S)\to W(A,S))$ if and only if for each prime $p\equiv 1\,\,\text{mod $4$}$ we have \[\sum_{l\,\,\text{odd}}\sigma_{(p),l}=(0,0)\in\Z/2\Z\oplus \Z/2\Z,\]and for each prime $p\equiv 3\,\,\text{mod $4$}$ we have \[\sum_{l\,\,\text{odd}}\sigma_{(p),l}=0\in\Z/4\Z.\]
\end{example}

We now move onto the example of Laurent polynomial rings over fields. All the terminology and results below concerning (single) Witt groups are well-known and borrowed from \cite[39C]{MR1713074}. Only the statements concerning double Witt groups in our examples are original.

Suppose $\FF=\R$ or $\C$ where $\R$ has the trivial involution and $\C$ has the involution given by complex conjugation. The \textit{laurent polynomial ring} $A=\FF[z,z^{-1}]$ has the involution extended linearly from $\FF$ by $\overline{z}=z^{-1}$. Define the set of involution invariant units in a commutative Noetherian ring $R$ to be \[U(R)=\{a\in R\,|\,\overline{a}a=1\}.\]Write the set of \textit{irreducible monic polynomials over $\FF$} as $\mathcal{M}(\FF)$ and define a subset as\[\{p(z)\in \mathcal{M}(\FF)\,|\,\overline{(p)}=(p)\}=:\overline{\mathcal{M}}(\FF)\subseteq\mathcal{M}(\FF).\]In other words, an irreducible monic polynomial $p(z)$ is in $\overline{\mathcal{M}}(\FF)$ if and only if there exists $u_p\in U(\FF[z,z^{-1}])$ such that $u_p\overline{p}=p$. It is well-known that the prime ideals of $\FF[z,z^{-1}]$ are in 1:1 correspondence with the elements of $\mathcal{M}(\FF)$. Define multiplicative subsets \[P:=\{p(z)\in\FF[z,z^{-1}]\,|\,p(1)\neq 0\},\]and \[Q:=\FF[z,z^{-1}]\sm \{0\}=P\cup (z-1)^\infty.\]Then $Q^{-1}\FF[z,z,^{-1}]=\FF(z)$, the fraction field.

Now consider Example \ref{lem:andrewripoff} where $p=(z-a)$ for $a\overline{a}=1$ and note that if we had taken an arbitrary $\eps$-symmetric linking form $(T,\lambda)$ over $(A,p^\infty)$, the $l$th auxiliary form $(\Delta_l(T),b_l(\lambda))$ over the field $A/pA$ is $v_p$-symmetric. The isomorphism \[f:A/pA\xrightarrow{\cong}\FF;\qquad z\mapsto a,\]affects $v_p$ in the following way\[f:\left\{\begin{array}{lll}\eps&&\text{}\\u_p\eps&&\text{}\end{array}\right.\mapsto \qquad\left\{\begin{array}{lll}f(\eps)&&\text{$l$ even,}\\-a^2f(\eps)&&\text{$l$ odd.}\end{array}\right.\]But now by our standard trick from the end of the proof of Theorem \ref{thm:MDT2} we can modify the $-a^2f(\eps)$-symmetry to $-f(\eps)$-symmetry. As single Witt groups are blind to auxiliary forms when $l$ is even, one consequence is the isomorphism:\[W^\eps(\FF[z,z^{-1}],(a-z)^\infty)\cong W^{-f(\eps)}(\FF).\]

\begin{example}\label{ex:C}Note that for any algebraically closed $\FF$, we have a 1:1 correspondence of sets \[U(\FF)\to\overline{\mathcal{M}}(\FF);\qquad a\mapsto (z-a).\]

If $\FF=\C$ then recall (\cite[39.22]{MR1713074}) that for $\eps=\pm1$ there is an isomorphism\[W^\eps(\C)\xrightarrow{\cong}\Z ,\]given by the signature of the hermitian pairing in the case $\eps=1$ and in the case $\eps=-1$ we send the skew-hermitian pairing $\theta(x,y)$ to the hermitian pairing $\theta(x,iy)$ and take the signature of this. By the discussion above and Example \ref{lem:andrewripoff} we have \[\begin{array}{rcl}W^\eps(\C[z,z^{-1}],Q)&\xrightarrow{\cong}&\bigoplus_{a\in S^1}\Z,\\
&&\\
DW^\eps(\C[z,z^{-1}],Q)&\xrightarrow{\cong}&\bigoplus_{a\in S^1}\bigoplus_{l=1}^\infty\Z.\end{array}\]and $(T,\lambda)$ is in the kernel of the forgetful map if and only if for each $a\in S^1$ we have\[\sum_{l\,\,\text{odd}}\sigma_{a,l}(T,\lambda)=0\in\Z.\]
\end{example}

\begin{example}\label{ex:R}If $\FF=\R$ then (by \cite[39.23]{MR1713074}) we have\[\overline{\mathcal{M}}(\R)=\{(z-1)\}\cup\{(z+1)\}\cup\{p_\theta(z)\,|\,0<\theta<\pi\},\]where \[p_\theta(z)=(z- e^{i\theta})(z-e^{-i\theta}).\]The corresponding residue class fields are given by the rings with involution\[\begin{array}{rclr}\R[z,z^{-1}]/(z\pm 1)&\xrightarrow{\cong}&\R;&z\mapsto \mp1,\\
\R[z,z^{-1}]/(p_\theta(z))&\xrightarrow{\cong}&\C;&z\mapsto e^{i\theta}.\end{array}\]Hence we have\[\begin{array}{rcl}W^\eps(\R[z,z]^{-1},Q)&\cong&\left\{\begin{array}{lll}\bigoplus_{0<\theta<\pi}\Z&&\eps=1,\\
\Z\oplus\Z\oplus\bigoplus_{0<\theta<\pi}\Z&&\eps=-1.\end{array}\right.\\
&&\\
DW^\eps(\R[z,z]^{-1},Q)&\cong&\left\{\begin{array}{lll}\bigoplus_{0<\theta<\pi}\bigoplus_{l=1}^\infty\Z&&\eps=1,\\
(\bigoplus_{l=1}^\infty\Z)\oplus(\bigoplus_{l=1}^\infty\Z)\oplus(\bigoplus_{0<\theta<\pi}\bigoplus_{l=1}^\infty\Z)&&\eps=-1.\end{array}\right.\end{array}\]Again, $(T,\lambda)$ is in the kernel of the forgetful map if and only if \[\begin{array}{rcll}\eps=1:&&&\sum_{l\,\,\text{odd}}\sigma_{\theta,l}(T,\lambda)=0\in\Z \qquad\text{for each $0<\theta<\pi$},\\
&&&\\
\eps=-1:&&&\sum_{l\,\,\text{odd}}\sigma_{\pm1,l}(T,\lambda)=0\in\Z,\\
&&\text{and}& \sum_{l\,\,\text{odd}}\sigma_{\theta,l}(T,\lambda)=0\in\Z\qquad\text{for each $0<\theta<\pi$}.\end{array}
\]
\end{example}

\begin{example}\label{ex:Q}The Witt group for $\FF=\Q$ is also well-known. We refer the reader to \cite[39.24]{MR1713074} for details of this group. In particular (and using the terminology of \cite[39.24]{MR1713074}), for each $(z\pm 1)\neq p(z)\in \overline{\mathcal{M}}(\Q)$ there is defined a natural number $t_{p(z)}$, and for each of $\eps=\pm1$ there are integer-valued signature invariants $\sigma^i_{p}$ for $i=1,\dots, t_{p(z)}$. Hence a non-singular $(\pm1)$-symmetric linking form $(T,\lambda)$ over $(\Q[z,z^{-1}],Q)$ determines an element of the kernel of the forgetful map only if for each $(z\pm 1)\neq p(z)\in \overline{\mathcal{M}}(\Q)$ and each $i=1,\dots,t_{p(z)}$\[\sum_{l\,\,\text{odd}}\sigma^i_{p,l}=0.\]
\end{example}

\begin{remark}If we restrict to the multiplicative subset $P$ in Example \ref{ex:R}, we recover the doubly-slice obstructions obtained in \cite[1.7]{MR1004605} coming from the $\R$-coefficient Blanchfield form (see Chapters \ref{chap:blanchfield} and \ref{chap:knots} for definitions). Restricting to the multiplicative subset $P$ in Examples \ref{ex:C} and \ref{ex:Q} we obtain doubly-slice obstructions similarly. As far as we can tell, our obstructions over the coefficient rings $\C$ and $\Q$ are new, although we have not included them as a theorem in Chapter \ref{chap:knots} as they are very similar to those of \cite{MR1004605}.
\end{remark}

\section{Localisation exact sequences}\label{sec:localisation}

Recall our heuristic that \[``\textit{$\H(A,S)$ is the difference between $\A(A)$ and $\A(S^{-1}A)$.''}\]A precise form that this statement takes is the well-known result of Milnor-Husemoller, which is our first example of what we will call a \textit{localisation exact sequence}:

\begin{theorem}For $A$ a Dedekind domain with trivial involution and fraction field $F$ there is an exact sequence of abelian groups \[0\to W(A)\to W(F)\to \bigoplus_{\p}W(A/\p),\]where $\p$ runs over all prime ideals (\cite[IV.3.3]{MR0506372}). In the case that $A=\Z$, there is moreover a split short exact sequence of abelian groups\[0\to W(\Z)\to W(\Q)\to \bigoplus_{p}W(\Z/p\Z)\to 0,\]where $p$ runs over all primes $p\in \Z$ (\cite[IV.2.1]{MR0506372}).
\end{theorem}

In light of Theorem \ref{thm:karoubi} we prefer to rephrase the first of these exact sequences as \[0\to W(A)\to W(S^{-1}A)\to W(A,S),\]for $S=A\sm\{0\}$. In fact a stronger version of the theorem due to Ranicki is proved using algebraic $L$-theory.

\begin{theorem}[{\cite[3.4.7]{MR620795}}]\label{thm:WLES}If $A$ contains a half-unit and $(A,S)$ defines a localisation, then there is an exact sequence of abelian groups\[M^\eps(A,S)\to W^\eps(A)\to W^\eps(S^{-1}A)\to W^\eps(A,S)\to M^{-\eps}(A),\]where the reader is referred to \cite[1.6, 3.5]{MR620795} for the definitions of the groups of so-called `formations' denoted `$M$' (although, of course, it is the case that $M^\eps(A,S)=0$ for $A$ a Dedekind domain, and that $M^{-\eps}(\Z)=0$).
\end{theorem}

We will not reprove the exactness of this sequence here. But as our main intention is to extend this to the context of the double Witt groups it is enlightening for us to understand the groups and maps involved. In particular, we wish to change perspective on the central term $W(S^{-1}A)$ in order to understand the map $W(S^{-1}A)\to W(A,S)$. The classical approach to this is to start by making a choice of `$A$-lattice' for a non-singular $\eps$-symmetric form $(P,\theta)$ over $S^{-1}A$, an $A$-submodule $K\subset P$ such that $(K,\theta|_K)$ is an $\eps$-symmetric form over $A$ and $S^{-1}K\cong P$. The perspective we want to take is that, actually, the pairings on these `$A$-lattices' are a more useful way to define the group $W(S^{-1}A)$ in the first place. The heuristic is that if we work over the same ring at all stages in the localisation exact sequence, comparison becomes easier. Now we describe how to change perspective in order to do so.

\begin{definition} Suppose $(K,\alpha)$ is an $\eps$-symmetric form over $A$ such that $S^{-1}(K,\alpha):=(S^{-1}K,S^{-1}\alpha)$ is a non-singular $\eps$-symmetric form over $S^{-1}A$, then $(K,\alpha)$ is called \textit{$S$-non-singular}. A submodule (not necessarily a direct summand) $j:L\hookrightarrow K$ is called an \textit{$S$-lagrangian} if $j^*\alpha j=0$ and $S^{-1}j:S^{-1}L\hookrightarrow S^{-1}K$ is a lagrangian for $S^{-1}(K,\alpha)$. $(K,\alpha)$ is called \textit{$S$-metabolic} if it admits an $S$-lagrangian.
\end{definition}

\begin{definition}
The \textit{$\eps$-symmetric Witt $\Gamma$-group of $i:A\to S^{-1}A$} is the abelian group defined by the monoid construction\[\Gamma^\eps(A\to S^{-1}A)=\{\text{$S$-non-singular, $\eps$-symmetric forms over $A$}\}/\{\text{$S$-metabolic forms}\}.\]
\end{definition}

\begin{proposition}\label{prop:gammachar}There is an isomorphism of groups\[S^{-1}A\otimes-:\Gamma^\eps(A\to S^{-1}A)\xrightarrow{\cong} W^\eps(S^{-1}A);\qquad(K,\alpha)\mapsto S^{-1}(K,\alpha).\]
\end{proposition}
\begin{proof}This follows, for instance, from taking $n=0$ in the isomorphism $\Gamma^n(A\to S^{-1}A,\eps)\cong L^n(S^{-1}A,\eps)$ from the proof of \cite[3.2.3(i)]{MR620795}.
\end{proof}

Using \ref{prop:gammachar} we may now rewrite the Milnor-Husemoller localisation exact sequence for $A$ a Dedekind domain in our preferred form \[0\to W^\eps(A)\xrightarrow{i} \Gamma^\eps(A\to S^{-1}A)\xrightarrow{\partial} W^\eps(A,S).\]The map $i: W^\eps(A)\xrightarrow{i} \Gamma^\eps(A\to S^{-1}A)$ is simply the map that forgets the non-singularity of the form. The map $\partial$ is classically called the `dual lattice construction' (\cite{MR770063}), although we will not make use of this construction here. Actually we prefer the following well-known formulation (which is simpler and amounts to the same thing).

\begin{definition}\label{def:algbound}If $(K,\alpha)$ is an $S$-non-singular $\eps$-symmetric form over $A$ then by Lemma \ref{lem:injective} there is a short exact sequence of $A$-modules \[0\to K\xrightarrow{\alpha} K^*\to T:=\coker(\alpha)\to 0,\]and $T$ is $S$-torsion. The \textit{boundary of $(K,\alpha)$} is the non-singular $\eps$-symmetric linking form $\partial(K,\alpha)=(T,\lambda)$, where \[\lambda:T\to T^\wedge;\qquad [x]\mapsto ([y]\mapsto x(z)/s),\]where $x,y\in K^*$, $s\in S$ and $\alpha(z)=sy$. The definition is easily checked to be independent of the choices of $x,y,z$ and $s$.
\end{definition}

The following lemma tells us how to build lagrangians for the boundary linking forms associated to $S$-non-singular $\eps$-symmetric forms.

\begin{lemma}\label{lem:boundarylag}Suppose $f:(K,\alpha)\to(K',\alpha')$ is a morphism of $\eps$-symmetric forms over $A$ that becomes an isomorphism of forms over $S^{-1}A$. Suppose that $(K',\alpha')$ is non-singular and that $(K,\alpha)$ is $S$-non-singular. Then the induced injective morphism\[j:L=\coker(f)\hookrightarrow \coker(\alpha)=T\] is a lagrangian for the boundary $\partial(K,\alpha)=(T,\lambda)$.
\end{lemma}

\begin{proof}There is a diagram with commuting squares\[\xymatrix{
0\ar[r]&K\ar[r]^-{f}\ar[d]^-{1}&K'\ar[r]\ar[d]^-{f^*\alpha'}&L\ar[r]\ar[d]^-{j}&0\\
0\ar[r]&K\ar[r]^-{\alpha}\ar[d]^-{\eps}&K^*\ar[r]\ar[d]^-{1}&T\ar[r]\ar[d]^{\lambda}&0\\
0\ar[r]&K\ar[r]^-{\alpha^*}\ar[d]^-{(\alpha')^*f}&K^*\ar[r]\ar[d]^-{1}&T^\wedge\ar[r]\ar[d]^-{j^\wedge}&0\\
0\ar[r]&(K')^*\ar[r]^-{f^*}&K^*\ar[r]&L^\wedge\ar[r]&0}\]which is a projective resolution of the rightmost vertical column. The exactness of this rightmost column must be checked to establish the lemma. $j$ is injective, $j^\wedge$ is surjective and $j^\wedge\lambda j=0$. Suppose $x=f^*(y)\in K^*$, then as $\alpha'$ is an isomorphism of $A$-modules, $y=\alpha'(z)$ and $x\in \im(f^*\alpha')$. 
\end{proof}

\begin{remark}There is no reason for the injective morphism $j:L\hookrightarrow T$ in the proof of \ref{lem:boundarylag} to be split. This is the crucial point for explaining the historical interest in the Witt group of linking forms and the lack of historical interest in something like the split Witt group of linking forms. The classical Witt group of linking forms is the correct term for the Milnor-Husemoller localisation exact sequence and that the morphism $j$ might not split is \textit{precisely} why.
\end{remark}

\begin{proposition}\label{prop:boundarywelldef}If $(K,\alpha)$ is an $S$-metabolic, $S$-non-singular $\eps$-symmetric form over $A$ then $[\partial(K,\alpha)]=0\in W^\eps(A,S)$. Hence the map\[\partial:\Gamma^\eps(A\to S^{-1}A)\to W^\eps(A,S);\qquad[(K,\alpha)]\mapsto [\partial(K,\alpha)]\]is well-defined.
\end{proposition}

\begin{proof}Let $j:L\hookrightarrow K$ be an $S$-lagrangian. We claim that there are metabolic linking forms $F$ and $F'$ such that $\partial(K,\alpha)\oplus F\cong F'$. 

There is a splitting of the $S^{-1}A$-module morphism $S^{-1}(j^*\alpha):S^{-1}K\to S^{-1}L^*$. By clearing denominators we may assume the splitting is given by $S^{-1}k$ for an $A$-module morphism $k:L^*\to K$ and hence $g:=j^*\alpha k:L^*\to L^*$ is an isomorphism over $S^{-1}A$. Using the approach of Lemma \ref{lem:splitishyp1} the morphism\[(j\,\,k):\left(L\oplus L^*,\left(\begin{matrix}0&g\\\eps g^* &k^*\alpha k\end{matrix}\right)\right)=(K',\alpha')\to(K,\alpha)\]becomes an isomorphism after passing to $S^{-1}A$. $\partial(K',-\alpha')$ will be our form $F$.

In the first instance there is a morphism of $S$-non-singular forms \[\left(\begin{matrix}\alpha'f&0\\\eps f&1\end{matrix}\right):(K,\alpha)\oplus(K',-\alpha')\to\left( (K')^*\oplus K',\left(\begin{matrix}0&1\\ \eps&-\alpha\end{matrix}\right)\right),\]which becomes an isomorphism of forms over $S^{-1}A$, and the latter form is non-singular. Applying Lemma \ref{lem:boundarylag}, we see that $F'$ is metabolic.

But there is also a morphism of $S$-non-singular forms \[\left(\begin{matrix}g^*&0\\0&1\end{matrix}\right):\left(L\oplus L^*,\left(\begin{matrix}0&g\\\eps g^* &k^*\alpha k\end{matrix}\right)\right)\to\left(L\oplus L^*,\left(\begin{matrix}0&1\\\eps  &k^*\alpha k\end{matrix}\right)\right),\]which becomes an isomorphism of forms over $S^{-1}A$. Hence $F$ is metabolic by Lemma \ref{lem:boundarylag}.
\end{proof}

\begin{remark}Note that given a morphism of $S$-non-singular forms $g:(K,\alpha)\to (K',\alpha)$ over $A$ that becomes an isomorphism of forms over $S^{-1}A$, we can `invert' the associated morphism of modules by clearing denominators on $(S^{-1}g)^{-1}$. However, the resulting morphism of modules may not be a morphism of forms over $A$.

In general, we do not expect this proof can be simplified so that $\partial(K,\alpha)$ is metabolic on the nose. This is a subtle but significant point. A full discussion of why one shouldn't expect this will require us to develop much more language, and is included in the $D\Gamma$-group discussion in Section \ref{sec:comparison}.
\end{remark}

\subsection{Double Witt group localisation exact sequence}\label{subsec:DWLES}

We will now state a localisation exact sequence in the context of our double Witt groups.

\medskip

We will need to define something like a `double Witt $\Gamma$-group'. Naively, we might define an $S$-non-singular form $(K,\alpha)$ over $A$ to be `$S$-hyperbolic'  if there exists a pair of $S$-complementary $S$-lagrangians. But if this naive approach is attempted, one immediately runs into trouble with the boundary map $\partial$ no longer being well-defined. In other words, there is no guarantee that the boundaries of a pair of $S$-complementary $S$-lagrangians for an $S$-non-singular form $(K,\alpha)$ over $A$ are complementary lagrangians for the linking form $\partial(K,\alpha)$.

To resolve this issue, the well-definedness of `$\partial$' in a double Witt group context must be ensured by the definition of the double Witt $\Gamma$-group itself. The following definition may seem a little contrived at first sight, but will be seen to be very natural later in the context of the algebraic $L$-theory of Chapter \ref{chap:DLtheory}.

\begin{definition}Let $(K,\alpha)$ be an $S$-non-singular $\eps$-symmetric form over $A$. Two $S$-lagrangians $j_\pm:L_\pm\hookrightarrow K$ are \textit{$\partial$-complementary} (read: `boundary complementary') if the sequence \[\xymatrix{0\ar[r]&L_+\oplus L_-\ar[rr]^-{\left(\begin{smallmatrix}j_+&j_-\\0&\alpha j_-\end{smallmatrix}\right)}&&K\oplus K^*\ar[rr]^-{\left(\begin{smallmatrix}-j_+^*\alpha&j_+^*\\0&j_-^*\end{smallmatrix}\right)}&&L_+^*\oplus L_-^*\ar[r]&0}\]is exact. In other words, if the inclusion\[\left(\begin{matrix}j_+&j_-\\0&\alpha j_-\end{matrix}\right):(L_+\oplus L_-,0)\to (K\oplus K^*,\left(\begin{smallmatrix}-\alpha&1\\\eps&0\end{smallmatrix}\right))\]is the inclusion of a lagrangian for the non-singular $\eps$-symmetric form $(K\oplus K^*,\left(\begin{smallmatrix}-\alpha&1\\\eps&0\end{smallmatrix}\right))$ over $A$. If $(K,\alpha)$ admits a pair of $\partial$-complementary $S$-lagrangians it is called is \textit{$\partial$-hyperbolic}.
\end{definition}

\begin{lemma}Suppose there exists a half-unit $s\in A$. If $f:(K,\alpha)\xrightarrow{\cong} (K',\alpha')$ is an isomorphism of $S$-non-singular $\eps$-symmetric forms over $A$ then $(K,\alpha)\oplus(K',-\alpha')$ is $\partial$-hyperbolic with $\partial$-complementary $S$-lagrangians \[\begin{array}{lccrcl}(&1&f&)&:&(K,0)\to (K\oplus K',\alpha\oplus -\alpha'),\\(&\bar{s}&-sf&)&:&(K,0)\to (K\oplus K',\alpha\oplus -\alpha').\end{array}\]
\end{lemma}

\begin{proof}By Lemma \ref{lem:halfunithyp} the submodules are indeed $S$-lagrangians. We need only show exactness of the sequence\[\xymatrix{0\ar[r]&K\oplus K\ar[rr]^-{\left(\begin{smallmatrix}1&\overline{s}\\f&-sf\\0&\overline{s}\alpha \\0&s\alpha'f\end{smallmatrix}\right)}&&(K\oplus K')\oplus (K\oplus K')^*\ar[rrr]^-{\left(\begin{smallmatrix}-\alpha&f^*\alpha'&1&f^*\\0&0&s&-\overline{s}f^*\end{smallmatrix}\right)}&&&K^*\oplus K^*\ar[r]&0}.\]The composition of the matrices is 0 and the sequence is exact at $K\oplus K$ and $K^*\oplus K^*$. To check exactness in the central term suppose \[\left(\begin{matrix}-\alpha&f^*\alpha'&1&f^*\\0&0&s&-\overline{s}f^*\end{matrix}\right)((x,y),(g,h))=0\]then $sg=(sf)^*h$ and $\alpha(x)=f^*\alpha'(y)+g+f^*h$, so $g=\overline{s}(\alpha(x)- \alpha f^*\alpha'(y))$ and, as $f^*$ is an isomorphism, $h=s(\alpha'f(x)-\alpha'(y))$. Consider\[\left(\begin{matrix}1&\overline{s}\\f&-sf\\0&\overline{s}\alpha \\0&s\alpha'f\end{matrix}\right)\left(\begin{matrix}sx+\overline{s}f^{-1}(y)\\x-f^{-1}(y)\end{matrix}\right)=(x,y,\overline{s}\alpha(x-f^{-1}(y)),s\alpha'(f(x)-y)).\]As $\alpha'=f^*\alpha f$, we are done.
\end{proof}

\begin{definition}Suppose $(A,S)$ defines a localisation. The set of isomorphism classes (over $A$) of $S$-non-singular $\eps$-symmetric forms over $A$, equipped with the addition $(K,\alpha)+(K',\alpha')=(K\oplus K',\alpha\oplus \alpha')$ forms a commutative monoid\[\mathfrak{G}^\eps(A\to S^{-1}A)=\{\text{$S$-non-singular $\eps$-symmetric forms over $A$}\}.\]If $A$ contains a half-unit then the monoid construction\[D\Gamma^\eps(A\to S^{-1}A)=\mathfrak{G}^\eps(A\to S^{-1}A)/\{\text{$\partial$-hyperbolic}\}\]is an abelian group called the \textit{$\eps$-symmetric double Witt $\Gamma$-group of $(A,S)$}.
\end{definition}

\begin{proposition}\label{prop:boundaryhypwelldef}Suppose an $S$-non-singular $\eps$-symmetric form $(K,\alpha)$ over $A$ is $\partial$-hyperbolic, then $\partial(K,\alpha)$ is hyperbolic. Hence the map\[D\partial:D\Gamma^\eps(A\to S^{-1}A)\to DW^\eps(A,S);\qquad[(K,\alpha)]\mapsto[\partial(K,\alpha)]\]is well-defined.
\end{proposition}

\begin{theorem}[Double Witt group localisation exact sequence]\label{thm:DWLES}Suppose $A$ is a Dedekind domain and $A$ contains a half unit. Then there is an exact sequence of group homomorphisms\[0\to DW^\eps(A)\xrightarrow{Di} \widetilde{D\Gamma}^\eps(A\to S^{-1}A)\xrightarrow{D\partial} DW^\eps(A,S),\]where $\widetilde{D\Gamma}^{\eps}(A\to S^{-1}A)\subseteq D\Gamma^\eps(A\to S^{-1}A)$ is a certain subgroup defined in \ref{def:reduced}. \end{theorem}

We will not prove Proposition \ref{prop:boundaryhypwelldef} and Theorem \ref{thm:DWLES} at this stage. The proofs of these are handled much more efficiently in Chapter \ref{chap:DLtheory} as corollaries of the double $L$-theory localisation exact sequence of Theorem \ref{thm:DLLES}. We have stated them here for completeness, and as an advertisement for the chain complex methods of later chapters!

\begin{remark}It is interesting to note at this stage that if $A$ is a Dedekind domain with a half-unit then there are surjective forgetful maps resulting in the following commutative diagram\[\xymatrix{0\ar[r] &DW^\eps(A)\ar[r]\ar[d]^-{\cong}&\widetilde{D\Gamma}^\eps(A\to S^{-1}A)\ar[r]\ar[d]&DW^\eps(A,S)\ar[d]&\\0\ar[r] &W^\eps(A)\ar[r]&\Gamma^\eps(A\to S^{-1}A)\ar[r]&W^\eps(A,S)\ar[r] &M^{-\eps}(A)}\]
\end{remark}

\begin{remark}For historical interest, we remark that this is not the first attempt at a localisation exact sequence for something like double Witt groups. Stoltzfus (\cite[3.6]{MR521738}) offers a kind of localisation exact sequence in the restricted setting of the split Witt groups he defined. In our terminology, and using the identifications of the Chapter \ref{chap:laurent}, he shows the following is exact for a Dedekind domain $A$ with fraction field $F$\[DW^\eps(A[z,z^{-1}],P)\to DW^\eps(F[z,z^{-1}],P')\to C^\eps(F/A),\]where $P=\{p(z)\in A[z,z^{-1}]\,|\,\text{$p(1)\in A$ is a unit}\}$, $P'$ is given similarly but with $F$-coefficients, and $C^\eps(F/A)$ is a certain Witt group of `$\eps$-symmetric isometric forms' where the underlying modules are torsion.

Our exact sequence is more closely related to the localisation exact sequence of Milnor-Husemoller and fits into the general scheme of classical localisation sequences in \cite{MR0249491}, \cite{MR0384894}, \cite{MR620795} et al.
\end{remark}

\chapter{Laurent polynomial rings}\label{chap:laurent}

Chapter \ref{chap:laurent} will review some standard topological motivation for the study of linking forms. For any cover of a space with Poincar\'{e} duality we will associate a series of \textit{linking pairings} between torsion cohomology modules. When the duality is odd-dimensional we discuss when it is possible to build a linking form from the middle-dimensional pairing, and under what conditions this form is non-singular.

Next we consider the simplest example of such a geometrically defined linking form for which the deck transformation group of the cover is not finite. If $R$ is a field, and an odd-dimensional manifold $M^{2k+1}$ has a map to the circle $S^1$ we obtain a non-singular linking form for the localisation of the group ring $R[\Z]$ to its fraction field. This example also gives rise to an `autometric form' - a form over $R$ equipped with an automorphism. We recall the geometric relationship between these linking and autometric forms.

On the algebraic side, we will set the following algebraic conventions for this chapter:\begin{itemize}
\item $R$ is a commutative Noetherian ring with unit and involution.
\item $R[\Z]=R[z,z^{-1}]$ with $\overline{z}=z^{-1}$, the \textit{Laurent polynomials with $R$ coefficients}.
\item $Q$ is the \textit{set of characteristic polynomials}\[Q=\left\{p(z)=\sum_{M}^Na_kz^k\,|\,a_{M},a_N\in R^\times\right\}\subset R[z,z^{-1}].\] If $R$ is a field then $(R[z,z^{-1}],Q)$ describes the full localisation to the field of fractions of $R[z,z^{-1}]$. Also, note that any $p\in Q$ is equivalent up to multiplication by units in $R[z,z^{-1}]$ to a \textit{bionic polynomial}, that is a polynomial $q(z)=a_0+a_1z+\dots+a_{L-1}z^{L-1}+z^{L-1}\in R[z]$ with $a_0\in R^\times$, a unit.
\item $P$ is the \textit{set of Alexander polynomials}\[P=\left\{p(z)=\sum_{M}^Na_kz^k\,|\, p(1)\in R^\times\right\}\subset R[z,z^{-1}].\]
\end{itemize}

Chapter \ref{chap:laurent} is largely intended as a motivation chapter for the study of linking forms, consisting of well-known background results in topology. However, we will prove some new algebraic results. The most important of which are the following isomorphisms:\[\begin{array}{rcll}DW^\eps(R[z,z^{-1}],Q)&\cong&DAut^{-\eps}(R),&\text{the \textit{double Witt group of autometric forms over $R$}},\\
DW^\eps(R[z,z^{-1}],P)&\cong& \widehat{DW}^{-\eps}(R),&\text{the \textit{double Witt group of Seifert forms over $R$}}.\end{array}\]Each isomorphism is given by the algebraic \textit{covering} operation, the inverse for the autometric case is the algebraic \textit{monodromy} operation.

\section{Linking forms from geometry}

\subsection{Motivation}\label{subsec:motivation}The motivation for the geometric linking forms we will consider is easiest to imagine in low dimensions, so this is where we will begin. Unfortunately the kind of `linking' that it measures in low dimensions is not as subtle as one might like because ultimately we will be defining a linking pairing on (co)homology classes.

A \textit{link} in a 3-manifold $M$ is the image in $M$ of an embedding of finitely many disjoint copies of $S^1$. How can we describe `linking' of two links $L,L'\subset M$ (for example as pictured below)?

\[\begin{xy} 
\def\piclinkone{\resizebox{0.5\textwidth}{!}{ \includegraphics{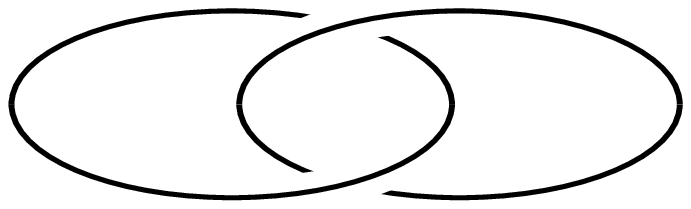}}}
\xyimport(259,139){\piclinkone}
,(-10,72)*!L{L}
,(260,72)*!L{L'}
\end{xy}\]

One way to detect linking would be to assign orientations to the links and find an oriented surface cobounding one of the links, e.g.\ $F\subset M$ such that $\partial F=L$, and count (with signs) the points of intersection $p\in F\cap L'$ to get $l=\sum_p(\pm1)$.

\[\begin{xy} 
\def\piclinktwo{\resizebox{0.5\textwidth}{!}{ \includegraphics{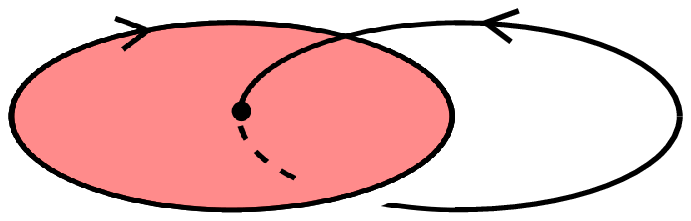}}}
\xyimport(259,139){\piclinktwo}
,(-10,72)*!L{L}
,(260,72)*!L{L'}
,(330,72)*!L{l=1}
,(40,72)*!L{F}
,(103,72)*!L{+1}
\end{xy}\]

This quantity tells you something, but in what way is it a topological invariant? For the sake of this motivational section, we will now do something fairly brutal - consider 2 links $L,L'$ to be equivalent if and only if they determine the same class $[L]=[L']\in H_1(M)$. This equivalence will allow us to generalise our example readily to high dimensions. However, this relationship is so crude that we can force our signed count to be anything we like. For example within a homology class we can make loops locally as in the picture below, and introduce as many extra intersections as we care to.

\[\begin{xy} 
\def\piclinkthree{\resizebox{0.5\textwidth}{!}{ \includegraphics{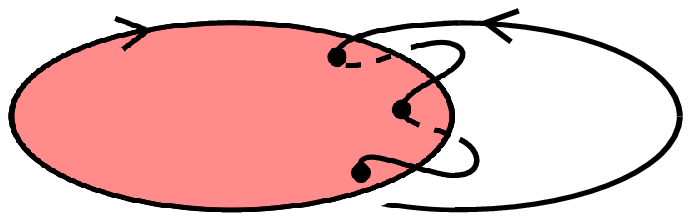}}}
\xyimport(259,139){\piclinkthree}
,(-10,72)*!L{L}
,(260,72)*!L{L'}
,(40,72)*!L{F}
,(330,72)*!L{l=3}
,(103,105)*!L{+1}
,(130,72)*!L{+1}
,(113,40)*!L{+1}
\end{xy}\]

So the quantity we have written down is not even well-defined in $\Z$, certainly not an invariant of a pair $[L],[L']\in H_1(M)$! But something can be salvaged from this discussion if we restrict to \textit{torsion} homology classes as we now see.

We now open the example up to a more general setting by extending to higher dimensions. Consider a closed, oriented, $n$-dimensional topological manifold $M^n$ and fix $0\leq r\leq n-1$. Suppose $U^r$ and $V^{n-1-r}$ are closed, oriented topological manifolds embedded as submanifolds $U,V\subset M$. Suppose $W$ is an oriented $(n-r)$-manifold with boundary and that it is embedded $W\subset M$ in such a way that for some $s\in\Z$, $[\partial W]=s[V]\in H_{n-1-r}(M)$. (For example the boundary might be a disjoint union $\partial W=\bigsqcup_{i=1}^s V_i$ and $[V_i]=[V]\in H_{n-1-r}(M)$ for each $i$.) Assume the embeddings are all in general position, so that $U\cap \partial W\neq\emptyset$ lies in the interior of $W$ and is a finite set of points of transverse intersection. We assign each intersection point $p$ a sign $\eps_p=\pm 1$ in the usual way, as the sign of the determinant of $T_pU\oplus T_pW\xrightarrow{\cong} T_pM$. We then define the \textit{geometric linking} of $U$ and $V$ to be \[\text{link}(U,V)=\frac{1}{s}\sum_p\eps_p\in \Q/\Z.\]As this sum takes values in $\Q/\Z$, the indeterminacy of the quantity $l$ modulo $\Z$ we observed above is no longer an issue. This is because introducing $n$ `loops' as before will introduce $n$ intersections for each homology copy of $V$ in $\partial W$ and \[\frac{1}{s}\sum\eps_p=\frac{1}{s}(\sum\eps_p+ns)\in \Q/\Z.\]Of course there is a lot more work now needed to show that our geometric linking is a well-defined invariant of our pair of homology classes $[U],[V]$ (independent of choice of $W$ etc.). We now turn to a precise description of this linking invariant and a proof of the claimed properties.

\subsection{Linking pairings from Poincar\'{e} duality}

Let $X$ be a topological space and $R$ be a commutative Noetherian ring with unit and involution. (Usually we will be thinking of the cases $R=\Z,\Z/2\Z,\Q,\R,\C$ and of this list only $\C$ has a non-trivial involution, given by complex conjugation). There are singular chain and cochain complexes with $R$-coefficients \[C_*(X;R),\qquad C^*(X;R)=\Hom_R(C_{*},R)\]defined in the usual way, giving (respectively) homology and cohomology with $R$-coefficients. A covering space $\overline{X}$ for $X$ with group of deck transformations $\pi$ is called a \textit{$\pi$-cover}. The ring with involution $R[\pi]$ is the group ring, consisting of finite sums $\sum r_gg$ where $r_g\in R$, $g\in \pi$ and the involution is given by\[\sum r_g g\mapsto \sum\overline{r_g} g^{-1}.\]The deck transformation group of a $\pi$-cover $\overline{X}$ acts on the left, so there are defined chain complexes of $R[\pi_1(X)]$-modules\[C_*(\overline{X};R),\qquad \Hom_{R[\pi]}(C_*(\overline{X};R),R[\pi]),\]with respective homology and cohomology written $H_*(\overline{X};R)$ and $H^*(\overline{X};R)$. When $R=\Z$ we will write $H_*(\overline{X})$ and $H^*(\overline{X})$. In particular, the universal cover $\widetilde{X}$ of $X$ is a $\pi_1(X)$-cover and we write\[C_*(X;R[\pi_1(X)]):=C_*(\widetilde{X};R),\qquad C^*(X;R[\pi_1(X)]):=\Hom_{R[\pi_1(X)]}(C_*(\widetilde{X};R),R[\pi_1(X)]).\]

For any class $[x]\in H_q(X;R)$ there is defined a \textit{chain level cap product} (see section \ref{sec:symmetric}, in particular apply $\Z^t\otimes -$ to equation \ref{eq:adjoint})\begin{equation}\label{eq:capproduct}-\cap [x]: \Hom_{R[\pi]}(C_p(\overline{X};R),R[\pi])\to C_{q-p}(\overline{X};R).\end{equation}

\begin{remark}This setup is an efficient way of encoding several classical systems for non-compact spaces that works particularly well when the non-compact space is a cover for a compact space. Here is how it compares:

Given a $\pi$-cover $\overline{X}$ of $X$ our definitions mean that the cohomology groups we get are in fact what is usually called the \textit{singular cohomology with compact supports}\[H^r(\overline{X};R):=H^r(\Hom_{R[\pi]}(C_*(\overline{X};R),R[\pi]))= H^r_{\text{cpt}}(\overline{X};R).\]We wish to emphasise that the lack of the decoration indicating compact support \textit{is the only non-standard part of our notation system}. In particular, for $\pi$ finite the singular cohomology with compact supports (and hence our cohomology groups of a space) are identified with the usual singular cohomology with $R$-coefficients, equipped with the induced $R[\pi]$-action:\[H^*(\Hom_R(C_*(\overline{X};R),R))\qquad\text{with $R[\pi]$-action.}\]

For any (non-compact) space $Y$, there is a cap product between compactly supported cochains and \textit{infinite but locally finite} singular chains. This product has values in the ordinary singular chain group:\[-\cap-:C_{\text{cpt}}^p(Y;R)\times C^{lf}_q(Y;R)\to C_{q-p}(Y;R).\]If $X$ is a finite $CW$ complex and there is a $\pi$-cover $p:\overline{X}\to X$, then we may choose a lift of the CW structure to the space $\overline{X}$. The \textit{transfer map} is the map $p^!:H_*(X;R)\to H_*(\overline{X};R)$ in homology induced by the cellular chain map sending a cell to the locally finite cellular chain consisting of the sum of all possible lifts of that cell. The cap product \ref{eq:capproduct} factors as\[C^p(\overline{X};R)\times C_q(X;R)\xrightarrow{\text{id}\otimes p^!}C_{\text{cpt}}^p(\overline{X};R)\times C^{lf}_q(\overline{X};R)\to C_{q-p}(\overline{X};R).\]
\end{remark}

\begin{definition}\label{def:universal}A topological space $X$ is called an \textit{$n$-dimensional Poincar\'{e} space} if there exists a homology class $[X]\in H_n(X;\Z)$, called a \textit{fundamental class}, such that\[-\cap[X]:H^{n-r}(\widetilde{X};\Z)\xrightarrow{\cong} H_r(\widetilde{X};\Z)\]is an isomorphism of $\Z[\pi_1(X)]$-modules for each $r\in\Z$. This version of Poincar\'{e} duality is called \textit{universal Poincar\'{e} duality} (\cite[4.5]{MR2061749}). \end{definition}

Given a $\pi$-cover $\overline{X}$ of $X$, the natural projection\[\Z[\pi]^t\otimes_{\Z[\pi_1(X)]}C_*(\widetilde{X};\Z)\to C_*(\overline{X};\Z)\]is well-known to be a chain equivalence (\cite[p207]{MR566491}) so that if $X$ is an $n$-dimensional Poincar\'{e} space then for each $\pi$-cover $\overline{X}$ (including the trivial one) there is induced a homology class $[X]\in H_n(\overline{X};\Z)$ and there are isomorphisms of $\Z[\pi]$-modules \[-\cap[X]:H^{n-r}(\overline{X};\Z)\xrightarrow{\cong} H_r(\overline{X};\Z).\]

\begin{example}\label{ex:compactman}A closed, oriented $n$-dimensional topological manifold has a distinguished generator $[M]\in H_n(M;\Z)\cong\Z$ determined by the choice of orientation. There is an isomorphism \[-\cap[M]:H^{n-r}(M;R)\xrightarrow{\cong} H_r(M;R).\]For any (now possibly non-compact) oriented $\pi$-cover $\overline{M}$ of $M$, the Poincar\'{e} duality is of the form\[-\cap[\overline{M}]^{lf}:H^{n-r}_{\text{cpt}}(\overline{M};R)\xrightarrow{\cong} H_r(\overline{M};R), \]where $[\overline{M}]^{lf}\in H^{lf}_n(\overline{M};R)$ is the locally finite fundamental class of $\overline{M}$ and the cap product is the one between locally finite chains and compactly supported cochains. But by the previous remark we have that $[\overline{M}]^{lf}$ is the transfer of $[M]$ (here we are defining the transfer using the fact that a compact manifold $M$ has the homotopy type of a finite $CW$ complex). In particular all of the above is true for the universal cover and $M$ is an $n$-dimensional Poincar\'{e} space with fundamental class $[M]$. See, for instance, \cite[3.1]{MR1366538} for a good account of this non-compact Poincar\'{e} duality.\end{example}

Now fix an $n$-dimensional Poincar\'{e} space $X$ and a $\pi$-cover $\overline{X}$. As we have a fixed $\pi$-cover in mind now, we will write $A=R[\pi]$ and for any $A$-module $G$ we can unambiguously write the chain complexes of $R[\pi]$-modules\[C_*(X;G):=G^t\otimes_AC_*(\overline{X};R),\qquad C^*(X;G):=G^t\otimes_A\Hom_A(C_*(\overline{X};R),A),\]with corresponding homology and cohomology written $H_*(X;G)$ and $H^*(X;G)$. There is then induced a Poincar\'{e} duality isomorphism of $A$-modules \[-\cap[X]=(1\otimes(-\cap[X]))_*:H^{n-r}(X;G)\xrightarrow{\cong} H_r(X;G).\]

Suppose $(A,S)$ defines a localisation so that there is a short exact coefficient sequence of $A$-modules\[0\to A\to S^{-1}A\to S^{-1}A/A\to 0.\]Now as $C_*(X;A)$, $C^*(X;A)$ are chain complexes of free $R[\pi]$-modules (using that the dual of a free module is free), and as free modules are flat, there are induced short exact sequences of chain complexes\[0\to C_*(X;A)\to C_*(X;S^{-1}A)\to C_*(X;S^{-1}A/A)\to 0\] and \[0\to C^*(X;A)\to C^*(X;S^{-1}A)\to C^*(X;S^{-1}A/A)\to 0.\]There is hence a commuting Poincar\'{e} duality diagram for the associated Bockstein long exact sequences\[\xymatrix{...\ar[r]&H^{r-1}(X;S^{-1}A/A)\ar[ddd]^{-\cap[X]}_\cong\ar[dr]\ar[rr]^{\delta^*}&&H^{r}(X;A)\ar[r]\ar[ddd]^{-\cap[X]}_\cong&H^{r}(X;S^{-1}A)\ar[ddd]^{-\cap[X]}_\cong\ar[r]&...\\
&&\im{\delta^*}\ar[ur]\ar[d]^{-\cap[X]}_\cong&&&\\
&&\im{\delta_*}\ar[dr]&&&\\
...\ar[r]&H_{n-r+1}(X;S^{-1}A/A)\ar[ur]\ar[rr]^{\delta_*}&&H_{n-r}(X;A)\ar[r]&H_{n-r}(X;S^{-1}A)\ar[r]&...}\]Recall that $S^{-1}H^r(X;A)\cong H^r(X;S^{-1}A)$ (by Lemma \ref{lem:locisexact}), and that the $S$-torsion submodule is \[TH^r(M;A):=\ker(H^r(M;A)\to S^{-1}H^r(M;A)),\]which is precisely $\im(\delta^*)$ by exactness of the sequence. A similar statement holds for the homological Bockstein sequence.

There is a non-singular bilinear coefficient pairing\[S^{-1}A/A\times A\to S^{-1}A/A;\qquad (a,b)\mapsto a\overline{b}\]inducing an isomorphism of $A$-modules $(S^{-1}A/A)^t\otimes_A A\cong S^{-1}A/A$ and hence for each $i,j$ there is a cup-product\[\cup:H^i(X;S^{-1}A/A)\times H^j(X;A)\to H^{i+j}(X;S^{-1}A/A).\]Using this cup-product we make the following definition.

\begin{definition}For each $r\in\Z$, the \textit{cohomology linking pairing} for the $\pi$-cover $\overline{X}$ is the pairing\[\lambda_X:TH^r(X;A)\times TH^{(n+1)-r}(X;A)\to S^{-1}A/A;\qquad(x,y)\mapsto \langle w\cup y,[X]\rangle\in S^{-1}A/A,\]where $w\in H^{r-1}(X;S^{-1}A/A)$ is a choice of element such that $\delta^*(w)=x$, and the angle brackets denote the Kronecker pairing.
\end{definition}

\begin{lemma}For each $r\in\Z$, the cohomology linking pairing for the $\pi$-cover $\overline{X}$ is well defined and can equivalently be described as\[\lambda_X:TH^r(X;A)\times TH^{(n+1)-r}(X;A)\to S^{-1}A/A;\qquad([x],[y])\mapsto \frac{1}{s}[\tilde{x}\cup y]\cap[X]\in S^{-1}A/A,\]where now $x\in C^r(X;A), y\in C^{(n+1)-r}(X;A),\tilde{x}\in C^{r+1}(X;A), s\in S$ and $d^*w=sx$.
\end{lemma}

\begin{proof}To show the pairing is well-defined is a straightforward diagram chase to show independence of choices, which we omit here.

The homology class $[s^{-1}\tilde{x}]\in H^{r+1}(X;S^{-1}A/A)$ is in the preimage of $[x]\in H^r(X;A)$ under the Bockstein $\delta^*$.
\end{proof}

\begin{proposition}The cohomology linking pairing for the $\pi$-cover $\overline{X}$ has the property that for $[x]\in TH^r(X;A), [y]\in TH^{(n+1)-r}(X;A)$\[\lambda_X([x],[y])=(-1)^{r(n+1-r)}\overline{\lambda_X([y],[x])}.\]
\end{proposition}

\begin{proof}Write representative cochains $x,y$, then choose $\tilde{x}\in C^{r-1}(X;A)$ and $\tilde{y}\in C^{n-r}(X;A)$ such that $d^*(\tilde{x})=sx, d^*(\tilde{y})=ty$ for some $s,t\in S$. We wish to prove that \[\left[\frac{1}{s}\tilde{x}\cup y-(-1)^{r(n+1-r)}\frac{1}{t}\overline{\tilde{y}\cup x}\right]\cap([X])\in A.\]Cup products in our setup are signed skew-commutative on the level of cohomology (consider the cup product of chains $u,v\in C^*(X)$ is given by $\Delta_0^*(u\otimes v )$ where $\Delta_0:C_*(X;A)\to C_*(X;A)^t\otimes_A C_*(X;A)$ is a choice of chain diagonal approximation, see \ref{sec:symmetric} for full description), so some rearranging of the cochains on the left hand side gives\[\frac{1}{s}\tilde{x}\cup y-(-1)^{r(n+1-r)+r(n-r)}\frac{1}{t}x\cup \tilde{y}+d^*z=\frac{1}{s}\tilde{x}\cup y+(-1)^{r-1}\frac{1}{t}x\cup \tilde{y}+d^*z,\]for some cochain $z\in C^{n-1}(X;A)$. But consider that the (co)differential is well-known to be a derivation for the cup product (see e.g.\ \cite[p326]{MR1224675}), viz.:\[\frac{1}{s}x\cup \tilde{y}+ \frac{1}{t}(-1)^{r-1}(\tilde{x}\cup y)=\frac{1}{st}d^*(\tilde{x}\cup \tilde{y})\in C^{n}(X;S^{-1}A).\]So we have \[\left[d^*(\frac{1}{st}\tilde{x}\cup\tilde{y}+z)\right]=0\in H^n(X;S^{-1}A).\]and the proposition follows.
\end{proof}

\begin{example}When $\pi$ is trivial and $(A,S)=(\Z,\Z\sm\{0\})$ we recover the linking quantity discussed in the motivation section above. Suppose $M^n$ is a closed oriented topological manifold with fundamental class $[M]\in H_n(M;\Z[\pi_1(X)])$. Then write the inverse of the Poincar\'{e} duality isomorphism\[D=(-\cap[M])^{-1}:H_r(M;\Z)\xrightarrow{\cong} H^{n-r}(M;\Z).\]Let $U^r,V^{n-1-r}\subset M$ be closed, oriented topological manifolds embedded as submanifolds, and $Z$ be an oriented $(n-r)$-manifold with boundary, embedded $Z\subset M$ in such a way that $[\partial Z]=s[V]\in H_{n-1-r}(M)$, for some $s\in\Z$. Suppose all embeddings are in general position, so in particular $U\cap \partial Z=\emptyset$, and intersections are transverse. Then it is clear that\[\begin{array}{rcl}\text{link}(U,V)=s^{-1}\sum_{p\in U\cap Z}\eps_p&=&s^{-1}D([U])\cap[Z]\\&=&s^{-1}(D([U])\cup D([Z]))\cap[M]\\&=&\pm\lambda_M(D([U]),D([V]))\in \Q/\Z.\end{array}\]
\end{example}

\begin{example}[Intersection on the interior is linking on the boundary] Suppose $\pi$ is trivial and $(A,S)=(\Z,\Z\sm\{0\})$. The following is a geometric motivation for the various localisation exact sequences considered in this thesis.

Let $(M^{n+1},\partial M)$ be a $(n+1)$-dimensional topological manifold with boundary. Suppose there are topological manifolds with boundary $(Y^{r+1},\partial Y),(Z^{n-r},\partial Z)\subset (M,\partial M)$ properly embedded as submanifolds in general position (so assume $\partial Y, \partial Z\subset \partial M$ and $\partial Y\cap \partial Z=\emptyset$). Assume that $Y$ intersects $Z$ transversely, so there is a well-defined intersection quantity\[\text{intersection}(Y,Z)=\sum_{p\in Y\cap Z}\eps_p\in\Z.\]Now assume that there exists a homotopy of the embedded $(Z,\partial Z)$ such that at each time in the homotopy $(Z,\partial Z)$ is properly embedded as a submanifold in general position and intersects $Y$ transversely. Suppose after this homotopy that $Z\subset \partial M$. A continuity argument shows the count of the points of intersection is unaffected by this homotopy so that\[\text{intersection}(Y,Z)=\sum_{p\in Y\cap Z}\eps_p=\sum_{p\in \partial Y\cap Z}\eps_p=\text{link}(\partial Y,\partial Z)\in\Z.\]Of course, as observed in the motivation section, this linking quantity is not well-defined by the homology classes of $\partial Y$ and $\partial Z$ in $H_*(\partial M;\Z)$. In fact, our homotopy shows the class $[\partial Z]\in H_{n-r-1}(\partial M;\Z)$ vanishes.

So now suppose $[\partial Y]=s[U]\in H_{r}(M;\Z)$ and  $[\partial Z]=t[V]\in H_{n-r-1}(M;\Z)$. In other words, suppose $\partial Y$ and $\partial Z$ realise torsion homology classes. Then, up to homology representatives of $U$ and $V$ in $\partial M$, there is a well-defined quantity\[\left(\frac{1}{st}\right)\cdot\text{intersection}(Y,Z)=\text{link} (U,V)\in \Q/\Z.\]

We can interpret this as the geometric-level reasoning behind the algebraic boundary map of monoids\[\partial:\{\text{$\eps$-symmetric forms over $\Z$ that are non-singular over $\Q$}\}\to \NN^\eps(\Z,\Z\sm\{0\}).\]To do so we need to assume that $n=2k+2$ and that $\partial M$ is a rational homology sphere. Then, defining $FH^{2k+1}(M;\Z)= H^{2k+1}(M;\Z)/TH^{2k+1}(M;\Z)$, the pairing\[b:FH^{2k+1}(M;\Z)\times FH^{2k+1}(M;\Z)\to \Z;\qquad([x],[y])\mapsto \langle [x\cup y],[M]\rangle\]defines a form over $\Z$ that is non-singular over $\Q$. The associated algebraic boundary as defined in \ref{def:algbound} then agrees with the geometric linking form by considering the commutative diagram\[\xymatrix{
FH^{k+1}(M)\ar[d]^-{\cong}_-{g}\ar[r]^-{b}&(H^{k+1}(M))^*\ar[r] \ar[d]^-{\cong}_-{\text{UCT}}&T\\
FH^{k+1}(M;\partial M)\ar[r]&FH^{k+1}(M)\ar[r]& H^{k+1}(\partial M)= TH^{k+1}(\partial M)}\]In this diagram, the lower row comes from the long exact  cohomology sequence of the pair $(M,\partial M)$, UCT denotes the isomorphism coming from the Universal Coefficient Theorem, and $g$ is the composition of isomorphisms\[FH^{k+1}(M)\xrightarrow{\text{UCT}}(H_{k+1}(M))^*\xrightarrow{(-\cap[M])^*} (H^{k+1}(M;\partial M))^*\xrightarrow{\text{UCT}}FH^{k+1}(M;\partial M).\]
\end{example}

\begin{example}Now a more concrete example. Let $U:S^1\hookrightarrow S^3=\partial D^4$ be the standard unknot. Push the unknot off along a trivially framed normal bundle to obtain $l$, the longitude. Choose a 2-disc transversely intersecting the unknot in a single point and define the meridian $m$ as the boundary of this disc choice. Define the Lens space \[L(p,q)=S^1\times D^2\cup_fD^2\times S^1\] by the diffeomorphism of the boundary $f:S^1\times S^1\to S^1\times S^1\cong \partial(\overline{S^3\sm U\times D^2})$, induced from the diffeomorphism\[\C\times \C\to \C\times \C;\qquad (x,y)\mapsto (x^uy^p,x^vy^q),\]where $u,v,p,q\in\Z$, $uq-vp=1$ and $p\neq0$. Suppose the open tubular neighbourhood was small enough so that $l$ and $m$ survived the construction, then the class $[l]$ of our longitude generates $\pi_1(L(p,q))\cong \Z/p\Z$, and note that we have $H_1(L(p,q);\Z)\cong H^2(L(p,q);\Z)\cong \Z/p\Z$. A copy $D\subset L(p,q)$ of  the disc of the glued-in solid torus, i.e.\ the first factor of $D^2\times S^1$, bounds $p$ concatenated longitudes $l$ in $L(p,q)$. Let $l'\subset L(p,q)$ be another homology copy of the longitude such that $l\cap l'=\emptyset$ and $l'$ intersects $D$ transversely. The diffeomorphism $f$ ensures that \[[\partial D]=p[l]+q[m]\in H_1(\partial(\overline{S^3\sm U\times D^2})),\] and so by the definition of $m$, we have that $l'$ meets $D$ in $q$ positive, transverse intersections. Hence\[\text{link}(l,l')=\frac{q}{p}.\]Now define $D'$ to be a 2-disc cobounding $p$ concatenated loops $l'$.

It is a standard fact that every closed, oriented 3-manifold is the boundary of an oriented 4-manifold ($\Omega^{\text{SO}}_3=0$) and hence $L(p,q)$ bounds a 4-manifold $M^4$. Push the 2-discs $D,D'$ into the interior of $M$ (via isotopies rel.\ boundary in general position) and assume the push-ins are in general position and transversely intersecting. By the previous examples, we know that the pushed-in discs intersect $q$ times positively transversely.
\end{example}

\subsection{Linking forms from Poincar\'{e} duality}\label{subsec:linkPD}

Suppose for this subsection that $n=2k+1$ and $r=k+1$, then the cohomology linking pairing is a $(-1)^{k+1}$-symmetric pairing on the module $TH^{k+1}(X;A)$ and the adjoint determines an injective morphism\[\lambda_X:TH^{k+1}(X;A)\hookrightarrow (TH^{k+1}(X;A))^\wedge.\]It is very tempting at this stage to call this a linking form over $(A,S)$, but we caution that the module is not necessarily in $\H(A,S)$.

\begin{definition}Suppose $\overline{X}$ is a $\pi$-cover of a $(2k+1)$-dimensional Poincar\'{e} space $X$ and that $TH^{k+1}(X;A)$ is an object of $\H(A,S)$. Then the \textit{linking form of the cover $\overline{X}$} is the $(-1)^{k+1}$-symmetric linking form over $(A,S)$ given by $(TH^{k+1}(X;A),\lambda_X)$
\end{definition}

One possible reason for the linking form to be defined is entirely algebraic:

\begin{definition}For $m\geq0$, the pair $(A,S)$ has \textit{dimension $m$} if every finitely generated $S$-torsion $A$-module $K$ admits a projective $A$-module resolution of length $m+1$\[0\to P_{m+1}\to P_m\to\dots\to P_1\to P_0\to K\to 0.\]
\end{definition}

Note that if $T$ is a f.g.\ $S$-torsion $A$-module and $(A,S)$ has dimension 0 then $T$ is in $\H(A,S)$. If $A$ has h.d.\ $m+1\geq 1$ then, for any possible $S$, this definition means the localisation pair $(A,S)$ has dimension $m$.

\begin{remark}If $(A,S)$ has dimension $m$ then it is possible for $A$ to have homological dimension greater than $m+1$. For instance $A=\Z\times\Q[x_1,x_2,\dots,x_n]$ has homological dimension $n$ and $(A,\Z\sm\{0\})$ has dimension 0. However we do not know if such examples can arise as group rings $A=R[\pi]$.
\end{remark}

\begin{example}Suppose $\overline{X}$ is a $\pi$-cover of a $(2k+1)$-dimensional Poincar\'{e} space $X$ and that $(A,S)$ defines a localisation of $A=R[\pi]$. If $(A,S)$ has dimension 0 then the linking form of $\overline{X}$ is defined. For example this happens when $A$ itself has homological dimension 1 (e.g.\ a Dedekind domain):
\begin{itemize}
\item If $\pi$ is trivial and $(A,S)=(\Z,\Z\sm\{0\})$ then we recover the classical linking form from the motivation section above\[\lambda_X:TH^{k+1}(X;\Z)\times TH^{k+1}(X;\Z)\to \Q/\Z.\]
\item If $\pi=\Z$ and $R$ is a field then $A=R[z,z^{-1}]$ is a principal ideal domain and \[\lambda_X:TH^{k+1}(X;R[z,z^{-1}])\times TH^{k+1}(X;R[z,z^{-1}])\to R(z)/R[z,z^{-1}].\] This is the field-coefficient case of an \textit{infinite cyclic cover}, which we investigate in the next section.
\end{itemize}
\end{example}

We now make a shallow foray into the problem of determining whether the linking form of a $\pi$-cover is non-singular.  In other words, we are searching for a reason that the adjoint of the linking form\[\lambda_X:TH^{k+1}(X;A)\to \Hom_A(TH^{k+1}(X;A),S^{-1}A/A)\cong\Ext^1_A(TH^{k+1}(X;A),A)\]is an isomorphism. The difficulty of this `universal coefficient problem' explodes when $A$ has homological dimension greater than 1, hence we will not be able to treat many cases.

Following the exposition in \cite{MR0461518}, we now briefly recall the standard spectral sequence governing the universal coefficient problem. Suppose we have a chain complex $(C,d)$ of projective $A$-modules such that $C_r=0$ for $r<0$, with associated cochain complex $(C^*,\delta)$, and that there is fixed an injective $A$-module resolution \[A\to Q_0\xrightarrow{\partial} Q_1\xrightarrow{\partial} Q_2\to\dots\]Then the \textit{universal coefficient spectral sequence} $\{E_r, d_r\}$ is one of the two first-quadrant (cohomology) spectral sequences associated to the bicomplex $K^{p,q}=\Hom_A(C_p,Q_q)$, with total complex $(\Tot K,D)$ where $\Tot^mK=\bigoplus_{p+q=m} K^{p,q}$ and differential $D=\delta+\partial$. By filtering in the two different homological grading `directions', there are 2 graded filtrations of this total complex\[0=F_{m+1}^m\subseteq F^m_m\subseteq \dots\subseteq F_1^m\subseteq F^m_0=\Tot^mK,\qquad\text{where } F^m_k=\bigoplus_{q\geq a}\bigoplus_{p+q=m}K^{p,q}.\]
\[0=\widetilde{F}_{m+1}^m\subseteq \widetilde{F}^m_m\subseteq\dots\subseteq \widetilde{F}_1^m\subseteq \widetilde{F}^m_0=\Tot^mK,\qquad\text{where } \widetilde{F}^m_k=\bigoplus_{p\geq a}\bigoplus_{p+q=m}K^{p,q}.\]
In a standard way \cite[\textsection 14]{MR658304}, these filtrations determine spectral sequences $\{E_r,d_r\}$ and $\{\widetilde{E}_r, \widetilde{d}_r\}$ respectively. If the sequences both converge, then $E^{p,q}_r,\widetilde{E}^{p,q}_r\implies H^{p+q}(\Tot K)$.

The latter filtration has\[\begin{array}{rcll}\widetilde{E}^{p,q}_1=H_{\partial}(K)^{p,q}&=&\Ext^q_A(C_p,A)\quad&\text{(as $Q$ is an injective resolution),}\\
&=&\left\{\begin{array}{lcl}\Hom_A(C_p,A)&&q=0,\\0&&q\neq 0\end{array}\right.\quad&\text{(as $C_p$ is projective)}.\\
&&&\\
\widetilde{E}^{p,q}_2=H_{\delta}H_\partial(K)^{p,q}&=&\left\{\begin{array}{lcl} H^p(C;A)&&q=0,\\0&&q\neq 0.\end{array}\right.\end{array}\]The differential $\widetilde{d}_r$ is of bidegree $(r-1,r)$, so the spectral sequence collapses on the second page and we have $H^m(\Tot K)\cong H^m(C;A)$.

The former filtration has
\[\begin{array}{rcll}E^{p,q}_1=H_{\delta}(K)^{p,q}&=&H^p(\Hom_A(C,Q_q)),&\\
&=&\Hom_A(H_p(C;A),Q_q)\qquad&\text{(as $Q_q$ is injective)}.\\
&&&\\
E^{p,q}_2=H_{\partial}H_\delta(K)^{p,q}&=&\Ext_A^q(H_p(C;A),A)&\text{(as $Q$ is an injective resolution)}.\end{array}\]Now consider that for each $m$, the filtration on $\Tot^mK$ induces \textit{some} filtration on the homology\[0=J_{-1,m+1}\subseteq J_{0,m}\subseteq J_{1,m-1}\subseteq\dots\subseteq J_{m-1,1}\subseteq J_{m,0}=H^m(\Tot K).\]Although it is generally far from obvious what the modules $J_{m-a,a}$ in this filtration are, it is a standard result (see e.g.\ \cite[\textsection 14]{MR658304}) that the subquotients of this filtration are determined by the $E_\infty$ page of the spectral sequence. More precisely, for each $m, a$, there is a short exact sequence\begin{equation}\label{eq:JSES}0\to J_{m-a-1,a-1}\to J_{m-a,a}\to E^{m-a,a}_\infty\to 0,\end{equation}where recall that $E_\infty^{p,q}$ denotes the stationary module in the $(p,q)$ position $E^{p,q}_r=E^{p,q}_{r+1}=\dots$. As a consequence, we always have that $J_{0,m}\cong E^{0,m}_\infty$. More carefully, we can draw similar consequences from vanishing of the $E_\infty^{p,q}$ terms along the diagonal $p+q=m$ as follows. Suppose there is some $t$, such that for $a\geq t$ we have $E^{m-a,a}_\infty=0$, then \[J_{m-t,t}=J_{m-t-1,t+1}=\dots=J_{-1,m+1}=0\qquad\text{and}\qquad J_{m-t+1,t-1}=E^{m-t+1,t-1}_\infty.\]

\begin{proposition}\label{prop:UCSS}Let $\overline{X}$ be a $\pi$-cover of a $(2k+1)$-dimensional Poincar\'{e} space $X$. Suppose one or both of the hypotheses\begin{enumerate}[(H1)]
\item $(A,S)$ has dimension 0 and for $r=0,1,\dots,k$,  $H_r(X;A)$ is $S$-torsion,
\item $A$ is a PID and $S=A\sm\{0\}$,
\end{enumerate}is the case. Then the linking form of the cover is well-defined and non-singular.
\end{proposition}

\begin{proof}For (H1), $H^{k+1}(X;A)\cong H_k(X;A)= TH_k(X;A)\cong TH^{k+1}(X;A)$ and $(A,S)$ is dimension 0, so the linking form is well-defined. Now consider the second page of the universal coefficient spectral sequence\[\arraycolsep=1.4pt\def\arraystretch{2.2}\begin{array}{ccccc}
\textcolor{red}{\Ext^{k+1}_A(H_0(X;A),A)}&\textcolor{red}{\cdots}&\textcolor{red}{\Ext^{k+1}_A(H_{k-1}(X;A),A)}&\textcolor{red}{\Ext^{k+1}_A(H_{k}(X;A),A)}&\\

\textcolor{red}{\vdots}&\textcolor{red}{\ddots}&\textcolor{red}{\vdots}&\textcolor{red}{\vdots}&\\

\textcolor{red}{\Ext^2_A(H_0(X;A),A)}&\textcolor{red}{\cdots}&\textcolor{red}{\Ext^{2}_A(H_{k-1}(X;A),A)}&\textcolor{red}{\Ext^{2}_A(H_{k}(X;A),A)}&\\

\Ext^{1}_A(H_{0}(X;A),A)&\cdots&\Ext^{1}_A(H_{k-1}(X;A),A)&\Ext^{1}_A(H_{k}(X;A),A)&\Ext^1_A(H_{k+1}(X;A),A)\\
\Hom_A(H_0(X;A),A)&\cdots&\Hom_A(H_{k-1}(X;A),A)&\Hom_A(H_{k}(X;A),A)&\Hom_A(H_{k+1}(X;A),A)\\
p=0&&p=k-1&p=k&p=k+1
\end{array}\]

By hypothesis, we have $\Ext^q_A(H_p(X;A),A)=0$ for the highlighted red region. Considering the the differential $d_2$ on this page has bidegree $(-1,2)$, the 0's in this region imply $E_2^{p,q}=E_\infty^{p,q}$ for the range written into the diagram above (but not necessarily for the entries in the $p=k+1$ column that we have left blank as they may be non-zero and be in the target of a differential). Thus the short exact sequence\[0\to J_{k,1}\to J_{k+1,0}\to E^{k+1,0}_\infty\]becomes\begin{equation}\label{eq:UCT}0\to \Ext_A^1(H_k(X;A),A)\to H^{k+1}(X;A)\to \Hom_A(H_{k+1}(X;A),A)\to 0.\end{equation}But the latter group vanishes because $H_{k+1}(X;A)$ is $S$-torsion and so, using also the natural isomorphism of Lemma \ref{lem:ext2} and Poincar\'{e} duality, we obtain an isomorphism \begin{equation}\label{eq:UCT2}\begin{array}{rcl}\Hom_A(TH^{k+1}(X;A);S^{-1}A/A)&\cong&\Ext_A^1(TH^{k+1}(X;A),A)\\&\cong&  \Ext_A^1(H_k(X;A),A)\\&\cong& TH^{k+1}(X;A).\end{array}\end{equation}

Before identifying this isomorphism as the adjoint of the linking form, we turn to (H2). In this case we obtain the exact sequence \ref{eq:UCT} for the same vanishing reasons on the $E_2$ page. This time the sequence is moreover split. This is because we are working over a PID and the central term splits as its free and torsion parts. The first term of the sequence is entirely torsion and the last term in the sequence is free. Hence we obtain exactly the same isomorphism as equation \ref{eq:UCT2} (but for slightly different reasons).

We now show our isomorphism agrees with the adjoint of the linking form. Let $x\in C^{k+1}(X)$ represent a class in $TH^{k+1}(X;A)$. There exists $s\in S$ and $\tilde{x}\in C^k$ such that $d^*(\tilde{x})=sx$ and we will take this to be our choice of element in the preimage of the cohomological Bockstein. We need to show that the cochain $s^{-1}\tilde{x}\in C^k(X;S^{-1}A/A)$ corresponds to $x:C_{k+1}\to A$ under our series of isomorphisms in \ref{eq:UCT2}. But unravelling the definitions of the isomorphisms in the series, this amounts to lifting the former cochain to $s^{-1}\tilde{x}\in C^k(X;S^{-1}A)$ and observing that \[d^*(s^{-1}\tilde{x})=i(x)\in C^{k+1}(X;S^{-1}A),\] where $i$ is the natural map $i:C^{k+1}(X;A)\to C^{k+1}(X;S^{-1}A)$.
\end{proof}

\subsubsection{Blanchfield forms from Poincar\'{e} duality}\label{subsec:blanchdual}

In Proposition \ref{prop:UCSS}, the spectral sequence collapsed because the $E_2$ page was trivial except at $q=0,1$. In the case of (H2), this happened for the entirely algebraic reason of the coefficient choice.
In \cite{MR0461518} Levine considers linking forms coming from a geometric example that are non-singular for a more geometric reason than we have so far considered. (This geometric reason will essentially be the `Alexander duality of a knot', see Chapter \ref{chap:blanchfield} and particularly, \ref{cor:Pacyclic}.)

On the spectral sequence level the example in question results in the $E_2$ page being trivial except when $q=2,3$. In this case, an analysis of the short exact sequences \ref{eq:JSES} returns a short exact sequence\[0\to \Ext^2_A(H_{r-2}(C;A),A)\to H^r(C;A)\to \Ext^1_A(H_{r-1}(C;A),A)\to 0.\]

Now suppose $T$ is an $A=\Z[z,z^{-1}]$-module and that $\Hom_A(T,A)=0$. Set $t(T)$ to be the \textit{$\Z$-torsion} \[t(T)=\ker(T\to \Q[z,z^{-1}]\otimes_A T)\qquad \text{and} \qquad f(T)=T/t(T).\]Note that $f(T)$ may still have torsion with respect to the multiplicative subset $\Z[z,z^{-1}]^\times$.

Set $P$ to be the set of Alexander polynomials. Levine (\cite{MR0461518}) shows that for any module $T$ in $\H(A,P)$\[\begin{array}{rcl}\Ext_A^2(T,A))&\cong& t(T),\\
\Ext_A^1(T,A))&\cong& f(T).\\
\end{array}\]Now suppose $C$ is a chain complex of f.g.\ free $A$-modules, that $P^{-1}A\otimes C$ is acyclic and that there exists a Poincar\'{e} duality isomorphism $H_{r-1}(C;A)\cong H^{n-(r-1)}(C;A)$. Then we obtain an isomorphism\[\begin{array}{rcl}f(H^r(C;A))&\xrightarrow{\cong}& H^r(C;A)/\Ext^2_A(H_{r-2}(C;A),A)\\&\xrightarrow{\cong}&\Ext_A^1(H_{r-1}(C;A,A))\\ &\xrightarrow{\cong}& \Ext_A^1(H^{(n+1)-r}(C;A,A))\cong \Hom_A(f(H^{(n+1)-r}(C;A)),P^{-1}A/A).\end{array}\]This isomorphism is adjoint to the following pairing:

\begin{theorem}[{\cite{MR0461518}}]\label{thm:levblanch}Let $(C,d)$ be a chain complex of f.g.\ free $A=\Z[z,z^{-1}]$-modules equipped with a chain homotopy equivalence $\phi:C^{n-*}\xrightarrow{\simeq} C_*$. Suppose $P^{-1}A\otimes C$ is acyclic. Then the \textit{Blanchfield pairing} is the pairing \[Bl:f(H^r(C;A))\times f(H^{(n+1)-r}(C;A))\to P^{-1}A/A;\qquad (x,y)\mapsto p^{-1}\overline{\tilde{y}(\phi(x))},\]where $x\in C^r$, $y\in C^{(n+1)-r}$, $\tilde{y}\in C^{n-r}$ and $p\in P$ such that $d^*\tilde{y}=py$. If $n=2k+3$ and $r=k+2$, the Blanchfield pairing is a non-singular linking form called the \textit{Blanchfield form}\[(f(H^{k+2}(C;A)),Bl).\] If $C$ is equipped with a `symmetric structure' (see Chapter \ref{chap:algLtheory}) then the linking form is moreover $(-1)^{k}$-symmetric (see Proposition \ref{prop:welldeflagrang2}).
\end{theorem}

\begin{corollary}If $\overline{X}$ is a $\Z$-cover of a $(2k+3)$-dimensional Poincar\'{e} space $X$ and $H^*(X;A)$ is $P$-acyclic, where $P$ is the set of Alexander polynomials, then the restriction to $f(H^{k+2}(X;A))$ of the linking form of the cover is a well-defined, non-singular $(-1)^{k}$-symmetric linking form over $(A,P)$, called the \textit{Blanchfield form} of the cover.
\end{corollary}

It is straightforward to see that if $A\hookrightarrow \mathbb{F}[z,z^{-1}]$ is the change of rings induced by the change of coefficient systems $\Z\hookrightarrow \mathbb{F}=\Q,\R$, the Blanchfield form of a cover (where defined) reduces to the linking form of the cover\[(f(H^{k+2}(C;A)),Bl)\mapsto (TH^{k+2}(\overline{X};\mathbb{F}[z,z^{-1}]),\lambda_X).\]

\section{Autometric forms and monodromy}

For this section, let $A=R[z,z^{-1}]$.

\medskip

As we saw in the last section, linking forms over $(A,S)$ arise from $\Z$-covers of Poincar\'{e} spaces (at least when $R$ is a field). We will show that these linking forms can also be described by the data of an `autometric form', that is form over $R$ together with an automorphism of this form. We show how this structure arises from the geometry using a cut-and-paste construction (this is well-known and we base our explanation on that in \cite{MR544156}). We then analyse the new perspective algebraically.

\subsection{Infinite cyclic covers via cut-and-paste}\label{subsec:infinite}

An \textit{infinite cyclic cover} of a space $X$ is a principal $\Z$-bundle over $X$, in other words it is a covering space for $X$ with group of deck transformations $\Z$. The pair $(X,f:X\to S^1)$ determines an infinite cyclic cover $\overline{X}$ of $X$ as the pullback of the standard infinite cyclic cover of $S^1\subset \C$\[\xymatrix{\overline{X}=f^*\R\ar[r]\ar[d]&\R\ar[d]^-{\exp(2\pi i x)}\\ X\ar[r]^-{f}&S^1}\] Recalling that $S^1$ is an Eilenberg-Maclane space $K(\Z,1)$, the set of homotopy classes of maps $[X,S^1]$ is in 1:1 correspondence with the elements of $H^1(X;\Z)$. Suppose now that $X$ has the homotopy type of a CW complex. Under this assumption, the pullback described above defines a 1:1 correspondence between the the set of homotopy classes of maps $f:X\to S^1$ and the set of isomorphism classes of infinite cyclic covers of $X$. (This last fact is an instance of a general theorem for principal $G$-bundles over a $CW$ complex.)

Now suppose $M$ is a smooth, closed, $n$-manifold with a map $f:M\to S^1$. We will describe the \textit{cut-and-paste} method for building the associated infinite cyclic cover. By Sard's Theorem, there is a point $p\in S^1$ at which $f$ is regular. Pulling this back we obtain an embedded smooth closed oriented submanifold $N^{n-1}=f^{-1}(p)\subset M$. (Observe that $[N]\in H_{n-1}(M)$ is Poincar\'{e} dual to the class $[f]\in H^1(M;\Z)$ determined by $f:M\to S^1$.) We now `cut $M$ open at $N$' to obtain a manifold with boundary.

\begin{figure}[h]\[\def\piccutone{\resizebox{0.5\textwidth}{!}{ \includegraphics{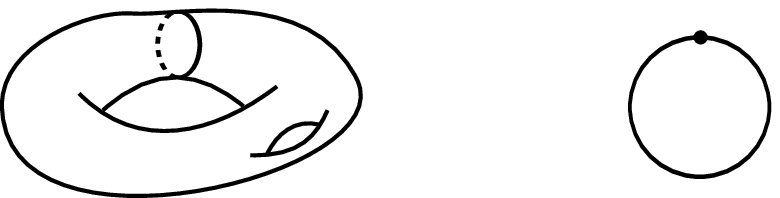}}}
\begin{xy} \xyimport(259,139){\piccutone}
,!+<7.2pc,2pc>*+!\txt{}
,(-15,72)*!L{M}
,(60,160)*!L{N}
,(265,72)*!L{S^1}
,(235,140)*!L{p}
,(130,90)*+{}="A";(205,90)*+{}="B"
,{"A"\ar@/^/"B"}
,(165,65)*!L{f}
\end{xy}\]
\end{figure}

More precisely, take a sufficiently small closed neighbourhood $p\in U\subset S^1$ so that the pullback $f^{-1}(U)$ is diffeomorphic to $N\times [0,1]\subset M$. The closure of $M\sm (N\times [0,1])$ is a manifold with boundary $(M',N_+\sqcup N_-)$ (where there is a diffeomorphism $N_\pm\cong N$). Take countably many copies of this manifold with boundary, labelled $z^r(M',N_+\sqcup N_-)$ for $r\in\Z$. It is now clear that the space obtained by identifying the boundaries of these copies so that $z^rN_+=z^{r+1}N_-$ is a smooth manifold that is isomorphic as an infinite cyclic cover to the pullback of the standard cover of $S^1$, i.e.\ $\overline{M}:=f^*(\R\to S^1)$.

\begin{figure}[h]\[\def\piccuttwo{\resizebox{0.95\textwidth}{!}{ \includegraphics{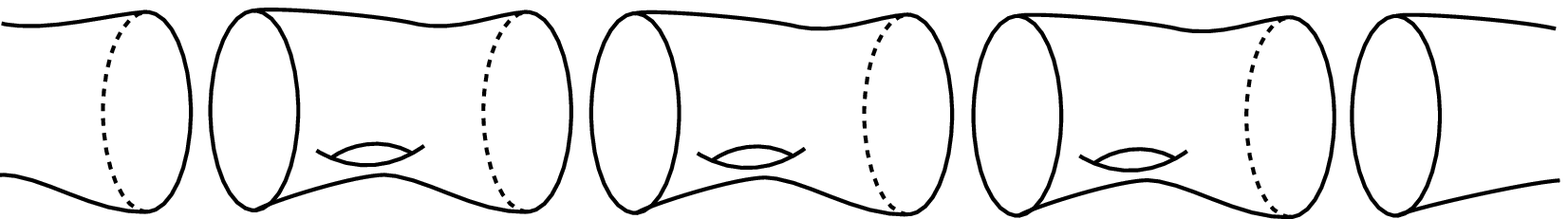}}}
\begin{xy} \xyimport(259,139){\piccuttwo}
,!+<7.2pc,2pc>*+!\txt{}
,(125,90)*!L{M'}
,(187,90)*!L{zM'}
,(57,90)*!L{z^{-1}M'}
,(103,-20)*!L{N_-}
,(150,-20)*!L{N_+}
,(165,-20)*!L{zN_-}
,(210,-20)*!L{zN_+}
,(35,-20)*!L{z^{-1}N_-}
,(80,-20)*!L{z^{-1}N_+}
\end{xy}\]
\end{figure}

\begin{remark}We have described the cut-and-paste method in the smooth category for convenience. The method is still valid in the topological category, although in this category it is extremely technical to describe, so would require too much of a departure to be described in this text. In the topological category we must in particular switch from normal vector bundles to normal microbundles \cite{MR0161346} and appeal to the notion of \textit{topological transversality}. The results of topological transversality were proved by Kirby-Siebenmann for the case of dimension not equal to 4 \cite{MR0645390} and by Quinn \cite{MR929089} in the remaining case of dimension 4.
\end{remark}

Denote by $i^+,i^-:N_\pm\hookrightarrow M'$ the inclusions (also thought of as pushing $N$ in the $\pm$ normal directions). Note that by removing small open neighbourhoods of each translate $z^kN$, $k\in\Z$ in $\overline{M}$, we obtain a Meier-Vietoris sequence\[\dots\to H_r(N;R)\otimes_RA\xrightarrow{\alpha} H_r(M';R)\otimes_RA\xrightarrow{\beta} H_r(\overline{M})\xrightarrow{\delta} H_{r-1}(N;R)\otimes_RA\to\dots\] The maps are $\alpha(u,\lambda)=i^+_*(u)\otimes(z\lambda)-i^-_*(u)\otimes\lambda$, and $\beta$ simply by inclusion.

\medskip

Now we wish to study particularly favourable circumstances (both algebraic or geometric) which result in $H_r(N;R)\cong H_r(M';R)$ and $\delta=0$. For instance, when $M$ is a fibre bundle over $S^1$, with fibre $N$, then this is the case. When $M$ is a knot exterior and $N$ is a Seifert surface then we also obtain these conditions (see Chapter \ref{chap:blanchfield}, particularly \ref{subsec:blanseif}). 

Under these favourable circumstances, there are short exact sequences\[0\to H_r(N;R)\otimes_RA\to H_r(N;R)\otimes_RA\to H_r(\overline{M})\to 0.\]The algebra of this viewpoint allows us to develop a second perspective on the linking forms of infinite cyclic covers. It is to the algebra of this situation that we dedicate the remainder of Chapter \ref{chap:laurent}.

As further motivation for the algebraic viewpoint we will consider in the next few subsections we recall the following example of Neumann \cite{MR544156}. Fix $n=2k+1$ and consider $(M^{2k+1},f)$ and $N^{2k}$ as above. Choose an isomorphism $R\cong H^1(\R;R)$ and define a homology class $[N_0]\in H_{2k}(\overline{M};R)$ as the image of $1\in R$ under the composition\[R\cong H^1(\R;R)\xrightarrow{\overline{f}^*} H^1(\overline{M};R)\xrightarrow{-\cap[M]}H_{2k}(\overline{M};R);\qquad 1\mapsto [N_0].\]There is defined a $(-1)^k$-symmetric pairing over $R$\[b':H^k(\overline{M};R)\times H^k(\overline{M};R)\to R;\qquad\langle x\cup y,[N_0]\rangle.\]In general this pairing is degenerate, but setting $\text{Rad}(b')=\ker(b':H^k(\overline{M};R)\to (H^k(\overline{M};R))^*)$ and defining $(H,b)=(H^k(\overline{M};R)/\text{Rad}(b'),b'|_{H})$ there is the following result.

\begin{lemma}[{\cite[2.1]{MR544156}}]If $R$ is a field then $(H,b)$ is a non-singular $(-1)^k$-symmetric form over $R$ (in particular, $H$ is finitely generated). The deck transformation $z:\overline{M}\to\overline{M}$ induces an automorphism $h:(H,b)\to (H,b)$ \end{lemma}
\begin{definition}The triple $(H,b,h)$ is called the \textit{monodromy} of $(M,f)$.\end{definition}

In fact, over a field, the monodromy and the linking form of the infinite cyclic cover carry the same information up to isomorphism.

\begin{theorem}[{\cite[11.1]{MR544156}}]With coefficients in a field $R$, the algebraic monodromy (definition \ref{def:monodromy}) of the linking form of the infinite cyclic cover determined by $(M,f)$ is isomorphic to the monodromy of $(M,f)$:\[\HH(TH^{k+1}(M;A),\lambda_M)\cong (H,b,h).\]In particular, there is a short exact sequence\[0\to H[z,z^{-1}]\xrightarrow{h-z} H[z,z^{-1}]\to TH^{k+1}(M;R[z,z^{-1}])\to 0.\]
\end{theorem}

As a corollary, in the next section we will show that a type of reverse of this theorem is also true for algebraic reasons. Specifically, given the monodromy of $(M,f)$ we can recover the linking form of the cover up to isomorphism in an entirely algebraic way.

\subsection{Algebraic viewpoint}

For the remainder of this subsection, set \[(A,S)=(R[z,z^{-1}],Q),\] where recall $Q$ is the set of characteristic polynomials.

\medskip

First we will set up a precise language for comparing $R$-modules to $R[z,z^{-1}]$-modules (this is essential - failure to do so leads to much confusion!). Write the natural ring morphism\[i:R\to R[z,z^{-1}].\]This induces two functors:\[\begin{array}{ccrcl}
\textit{induction}&\qquad&i_!:\text{$R$-modules}&\to&\text{$R[z,z^{-1}]$-modules,}\\
&&K&\mapsto&K[z,z^{-1}]=K\otimes_RR[z,z^{-1}].\\
&&&&\\
\textit{restriction}&&i^!:\text{$R[z,z^{-1}]$-modules}&\to&\text{$R$-modules,}\\
&&L&\mapsto& L\quad\text{(forget the action of $z$).}\end{array}\]It can be shown that induction and restriction are adjoint with respect to the $\Hom$ functor in the sense that if $K$ is a f.g.\ $R$-module and $L$ is a f.g.\ $R[z,z^{-1}]$-module, then there is a group isomorphism\[\Hom_{R[z,z^{-1}]}(i_!K,L)\cong \Hom_R(K,i^!L).\]Both induction and restriction preserve the property of being projective. Now, whenever $L$ is an $R[z,z^{-1}]$-module, consider the action of $z$ on $L$ as a morphism of $R[z,z^{-1}]$-modules, and set \[\zeta(L)=i^!(z):i^!L\to i^!L.\]So instead of just forgetting the action of $\Z=\langle z\rangle$ on $L$, we remember it as an $R$-module morphism after the restriction.

\begin{example}Taking $R$ and $R[z,z^{-1}]$ as modules over themselves:\[i_!R=R[z,z^{-1}]\qquad \text{and}\qquad i^!R[z,z^{-1}]=\bigoplus_{-\infty}^{\infty}R,\]with one $R$-summand for each $z^{k}R$. Indeed generally, induction preserves finite generation but restriction turns a finitely generated module into a (countably) infinitely generated one.
\end{example}

\begin{lemma}\label{lem:MV}Suppose $K$ is in $\A(R)$ and $h:K\to K$ is an automorphism. Then $K$ admits the structure of a f.g.\ $R[z,z^{-1}]$-module of homological dimension 1.
\end{lemma}

\begin{proof}Regard $K$ as an $A$-module denoted $K_A$ by defining the action of $z$ on $K_A$ as \[z:K_A\to K_A;\qquad x\mapsto h(x).\]Next, define a morphism of $A$-modules $f:i_!K\to K_A$ by sending $\sum_{M}^Nx_kz^k\in i_!K$ (where the coefficients $x_k\in K$) to $\sum_M^Nh^k(x_k)$. There is then a resolution of $K_A$ by f.g.\ projective $A$-modules\[\xymatrix{0\ar[r]&i_!K\ar[r]^-{z-h}&i_!K\ar[r]^-{f}&K_A\ar[r]&0.}\]
\end{proof}

When does restriction result in something finitely generated?

\begin{claim}If $R$ is an integral domain and $L$ is a f.g.\ projective $A$-module, then $i^!L$ is in $\A(R)$ (i.e.\ is f.g.\ projective) if and only if $L$ is in $\H(A,Q)$ i.e.\ if $L$ is $Q$-torsion and admits a length 1 resolution by projective $A$-modules.
\end{claim}

\begin{proof}This is a standard idea from algebraic number theory. Assume $L$ is generated by $g_1,\dots, g_n$. 

Now suppose $i^!L$ is in $\A(R)$. As $\zeta(L)$ is an automorphism of $i^!L$, we apply Lemma \ref{lem:MV} (and using the terminology in the proof of that lemma) we have that $(i^!L)_A$ is a f.g.\ $A$-module of homological dimension 1. Moreover, in this case we can identify $(i^!L)_A=L$. We must show $L$ is $Q$-torsion. As $i^!L$ is f.g.\ there exists a positive integer $N$ such that $\zeta^{N+1}(g_1)\in\langle \zeta^kg_1,\dots,\zeta^kg_n\,|\,k=-N,\dots,N\rangle$. But furthermore, we must have $\zeta^{N+1}(g_1)=\sum_{-N}^Na_k\zeta^kg_1$. Therefore\[p_1(z):=z^{N+1}-\sum_{-N}^Na_kz^k\in\text{Ann}(g_1)\subseteq L.\]Repeating for the other generators, form the Laurent polynomial $p'(z)=p_1(z)p_2(z)\dots p_n(z)$. Now there exists $M>0$ so that there is a bionic polynomial\[p(z)=z^Mp'(z)=\sum_{k=0}^Lb_kz^k\in R[z]\qquad\text{and}\qquad b_0\neq0.\]

Conversely, if $L$ is in $\H(A,Q)$ then it is annihilated by some degree $N+1$ monic polynomial $p(z)$ with constant term a unit in $R$ (i.e.\ a bionic polynomial). Thus, the basis for $i^!L$ given by $\{\zeta^kg_l\,|\,k\in\Z,l=1,\dots,n\}$ can be reduced to a finite one.
\end{proof}

Now we add structure to projective modules over $R$ so that we may use the restriction without losing information. Firstly, define a category\[\Aut(R)=\{\text{Pairs $(K,h)$ where $K$ is in $\A(R)$ and $h:K\to K$ an automorphism.}\}\]with morphisms $f:(K,h)\to (K',h')$ such that $f:K\to K$ is a morphism in $\A(R)$ and $fh=h'f$. The \textit{dual} of $(K,h)$ is given by $(K^*,(h^*)^{-1})$.

In light of which, define a functor of additive categories\[\textit{monodromy}\qquad\HH:\H(A,Q)\to \Aut(R);\qquad T\mapsto (i^!T,\zeta(T)).\]We also define a functor going the other way, one that takes an object $(K,h)$ of $\Aut(R)$ and returns an object $T$ of $\H(A,Q)$. Consider the sequence\[0\to i_!K\xrightarrow{z-h}i_!K\to T:=\coker(z-h)\to 0.\]\begin{claim} $T$ is in $\H(A,Q)$.
\end{claim}
\begin{proof}It is clear that T is a f.g.\ $A$-module and has homological dimension 1. We claim it is also $Q$-torsion. To see this, observe that as $K$ is f.g.\ as an $R$-module and that $h$ is an endomorphism of $K$ there exists a monic polynomial\[p_h(z)=\det(z-h:i_!K\to i_!K),\]called the \textit{characteristic polynomial of $h$}. The constant coefficient of the characteristic polynomial is also called the determinant of $h$. As $h$ is an automorphism, the determinant is a unit in $R$ and hence $p_h(z)\in Q$. Now we apply the Cayley-Hamilton Theorem for automorphisms of f.g.\ projective modules (see \cite[11.12]{MR1713074}) to conclude that $p_hT=0$.
\end{proof}

Putting all this together, there is a functor of additive categories\[\textit{covering}\qquad B:\Aut(R)\to \H(A,Q);\qquad K\mapsto T.\]Combining the considerations so far in this section, we have proved:

\begin{proposition}[{\cite[13.16(i)]{MR1713074}}]\label{prop:covmonod1}The covering and monodromy operations are functors defining an equivalence of categories\[\xymatrix{\H(A,Q)\ar@<1ex>[rr]^-\HH&\,&\Aut(R).\ar@<1ex>[ll]^-{\quad\quad B\quad\quad}}\] 
\end{proposition}

\subsubsection*{Covering and monodromy for forms}

For the most part, everything so far was just language and basic algebraic number theory. We now want to extend our covering and monodromy category equivalences to the setting of forms, where some actual work will need to be done as the duality for the Autometric category does not quite match the duality in $\H(A,Q)$ in the way one might naively expect. We begin by defining a form in the context of $\Aut(R)$.

\begin{definition}An \textit{$\eps$-symmetric autometric form} over $R$ is a triple $(K,\theta, h)$, where $(K,\theta)$ is an $\eps$-symmetric form over $R$ and $h:(K,\theta)\to (K,\theta)$ is an automorphism. $(K,\theta,h)$ is called \textit{non-singular} if $(K,\theta)$ is non-singular.
\end{definition}

We now describe how to extend the covering and monodromy operations to the full data of forms. First we describe the covering of an autometric form.

\begin{definition}
The \textit{covering} of a (non-singular) $\eps$-symmetric autometric form $(K,h,\theta)$ over $R$ is the (non-singular) $(-\eps)$-symmetric linking form $B(K,h,\theta)=(T,\lambda)$ over $(A,Q)$ given by $T=B(K,h)$ and with $\lambda$ induced by the following chain map of resolutions of $T$

\[\xymatrix{0\ar[r]&K[z,z^{-1}]\ar[rrr]^{z-h}\ar[d]^{\theta}&&&K[z,z^{-1}]\ar[r]\ar[d]^-{\theta}&B(K,h)=T\ar[r]\ar[d]_\cong^{B(\theta)}&0\\
0\ar[r]&K^*[z,z^{-1}]\ar[d]^{(h^{-1})^*}\ar[rrr]^{z-(h^*)^{-1}}&&&K^*[z,z^{-1}]\ar[d]^{-z^{-1}}\ar[r]&B((K,h)^*)\ar[r]\ar[d]_\cong^\xi&0\\
0\ar[r]&K^*[z,z^{-1}]\ar[d]_\cong^{F}\ar[rrr]^{-(z-(h^*)^{-1})}&&&K^*[z,z^{-1}]\ar[d]_\cong^{F}\ar[r]&B((K,h)^*)\ar[r]\ar[d]_\cong&0\\
0\ar[r]&(K[z,z^{-1}])^*\ar[rrr]^{-(z-h)^*}&&&(K[z,z^{-1}])^*\ar[r]&\Ext^1_A(T,A)\cong T^\wedge\ar[r]&0}\]


where the natural isomorphism $F$ is defined in Claim \ref{clm:funnyF}.  This linking form has associated pairing \[T\times T\to Q^{-1}A/A;\qquad ([x],[y])\mapsto -z^{-1}\theta(x,(z-h)^{-1}(y))\]where $(z-h)^{-1}(y)\in Q^{-1}A\otimes K[z,z^{-1}]$ is a choice of element in the preimage \textit{after localisation} (otherwise $z-h$ is not invertible).
\end{definition}

\begin{remark}Why is this the correct definition to take? One explanation requires methodology we haven't yet developed, but we mention this now anyway for future reference. We are using $\theta$ and $h$ to construct a 1-dimensional $\eps$-symmetric $Q$-acyclic Poincar\'{e} complex (see Proposition \ref{prop:correspondence}) with $C_0=C_1=K^*[z,z^{-1}]$ and \[\begin{array}{rrcl}\phi_0:&C^0&\xrightarrow{\theta h}& C_1,\\
\phi_0:&C^1&\xrightarrow{-z^{-1}\theta}& C_0,\\
\phi_1:&C^1&\xrightarrow{-\theta} &C_1.\end{array}\]The signs here are coming from the convention that the $\phi_0$ must be a chain map and $d_{C^{n-*}}=(-1)^r(d_C)^*$ (see conventions in Chapter \ref{chap:algLtheory}). The linking form we get is simply the linking form of this $\eps$-symmetric Poincar\'{e} complex (see Proposition \ref{prop:correspondence}).
\end{remark}

%
%

To extend our monodromy $\HH:\H(A,Q)\to \Aut(R)$ to the full data of a linking form, we will need the concept of a \textit{trace function}. For details of the trace function in this context, see Appendix \ref{chap:trace}. The trace function is a homomorphism of $R$-modules\[\chi:Q^{-1}A/A\to R\]with the properties\begin{enumerate}
\item For any $T$ in $\H(A,Q)$ there is naturally defined an isomorphism (of $R$-modules) \[\chi_*:\Hom_A(T,Q^{-1}A/A)\to \Hom_R(i^!T,R).\]
\item For any $x\in Q^{-1}A/A$, we have $\chi(\overline{x})=\overline{\chi(x)}$.
\end{enumerate}

\begin{definition}\label{def:monodromy}If $(T,\lambda)$ is a (non-singular) $\eps$-symmetric linking form over $(A,Q)$, the \textit{monodromy} is the $(-\eps)$-symmetric  autometric form $\HH(T,\lambda)=(i^!T,\zeta(T),\theta)$ over $R$, with $\theta$ given by the composition\[i^{!}T\xrightarrow{B}T\xrightarrow{\lambda}T^\wedge\xrightarrow{\chi_*}(i^!T)^*.\]
\end{definition}

\begin{proposition}[{\cite[28.27(iii)]{MR1713074}}]\label{prop:covmonod2}The functors $B$ and $\HH$ (now considered on the full data of the forms) are inverses of one another and define a natural isomorphism of commutative monoids (the monoid operation is via direct sum) \[\xymatrix{\left\{\text{(non-singular) $\eps$-symmetric linking forms over $(A,Q)$}\right\}\ar@<1ex>[d]^\HH\\ \left\{\text{(non-singular) $(-\eps)$-symmetric autometric forms over $R$}\right\}.\ar@<1ex>[u]^B}\]
\end{proposition}

\begin{proof}Both $B$ and $\HH$ are natural morphisms of the underlying categories and each preserve the monoid structure and extend the equivalence of \ref{prop:covmonod1}. Moreover, they are inverses of each other up to isomorphism as shown in Proposition \ref{prop:traceisgood}.
\end{proof}

\begin{definition}A \textit{(split) lagrangian} for a non-singular $\eps$-symmetric autometric form $(K,h,\theta)$ over $R$ is an $\Aut(R)$-submodule $j : (L,h_L) \hookrightarrow (K,h)$ such that the sequence\[0\to (L,h_L)\xrightarrow{j}(K,h)\xrightarrow{j^*\theta} (L,h_L)^*\to0\]is (split) exact in $\Aut(R)$. If $(K,h,\theta)$ admits a (split) lagrangian it is called \textit{(split) metabolic}. If $(K,h,\theta)$ admits two lagrangians $j_\pm:(L_\pm,h_\pm)\hookrightarrow (K,h)$ such that they are complementary as $\Aut(R)$-submodules\[\left(\begin{matrix}j_+\\j_-\end{matrix}\right):(L_+,h_+)\oplus (L_-,h_-)\xrightarrow{\cong} (K,h),\]then the form is called \textit{hyperbolic}. The $\eps$-symmetric \textit{Witt} and \textit{double Witt} groups of non-singular autometric forms over $R$ are defined as usual and respectively written\[WAut^\eps(R),\qquad DWAut^\eps(R).\]For the double Witt group to be well-defined, we must, for the usual reasons, assume that there is a half-unit $s\in R[z,z^{-1}]$.
\end{definition}

\begin{remark}As modules in the category $\A(A)$ are projective, all surjective morphisms split, and a lagrangian is always a direct summand. However in the category $\Aut(R)$, such a splitting on the module level might not respect the automorphisms. Hence the possible existence of non-split lagrangians in this setting.\end{remark}

\begin{lemma}\label{lem:algtrans2}If $j:(L,h,0)\hookrightarrow (K,h,\theta)$ is (split) lagrangian of a non-singular $\eps$-symmetric autometric form then\[B(L,h,0)\hookrightarrow B(K,h,\theta)=(T,\lambda)\]is a (split) lagrangian.
\end{lemma}

\begin{proof}Entirely analogous to the proof of Proposition \ref{algtrans} so we defer the proof until then.
\end{proof}

\begin{theorem}The covering functor $B$ defines isomorphisms of groups\[\begin{array}{rcl}WAut^{\eps}(R)&\xrightarrow{B}&W^{-\eps}(A,Q),\\ DWAut^{\eps}(R)&\xrightarrow{B}&DW^{-\eps}(A,Q),\end{array}\]each with inverse defined by the monodromy functor $\HH$.
\end{theorem}

\begin{proof}This follows from Proposition \ref{prop:covmonod2} and Lemma \ref{lem:algtrans2}.
\end{proof}

\section{Blanchfield and Seifert forms}

We now extend the algebraic correspondence of the previous section to the case that arises from knot theory. We will prove the double Witt group of Seifert forms is isomorphic to the double Witt group of Blanchfield forms. The duality in the linking forms is now called \textit{Blanchfield duality} and will correspond to the \textit{Seifert duality} in forms with endomorphism via a covering construction as before. Reversing the covering construction is no longer so easy and we will rely more on our references than before.

The algebraic language we use for Seifert/Blanchfield correspondence is based on the paper \cite{MR2058802} and we recap this in the necessary detail.

\subsection{Seifert forms}

The following definitions are due to Ranicki \cite{MR2058802}. Suppose $R$ is a commutative Noetherian ring with involution and unit.

\begin{definition}A \textit{Seifert $R$-module} is a pair $(K,e)$ where $K$ is a f.g.\ projective $R$-module and $e:K\to K$ is a morphism. The (additive) category of Seifert $R$-modules has objects $(K,e)$ and morphisms $g:(K,e)\to (K',e')$ such that $g:K\to K'$ an $R$-module morphism with $e'g=ge$. The \textit{dual} of a Seifert $R$-module $(K,e)$ is the Seifert $R$-module $(K^*,1-e^*)$. The \textit{dual} of a morphism $g$ of Seifert $R$-modules is a morphism $g^*:(K^*,1-e^*)\to ((K')^*,1-e'^*)$.
\end{definition}

\begin{definition}An \textit{$\eps$-symmetric Seifert form} $(K,e,\psi)$ over $R$ is a morphism of Seifert $R$-modules\[\psi:(K,e)\to(K^*,1-e^*)\]such that $\psi=(\psi+\eps\psi^*)e$. The form $(K,e,\psi)$ is called \textit{non-singular} if $\psi+\eps\psi^*:K\to K^*$ is an isomorphism, in which case there is redundancy in the data $(K,e,\psi)$ and we may specify the form by $(K,\psi)$ and recover $e=(\psi+\eps\psi^*)^{-1}\psi$ (note that this gives $1-e=\eps(\psi+\eps\psi^*)^{-1}\psi^*$). 

A morphism of Seifert forms $g:(K,e,\psi)\to(K',e',\psi')$ is a morphism of Seifert $R$-modules $g:(K,\psi)\to(K',\psi')$ such that $g^*\psi'g=\psi$.
\end{definition}

\begin{remark}
If $(K,e,\psi)$ is an $\eps$-symmetric Seifert form over $R$ then $(K,\psi)$ is an $\eps$-quadratic form over $R$ (see \cite{MR560997} for definition) and the symmetrisation $(K,\psi+\eps\psi^*)$ is the associated $\eps$-symmetric form over $R$.
\end{remark}

From now on, when $\eps$-symmetric Seifert forms over $R$ are non-singular, we will often specify them by $(K,\psi)$.

\begin{definition}A \textit{(split) lagrangian} for a non-singular $\eps$-symmetric autometric form $(K,\psi)$ over $R$ is a Seifert submodule $j : (L,e_L) \hookrightarrow (K,e_K)$ such that the sequence in the category of Seifert modules\[0\to (L,e_L)\xrightarrow{i}(K,e_K)\xrightarrow{j^*(\psi+\eps\psi^*)} (L,e_L)^*\to0\]is (split) exact. If $(K,\psi)$ admits a (split) lagrangian it is called \textit{(split) metabolic}. If $(K,\psi)$ admits two lagrangians $j_\pm:(L_\pm,e_\pm)\hookrightarrow (K,e)$ such that they are complementary as Seifert submodules\[\left(\begin{matrix}j_+\\j_-\end{matrix}\right):(L_+,e_+)\oplus (L_-,e_-)\xrightarrow{\cong} (K,e),\]then the form is called \textit{hyperbolic}.
\end{definition}

\begin{proposition}[{\cite[Proof of Prop.\ 2.1]{MR1004605}}]A Seifert submodule described by $f:(L,e_L)\to(K,e_K)$ is a (split) lagrangian of $(K,e_K,L)$ if and only if the sequence of Seifert modules \[\xymatrix{0\ar[r] & (L,e_L)\ar[r]^{f} &(K,e_K)\ar[r]^-{f^*\psi} & (L,e_L)^*\ar[r] & 0}\]is (split) exact.
\end{proposition}

Consequently we could have defined the notions of (split) lagrangian, (split) metabolic and hyperbolic using this characterisation.

\begin{lemma}\label{doublyslice}For any non-singular $\eps$-symmetric Seifert form $(K,\psi)$, the form $(K\oplus K,\psi\oplus-\psi)$ is hyperbolic.
\end{lemma}

\begin{proof}The Seifert submodules\[\begin{array}{lccrcl}(&1&1&)&:&(K,e,0)\to (K\oplus K,e\oplus e,\psi\oplus-\psi)\\(&(1-e)&-e&)&:&(K,e,0)\to (K\oplus K,e\oplus e,\psi\oplus-\psi)\end{array}\]are complementary. The diagonal submodule is well known to be a lagrangian. To see that the latter submodule is a lagrangian, consider the associated sequence:\[\xymatrix{0\ar[r] & K\ar[rr]^-{\lmat 1-e\\-e\rmat} && K\oplus K\ar[rrrr]^-{\lmat (1-e)^*(\psi+\eps\psi^*)&& e^*(\psi+\eps\psi)\rmat} &&&& K^*\ar[r] &0}\]It is easy to see that the sequence is exact at $K$ and $K^*$. To see it is exact at $K\oplus K$, note that if $(a,b)\in\ker((1-e)^*(\psi+\eps\psi^*)\,\,\,\,e^*(\psi+\eps\psi))$ then, substituting the fact that $e=(\psi+\eps\psi^*)^{-1}\psi$, we have \begin{eqnarray*}\psi a+\eps\psi^*b=0&\implies&b=(\psi+\eps\psi^*)^{-1}\psi(b-a),\\&\implies&b=-e(a-b).\end{eqnarray*}A similar calculation shows $(a,b)=((1-e)(a-b),-e(a-b))$. To complete the proof, a simple calculation shows that ${\lmat 1-e\\-e\rmat}{\lmat (1-e)^*(\psi+\eps\psi^*)&& e^*(\psi+\eps\psi)\rmat}=0$.\end{proof}

\begin{remark}Note that we \textit{did not} need a half-unit in the ring to prove Lemma \ref{doublyslice}.
\end{remark}

\begin{definition}
The monoid construction (see Definition \ref{def:monoid})\[\widehat{W}_\eps(R):=\{\text{non-sing, $\eps$-symm.\ Seifert forms over $R$}\}/\{\text{metabolic}\}\]is an abelian group called the \textit{$\eps$-symmetric Witt group of Seifert forms over $R$}. The monoid construction \[\widehat{DW}_\eps(R):=\{\text{non-sing, $\eps$-symm Seifert forms over $R$}\}/\{\text{hyperbolic}\}\]is an abelian group called the \textit{$\eps$-symmetric double Witt group of Seifert forms over $R$}.
\end{definition}

\subsection{Blanchfield forms}

Recall that $P\subset R[z,z^{-1}]$ is the set of Alexander polynomials.

\begin{definition}A \textit{Blanchfield $R[z,z^{-1}]$-module} $B$ is an object of $\H(R[z,z^{-1}],P)$.
\end{definition}

\begin{lemma}[{\cite[10.21(iv)]{MR1713074}}]\label{halfunit}There is an equivalence of exact categories\[\H(R[z,z^{-1}],P)\xrightarrow{\cong} \H(R[z,z^{-1},(1-z)^{-1}],P).\]Under this equivalence, a Blanchfield $R[z,z^{-1}]$ module $T$ is a homological dimension 1, f.g.\ $R[z,z^{-1}]$-module $T$ such that $1-z:T\to T$ is an isomorphism (where we consider `$z$' to be the endomorphism of $T$ given by multiplication by $z\in R[z,z^{-1}]$). 
\end{lemma}

\begin{proof}[Proof (sketch)]The inclusion\[(R[z,z^{-1}],P)\to(R[z,z^{-1},(1-z)^{-1}],P)\]is a cartesian morphism (see Definition \ref{def:cartmorph}). The proof of this essentially follows from the fact that the set $P$ is coprime to the element $(1-z)$. Now apply Theorem \ref{thm:cartmorph}.
\end{proof}

\begin{remark}The reason that this is such an excellent idea is that $(1-z)^{-1}\in R[z,z^{-1},(1-z)^{-1}]$ is a half unit. Lemma \ref{halfunit} allows us to formally adjoin the half-unit $(1-z)^{-1}$ without affecting the category or indeed the (double) Witt group of linkings (as we shall see).

\end{remark}

\begin{definition}The \textit{covering} of a Seifert $R$-module $(K,e)$ is the Blanchfield $R[z,z^{-1}]$-module \[B(K,e):=\coker((1-e)+ez:K[z,z^{-1}]\to K[z,z^{-1}]).\](To see that $1-z$ is an automorphism of $B(K,e)$ we note the equivalent fact that $B(K,e)$ vanishes under the augmentation $R[z,z^{-1}]\to R$ sending $z\mapsto 1$.) The covering operation is a functor of additive categories \[B:\left\{\text{Seifert $R$-modules}\right\}_{\text{\big{/iso}}}\to \left\{\text{Blanchfield $R[z,z^{-1}]$-modules}\right\}_{\text{\big{/iso}}}\]
\end{definition}

\begin{lemma}[{\cite[2.3]{MR2058802}}]\label{lem:kerofB}If $B(K,e)=0$ then the the morphism $e(1-e):K\to K$ is nilpotent and we say $(K,e)$ is a \textit{near projection}. This is equivalent to the condition that\[(K,e)\cong (K_+,e_+)\oplus (K_-,e_-),\]with $e_-$, $1-e_+$ both nilpotent.
\end{lemma}

\begin{definition}A  \textit{(non-singular) $\eps$-symmetric Blanchfield form} over $R$ is a (non-singular) $\eps$-symmetric linking form over $(R[z,z^{-1}],P)$.
\end{definition}

\begin{definition}\label{def:coveringsief}
The \textit{covering} of a (non-singular) $\eps$-symmetric Seifert form $(K,e,\psi)$ over $R$ is the (non-singular) $(-\eps)$-symmetric Blanchfield form $B(K,e,\psi)=(T,\lambda)$ over $R$ given by $T=B(K,e)$ and\[\lambda: T\times T\to P^{-1}R[z,z^{-1}]/R[z,z^{-1}];\qquad ([x],[y])\mapsto-(1-z^{-1})(\psi+\eps\psi^*)(x,((1-e)+ez)^{-1}(y))\]so that $\lambda$ is resolved by the following chain map\[\xymatrixcolsep{4pc}\xymatrix{
0\ar[r]&K[z,z^{-1}]\ar[rr]^{(1-e)+ez}\ar[d]^{(1-z)(\psi+\eps\psi^*)}&&K[z,z^{-1}]\ar[r]\ar[d]^-{-(1-z^{-1})(\psi+\eps\psi^*)}&T\ar[r]\ar[d]^-{\lambda}&0\\
0\ar[r]&K^*[z,z^{-1}]\ar[rr]^{((1-e)+ez)^*}&&K^*[z,z^{-1}]\ar[r]&T^\wedge\ar[r]&0}\]and the $(1-z)$ factor forces $(-\eps)$-symmetry.
\end{definition}

\begin{remark}Another way to look at this is to note that, although the covering of $(K^*,1-e^*)$ is isomorphic to the covering of $(K,e)$, it is not the dual Blanchfield module (resolved by the dual chain). We have explicitly forced this situation:\[\xymatrix{
0\ar[r]&K[z,z^{-1}]\ar[rrr]^{(1-e)+ez}\ar[d]^{\psi+\eps\psi^*}&&&K[z,z^{-1}]\ar[r]\ar[d]^-{\psi+\eps\psi^*}&B(K,e)\ar[r]\ar[d]^\alpha&0\\
0\ar[r]&K^*[z,z^{-1}]\ar[d]^z\ar[rrr]^{(1-(1-e^*))+(1-e^*)z}&&&K^*[z,z^{-1}]\ar[d]^1\ar[r]&B(K^*,1-e^*)\ar[r]\ar[d]^\beta&0\\
0\ar[r]&K^*[z,z^{-1}]\ar[rrr]^{((1-e)+ez)^*}&&&K^*[z,z^{-1}]\ar[r]&B(K,e)^\wedge\ar[r]&0}\] so that \[\lambda=-(1-z^{-1})\circ\beta\circ\alpha:T\to T^\wedge.\]The added $-(1-z^{-1})$ ensures $\lambda$ defines a $(-\eps)$-symmetric form and not a $(\eps z^{-1})$-symmetric form. Compare with the `modified' Blanchfield form of \cite[5.6]{MR0358796}.
\end{remark}

\begin{theorem}[{\cite[1.8(i)]{MR2058802}}] Up to isomorphism, every Blanchfield module $T$ admits a resolution \[0\to K[z,z^{-1}]\xrightarrow{(1-e) +ez} K[z,z^{-1}]\to T\to 0,\] where $(K,e)$ is a Seifert $R$-module i.e. $T$ is the covering of $(K,e)$.

{\cite[3.10]{MR2058802}} If $(T,\lambda)$ is an $\eps$-symmetric Blanchfield form, then the Seifert module $(K,e)$ above can be improved so that $(T,\lambda)$ is the covering of a $(-\eps)$-symmetric Seifert form $(K',e',\psi')$.\end{theorem}

\begin{lemma}\label{lagrangians}The covering functor sends a lagrangian of a non-singular Seifert form $j:(L,e_L,0)\to(K,e_K,\psi)$ to a lagrangian of the covering Blanchfield form\[B(j):(B(L,e_L),0)\to B(K,e_K,\psi)=(T,\lambda).\]
\end{lemma}

\begin{proof}We wish to show that the sequence\[0\to B(L,e_L)\xrightarrow{B(j)} T\xrightarrow{B(j)^\wedge\circ\lambda} B(L,e_L)^\wedge\to 0\] is exact. This sequence is resolved by the commutative diagram\[\xymatrix{
0\ar[r]&L[z,z^{-1}]\ar[rrr]^{(1-e)+ez}\ar[d]^{j}&&&L[z,z^{-1}]\ar[r]\ar[d]^-{j}&B(L,e)\ar[r]\ar[d]^{B(j)}&0\\
0\ar[r]&K[z,z^{-1}]\ar[rrr]^{(1-e)+ez}\ar[d]^{\psi+\eps\psi^*}&&&K[z,z^{-1}]\ar[r]\ar[d]^-{\psi+\eps\psi^*}&B(K,e)\ar[r]\ar[d]_\cong^{B(\psi+\eps\psi^*)}&0\\
0\ar[r]&K^*[z,z^{-1}]\ar[d]^{(1-z)}\ar[rrr]^{e^*+(1-e^*)z}&&&K^*[z,z^{-1}]\ar[d]^{-(1-z^{-1})}\ar[r]&B(K^*,1-e^*)\ar[r]\ar[d]_\cong^\xi&0\\
0\ar[r]&K^*[z,z^{-1}]\ar[d]^{j^*}\ar[rrr]^{((1-e)+ez)^*}&&&K^*[z,z^{-1}]\ar[d]^{j^*}\ar[r]&B(K,e)^\wedge\ar[r]\ar[d]^{B(j)^\wedge=B(j^*)}&0\\
0\ar[r]&L^*[z,z^{-1}]\ar[rrr]^{((1-e)+ez)^*}&&&L^*[z,z^{-1}]\ar[r]&B(L,e)^\wedge\ar[r]&0}\]where $\xi$ is defined to make the diagram commute. As the sequence of $A$-modules \[0\to L\xrightarrow{j} K\xrightarrow{j^*\circ(\psi+\eps\psi^*)} L^*\to 0\]is exact, we need only show that $\ker(B(j^*))=\ker(B(j^*)\circ\xi)$ so that inserting the isomorphism $\xi$ preserves exactness of the sequence\[0\to B(L,e)\xrightarrow{B(j)} B(K,e) \xrightarrow{B(j^*)\circ B(\psi+\eps\psi^*)} B(L,e)\to 0.\]But $j^*$ commutes with the maps $(1-z)$ and $-(1-z^{-1})$, so the result follows.
\end{proof}

\begin{theorem}\label{algtrans}$\,$
\begin{enumerate}[(i)]\item \underline{Excision isomorphisms:} The inclusion $i:(R[z,z^{-1}],P)\to(R[z,z^{-1},(1-z)^{-1}],P)$ induces isomorphisms of Witt groups
\begin{eqnarray*}W^\eps(R[z,z^{-1}],P)&\cong&W^\eps(R[z,z^{-1},(1-z)^{-1}],P),\\ DW^\eps(R[z,z^{-1}],P)&\cong&DW^\eps(R[z,z^{-1},(1-z)^{-1}],P).\end{eqnarray*}These isomorphisms are referred to as `excision' isomorphisms because of the formal resemblance to geometric excision.
\item \underline{Covering isomorphisms:} The covering operation induces isomorphisms of Witt groups
\begin{eqnarray*}\widehat{W}_\eps(R)&\cong&W^\eps(R[z,z^{-1}],P),\\ \widehat{DW}_\eps(R)&\cong&DW^\eps(R[z,z^{-1}],P).\end{eqnarray*}
\end{enumerate}
\end{theorem}

\begin{proof}To ease notation, write $R[z,z^{-1}]=R_z$, $R[z,z^{-1},(1-z)^{-1}]=R_{z,1-z}$
\begin{enumerate}[(i)]
\item Suppose $M,N$ are $(R_z,P)$-modules. The cartesian morphism $i$ induces an isomorphism of exact categories with involution $\H(R_z,P)\cong \H(R_{z,1-z},P)$ and there are induced isomorphisms of abelian groups\[\begin{array}{rclc}
M&\xrightarrow{\cong}& R_{z,1-z}\otimes_{R_z}M;\qquad\quad\,\,1\mapsto1\otimes x,\\
M^\wedge&\xrightarrow{\cong}&(R_{z,1-z}\otimes_{R_z}M)^\wedge;\qquad g\mapsto(b\otimes x\mapsto b\otimes i(g(x))),\\
\Hom_{R_z}(M,N)&\xrightarrow{\cong}& \Hom_{R_{z,1-z}}(R_{z,1-z}\otimes_{R_z}M,R_{z,1-z}
\otimes_{R_z}N);\\ &&\phantom{(R_{z,1-z}\otimes_{R_z}M)^\wedge;\qquad } \,g\mapsto(b\otimes x\mapsto b\otimes g(x))
\end{array}\](see \cite[3.13(i)]{MR620795} for proof). Thus there is also an isomorphism of categories\[\xymatrix{\{\text{non-sing, $\eps$-symm linking forms over $(R_z,P)$}\}\ar[d]^{\cong}\\
\{\text{non-sing, $\eps$-symm linking forms over $(R_{z,1-z},P)$}\}.}\]This isomorphism sends (split) lagrangians to (split) lagrangians, hence the isomorphisms of Witt groups claimed.

\item The covering map $B$ is surjective on the level of isomorphism classes of forms. Lemma \ref{lagrangians} shows that the covering operation preserves lagrangians so on the level of Witt- and double Witt-groups, the covering morphism is well-defined.

To show that the covering map on Witt and double Witt groups is moreover injective, we will show that the Seifert forms with vanishing covering are hyperbolic. It is shown in {\cite[1.8(i)]{MR2058802}} that there exists $k>0$ such that, defining a projection \[p_e:=(e^k+(1-e)^k)^{-1}e^k:K\to K,\] the direct sum decomposition in Lemma \ref{lem:kerofB} may be taken to be $K_+=\im(p_e)$, $K_-=\im(1-p_e)$.

Now note that $(K,e)$ being a near-projection is equivalent to any one of $(K,1-e)$, $(K^*,e^*)$, $(K^*,1-e^*)$ being a near projection. Hence we may similarly decompose the Seifert dual module\[(K,e)^*=(K^*,1-e^*)=(K^*,1-e^*)_+\oplus (K^*,1-e^*)_-\]using the projection $p_{1-e^*}:K^*\to K^*$. But as \[p_e^*= ((e^*)^k+(1-e^*)^k)^{-1}(e^*)^k=p_{e^*}=1-p_{1-e^*}\] and $(\im(p_e))^*=\im(p_e^*)$ we must swap the summands (with respect to the previous decomposition)\[\begin{array}{rcl}(K^*,1-e^*)_+&=&((K_-)^*,(1-e^*)|_{(K_-)^*})\\ (K^*,1-e^*)_-&=&((K_+)^*,(1-e^*)|_{(K_+)^*})\end{array}\]Now if $(K,e,\psi)$ is a Seifert form such that $B(K,e)=0$, then as $\psi e=(1-e^*)\psi$ we have an isomorphism of Seifert forms \[\lmat p\\ 1-p\rmat:(K,e_K,\psi)\xrightarrow{\cong}\left(K_+\oplus K_-,e_+\oplus e_-,\lmat 0&\psi|_{K_-}\\ \psi|_{K_+}&0\rmat\right)\] so that $(K,e,\psi)$ is hyperbolic.
\end{enumerate}
\end{proof}

We say that two non-singular $\eps$-symmetric Seifert forms over  $R$ are \textit{$S$-equivalent} if they have isomorphic coverings. This condition has been alternatively expressed by many authors in terms of a series of moves that can be performed on a matrix representation of the adjoint of the Seifert form, see for example \cite{MR0420639}. The proof of the previous theorem shows the following:

\begin{corollary}\label{cor:sequiv}If two non-singular $\eps$-symmetric Seifert forms over $R$ are $S$-equivalent then they are moreover double Witt-equivalent.
\end{corollary}

\chapter{Algebraic $L$-theory}\label{chap:algLtheory}

We now intend to shift up a gear and develop our chain complex generalisation of the ideas considered so far, which we call \textit{double $L$-theory}. As well as being a way to extend the concept of `hyperbolic versus metabolic' to a richer setting, this new theory has a natural application to (high-dimensional) knot theory in the problem of detecting doubly-slice knots.

Chapter \ref{chap:algLtheory} will review the concepts from the Algebraic Theory of Surgery necessary to develop double $L$-theory in Chapter \ref{chap:DLtheory} and the subsequent topological applications. Our intention in this chapter is to present a self-contained and conceptually coherent account of the necessary ideas from symmetric $L$-theory, focussing on the chain homotopy theoretic reasons the constructions work. As algebraic $L$-theory is not a common technique we have included a level of detail we hope will be helpful to a beginner in this subject, however the nuances included will satisfy a more experienced reader as well.

Our account is largely based on the foundational \cite{MR560997}, \cite{MR566491}, \cite{MR620795} and also \cite{MR1211640}. Some use of the homotopy theoretic techniques from \cite{MR788854}, \cite{MR794111} is made to clarify the details of some constructions. We are also greatly indebted to Tibor Macko for offering us a preview of his, as yet unpublished, chapter of a forthcoming surgery book \cite{TIBOR}. We have directly built on Macko's description of the reversal of the algebraic Thom construction to analyse triads.

While the homotopy theory provides pleasing explanations of the ideas, it will be necessary in later chapters to do calculations using specific choices within homotopy classes. At that stage, our account will become less self-contained as we borrow calculations from our source texts rather than derive everything from first principles.

We will begin with a description of `stable' $L$-theory of \cite{MR1211640} but from Section \ref{sec:connective} onwards, we will restrict ourselves to the original `non-stable' or `connective' description of the theory from \cite{MR560997}, which coincides with Mischencko's earlier formulation \cite{mischencko}. The difference between the original description and that in \cite{MR1211640} is the assumption in the connective case of a positivity condition on the underlying chain complex with which we work. One major consequence is that the eventual symmetric $L$-groups will not generally be periodic in the connective theory (see \cite[3.18, 15.8]{MR1211640}). However, for practical purposes, connective $L$-theory has the advantage of corresponding more directly to $\eps$-symmetric forms and to the associated geometry. Connective $L$-theory has several major algebraic disadvantages, one of which will be the need to carry around various other connectivity conditions in order to make the mapping cone and desuspension constructions of boundary, surgery etc.\ well-defined.

We will organise this review by first presenting the basic objects purely algebraically, then showing that the geometry naturally gives rise to the algebraic structures described. Next we will define the constructions and groups needed to understand double $L$-theory and the double $L$-theory localisation exact sequence of Chapter \ref{chap:DLtheory}.

\section{Algebraic symmetric structures}

\subsection{Conventions and standard resolutions}\label{subsec:res}
Given chain complexes $(C,d_C),(D,d_D)$ of $A$-modules a \textit{chain map of degree $n$} is a collection of morphisms $f_r:C_r\to D_{r+n}$ with $d_Df_r=(-1)^nf_{r-1}d_C$. The category of chain complexes of $A$-modules with morphisms degree 0 chain maps is denoted $\Ch(A)$. A chain complex $C$ in $\Ch(A)$ is \textit{finite} if it is concentrated in finitely many dimensions. The category of finite chain complexes of projective, finitely generated (\text{f.g.}) $A$-modules is denoted $\B(A)\subset \Ch(A)$. If $C$ is in $\Ch(A)$, let $C^t$ denote the chain complex of f.g.\ projective, right $A$-modules $(C^t)_r:=(C_r)^t$. The \textit{dual chain complex} of $C$ in $\Ch(A)$ is $C^{-*}$ in $\Ch(A)$ with modules $(C^{-*})_r:=(C_{-r})^*=:C^{-r}$ and differential $(-1)^rd^*_C:C^{-r}\to C^{-r+1}$. The \textit{suspension} of $C$ in $\Ch(A)$ is the chain complex $\Sigma C$ in $\Ch(A)$ with modules $(\Sigma C)_r=C_{r-1}$ and differential $d_{\Sigma C}=d_C$. The \textit{desuspension} $\Sigma^{-1}C$ is defined by $\Sigma(\Sigma^{-1}C)=C$. Morphisms $f,f':C\to D$ are \textit{homotopy equivalent} if there exists a collection of $A$-module morphisms $g=\{g_r:C_r\to D_{r+1}\,|\,r\in \Z\}$ so that $f-f'=d_Dg+gd_C$, in which case the collection is called a \textit{chain homotopy} and we write $f\simeq f'$. A morphism $f:C\to D$ is a \textit{chain homotopy equivalence} if there exists a morphism $g:D\to C$ such that $fg\simeq 1_D$ and $gf\simeq 1_C$.

For $C,D$ in $\Ch(A)$, there are chain complexes of $\Z$-modules \[(C^t\otimes_A D)_r:=\bigoplus_{p+q=r}C^t_p\otimes_A D_q;\quad d(x\otimes y)=x\otimes d_D(y)+(-1)^qd_C(x)\otimes y,\]\[(\Hom_A(C,D))_r:=\prod_{q-p=r}\Hom_A(C_p,D_q);\quad d(f)=d_D(f)-(-1)^rfd_C,\]and the \textit{slant map} is defined as \[\setminus-:C^t\otimes_A D\to \Hom_A(C^{-*},D);\quad x\otimes y\mapsto (f\mapsto \overline{f(x)}y).\]In the sequel we will often write $C\otimes D$ in place of $C^t\otimes_AD$ in order to ease notation. If $C,D$ are (chain homotopy equivalent to) objects of $\B(A)$ then the slant map is a chain (homotopy) equivalence. When $C,D$ are chain homotopy equivalent to objects of $\B(A)$, there is an isomorphism of groups\[\{\text{$n$-cycles in $\Hom_A(C,D)$}\}\cong\{\text{chain maps of degree $n$ from $C$ to $D$}\}.\]

A morphism $f:C\to D$ in $\Ch(A)$ is a \textit{cofibration} if it is degreewise split injective and a \textit{fibration} if it degreewise split surjective. A sequence of morphisms in $\Ch(A)$ is a \textit{(co)fibration sequence} if each morphism in the sequence is a (co)fibration. The \textit{algebraic mapping cone} of $f$ is the chain complex $C(f)$ in $\Ch(A)$ with $C(f)_r=D_r\oplus C_{r-1}$ and \[d_{C(f)}=\left(\begin{matrix} d_D&(-1)^{r-1}f\\0&d_C\end{matrix}\right):D_r\oplus C_{r-1}\to D_{r-1}\oplus C_{r-2}.\]There is an obvious inclusion morphism $e:D\to C(f)$ and the composite $ef:C\to C(f)$ is easily seen to be nullhomotopic. Moreover the mapping cone has the universal property that if there is a sequence of morphisms and a nullhomotopy \[\xymatrix{C\ar@/_1pc/[rr]_-{j:gf\simeq 0}\ar[r]^-f&D\ar[r]^g&E}\]there exists a morphism $\Phi_j:C(f)\to E$ factoring $\Phi_je=g$. If two nullhomotopies $j_1$ and $j_2$ are homotopic (in the sense that there exists a collection of $A$-module morphisms $h=\{h_r:C_r\to D_{r+2}\,|\,r\in\Z\}$ such that $j_1-j_2=d_Dh+hd_C$) then $h$ can be shown to induce $\Phi_{j_1}\simeq\Phi_{j_2}$.

A \textit{homotopy cofibration sequence} is a sequence of morphisms in $\Ch(A)$ such that any two successive morphisms \[\xymatrix{C\ar[r]^-{f}&D\ar[r]^g&E}\] have nullhomotopic composition and such that any choice of nullhomotopy $j$ induces a chain equivalence $\Phi_j:C(f)\simeq E$. A sequence of morphisms in $\Ch(A)$ is a \textit{homotopy fibration sequence} if the dual sequence of morphisms is a homotopy cofibration sequence. Using the obvious projection morphisms $\text{proj}:C(f)\to \Sigma C$, every morphism $f:C\to D$ in $\Ch(A)$ has an associated \textit{Puppe sequence}\[\dots\to\Sigma^{-1}D\to \Sigma^{-1}C(f)\xrightarrow{\Sigma^{-1}\text{proj}} C\xrightarrow{f} D\xrightarrow{e} C(f)\xrightarrow{\text{proj}} \Sigma C\xrightarrow{\Sigma f} \Sigma D\to \dots\]which is both a homotopy fibration sequence and a homotopy cofibration sequence. In particular this shows that in $\Ch(A)$, homotopy fibration sequences agree with homotopy cofibration sequences.

Given diagrams\[\xymatrix{D&C\ar[l]_-{f}\ar[r]^-{f'}& D'}\quad\text{ and }\quad\xymatrix{D\ar[r]^-{g}&E&D'\ar[l]_-{g'}}\]define the \textit{homotopy pushout} as \[D\cup_C D':=C\left(\lmat -f\\f'\rmat :C\to D\oplus D'\right)\]and the \textit{homotopy pullback} as \[D\times_E D':= \Sigma^{-1}C((g\,\,-g'):D\oplus D'\to E).\]

\begin{remark}We have defined both pushout and pullback in terms of mapping cone. While this will be fine for our purposes as we generally will only need these up to homotopy equivalence, it is worth mentioning that a more pedagogically correct way of making these definitions would be to define pushout using mapping cone and pullback using `mapping fibre' as is carried out in \cite{TIBOR}.
\end{remark}

A \textit{homotopy commuting square} $\Gamma$ in $\Ch(A)$ is a diagram \[\xymatrix{C\ar@{~>}[dr]^-{h}\ar[r]^-{f'}\ar[d]_-{f}&D'\ar[d]^-{g'}\\D\ar[r]^{g}&E}\]consisting of a square of morphisms $f,f',g,g'$ in $\Ch(A)$ together with a homotopy $h:g'f'\simeq gf$. A homotopy commuting square induces the obvious maps of cones\[C(g',f):C(f')\to C(g),\qquad C(g,f'):C(f)\to C(g').\]Taking cones again, there is not just homotopy equivalence, but actual equality $C(C(g',f))=C(C(g,f'))$. We define the \textit{iterated cone on $\Gamma$} to be that chain complex \[C(\Gamma)=C(C(g',f))=C(C(g,f')).\]Note that, strictly speaking, the morphisms $C(g',f)$ and $C(g,f')$ in $\Ch(A)$, and hence the complex $C(\Gamma)$, depend on the choice of $h$ but this is suppressed from the notation.

A \textit{homotopy pushout square} is a homotopy commuting square $\Gamma$ such that the induced map $\Phi_h:D\cup_C D'\to E$ is a homotopy equivalence. A \textit{homotopy pullback square} is defined analogously using the homotopy pullback.

\subsubsection*{The $W^\%$ functor}

To package the symmetries in the homology of a manifold efficiently, we introduce a functor that will be geometrically motivated in Section \ref{sec:symmetric} when we build the `symmetric structure' for a topological space. For now, we describe the necessary algebra.

The cyclic group of order 2 is denoted $\Z/2\Z=\{1,T\}$, and let $\eps\in A$ be a unit such that $\eps^{-1}=\overline{\eps}$ (for instance $\eps=\pm1$). Let $C$ be in $\Ch(A)$ and define the standard $\eps$-involution\begin{eqnarray*}T=T_\eps:C_p^t\otimes_A C_q&\to& C_q^t\otimes_A C_p\\x\otimes y&\to& \eps(-1)^{pq}y\otimes x\end{eqnarray*}so that $C^t\otimes_A C$ in $\Ch(\Z)$ may moreover be regarded as a chain complex of $\Z[\Z/2\Z]$ modules.

The standard free $\Z[\Z/2\Z]$-resolution of $\Z$ is the chain complex\[W:\quad\dots\to\Z[\Z/2\Z]\xrightarrow{1-T} \Z[\Z/2\Z]\xrightarrow{1+T} \Z[\Z/2\Z]\xrightarrow{1-T} \Z[\Z/2\Z]\to 0\]and $W_0\hookrightarrow W$ denotes the inclusion at dimension 0 of the truncation $0\to \Z[\Z/2\Z]\to 0$. We use $W$ to find the `homotopy fixed points' of the involution $T_\eps$ on $C^t\otimes_A C$ and $W_0$ to find a `level 0 approximation' to these homotopy fixed points that will later be analogous to the intersection pairings associated to the chain complex of a topological space. The `homotopy fixed points' are in the form of the complex of $\Z$-modules\[W^\%C:=\Hom_{\Z[\Z/2\Z]}(W,C^t\otimes_A C).\]Given a morphism $f:C\to D$, the $W^\%$ construction induces a morphism of abelian groups\[f^\%:W^\%C\to W^\% D\]so that $W^\%$ is a functor $\Ch(A)\to \Ch(\Z)$. It is possible to show (see \cite[p.101]{MR560997}) that $W^\%$ is moreover a homotopy functor; given a homotopy $h: f_1 \simeq f_2:C\to D$ there exists a (non-canonical) choice of homotopy $h^\%:f_0^\%\simeq f_1^\%:W^\% C\to W^\% D$, and we will denote any such choice by $h^\%$.

It is sometimes the case that we will need to keep track of the choice of $\eps$ in the involution $T_\eps$ that is suppressed from the notation. In these cases we will write $W^\%_\eps C$ instead of $W^\%C$.

\begin{remark}If $C$ in $\Ch(A)$ is chain homotopy equivalent to an object of $\B(A)$ then $W^\%C$ is chain homotopy equivalent to an object of $\B(\Z)$.
\end{remark}

The $W^\%$ functor does not behave additively with respect to direct sum of chain complexes $C,C'$ in $\Ch(A)$. Indeed, by considering the identity \[(C\oplus C)\otimes (C\oplus C')= (C\otimes C)\oplus (C'\otimes C') \oplus (C\otimes C')\oplus (C'\otimes C),\] and the effect of the $\Z/2\Z$-action on the right hand side, we easily see that\[W^\%(C\oplus C')= W^\%C\oplus W^\% C'\oplus (C\otimes C').\]Carrying around these `cross-terms' $(C\otimes C')$ will become cumbersome in some of the constructions later. For this reason we introduce the following innovation according to \cite[\textsection 4.2]{Borodzik:2012fk}.

\begin{definition}For $A$-modules $P$, $P'$,  the set $P\oplus P'$ has the structure of an $A\times A$-module by the left action $((a,b),(x,y))\mapsto (ax,by)$. Define the \textit{disjoint union} of $P$ and $P'$ to be the $A\times A$-module \[P\sqcup P':= P\oplus P'\text{ regarded as an $A\times A$-module.}\] If $P,P'$ are f.g.\ projective $A$-modules then $P\sqcup P'$ will in general be a f.g.\ projective $A\times A$-module and the functor described by $(P,P')\mapsto P\sqcup P'$ is an equivalence of categories with involution for f.g.\ projective modules.

For complexes $C,C'$ in $\Ch(A)$ we define $C\sqcup C'$ in $\Ch(A\times A)$ in the obvious way.
\end{definition}

The following proposition shows our main reason for introducing this concept.

\begin{proposition}If $C,C'$ are in $\Ch(A)$, there is an equality in $\Ch(\Z)$\[W^\%(C\sqcup C')= W^\%C\oplus W^\% C'.\]
\end{proposition}
\begin{proof}Consider the identity\[(C\sqcup C')^t\otimes_{A\times A}(C\sqcup C')=(C^t\otimes _A C)\oplus ((C')^t\otimes_A C').\]
\end{proof}

\subsection{Symmetric complexes, pairs and triads}\label{subsec:definitions}

\subsubsection*{Complexes and pairs}

\begin{definition}An \textit{$n$-dimensional $\eps$-symmetric structure} on $C$ in $\Ch(A)$ is an $n$-dimensional cycle $\phi\in W^\%C_n$. The pair $(C,\phi)$ is called an \textit{$n$-dimensional $\eps$-symmetric complex}. Two $n$-dimensional $\eps$-symmetric complexes $(C,\phi)$, $(C',\phi')$ are \textit{homotopy equivalent} if there is a homotopy equivalence $h:C\xrightarrow{\simeq} C'$ such that $h^\%\phi=\phi'$.

An \textit{$(n+1)$-dimensional $\eps$-symmetric structure} on a morphism $f:C\to D$ in $\Ch(A)$ is an $(n+1)$-dimensional cycle $(\delta\phi,\phi)\in C(f^\%)_{n+1}\cong W^\%D_{n+1}\oplus W^\%C_{n}$. Note this implies $\phi$ is a cycle and $\delta\phi$ is a nullhomotopy of $f^\%(\phi)$. The pair $(f:C\to D,(\delta\phi,\phi))$ is called an \textit{$(n+1)$-dimensional $\eps$-symmetric pair}. Two $(n+1)$-dimensional $\eps$-symmetric pairs $(f:C\to D,(\delta\phi,\phi))$, $(f':C'\to D',(\delta\phi',\phi'))$ are \textit{homotopy equivalent} if there is a homotopy commuting square\[\xymatrix{C\ar@{~>}[rd]\ar[r]^-{f}\ar[d]_-{h}&D\ar[d]^-{k}\\C'\ar[r]^-{f'}&D'}\]inducing a homotopy equivalence $C(f)\simeq C(f')$ and such that $C(k^\%,h^\%)(\delta\phi,\phi)=(\delta\phi',\phi')$.
\end{definition}

\begin{figure}[h]\[\def\picnullcob{\resizebox{0.27\textwidth}{!}{ \includegraphics{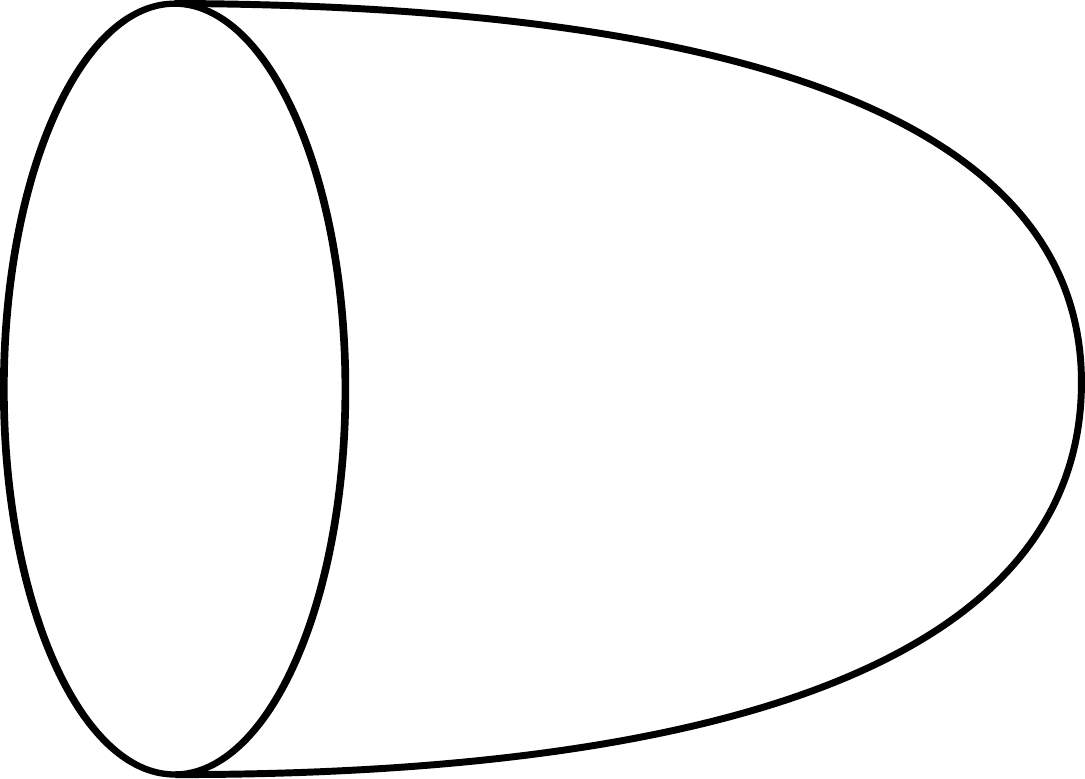}}}
\begin{xy} \xyimport(300,300){\picnullcob}
,!+<7.2pc,2pc>*+!\txt{}
,(35,148)*!L{C}
,(165,148)*!L{D}
\end{xy}\]\caption{A schematic of an $\eps$-symmetric pair}
\end{figure}


\begin{example}[$\eps$-symmetric forms are short $\eps$-symmetric chain complexes]If $C$ is a complex $0\to C_0\to 0$ concentrated in dimension 0 and \[\phi\in (W^\%C)_0=\Hom_{\Z[\Z/2\Z]}(W,C\otimes C)_0\cong\Hom_{\Z[\Z/2\Z]}(W_0,C\otimes C)_0,\]then $\phi$ is a chain homotopy class described by a single degree 0 chain map $\phi_0:C^{-*}\to C$ with the $\eps$-symmetry property $\phi_0\circ\text{switch}=T_\eps\phi_0$ where `switch' is the $\Z/2\Z$-action switching the factors in $C\otimes C$. Adjointly, it is an $\eps$-symmetric sesquilinear pairing\[C^0\times C^0\to A;\qquad (x,y)\mapsto y(\phi_0(x)).\]This correspondence is clearly reversible.
\end{example}

\begin{example}[$(-\eps)$-symmetric linking forms are short $\eps$-symmetric $S$-acyclic complexes] If $C$ is a complex $0\to C_1\to C_0\to 0$ whose cohomology is torsion with respect to some localisation $(A,S)$ and \[\phi\in (W^\%C)_1=\Hom_{\Z[\Z/2\Z]}(W,C\otimes C)_1\cong\Hom_{\Z[\Z/2\Z]}(0\to\Z[\Z/2\Z]\xrightarrow{1-T}\Z[\Z/2\Z]\to 0,C\otimes C)_1\] then $\phi$ is a chain homotopy class described by one degree 0 chain map $\phi_0:C^{1-*}\to C$ and one degree 1 chain map $\phi_1:C^{2-*}\to C$. The cohomology module $H^1(C)$ then admits a $(-\eps)$-symmetric linking form over $(A,S)$ given by \[H^1(C)\times H^1(C)\to S^{-1}A/A;\qquad([x],[y])\mapsto s^{-1}y(\phi_0(z)),\]where $x,y\in C^1$, $z\in C^0$ such that $d_C^*z=sx$ for some $s\in S$. In fact this correspondence is also reversible in some sense - of which more later! (See Proposition \ref{prop:correspondence}, onwards.)
\end{example}

More generally, the inclusion $W_0\hookrightarrow W$ induces a natural transformation of functors. So that for a given complex $C$ in $\Ch(A)$ there are `evaluation' morphisms \[\ev:W^\%C\to C\otimes C;\quad \phi\mapsto \ev(\phi)\]and the image of the evaluation $\ev(\phi)\in (C\otimes C)_n$ under the slant map is written\[\setminus-:C\otimes C\to \Hom_A(C^{-*},C);\qquad \phi_0:=\sm(\ev(\phi)):C^{n-*}\to C.\]If $f:C\to D$ is a morphism in $\Ch(A)$ then there is a relative evaluation morphism\[\ev:C(f^\%)\simeq\Hom_{\Z[\Z/2\Z]}(W,C(f\otimes f))\to C(f\otimes f);\quad (\delta\phi,\phi)\mapsto \ev(\delta\phi,\phi)\]and the image of the evaluation $\ev(\delta\phi,\phi)\in C(f\otimes f)_{n+1}=(D\otimes D)_{n+1}\oplus (C\otimes C)_n$ under the slant map is written\[\left(\begin{matrix}\delta\phi_0 &0\\0&\phi_0\end{matrix}\right):=\setminus(\ev(\delta\phi,\phi)),\quad \delta\phi_0:D^{n+1-*}\to D,\quad \phi_0:C^{n-*}\to C\]Moreover there is a homotopy cofibration sequence\[\xymatrix{C\ar@/_1pc/[rr]_{j:ef\simeq 0}\ar[r]^f&D\ar[r]^-e&C(f)}\] with $j$ a choice of nullhomotopy so that $j\otimes f:C\otimes C\to C(f)\otimes D$ and $f\otimes j:C\otimes C\to D\otimes C(f)$ induce morphisms \[\begin{array}{rcl}\ev_l:C(f\otimes f)&\to& C(f)\otimes D\\\ev_r:C(f\otimes f)&\to&D\otimes C(f)\end{array}\] respectively. Under the slant maps, the images of $(\delta\phi,\phi)\in C(f\otimes f)_{n+1}$ are\[\begin{array}{rccl}
(\delta\phi_0\,\,f\phi_0)=\sm(\ev_l(\delta\phi,\phi)):&C(f)^{n+1-*}&\to& D\\
\lmat \delta\phi_0\\\phi_0 f^* \rmat=\setminus(\ev_r(\delta\phi,\phi)):&D^{n+1-*}&\to& C(f)\end{array}\]respectively.

\begin{proposition}[{\cite{TIBOR}}]\label{prop:pairev}The following diagram is a homotopy pullback square\[\xymatrix{C(f\otimes f)\ar@{~>}[dr]^{\Phi_h}\ar[r]^{\ev_r}\ar[d]_{\ev_l}&D\otimes C(f)\ar[d]^{e\otimes\id}\\C(f)\otimes D\ar[r]_{\id\otimes e}&C(f)\otimes C(f)}\]with $h=j\otimes j:j\otimes ef\simeq ef\otimes j$ inducing the chain homotopy commutativity $\Phi_h$.
\end{proposition}

However, it is not true that $W^\%$ preserves homotopy fibrations and homotopy cofibrations in general. This will be significant (and a source of technical challenge) when we compare complexes to pairs and pairs to triads later.  The difference between $C(f^\%)$ and $W^\%C(f)$ is given by the following proposition.

\begin{proposition}[{\cite{TIBOR}}]\label{prop:pullback} Let $f:C\to D$ be a morphism in $\B(A)$ so that there is a homotopy cofibration sequence\[\xymatrix{C\ar@/_1pc/[rr]_{j:ef\simeq 0}\ar[r]^f&D\ar[r]^-e&C(f)}\] with $j$ a choice of nullhomotopy inducing $\Phi_{j^\%}:C(f^\%)\to W^\%C(f)$. Then there is a homotopy pullback square\[\xymatrix{C(f^\%)\ar@{~>}[dr]\ar[r]^{\ev_r}\ar[d]_{\Phi_{j^\%}}&D\otimes C(f)\ar[d]^{e\otimes{\textit{id}}}\\W^\%C(f)\ar[r]_{\ev}&C(f)\otimes C(f)}\]
\end{proposition}

\begin{corollary}[{\cite[Prop. 18]{MR2189218}}]\label{cor:pullback}There is a long exact sequence\[\dots\to H_k(C\otimes C(f))\to H_k(C(f^\%))\xrightarrow{\Phi_{j^\%}} H_k(W^\%C(f))\to H_{k-1}(C\otimes C(f))\to\dots\]\end{corollary}

\subsubsection*{Algebraic glueing}

Given morphisms $f:C\to D$ and $f':C\to D'$ in $\B(A)$, denote the homotopy pushout by \[D'':=D\cup_C D'=C\left(\lmat f\\ -f'\rmat:C\to D\oplus D'\right).\] There is correspondingly a homotopy pushout \[X:=W^\%D\cup_{W^\%C}W^\%D'=C\left(\lmat f^\%\\-(f')^\%\rmat:W^\%C\to W^\%(D\sqcup D')\right).\]and as fibration sequences agree with cofibration sequences in the homotopy category of $\Ch(A)$, we can alternatively describe $X$ as the homotopy pullback\[X=C(f^\%)\times_{\Sigma W^\%C}C((f')^\%).\]However, $X$ is not equal to $W^\%D''$. There is a a map \[G:X\to W^\% D''\]induced by a choice of nullhomotopy \[\xymatrix{C\ar@/_1pc/[rr]_{\simeq 0}\ar[r]&D\oplus D'\ar[r]&D''}\]and the difference between $X$ and $W^\% D''$ is precisely given by Proposition \ref{prop:pullback}. 

\begin{figure}[h]\[\def\picgluewhole{\resizebox{0.6\textwidth}{!}{ \includegraphics{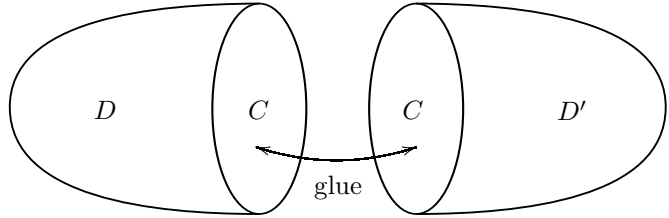}}}
\begin{xy} \xyimport(300,300){\picgluewhole}
,!+<7.2pc,2pc>*+!\txt{}
,(110,148)*!L{C}
,(180,148)*!L{C}
,(40,148)*!L{D}
,(250,148)*!L{D'}
,(110,100)*+{}="A";(190,100)*+{}="B"
,{"A"\ar@/_/"B"}
,{"B"\ar@/^/"A"}
,(140,40)*!L{\text{glue}}
\end{xy}\]\caption{A schematic of algebraic glueing along the whole boundary}
\end{figure}

\begin{definition}[Algebraic glueing along the whole boundary]\label{def:glue1} Given $(n+1)$-dimensional $\eps$-symmetric pairs \[x:=(f:C\to D,(\delta\phi,\phi)),\qquad x':=(f':C\to D',(\delta'\phi,\phi)),\] the \textit{algebraic glue} of the pairs is the $(n+1)$-dimensional $\eps$-symmetric complex \[x\cup x'=(D\cup_C D',\delta\phi\cup_\phi\delta'\phi)\] where $\delta\phi\cup_\phi\delta'\phi=G((\delta\phi,\phi,\delta'\phi))\in W^\% D''$.
\end{definition}

Now given morphisms $(f\,\,f'):C\oplus C'\to D$ and $(\tilde{f}'\,\,f''):C'\oplus C''\to D'$ in $\B(A)$, there is a homotopy commutative diagram\[\xymatrix{
C \oplus C' \oplus C''\ar[d]^-{\text{pr.}}\ar[rrr]^-{\lmat f &f'& 0\\ 0&-\tilde{f}' &-f''\rmat}&&&D\oplus D'\ar[d]\\
C\oplus C''\ar[rrr]^{(g\,\,g'')}&&&D\cup_{C'} D'}\]with $(g\,\,g'')$ defined to make the diagram commute. This diagram is a homotopy pushout (consider the homotopy fibres of the vertical maps). The diagram induces a homotopy commutative diagram (again a homotopy pushout) \[\xymatrix{
W^\%(C \sqcup C' \sqcup C'')\ar[d]^-{\text{pr.}}\ar[rrr]^-{\lmat f^\% &f'^\%& 0\\ 0&-\tilde{f}'^\% &-f''^\%\rmat}&&&W^\%(D\sqcup D')\ar[d]\\
W^\%(C\sqcup C'')\ar[rrr]^-{(g^\%\,\,g''^\%)}&&& X}\]Consider taking the cone on the bottom row of the previous diagram:\[\xymatrix{W^\%(C\sqcup C'')\ar[rr]^-{(g^\%\,\,g''^\%)}&& X\ar[r]& C((g^\%\,\,g''^\%))=:Y.}\]To understand this homotopy fibration sequence, we observe that up to homotopy, all three terms are homotopy pushouts of the respective columns of the following diagram, in which the rows are homotopy cofibration sequences\[\xymatrix{
W^\%C\ar[rr]^-{f^\%}&&W^\% D\ar[rr] &&C(f^\%)\\
0\ar[u]\ar[d]\ar[rr]&&W^\% C'\ar[u]^-{f'^\%}\ar[d]_-{\tilde{f}'^\%}\ar[rr]^-{\text{id.}}&&W^\%C'\ar[u]\ar[d]\\
W^\%C''\ar[rr]^-{f''^\%}&&W^\% D'\ar[rr] &&C(f''^\%)}\]So we may describe $Y$ as either a homotopy pushout or homotopy pullback:\[Y=C(f^\%)\cup_{W^\% C'}C(f''^\%)\simeq C((f^\%\,\,f'^\%))\times_{\Sigma W^\% C'}C((\tilde{f}'^\%\,\,f''^\%))\]with the map $X\to Y$ given accordingly.

\begin{figure}[h]\[\def\picgluepart{\resizebox{0.8\textwidth}{!}{ \includegraphics{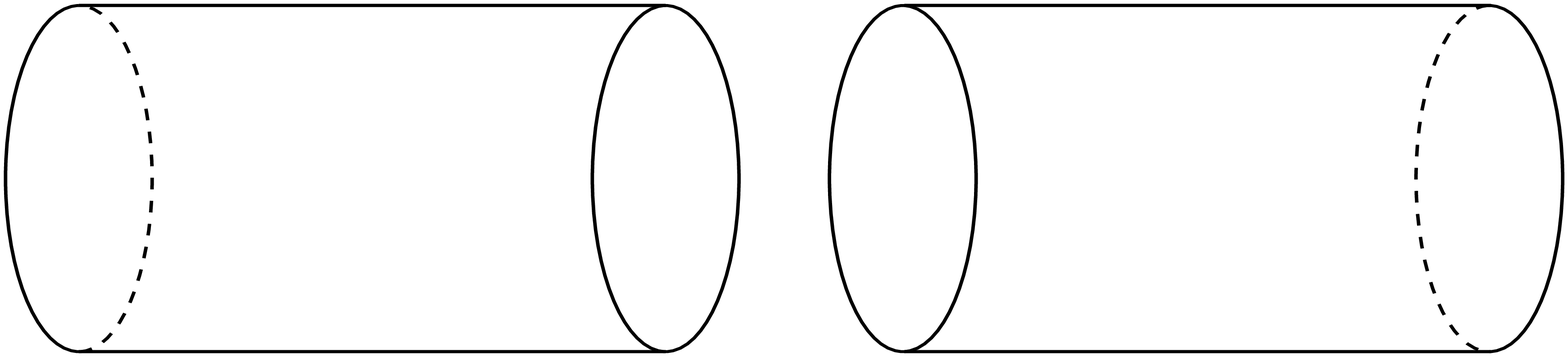}}}
\begin{xy} \xyimport(300,300){\picgluepart}
,!+<7.2pc,2pc>*+!\txt{}
,(125,148)*!L{C}
,(170,148)*!L{C}
,(12,148)*!L{C'}
,(280,148)*!L{C''}
,(65,148)*!L{D}
,(230,148)*!L{D'}
,(125,100)*+{}="A";(175,100)*+{}="B"
,{"A"\ar@/_/"B"}
,{"B"\ar@/^/"A"}
,(142,40)*!L{\text{glue}}
\end{xy}\]\caption{A schematic of algebraic glueing along a boundary component}
\end{figure}

\begin{definition}[Algebraic glueing along a boundary component]\label{def:glue2} Given $(n+1)$-dimensional $\eps$-symmetric pairs \[x:=((f\,\,f'):C\oplus C'\to D,(\delta\phi,\phi\oplus \phi')),\qquad x':=((\tilde{f}'\,\,f''):C'\oplus C''\to D',(\delta'\phi,\phi'\oplus \phi'')),\] the \textit{algebraic glue of $x$ and $x'$ along $(C',\phi')$} is the $(n+1)$-dimensional $\eps$-symmetric pair\[x\cup x':=((g\,\,g''):C\oplus C''\to D\cup_{C'}D',(\delta\phi\cup_{\phi'}\delta'\phi,\phi\oplus \phi'')),\]with the $(n+1)$-dimensional $\eps$-symmetric structure defined by the image of the $(n+1)$ dimensional cycle\[((\delta\phi,\phi\oplus\phi'),\Sigma\phi',(\delta'\phi,\phi'\oplus\phi''))\in Y=C((f^\%\,\,f'^\%))\times_{\Sigma W^\% C'}C((\tilde{f}'^\%\,\,f''^\%))\]in $C((g\,\,g'')^\%)$ using the homotopy commutative diagram defined by the universality of the pushouts along its top row \[\xymatrix{
W^\%(C\sqcup C'')\ar[d]^-{\text{incl.}}\ar[rr]^-{(g^\%\,\,g''^\%)}&&X\ar[d]^-{G}\ar[r] &Y\ar[d]\\
W^\%(C\oplus C'')\ar[rr]^-{(g\,\,g'')^\%}&&W^\%(D\cup_C D')\ar[r]&C((g\,\,g'')^\%)
}\]
\end{definition}

\subsubsection*{Triads}

Suppose we have a homotopy commuting square $\Gamma$ in $\Ch(A)$ given by\[\xymatrix{C\ar[r]^-{f'}\ar[d]_-{f}\ar@{~>}[dr]^{h}&D'\ar[d]^-{g'}\\D\ar[r]_-{g}&E}\]Then the homotopy commutativity induces a morphism $(g'')^\%:W^\%D''\to W^\%E$ and a homotopy commutative diagram\begin{equation}\label{eq:lift}\xymatrix{X\ar[r]^-F\ar[d]_-G& W^\%E\ar[r]\ar[d]_-{=}&C(F)\ar[d]\\W^\%D''\ar[r]_{(g'')^\%}&W^\%E\ar[r]&C((g'')^\%)}\end{equation}

\begin{definition}An \textit{$(n+2)$-dimensional $\eps$-symmetric structure} on a homotopy commuting square $\Gamma$ in $\Ch(A)$ is an $(n+2)$-dimensional cycle $(\Phi, \delta\phi,\delta'\phi,\phi)\in C(F)_{n+2}= W^\%E_{n+2}\oplus W^\%D_{n+1}\oplus W^\%D'_{n+1}\oplus W^\%C_n$. The pair $(\Gamma,(\Phi,\delta\phi,\delta'\phi,\phi))$ is called an \textit{$(n+2)$-dimensional $\eps$-symmetric triad}.
\end{definition}

\begin{figure}[h]
\def\picrelboundary{\resizebox{0.3\textwidth}{!}{ \includegraphics{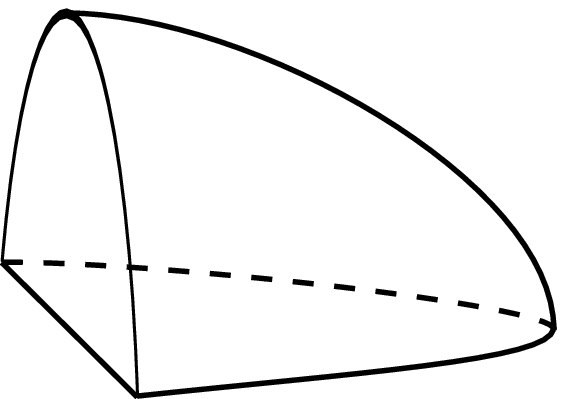}}}

\[\begin{xy} \xyimport(219.73,159.08){\picrelboundary}
,(20,90)*!L{D'}
,(100,87)*!L{E}
,(0,20)*!L{C}
,(83,29)*!L{D}

\end{xy}\]
                \caption{A schematic of an $\eps$-symmetric triad}
\end{figure}

There are several other equivalent ways of packaging the information of an $(n+2)$-dimensional $\eps$-symmetric triad. Consider that applying $W^\%$ to the square $\Gamma$ and then iterating cones yields a homotopy commuting diagram\begin{equation}\label{eq:level1}\xymatrix{
W^\%C\ar[rr]^-{(f')^\%}\ar[d]_-{f^\%}\ar@{~>}[drr]^{h^\%}&&W^\%D'\ar[d]^-{(g')^\%}\ar[r]&C((f')^\%)\ar[d]^{C((g')^\%,f^\%)}\\
W^\%D\ar[d]\ar[rr]_-{g^\%}&&W^\%E\ar[d]\ar[r]&C(g^\%)\ar[d]\\
C(f^\%)\ar[rr]_{C(g^\%,(f')^\%)}&&C((g')^\%)\ar[r]&C(F)}\end{equation}so that $C(F)$ is homotopy equivalent to the cone on $C(g^\%,(f')^\%)$ and to the cone on $C((g')^\%,f^\%)$. The homotopy cofibration sequences on the right-most vertical and the lower-most horizontal allow us to view the elements $(\delta\phi,\phi)\in C(f^\%)_{n+1}$ and $(\delta'\phi,\phi)\in C((f')^\%)_{n+1}$ as cycles, defining $(n+1)$-dimensional $\eps$-symmetric pairs. Furthermore, the diagram \ref{eq:lift} shows that a cycle in $C(F)$ induces an $(n+2)$-dimensional $\eps$-symmetric structure on the morphism $g'':D\cup_C D'\to E$. In fact the construction of these 3 data is reversible:

\begin{proposition}\label{prop:triple}An $(n+2)$-dimensional $\eps$-symmetric triad $(\Gamma,(\Phi,\delta\phi,\delta'\phi,\phi))$ determines and is determined by the triple consisting of two $(n+1)$-dimensional $\eps$-symmetric pairs \[(f:C\to D,(\delta\phi,\phi)),\quad(f':C\to D',(\delta'\phi,\phi))\] and one $(n+2)$-dimensional $\eps$-symmetric pair \[(g'':D\cup_CD'\to E,(\Phi,\delta\phi\cup_\phi\delta'\phi))\]
\end{proposition}

\begin{proof}To reverse our construction above, note that the three pairs comprising the triple in the theorem define an obvious homotopy commuting square $\Gamma$. We only need to lift the cycle $(\Phi,\delta\phi\cup_\phi\delta'\phi)\in C((g'')^\%)$ to $C(F)$. But this is possible as $\delta\phi\cup_\phi\delta'\phi\in\im(G)$ by definition, and the homotopy commutativity of diagram \ref{eq:lift} ensures that the corresponding lift of $(\Gamma,\delta\phi\cup_\phi\delta'\phi)$ to $C(F)$ is a cycle.
\end{proof}

Two $(n+2)$-dimensional $\eps$-symmetric triads $(\Gamma,(\Phi,\delta\phi,\delta'\phi,\phi))$, $(\tilde{\Gamma},(\tilde{\Phi},\delta\tilde{\phi},\delta'\tilde{\phi},\tilde{\phi}))$ are \textit{homotopy equivalent} if there are morphisms $h:C\to \tilde{C}$, $k:D\to\tilde{D}$, $k':D'\to\tilde{D'}$, $l:E\to\tilde{E}$ commuting up to homotopy in the obvious way with the maps within $\Gamma$ and $\tilde{\Gamma}$ and inducing homotopy equivalences of the respective $\eps$-symmetric pairs determined by Proposition \ref{prop:triple}.

\subsubsection*{Symmetric Poincar\'{e} structures}

We are finally in a position to define the chain complex analogues of Poincar\'{e} duality and Poincar\'{e}-Lefschetz duality for our various $\eps$-symmetric structures.

\begin{definition}[Symmetric Poincar\'{e} structures]

An $n$-dimensional $\eps$-symmetric complex $(C,\phi)$ is \textit{Poincar\'{e}} if the image under the slant map $\phi_0:C^{n-*}\to C$ of $\ev(\phi)\in (C\otimes C)_n$ is a chain homotopy equivalence.

An $(n+1)$-dimensional $\eps$-symmetric pair $(f:C\to D,(\delta\phi,\phi))$ is \textit{Poincar\'{e}} if the image $(\delta\phi_0\,\,f\phi_0):C(f)^{n+1-*}\to D$ of $\ev_l(\delta\phi,\phi)\in (C(f)\otimes D)_{n+1}$ under the slant map is a chain homotopy equivalence. This is an equivalent condition to the image of $\ev_r(\delta\phi,\phi)\in (D\otimes C(f))_{n+1}$ under the slant map being a chain homotopy equivalence.

An $(n+2)$-dimensional $\eps$-symmetric triad $(\Gamma,(\Phi,\delta\phi,\delta'\phi,\phi))$ is \textit{Poincar\'{e}} if each of the associated pairs \[(f:C\to D,(\delta\phi,\phi)),\quad(f':C\to D',(\delta'\phi,\phi)),\quad(g'':D\cup_CD'\to E,(\Phi,\delta\phi\cup_\phi\delta'\phi))\] is Poincar\'{e}.
\end{definition}

In particular, this means that given an $(n+1)$-dimensional $\eps$-symmetric Poincar\'{e} pair $(f:C\to D,(\delta\phi,\phi))$, the $n$-dimensional $\eps$-symmetric complex $(C,\phi)$ is in fact always Poincar\'{e} as well. This is seen directly from the the definitions and by considering the homotopy commutative diagram \ref{eq:ladder}:\begin{equation}\label{eq:ladder}\xymatrix{C^{n-*}\ar[d]_{\setminus(ev(\phi))}\ar[r] &C(f)^{n+1-*}\ar@{~>}[drr]^-{\sm\Phi_h(\delta\phi,\phi)}\ar[rr]^-{e^*}\ar[d]_{\setminus(\ev_l(\delta\phi,\phi))} &&D^{n+1-*}\ar[r]\ar[d]^{\setminus(\ev_r(\delta\phi,\phi))} &C^{n+1-*}\ar[d]^{\Sigma(\setminus(\ev(\phi)))}\\
C\ar[r]&D\ar[rr]^-{e}&&C(f)\ar[r]&\Sigma C}\end{equation}To check that this is indeed homotopy commutative, start by applying the evaluation and slant maps to the homotopy pullback square of \ref{prop:pairev}. This will show that the central square of diagram \ref{eq:ladder} is a homotopy commuting square. The rows of \ref{eq:ladder} are clearly cofibration sequences and the remaining two squares are actually seen to commute on the nose.

\subsection{Collections of chain maps and the effect of a half-unit}\label{subsec:chainmaps}

In Chapter \ref{chap:DLtheory} we develop the chain complex analogue of the double Witt groups $DW^\eps(A,S)$. Just as in the case of linking forms, in order to construct a group using this chain complex analogue we will need to assume the ring $A$ contains a half-unit $s$. This will have the additional effect of permitting a dramatic simplification of the symmetric structures so far defined - to whit  we will be able to reduce all symmetric structure information of an $\eps$-symmetric complex $(C,\phi)$ to our `level 0' approximation $\phi_0$. In order to explain this we will quote another formulation for symmetric structures consisting of a collection of interrelated chain maps of various degrees.

\medskip

Suppose $(C,d)$ in $\B(A)$ and $\phi\in W^\%C_n$. Then, writing $W_s=\Z[1_s,T]$ we have the slant $\phi_s:=\sm(\phi(1_s))$ defining a collection of (degree $s$) chain maps, which we write as\[\{\phi_s:C^{n+s-r}\to C_r\,|\,s\geq0, r\in\Z\},\]and this is a complete description of $\phi$. Note that this description agrees with our previous definition of $\phi_0$. Now if we assume that $\phi$ is a cycle we must moreover insist that it is in the kernel of the differential \[d^\%:W^\%C_n\to W^\%C_{n-1}.\]This differential is calculated to be\[d^\%:\phi_s\mapsto d\phi_s+\phi_sd^*+(-1)^{n+s-1}(\phi_{s-1}+(-1)^sT_\eps\phi_{s-1}):C^{n-r+s-1}\to C_r,\]for $s\geq0$, $r\in\Z$ and $\phi_{-1}:=0$ by convention.

\begin{proposition}[{\cite[1.2, 1.4(i), 3.3]{MR560997}}]When there is a half unit $s\in A$, and $C$ is chain homotopy equivalent to an object of $\B(A)$, there is a chain homotopy equivalence $W^\%C\simeq W^\%_0C$. 
\end{proposition}

\begin{proof}Write a $\Z[\Z/2\Z]$-module chain complex\[\widehat{W}:\qquad\dots\to\Z[\Z/2\Z]\xrightarrow{1-T_\eps}\Z[\Z/2\Z]\xrightarrow{1+T_\eps}\Z[\Z/2\Z]\xrightarrow{1-T_\eps}\Z[\Z/2\Z]\to\dots,\]homologically graded as $1+(-1)^iT_\eps:(\widehat{W})_i\to(\widehat{W})_{i-1}$. Then there is an obvious degree-wise split short exact sequence in $\Ch(\Z[\Z/2\Z])$\[0\to \Sigma^{-1}W^{-*}\to \widehat{W}\to W\to 0, \]and there is a corresponding homotopy cofibration sequence\[\Hom_{\Z[\Z/2\Z]}(W^{-*},C\otimes C)\xrightarrow{1+T_\eps} \Hom_{\Z[\Z/2\Z]}(W,C\otimes C) \to \Hom_{\Z[\Z/2\Z]}(\widehat{W},C\otimes C),\]where $((1+T_\eps)\psi)_s:=0$ for $s\neq 0$ and $((1+T_\eps)\psi)_0:=(1+T_\eps)\psi_0)$, so in fact $1+T_\eps$ factors through $W^\%_0C$. When there is a half-unit in $A$, we have $\Hom_{\Z[\Z/2\Z]}(\widehat{W},C\otimes C)\simeq 0$ (see \cite[3.3]{MR560997}). Hence, there is a chain homotopy equivalence $1+T_\eps:\Hom_{\Z[\Z/2\Z]}(W^{-*},C\otimes C)\xrightarrow{\simeq} W^\%C$. But the image of this map is contained in $W^\%_0C$ so indeed $W^\%C\simeq W^\%_0C$ in this case.
\end{proof}

\subsection{Skew-suspension}\label{subsec:skewsusp}

The idea of `skew-suspending' will allow us to compare symmetric chain complexes of different dimensions. This will be exploited in Chapter \ref{chap:DLtheory} to analyse dimensional periodicity in the groups of symmetric Poincar\'{e} complexes we shall introduce later.

\medskip

Given a chain complex $C$ in $\B(A)$, the suspension $\Sigma C$ contains the same information as $C$. In some sense, this fact is true for chain complexes equipped with an $\eps$-symmetric structure. The naive approach of taking an $n$-dimensional $\eps$-symmetric complex $(C,\phi)$ and hoping for an $(n+1)$-dimensional $\eps$-symmetric structure on $\Sigma C$ does indeed work, and it is possible (\cite[11.54]{TIBOR}) to construct a map \[S:\Sigma W^\%C\to W^\%(\Sigma C)\]The problem is that $S$ is not in general a chain homotopy equivalence, so $W^\%C$ does not carry the same information as $W^\%\Sigma C$, as hoped.

The interactions between chain complexes and their suspensions via the map $S$ is an important part of the surgery programme for manifold classification due to the idea of using suspensions of spaces to obtain `quadratic refinements' of symmetric structure. We will not need this in the sequel so we simply quote the following result (which we \textit{will} need later).

\begin{proposition}[{\cite[Case $p=1$ of 1.1.3]{MR620795}}]\label{prop:ultra}Suppose $\overline{\eps}=\eps^{-1}$ and $C$ is in $\B(A)$. Then there is a short exact sequence of chain complexes\[0\to \Sigma W^\% C\xrightarrow{S} W^\%\Sigma C \to \Sigma C\otimes \Sigma C\to 0,\]and the connecting morphism in the associated long exact homology sequence is induced by a chain map\[1+T_\eps:\Hom_A(C^{-*},C))\to W^\% C;\qquad (\hat{\psi}+\eps\hat{\psi}^*)=(\hat{\psi}+\eps\hat{\psi}^*)_0.\]
\end{proposition}

In order to retain the same symmetry information upon chain complex suspension, we can take a less naive approach than above and insert a `dummy' top dimension after we suspend. Then the 0 we have introduced at the bottom of the chain complex is paired to a 0 at the top of the chain complex. Here is a simple example illustrating the idea:

\begin{example}If we equip a chain complex \[C:\qquad\qquad0\to C_0\to 0,\qquad\qquad\]with a 0-dimensional $\eps$-symmetric structure $\phi\in W^\%C\simeq \Hom_{\Z[\Z/2\Z]}(W_0,C\otimes C)$ it is completely defined by $\phi_0:C^0\to C_0$. From this, we may define a 2-dimensional $(-\eps)$-symmetric structure on \[\Sigma C:\qquad\qquad 0\to 0\to C_0\to0\to0\qquad\qquad\]by \[(\overline{S}\phi)_0=\pm\phi_0:(\Sigma C)^1\to (\Sigma C)_1,\] and $(\overline{S}\phi)_s=0$ otherwise. This structure is $(-\eps)$-symmetric as transposing elements in $(\Sigma C)_1$ means we pick up a minus sign.
\end{example}

More generally, given a chain complex $C$ in $\B(A)$, there is a homotopy equivalence defined by\[\overline{S}:\Sigma^2(C^t\otimes_A C)\xrightarrow{\simeq} (\Sigma C)^t\otimes_A(\Sigma C);\qquad x\otimes y\mapsto (-1)^{|x|}x\otimes y\]in $\Ch(\Z[\Z/2\Z])$, where it is understood that the involution on $C^t\otimes_A C$ uses $T_\eps$ and the involution on $(\Sigma C)^t\otimes_A(\Sigma C)$ uses $T_{-\eps}$. Applying the $W^\%$ functor to this chain equivalence we obtain the homotopy equivalence\[\overline{S}:\Sigma^{2}W_\eps ^\%C\xrightarrow{\simeq} W^\%_{-\eps}\Sigma C\]in $\Ch(\Z)$. By similar reasoning, given a morphism $f:C\to D$ in $\B(A)$ there is a homotopy equivalence in $\Ch(\Z)$\[\overline{S}:\Sigma^2C(f^\%:W_\eps^\%C\to W^\%_\eps D)\xrightarrow{\simeq}C((\Sigma f)^\%:W_{-\eps}^\%(\Sigma C)\to W^\%_{-\eps}(\Sigma D)).\]

\begin{definition}If $(C,\phi\in W_\eps^\%C_n)$ is an $n$-dimensional, $\eps$-symmetric (Poincar\'{e}) complex over $A$ then the \textit{skew-suspension} of $(C,\phi)$ is the $(n+2)$-dimensional, $(-\eps)$-symmetric (Poincar\'{e}) complex $\overline{S}(C,\phi)=(\Sigma C,\overline{S}\phi\in W_{-\eps}^\%\Sigma C_{n+2})$.

If $(f:C\to D,(\delta\phi,\phi)\in C(f^\%)_{n+1})$ is an $(n+1)$-dimensional, $\eps$-symmetric (Poincar\'{e}) pair over $A$ then the \textit{skew-suspension} of $(f:C\to D,(\delta\phi,\phi))$ is the $(n+3)$-dimensional, $(-\eps)$-symmetric (Poincar\'{e}) pair \[\overline{S}(f:C\to D,(\delta\phi,\phi)):=(\Sigma f:\Sigma C\to \Sigma D,\overline{S}(\delta\phi,\phi)\in C((\Sigma f)^\%)_{n+3}).\]
\end{definition}

\section{The symmetric constructions}\label{sec:symmetric}

We now show how the structures described so far arise very naturally from geometry. In the following account, $X$ will always denote a topological space, $C(X)$ its singular chain complex and $\widetilde{X}$ its universal covering space. Given a group $\pi$, the group ring $\Z[\pi]$ is considered to have the involution extended linearly from $g\mapsto g^{-1}$\text{.} 

\medskip

First we recall two basic chain-level constructions.

\medskip

Given $X$ and a natural choice of chain homotopy equivalence $\alpha:C(X\times X)\xrightarrow{\simeq} C(X)\otimes C(X)$, the diagonal map $\Delta:X\to X\times X$ induces a natural morphism\[\Delta_0:C(X)\xrightarrow{\Delta_*}C(X\times X)\xrightarrow{\alpha}C(X)\otimes C(X).\]

Suppose $m=p+q$. Fixing a choice of $\alpha$, an $m$-dimensional cycle of the form $x\otimes y\in C^p(X)\otimes C^q(Y)$ is pulled back to an $m$-cycle $x\cup y:=\Delta_0^*(x\otimes y)\in C^m(X)$ called the \textit{(chain-level) cup product of $x$ and $y$}. Also, for a fixed choice of $\alpha$, the composition \[C(X)\xrightarrow{\Delta_0}C(X)\otimes C(X)\xrightarrow{\sm}\Hom(C(X)^{-*},C(X))\] sends an $n$-cycle $x\in C_n(X)$ to a chain map called the \textit{(chain-level) cap product with $x$} and written\[x\cap-:C^{n-*}(X)\to C(X).\] 

\subsection*{Chain diagonal approximation (idea)}

We now sketch a standard extension of these chain-level constructions that can be found in many textbooks, for instance we base our explanation on \cite[VI.16]{MR1224675}. We will need the equivariant version to deal with covering spaces, particularly the infinite cyclic cover when we move on to knot theory. We will not develop these entirely but will rely on our references.

The existence of a natural chain equivalence $\alpha:C(X\times X)\xrightarrow{\simeq} C(X)\otimes C(X)$ is ensured by the method of acyclic models (see e.g. \cite[IV.16]{MR1224675}), and a choice of $\alpha$ is unique up to natural homotopy equivalence (again, by the method of acyclic models). In particular, fix a choice of $\alpha$ and denote by $T$ the involution on $C(X)\otimes C(X)$ defined by $T(x\otimes y)=(-1)^{pq}(y\otimes x)$ for $x\in C^p(X), y\in C^q(X)$. Then $T\circ\alpha:C(X\times X)\to C(X)\otimes C(X)$ is also natural chain homotopy equivalence and hence there exists a natural chain homotopy equivalence $\alpha\simeq T\circ\alpha$.

Denote by $\Delta_1:C(X)\to C(X)\otimes C(X)$ a natural degree 1 chain map such that $(T-1)\Delta_0=\Delta_1d+d\Delta_1$ (i.e. $\Delta_1$ is a nullhomotopy of $(T-1)\Delta_0$), induced by a choice of natural chain homotopy equivalence $\alpha\simeq T\circ\alpha$. We have \[(T+1)\Delta_1d+d(T+1)\Delta_1=(T+1)(\Delta_1d+d\Delta_1)=0,\]as $T^2=1$. So that $(T+1)\Delta_1$ is a natural chain map of degree 1. 

Denote by $\Delta_2:C(X)\to C(X)\otimes C(X)$ a natural degree 2 chain map that is a nullhomotopy of the degree 1 chain map $(T+1)\Delta_1$ (in the sense that $(T+1)\Delta_1=d\Delta_2-\Delta_2d$). Again this is assured to exist by the method of acyclic models. 

Repeating this process we obtain a sequence $\Delta_s:C(X)\to C(X)\otimes C(X)$ of degree $s$ chain maps for $s\geq0$ that are natural in $X$ and interrelated by \[(T+(-1)^{s+1})\Delta_s=d\Delta_{s+1}+(-1)^s\Delta_{s+1}d.\]This has a more slick expression in the form of a natural (degree 0) chain map\[\Delta_X:W\otimes C(X)\to C(X)\otimes C(X);\qquad 1_s\otimes x\mapsto \Delta_s(x).\]

Now assume $\pi$ is a group and $\overline{X}$ is a $\pi$-space so that $C(\overline{X})$ admits the structure of a $\Z[\pi]$-module chain complex. By extending the $\pi$-action trivially to $W$ over the tensor $W\otimes C(\overline{X})$ and by equipping $C(\overline{X})\otimes C(\overline{X})$ with the diagonal $\pi$-action, we may regard $\Delta_{\overline{X}}$ as a morphism in $\Ch(\Z[\pi])$. The adjoint morphism to $\Delta_{\overline{X}}$ in $\Ch(\Z[\pi])$ is \begin{equation}\label{eq:adjoint}C(\overline{X})\to\Hom_{\Z[\Z/2\Z]}(W,C(\overline{X})\otimes_{\Z}C(\overline{X})).\end{equation}Note that the image of a cycle $x\in C(\overline{X})$ under the morphism \ref{eq:adjoint}, followed by the slant map is the chain level cap product with $x$. 

Now, using the trivial $\pi$-action on $\Z$, apply $\Z^t\otimes_{\Z[\pi]}-$ to both sides of the morphism \ref{eq:adjoint} to obtain a morphism in $\Ch(\Z)$\[\phi_{\overline{X}}:\Z^t\otimes_{\Z[\pi]}C(\overline{X})\to\Hom_{\Z[\Z/2\Z]}(W,C(\overline{X})^t\otimes_{\Z[\pi]}C(\overline{X})),\]called the \textit{Alexander-Whitney diagonal approximation} of the $\pi$-space $\overline{X}$. In particular, if $\overline{X}$ is a covering space for $X$ with group of covering translations $\pi$, we have $\Z^t\otimes_{\Z[\pi]}C(\overline{X})=C(X)$ in $\Ch(\Z)$ so that \[\phi_{\overline{X}}:C(X)\to W^\%C(\overline{X}).\]

\begin{remark}The Alexander-Whitney diagonal approximation described above is most commonly used to construct the classical Steenrod square cohomology operations on a $\pi$-space $\overline{X}$. Indeed, suppose for simplicity that we are working with $\Z/2\Z$-coefficients, that $\pi$ is trivial and that for each $s\geq0$ we have constructed the degree $s$ chain map $\Delta_s:C(X)\to C(X)\otimes C(X)$. The $s$th Steenrod square is the homomorphism defined by \[\text{Sq}^s:H^r(X;\Z/2\Z)\to H^{r+s}(X;\Z/2\Z);\qquad a\mapsto \Delta_s^*(a\otimes a).\]See \cite[VI.16]{MR1224675} for details.
\end{remark}

\begin{remark}It will be very important later on when we define the algebraic $L$-groups that the chain complexes we use are homotopy equivalent to chain complexes in $\B(A)$, this is necessary to avoid certain pathologies related to the Eilenberg swindle. We have been trying to work in the fullest generality and have hence used the singular chain complex of a topological space $C(X)$. In the case of compact manifolds we will be able to fix this problem by choosing a finite $CW$ structure.
\end{remark}

\subsubsection*{Pair and triad versions}

A continuous map of $\pi$-spaces $f:\overline{X}\to \overline{Y}$ is a $\pi$-map if $f(gx)=g(f(x))$ for each $g\in \pi$ and $x\in \overline{X}$. Given a $\pi$-map of $\pi$-spaces, the \textit{Alexander-Whitney diagonal approximation for the $\pi$-map} is defined simply by taking the algebraic mapping cone $\phi_f:=C(\phi_{\overline{Y}},\phi_{\overline{X}})$, giving\[\phi_f:C(\Z^t\otimes_{\Z[\pi]}f)\to C(f^\%:W^\%C(\overline{X})\to W^\%C(\overline{Y})).\]Similarly, suppose we have a homotopy commuting diagram $\Gamma$ of $\pi$-spaces and $\pi$-maps \[\xymatrix{\overline{X}\ar[r]^-{f}\ar[d]_-{f'}&\overline{Y}\ar[d]^-{g}\\\overline{Y'}\ar[r]^-{g'}&\overline{Z}}\]such that the induced diagram of singular chain complexes in $\Ch(\Z[\pi])$ homotopy commutes. Then the Alexander-Whitney diagonal approximations so far defined determine a morphism of homotopy commutative squares in $\Ch(\Z)$\[\begin{array}{rcl}\xymatrix{\Z^t\otimes C(\overline{X})\ar[r]^-{f}\ar[d]_-{f'}&\Z^t\otimes C(\overline{Y})\ar[d]^-{g}\\ \Z^t\otimes C(\overline{Y'})\ar[r]^-{g'}&\Z^t\otimes C(\overline{Z})}&\quad\to\quad&\xymatrix{W^\%C(\overline{X})\ar[r]^-{f^\%}\ar[d]_-{(f')^\%}&W^\%C(\overline{Y})\ar[d]^-{g^\%}\\W^\%C(\overline{Y'})\ar[r]^-{(g')^\%}&W^\%C(\overline{Z})}\end{array}\]inducing a morphism of iterated cones\[\phi_\Gamma:C(C(g,f'))\to C(C(g^\%,f'^\%)),\]called the \textit{Alexander-Whitney chain diagonal approximation for the homotopy commuting square $\Gamma$}.

\subsection*{Implicit choice of $CW$ structure}

If $M$ is a compact oriented manifold and $\overline{M}\to M$ is an oriented covering with group of covering translations $\pi$ then we may choose a finite $CW$ complex $K\simeq M$. We may also choose a lift of the $CW$ structure to $\overline{M}$ so that there is a $CW$ complex $\overline{K}\simeq \overline{M}$ with $C^{cell}(\overline{K})$ a finite f.g.\ free $\Z[\pi]$-module chain complex. As such there is a chain homotopy equivalence $f:C(M;\Z[\pi])\simeq C(\overline{K})$ to an object of $\B(\Z[\pi])$. This induces a chain homotopy equivalence $f^{\%}:W^\%C(M;\Z[\pi])\simeq W^\%C(\overline{K})$

This procedure works equally well in the cases of compact manifolds with boundary and compact manifold triads which are considered later. 

\textbf{From now on, all compact manifolds are implicitly equipped with a choice of finite $CW$ structure. When there is a covering of the manifold, a choice of lift of the $CW$ structure is implicit.}

\subsection*{Compact, oriented manifolds and Umkehr maps}

We now recall the definition of universal Poincar\'{e} duality from \ref{def:universal} and \ref{ex:compactman}. Suppose $M$ is a compact, oriented, $n$-dimensional topological manifold. Then the universal cover $\widetilde{M}$ has singular chain complex $C(\widetilde{M};\Z)=:C(M;\Z[\pi_1(M)])$, a finite chain complex of free $\Z[\pi_1(X)]$-modules. If $\overline{M}\to M$ is an oriented covering with group of covering translations $\pi$ then there is a natural chain homotopy equivalence\begin{equation}\label{eq:univpd}\Z[\pi]\otimes_{\Z[\pi_1(M)]}C(\widetilde{M})\simeq C(\overline{M};\Z)=:C(M;\Z[\pi]).\end{equation}The orientation on $M$ determines a choice of \textit{fundamental class} $[M]\in H_n(\widetilde{M})$ inducing \[[M]\cap-:H^r(\widetilde{M};\Z)\xrightarrow{\cong}H_{n-r}(\widetilde{M};\Z)\]for each $r\geq 0$. Moreover $M$ satisfies Poincar\'{e} duality with respect to any oriented covering $\overline{M}\to M$ in the sense that if the group of covering translations is $\pi$ then the equation \ref{eq:univpd} maps the fundamental class $[M]\in H_n(\widetilde{M};\Z)$ to some class $[M]\in H_n(\overline{M};\Z)$, inducing $\Z[\pi]$-module isomorphisms \[[M]\cap-:H^r(\overline{M};\Z)\xrightarrow{\cong}H_{n-r}(\overline{M};\Z)\]for each $r\geq 0$.

\begin{proposition}[{\cite[2.1]{MR566491}}]\label{prop:umkehr1}If $M$ is a compact, oriented, $n$-dimensional manifold and $\overline{M}\to M$ is an oriented covering of $M$ with group of covering translations $\pi$ then a choice of finite $CW$ structure defines, in a natural way up to homotopy equivalence, an $n$-dimensional symmetric Poincar\'{e} complex in $\B(\Z[\pi])$\[\sigma^*(\overline{M})\simeq(C(M;\Z[\pi]),\phi_{\overline{M}}([M])).\]
\end{proposition}

A continuous map $f:M\to N$ of compact, oriented, $n$-dimensional manifolds is \textit{degree 1} if $f_*([M])=[N]\in H_n(\overline{N})$. If $\overline{N}\to N$ is an oriented covering of $N$ with group of covering transformations $\pi$ then there is an induced cover $\overline{M}:=f^*\overline N$ of $M$ and a $\pi$-map of $\pi$-spaces $\overline{f}:\overline{M}\to \overline{N}$. If $f$ is degree 1, there is then defined an \textit{Umkehr} chain map \[\xymatrix{\overline{f}^!:C(\overline{N})\ar[rr]^-{[N]\cap-} &&C^{n-*}(\overline{N})\ar[r]^-{\overline{f}^*} &C^{n-*}(\overline{M})\ar[rr]^-{[M]\cap-}&&C(\overline{M}).}\]

\begin{claim}[{\cite[2.2]{MR566491}}]Denoting by $e$ the inclusion into the cone, there is a split short exact sequence of chain complexes\[0\to C(\overline{N})\xrightarrow{\overline{f}^!} C(\overline{M})\xrightarrow{e} C(\overline{f}^!)\to 0.\]
\end{claim}

\begin{proposition}[{\cite[2.2]{MR566491}}]\label{prop:umkehr2}If $f:M\to N$ is a continuous, degree 1 map of compact, oriented, $n$-dimensional manifolds and $\overline{N}\to N$ is an oriented covering of $N$ with group of covering transformations $\pi$ then choices of finite $CW$ structure define, in a natural way up to homotopy equivalence, an $n$-dimensional symmetric Poincar\'{e} complex in $\B(\Z[\pi])$ called the \textit{kernel complex}\[\sigma^*(f)\simeq(C(\overline{f}^!),e^\%\phi_{\overline{M}}([M])).\]Moreover, there is an equivalence of symmetric complexes\[\left(\begin{matrix} e\\\overline{f}\end{matrix}\right):\sigma^*(\overline{M})\xrightarrow{\simeq}\sigma^*(f)\oplus \sigma^*(\overline{N}).\]\end{proposition}

The pair and triad versions will follow formally from Propositions \ref{prop:umkehr1} and \ref{prop:umkehr2} using the functoriality of the symmetric construction and of the $W^\%$ functor in the homotopy categories of finite CW structures and of $\B(\Z[\pi]))$ respectively. We will now spell out the triad and pair versions, for use in subsequent chapters.

\subsubsection*{Pair versions}

Suppose $(M,\partial M)$ is a compact, oriented $n$-dimensional topological manifold with boundary. Write the singular chain complex $C(\widetilde{M},\overline{\partial M};\Z)=:C(M,\partial M;\Z[\pi_1(M)])$ in $\B(\Z[\pi_1(M)])$ where $\widetilde{M}$ is the universal cover of $M$ and $\overline{\partial M}$ is the induced oriented cover of $\partial M$. The orientation determines a \textit{fundamental class} $[M]\in H_n(M,\partial M;\Z[\pi_1(M)])$ inducing Poincar\'{e}-Lefschetz duality isomorphisms\[[M]:H^r(M,\partial M;\Z[\pi_1(M)])\xrightarrow{\cong} H_{n-r}(M;\Z[\pi_1(M)]).\]A covering $\overline{M}\to M$ induces a covering $\overline{\partial M}\to \partial M$, pulled back via the inclusion $i_M:\partial M\to M$, and there is a \textit{universal Poincar\'{e}-Lefschetz duality} analogously to the universal Poincar\'{e} duality described above.

\begin{proposition}If $(M,\partial M)$ is a compact, oriented, $n$-dimensional manifold with boundary and $\overline{M}\to M$ is an oriented covering of $M$ with group of covering translations $\pi$ then a choice of finite $CW$ structure defines, in a natural way up to homotopy equivalence, an $n$-dimensional symmetric Poincar\'{e} pair in $\B(\Z[\pi])$\[\sigma^*(\overline{M},\overline{\partial M})=(\overline{i_M}:C(\overline{\partial M})\to C(\overline{M}),\phi_{\overline{i_M}}([M])).\]
\end{proposition}

A continuous map of compact, oriented, $n$-dimensional, manifolds with boundary $(f,\partial f): (M,\partial M)\to (N,\partial N)$ is \textit{degree 1} if $f_*([M])=[N]\in H_n(N,\partial N;\Z)$. If $(f,\partial f)$ is degree 1 then $\partial f: \partial M\to \partial N$ is a degree 1 map of compact, oriented, $n$-dimensional manifolds. If $\overline{N}\to N$ is an oriented cover of $N$ then there is defined an \textit{Umkehr} chain map\[\overline{f}^!: C(\overline{N})\xrightarrow{[N]\cap-} C^{n-*}(\overline{N},\overline{\partial N})\xrightarrow{\overline{f}^*}C^{n-*}(\overline{M},\overline{\partial M})\xrightarrow{[M]\cap-}C(\overline{M}).\]

\begin{proposition}\label{prop:pairdiag}If $(f,\partial f):(M,\partial M)\to (N,\partial N)$ is a degree 1 map of compact, oriented, $n$-dimensional manifolds with boundary and $\overline{N}\to N$ is an oriented covering of $N$ with group of covering translations $\pi$ then choices of finite $CW$ structure define, in a natural way up to homotopy equivalence, an $n$-dimensional symmetric Poincar\'{e} pair in $\B(\Z[\pi])$ called the \textit{kernel pair}\[\sigma^*(\overline{f},\overline{\partial f})\simeq(C(\overline{\partial f}^!)\to C(\overline{f}^!),C(e^\%,\partial e^\%)\phi_{\overline{i_M}}([M])).\]Moreover, there is an equivalence of symmetric pairs\[\sigma^*(\overline{M},\overline{\partial N})\xrightarrow{\simeq}\sigma^*(\overline{f},\overline{\partial f})\oplus \sigma^*(\overline{N},\overline{\partial N}).\]
\end{proposition}

\subsubsection*{Triad versions}

Now suppose $(W,\partial W)$, $(M,N)$ and $(M',N)$ are compact oriented topological manifolds with boundary and that $\partial W\cong M\cup_NM'$ is an homeomorphism of compact oriented $(n-1)$-dimensional manifolds. The quadruple $(W;M,M';N)$ is called a \textit{compact, oriented manifold triad} and the compatible orientations determine a \textit{fundamental class} $[W]\in H_n(W;M\cup_N M';\Z[\pi_1(W)])$ compatible with $[M]\in H_{n-1}(M,N;\Z[\pi_1(W))$ and $[M']\in H_{n-1}(M',N;\Z[\pi_1(W)])$. A compact oriented manifold triad together with a covering $\overline{W}\to W$ with group of covering translations $\pi$ determines (via pullback) a homotopy commuting diagram $\Gamma$ of $\pi$-spaces and $\pi$-maps \[\xymatrix{\overline{N}\ar[r]^-{i_{\overline{N}'}}\ar[d]_-{i_{\overline{N}}}&\overline{M'}\ar[d]^-{i_{\overline{M'}}}\\\overline{M}\ar[r]^-{i_{\overline{M}}}&\overline{W}}\]Choosing a finite $CW$ structure, this induces a homotopy commuting diagram of homotopy cofibration sequences of chain complexes of f.g.\ free modules in $\B(\Z[\pi])$, the top-left square of which we denote $C(\Gamma)$ \begin{equation}\label{eq:relcob}\xymatrix{C(\overline{N})\ar[r]\ar[d]&C(\overline{M'})\ar[r]\ar[d]&C(\overline{M'},\overline{N})\ar[d]\\C(\overline{M})\ar[r]\ar[d]&C(\overline{W})\ar[r]\ar[d]&C(\overline{W},\overline{M})\ar[d]\\C(\overline{M},\overline{N})\ar[r]&C(\overline{W},\overline{M'})\ar[r]&C(\overline{W},\overline{M}\cup_{\overline{N}}\overline{M'})}\end{equation}

\begin{proposition}If $(W;M,M';N)$ is a compact, oriented, $n$-dimensional manifold triad and $\overline{W}\to W$ is an oriented covering of $W$ with group of covering translations $\pi$ then a choice of finite $CW$ structure defines, in a natural way up to homotopy equivalence, an $n$-dimensional symmetric Poincar\'{e} triad in $\B(\Z[\pi])$\[\sigma^*(\overline{W};\overline{M},\overline{M'};\overline{N})=(C(\Gamma),\phi_\Gamma([W])).\]
\end{proposition}

A continuous map of compact, oriented $n$-dimensional manifold triads\[F=(f,\partial f,\partial' f,\partial \partial f):(W_1;M_1,M_1';N_1)\to (W_2;M_2,M_2';N_2),\]is \textit{degree 1} if \[(f,\partial f\cup_{\partial \partial f}\partial' f):(W_1,M_1\cup_{N_1 }M_1')\to (W_2,M_2\cup_{N_2}M_2')\] is a degree 1 map of manifolds with boundary. If $(f,\partial f,\partial'f,\partial\partial f)$ is degree 1 and there is a covering $\overline{W_2}\to W_2$ with group of covering translations $\pi$ then the associated Umkehr maps of the 3 maps of pairs $(f,\partial f\cup \partial' f)$, $(\partial f, \partial \partial f)$ and $(\partial' f,\partial \partial f)$, define a homotopy commuting square $C(\overline{F}^!)$ in $\B(\Z[\pi])$ given by\[\xymatrix{C((\overline{\partial\partial f})^!)\ar[r]\ar[d]&C((\overline{\partial' f})^!)\ar[d]\\C((\overline{\partial f})^!)\ar[r]&C(\overline{f}^!)}\]Using the 3 maps of pairs $(f,\partial f\cup \partial' f)$, $(\partial f, \partial \partial f)$ and $(\partial' f,\partial \partial f)$, we may construct:

\begin{proposition}\label{prop:triaddiag}If $F:(W_1;M_1,M_1';N_1)\to (W_2;M_2,M_2';N_2)$ is a degree 1 map of compact, oriented, $n$-dimensional manifold triads and $\overline{W_2}\to W_2$ is a covering with group of covering translations $\pi$ then choices of finite $CW$ structures define, in a natural way up to homotopy equivalence, an $n$-dimensional symmetric Poincar\'{e} triad in $\B(\Z[\pi])$ called the \textit{kernel triad}\[\sigma^*(\overline{F}^!)=(C(\overline{F}^!),C(e^\%;\partial e^\%,\partial' e^\%;\partial\partial e^\%)\phi_\Gamma([W])).\]Moreover, there is a homotopy equivalence of symmetric triads\[\sigma^*(\overline{W_1};\overline{M_1},\overline{M_1'};\overline{N_1})\xrightarrow{\simeq} \sigma^*(\overline{F}^!)\oplus \sigma^*(\overline{W_2};\overline{M_2},\overline{M_2'};\overline{N_2}).\]
\end{proposition}

\section{Algebraic Thom constructions and thickenings}

One beautiful flexibility of the Algebraic Theory of Surgery is the dual perspective on an object described by `thickening' or conversely by an algebraic `Thom construction'. This will allow us to change  perspective between complexes/pairs or pairs/triads later - in particular when we construct the Blanchfield complex of an $n$-knot exterior in Chapter \ref{chap:blanchfield} it will be a symmetric complex, rather than a pair as one might expect from a manifold with boundary. The cost of the change in perspective is that, while thickening always improves an object to a Poincar\'{e} object, the Thom construction will in general destroy the property of being Poincar\'{e}.

\subsection{Symmetric Poincar\'{e} pairs vs. symmetric complexes}\label{subsec:cxpair}

\subsubsection*{Thom construction}

Given an $(n+1)$-dimensional $\eps$-symmetric pair $(f:C\to D,(\delta\phi,\phi))$ recall that a choice of nullhomotopy \[\xymatrix{C\ar@/_1pc/[rr]_{j:ef\simeq 0}\ar[r]^f&D\ar[r]^e&C(f)}\] induces a morphism $\Phi_{j^\%}:C(f^\%)\to W^\%C(f)$. Define the $(n+1)$-dimensional cycle $\delta\phi/\phi:=\Phi_{j^\%}(\delta\phi,\phi)\in W^\%C(f)_{n+1}$. The \textit{algebraic Thom construction} for $(f:C\to D, (\delta\phi,\phi))$ is the $(n+1)$-dimensional $\eps$-symmetric complex $(C(f),\delta\phi/\phi)$.

\subsubsection*{Algebraic thickening}

We wish to reverse the algebraic Thom construction. Fix an $(n+1)$-dimensional $\eps$-symmetric complex $(C,\phi)$ and consider the Puppe sequence of chain complexes in $\B(A)$ associated to the evaluation $\setminus(\ev(\phi))=\phi_0$ with a fixed choice of nullhomotopy $j:\phi_0i\simeq 0$\[\xymatrix{\dots\ar[r]&\Sigma^{-1}C(\phi_0)\ar@/_1pc/[rr]_{j:\phi_oi\simeq 0}\ar[r]^i& C^{n+1-*}\ar[r]^-{\phi_0}&C\ar[r]&C(\phi_0)\ar[r]&\dots}\]so that $C\simeq C(i)$. Define the chain complex $\partial C:=\Sigma^{-1}C(\phi_0)$. There is then a Puppe sequence of chain complexes of abelian groups \[\xymatrix{\dots\ar[r]&W^\%(\partial C)\ar[r]^-{i^\%}& W^\%(C^{n+1-*})\ar[r]&C(i^\%)\ar[r]^-{\text{proj}}&\Sigma W^\%\partial C\ar[r]&\dots}\]We must now lift $\phi\in W^\%C\simeq W^\%C(i)$ to $C(i^\%)$ via $\Phi_{j^\%}$. Recall the homotopy pullback square of Proposition \ref{prop:pullback} \[\xymatrix{C(i^\%)\ar[r]^-{\ev_r}\ar[d]_{\Phi_{j^\%}}&C^{n+1-*}\otimes C\ar[d]\cong \Hom_A(C_{*-(n+1)},C)\\W^\%C\ar[r]_-{\ev}&C\otimes C\cong\Hom_A(C^{-*},C)}\]We thus have a distinguished choice of cycle \[(\phi,\phi_0,\id)\in C(i^\%)\simeq W^\%C\times_{C\otimes C}C^{n+1-*}\otimes C,\]lifting $\phi\in W^\%C$. If we like, we can push this cycle along the cofibration sequence and formally desuspend to obtain an $n$-dimensional cycle $\partial \phi:=\Sigma^{-1}\text{proj}(\phi,\phi_0,\id)$ and a nullhomotopy $\overline{\phi}:=(\Sigma i^\%)(\text{proj}(\phi,\phi_0,\id))$. Using these, we can describe our distinguished choice of cycle in the more standard form and write an $(n+1)$-dimensional $\eps$-symmetric pair \[(i:\partial C\to C^{n+1-*},(\overline{\phi},\partial\phi))\] called the \textit{algebraic thickening} of $(C,\phi)$. In fact it is straightforward to show that the algebraic thickening of $(C,\phi)$ is always Poincar\'{e} (even when $(C,\phi)$ was not!) so that there is an $n$-dimensional $\eps$-symmetric Poincar\'{e} complex $(\partial C,\partial\phi)$ called the \textit{boundary of $(C,\phi)$}. With a little more work there is:

\begin{proposition}[{\cite[1.15]{MR1211640}}]\label{prop:thomthick1}The algebraic Thom construction and algebraic thickening describe a natural 1:1 correspondence between homotopy equivalence classes of $(n+1)$-dimensional $\eps$-symmetric complexes and homotopy equivalence classes of $(n+1)$-dimensional $\eps$-symmetric Poincar\'{e} pairs. An $\eps$-symmetric complex is Poincar\'{e} if and only if it is homotopy equivalent to the Thom complex of a pair of the form\[(0:0\to D,(\phi,0)).\]
\end{proposition}

\begin{remark}The result of Proposition \ref{prop:thomthick1} is designed to mimic the process of embedding a finite CW complex in a high dimensional sphere, taking a regular neighbourhood (i.e. manifold with boundary) and then applying the Pontryagin-Thom collapse map. As such it is not so shocking that the algebraic thickening of a symmetric complex should have Poincar\'{e} symmetric structure.
\end{remark}

\subsection{Symmetric Poincar\'{e} triads vs. symmetric surgery dual pairs}

It is important for knot-theoretic applications in Chapter \ref{chap:blanchfield} and various technical constructions in Chapter \ref{chap:DLtheory} to have a firm grasp of the connections between triads and pairs. Along the way, we will need to define what it means for two $(n+2)$-dimensional $\eps$-symmetric pairs to be \textit{surgery dual} to one another up to homotopy. The goal of this subsection is then to prove the triad/pair version of Proposition \ref{prop:thomthick1}:

\begin{proposition}\label{prop:triad}The \textit{relative algebraic Thom construction} and the \textit{algebraic thickening} (definitions below) describe a 1:1 correspondence between sets $\{x,x'\}$ of $(n+2)$-dimensional $\eps$-symmetric pairs \[\begin{array}{rcl}x&=&(f:C\to D,(\delta\phi,\phi))\\x'&=&(f':C'\to D',(\delta\phi',\phi'))\end{array}\]that are \textit{surgery dual} to one another up to homotopy (precise definition below) and homotopy equivalence classes of $(n+2)$-dimensional $\eps$-symmetric Poincar\'{e} triads.
\end{proposition}

While the content of Chapter \ref{chap:algLtheory} is otherwise only a review, Proposition \ref{prop:triad} is original (although it was anticipated to some extent by Weiss in \cite[Section 4]{MR794111}).

\subsubsection*{Relative Thom construction}

Given an $(n+2)$-dimensional $\eps$-symmetric triad $(\Gamma,(\Phi,\delta\phi,\delta'\phi,\phi))$ consider the induced maps of cones\[\begin{array}{rrcl}\nu=C(g,f'):&C(f)&\to&C(g')\\\nu'=C(g',f):&C(f')&\to &C(g)\end{array}\]Using diagrams \ref{eq:lift}, \ref{eq:level1} and Proposition \ref{prop:pullback} we obtain a series of morphisms\[\xymatrix{&C(F)\ar[d]&\\C(\nu^\%)&C((g'')^\%)\ar[r]\ar[l]&C((\nu')^\%)}\]Using the two images of $(\Phi,\delta\phi,\delta'\phi,\phi)\in C(F)$ under these two composites, a triad defines two $(n+2)$-dimensional $\eps$-symmetric pairs \[\begin{array}{rcl}x&=&(\nu:C(f)\to C(g'),(\Phi/\delta\phi,\delta\phi'/\phi))\\x'&=&(\nu':C(f')\to C(g),(\Phi/\delta\phi,\delta\phi'/\phi))\end{array}\]The \textit{relative algebraic Thom construction} for $(\Gamma,(\Phi,\delta'\phi,\delta\phi,\phi))$ is defined to be the set $\{x,x'\}$.

\subsubsection*{Algebraic Surgery}

The pairs $x$ and $x'$ of the relative algebraic Thom construction are related to each other in the following way, that is best motivated geometrically. Recall the setup of diagram \ref{eq:relcob} and for this motivating section assume $\overline{W}=W$ is the trivial cover. Suppose $N=\emptyset$ so that $W$ is a cobordism from $M$ to $M'$ (the case where $N\neq\emptyset$ is more fiddly but the idea is the same). Up to homotopy, the mapping cone on the inclusion $M'\hookrightarrow W/M$ is the space $W/(M\sqcup M')$, but this is also the cone on the inclusion $M\hookrightarrow W/M'$.  A Poincar\'{e} pair $(f:C\to D,(\delta\phi,\phi))$ represents the pair of spaces $(W/M',M)$, which should be thought of as the data for one `end' of the cobordism. There is a well-known correspondence between cobordism of manifolds and `(geometric) surgery' so that doing geometric surgery on $M$ corresponding to the trace $W$ finds $(W/M,M')$, the other `end'. On the chain level, Poincar\'{e}-Lefschetz duality tells us that there is a chain equivalence of singular chain complexes $C^{n+1-*}(W,M')\simeq C(W,M)$, so that there is a homotopy cofibration sequence \[C(M')\to C^{n+1-*}(W,M')\xrightarrow{[W]\cap-} C(W,M\sqcup M').\]Hence it possible to find the chain homotopy type of $M'$ by just working with the Puppe sequence of the morphism\[[W]\cap-:C^{n+1-*}(W,M')\to C(W,M\sqcup M').\]We now take this Puppe sequence approach as a general strategy on the chain level for finding the ``other end'' of a cobordism. This is the basic concept of algebraic surgery.

Purely algebraically, then, begin with a (possibly non-Poincar\'{e}) $(n+2)$-dimensional $\eps$-symmetric pair \[x=(f:C\to D,(\delta\phi,\phi))\] with homotopy cofibration sequence \[\xymatrix{C\ar[r]^-{f}&D\ar[r]^-{e} &C(f)}.\]Recall we have $\setminus(\ev_r(\delta\phi,\phi)):D^{n+2-*}\to C(f)$ and define $C':=\Sigma^{-1}C(\setminus(\ev_r(\delta\phi,\phi)))$ so that there is a homotopy cofibration sequence\[\xymatrix{\dots\ar[r]&C'\ar[r]^-{f'}& D^{n+2-*}\ar[rr]^-{\setminus(\ev_r(\delta\phi,\phi))}&& C(f)\ar[r]^-{e'}& \Sigma C'\ar[r]&\dots}\]with $f'$ the inclusion and $e'$ the projection. Next we wish to lift the $\eps$-symmetric structure $\Phi:=\delta\phi/\phi\in W^\%C(f)$ to $C((f')^\%)$. But now there is a distinguished choice of cycle\[(\Phi,\Phi_0,e)\in C((f'^\%))\simeq W^\%C(f)\times_{C(f)\otimes C(f)}(D^{n+2-*}\otimes C(f)).\]If we like, we can describe this cycle in our more standard fashion by pushing this element along the homotopy cofibration sequence \[\dots\to W^\%C'\xrightarrow{(f')^\%} W^\%(D^{n+2-*})\to C((f')^\%)\xrightarrow{\text{proj}}\Sigma W^\%C'\xrightarrow{\Sigma (f')^\%}\Sigma W^\%(D^{n+2-*})\to\dots\]defining $\phi':=\Sigma^{-1}\text{proj}(\Phi,\Phi_0,e)$and $\delta\phi':=\Sigma(f')^\%(\Sigma\phi')$ to give a pair\[x':=(f':C'\to D^{n+2-*},(\delta\phi',\phi')).\]The pair $x'$ is called the \textit{surgery dual pair} to $x$ and the $n$-dimensional $\eps$-symmetric complex $(C',\phi')$ is called the \textit{effect of surgery on $(C,\phi)$ with data $x$}. The following is clear from the construction:

\begin{proposition}$x'$, the surgery dual pair to $x$, is well defined up to homotopy equivalence by the homotopy class of $x$. The surgery dual pair to $x'$ is homotopy equivalent to $x$.
\end{proposition}

This justifies the definition of a \textit{surgery dual set} $\{x,x'\}$ of $(n+2)$-dimensional $\eps$-symmetric pairs that are surgery dual to each other (up to homotopy).

\begin{corollary}The relative algebraic Thom construction defines a surgery dual set $\{x,x'\}$.
\end{corollary}

\begin{remark}The effect of surgery $(C',\phi')$ on an $n$-dimensional $\eps$-symmetric complex $(C,\phi)$ was introduced in \cite[\textsection 4]{MR560997} where a formula for a choice of $n$-dimensional $\eps$-symmetric structure $\phi'$ is specified by a matrix of morphisms involving the collection of higher chain maps $\phi_s$. While the homotopy theory we have presented gives a pleasing explanation for the concept of algebraic surgery, if the reader actually wants to calculate a surgery effect then \cite[\textsection 4]{MR560997} is more useful.
\end{remark}

\subsubsection*{Algebraic thickening}

We wish to reverse the relative algebraic Thom construction on a triad. That is, given an $(n+2)$-dimensional $\eps$-symmetric pair we will rebuild a triad. The strategy will be to build the appropriate homotopy commuting square using interlocking homotopy cofibration sequences. We then show there is an $(n+2)$-dimensional $\eps$-symmetric structure on the square by using diagram $\ref{eq:lift}$ to lift the structure from the algebraic thickening of the algebraic Thom construction on the original $(n+2)$-dimensional $\eps$-symmetric pair.

Fix an $(n+2)$-dimensional $\eps$-symmetric pair $x=(g:D\to E,(\delta\phi,\phi))$. Then $x$ and the surgery dual pair $x'=(g':D'\to E^{n+2-*},(\delta\phi',\phi'))$ fit into a homotopy commutative diagram\begin{equation}\label{eq:thickening}\xymatrix{\partial D\ar[d]_-{i}\ar[r]^-{i'}&(D')^{n+1-*}\ar[rr]^-{\ev(\phi')}\ar[d]^-{(e')^*}&&D'\ar[d]^-{g'}\\D^{n+1-*}\ar[d]_-{\ev(\phi)}\ar[r]^-{\text{proj}^*}&C(g)^{n+2-*}\ar[d]^-{\ev_l(\delta\phi,\phi)}\ar[rr]^-{e^*}&&E^{n+2-*}\ar[d]^-{\ev_r(\delta\phi,\phi)}\\ D\ar[r]^-{g}&E\ar[rr]^-{e}&&C(g)}\end{equation}where the bottom right square was observed to homotopy commute in diagram \ref{eq:ladder} and the rest is defined up to homotopy by desuspending iterated cones.

\begin{corollary}\label{cor:fudge}If $x'$ is the surgery dual pair to $x$ then $\partial D\simeq \partial D'$.
\end{corollary}

Write the pushout $D'':=D^{n+1-*}\cup_{\partial D}(D')^{n+1-*}$ and observe that there is a morphism\[g'':D''\to C(g)^{n+2-*}\]such that $C(g'')\simeq C(g)$. As we have \[C(g)\simeq C(g'')\simeq C(g')\]we can lift the same element $\Phi:=\delta\phi/\phi\in W^\%C(g)_{n+2}$ to the three distinguished choices:\[\begin{array}{rclcl}(\Phi,\Phi_0,\id)&=&(\overline{\Phi},\partial\Phi)&\in&C((g'')^\%)\\(\Phi,\Phi_0,\setminus(\ev_r(\delta\phi,\phi)))&=&(\delta\phi,\phi)&\in&C((g)^\%)\\(\Phi,\Phi_0,e)&=&(\delta\phi',\phi')&\in&C((g')^\%)\end{array}\]i.e. the same element recovers $x$, its surgery dual $x'$ and the algebraic Thom construction $(i'':\partial C(g)\to C(g)^{n+2-*},(\overline{\Phi}.\partial\Phi))$ on the $(n+2)$-dimensional $\eps$-symmetric complex $(C(g),\Phi)$.

\begin{claim}The cycle $(\overline{\Phi},\partial\Phi)\in C((g'')^\%)$ lifts to an $(n+2)$-dimensional $\eps$-symmetric structure on the homotopy commuting square $\Gamma$, where $\Gamma$ is given by the top left square of diagram \ref{eq:thickening}.
\end{claim}
\begin{proof}
Diagram \ref{eq:lift} now looks like\[\xymatrix{X\ar[r]^-{F}\ar[d]_-{G}&W^\%(C(g)^{n+2-*})\ar[d]^-{=}\ar[r]&C(F)\ar[d]\\W^\%(D'')\ar[r]^-{(g'')^\%}&W^\%(C(g)^{n+2-*})\ar[r]&C((g'')^\%)}\]with $D''$ and $X$ defined by the homotopy pullback squares\[\xymatrix{D''\ar[r]^-{k}\ar[d]_-{k'}&D\ar[d]^-{l}\\D'\ar[r]^{l'}&\Sigma\partial D}\qquad\qquad\xymatrix{X\ar[r]\ar[d]&W^\%D\ar[d]^-{l^\%}\\ W^\% D'\ar[r]^{(l')^\%}&\Sigma W^\%\partial D}\]

The obstruction to lifting $(\overline{\Phi},\partial\Phi)\in C((g'')^\%)$ to a cycle in $C(F)$ is given by the corresponding element in $C(\Sigma G)=\Sigma C(G)$. There is an obvious map of diagrams\[(W^\%\Sigma\partial D\to\Sigma\partial D\otimes\Sigma\partial D\xleftarrow{}0)\quad\longrightarrow\quad (W^\%\Sigma\partial D\to\Sigma\partial D\otimes\Sigma\partial D\xleftarrow{} (D\oplus D')\otimes \Sigma\partial D)\]determining a map on the homotopy pullbacks \[H:\Sigma W^\%\partial D\to C\left(\lmat -k^\%\\(k')^\%\rmat \right)\]with mapping cone $C(H)\simeq (D\oplus D')\otimes \Sigma\partial D$. Using this and \ref{prop:pullback}, \ref{cor:pullback}, there is a homotopy commutative diagram of homotopy cofibration sequences\[\xymatrix{W^\%\partial D\ar[r]\ar[d]_-{\Sigma^{-1}H}&X\ar[r]\ar[d]^-{G}&C(i^\%)\oplus C((i')^\%)\ar[r]\ar[d]&\Sigma W^\%\partial D\ar[d]^-{H}\\
\Sigma^{-1}C\left({\lmat -k^\%\\(k')^\%\rmat} \right)\ar[r]^-{(0\,\,1)}\ar[d]&W^\% D''\ar[r]^-{{\lmat-k^\%\\(k')^\%\rmat}}\ar[d]&W^\%D\oplus W^\% D'\ar[r]\ar[d]&C\left({\lmat -k^\%\\(k')^\%\rmat }\right)\ar[d]\\
\Sigma^{-1}((D\oplus D')\otimes \Sigma\partial D)\ar[r]&C(G)\ar[r]&(\Sigma\partial D\otimes D)\oplus(\Sigma\partial D\otimes D')\ar[r]&(D\oplus D')\otimes \Sigma\partial D}\]

Considering the lower row of diagram \ref{eq:lift} and the definition of $\partial \Phi$ as a formal desuspension, it is sufficient to check that there is no obstruction to lifting $\partial\Phi\in C(G)$ to $X$.

Firstly, we have already seen in the algebraic Thom construction that $k^\%\partial \Phi=\phi$ and $(k')^\%\partial \Phi=\phi'$ lift to $(\overline{\phi},\partial\phi)$ and $(\overline{\phi'},\partial\phi')$ respectively. So it is enough to show that there is a homotopy equivalence $(\partial D,\partial \phi)\simeq (\partial D',\partial \phi')$. But indeed the homotopy equivalence of \ref{cor:fudge} induces a homotopy equivalence of $\eps$-symmetric Poincar\'{e} complexes as required.

\end{proof}

\begin{corollary}\label{corr:cobordism}If $x$ and $x'$ as in diagram \ref{eq:thickening} and $\partial D\simeq 0$, then $D''\simeq D\oplus D'$ and there is an $(n+2)$-dimensional $\eps$-symmetric Poincar\'{e} pair given by\[(g''=(f\,\,f'):D\oplus D'\to C(g)^{n+2-*},(\overline{\Phi},\phi\oplus-\phi')),\]with $f$ and $f'$ maps induced by diagram \ref{eq:thickening}.
\end{corollary}

\begin{remark}More concretely than Corollary \ref{corr:cobordism}, but under the same hypotheses, we can write an $(n+2)$-dimensional $\eps$-symmetric Poincar\'{e} pair as in \cite[4.1(ii)]{MR560997}\[((h\,\,h'):D\oplus D'\to \tilde{E},(0,\phi\oplus-\phi'))\]where\[d_{\tilde{E}}=\left(\begin{matrix}d_D&(-1)^{n+2}\phi_0 f^*\\0&(-1)^rd^*_E\end{matrix}\right):\tilde{E}_r=D_r\oplus E^{n+2-r}\to\tilde{E}_{r-1}=D_{r-1}\oplus E^{n+2-(r-1)},\]\[h=\left(\begin{matrix}1\\0\end{matrix}\right):D_r\to D_r\oplus E^{n+2-r},\]\[h'=\left(\begin{matrix}1&0&0\\0&0&1\end{matrix}\right):D'_r=D_r\oplus E_{r+1}\oplus E^{n+2-r}\to D_r\oplus E^{n+2-r}.\]Actually $\tilde{E}\simeq C(g)^{n+2-*}$ but $\tilde{E}$ is carefully chosen so that $h, h'$ are split, which has the obvious advantage that the $\eps$-symmetric structure $(0,\phi\oplus-\phi')$ will be easier to calculate and manipulate.
\end{remark}

\begin{remark}Just as a manifold triad can be thought of as a nullcobordism of a manifold with boundary, so a symmetric Poincar\'{e} triad should be thought of as a nullcobordism of a symmetric Poincar\'{e} pair. There is a corresponding higher notion; a manifold $n$-ad can be thought of as a nullcobordism of a manifold $(n-1)$-ad and correspondingly one can make a definition of a symmetric Poincar\'{e} $n$-ad which is a nullcobordism of a symmetric Poincar\'{e} $(n-1)$-ad. Presumably this dual perspective afforded by the operations of Thom construction vs.\ thickening has an analogue here. Although I can't see what use it would be right now.
\end{remark}

\section{Connective $L$-groups, $\Gamma$-groups and torsion $L$-groups}\label{sec:connective}

A chain complex $C$ in $\B(A)$ is \textit{positive} if $H_r(C)=0$ for $r<0$. The subcategory of $\B(A)$ consisting of finite, positive chain complexes is denoted $\B_+(A)$. In Chapter \ref{chap:DLtheory} we will need to work with positive chain complexes in order to relate $DL$-groups to double Witt groups without the need for algebraic surgery. These are the relevant definitions and modifications required before we pass to the $L$-theory of $\B_+(A)$.

\begin{definition}An $n$-dimensional $\eps$-symmetric complex $(C,\phi)$ is \textit{positive} if $C$ is in $\B_+(A)$. An $(n+1)$-dimensional $\eps$-symmetric pair $(f:C\to D,(\delta\phi,\phi))$ is \textit{positive} if $f:C\to D$ is in $\B_+(A)$. An $(n+2)$-dimensional $\eps$-symmetric triad $(\Gamma, (\Phi,\delta\phi,\delta'\phi,\phi))$ is \textit{positive} if $\Gamma$ is a homotopy commuting square in $\B_+(A)$.
\end{definition}

If an $(n+1)$-dimensional $\eps$-symmetric pair $(f:C\to D,(\delta\phi,\phi))$ is positive then the $(n+1)$-dimensional $\eps$-symmetric complex defined by the algebraic Thom construction is positive. The converse is not true in general. 

An $n$-dimensional $\eps$-symmetric complex $(C,\phi)$ is called \textit{connected} if the boundary $(\partial C,\partial \phi)$ is positive. If $(C,\phi)$ is connected and positive then the algebraic thickening of $(C,\phi)$ is positive.

The version of the Proposition \ref{prop:thomthick1} using positivity conditions is:

\begin{proposition}[{\cite[Proposition 3.4]{MR560997}}]The algebraic Thom construction and algebraic thickening describe a natural 1:1 correspondence between homotopy equivalence classes of positive, connected $(n+1)$-dimensional $\eps$-symmetric complexes and homotopy equivalence classes of positive $(n+1)$-dimensional $\eps$-symmetric Poincar\'{e} pairs.
\end{proposition}

If an $(n+2)$-dimensional $\eps$-symmetric triad is positive, then the $(n+2)$-dimensional $\eps$-symmetric pair defined by the relative algebraic Thom construction is positive. The converse is not true in general. It is also not generally true that if an $(n+1)$-dimensional pair $x$ is positive then its surgery dual $x'$ is positive.

An $(n+1)$-dimensional $\eps$-symmetric pair $x=(f:C\to D,(\delta\phi,\phi))$ is \textit{connected} if the effect of surgery on $C$ with data $x$ is positive. If $x$ is positive and connected then consequently $x'$ is positive and connected. Unravelling the definitions, this means in particular that if $x$ is positive and connected, then the $(n+2)$-dimensional $\eps$-symmetric Poincar\'{e} triad given by the algebraic thickening of $x$ is positive.

The version of Proposition \ref{prop:triad} using positivity conditions is

\begin{proposition}\label{prop:triad2}The relative algebraic Thom construction and the algebraic thickening describe a natural 1:1 correspondence between surgery dual sets $\{x,x'\}$ (defined up to homotopy) of positive, connected $(n+2)$-dimensional $\eps$-symmetric pairs and homotopy equivalence classes of positive $(n+2)$-dimensional $\eps$-symmetric Poincar\'{e} triads.
\end{proposition}

\subsection{Projective $L$-groups}

\begin{definition}A \textit{cobordism} between $n$-dimensional $\eps$-symmetric complexes $(C,\phi)$ and $(C',\phi')$ is an $(n+1)$-dimensional $\eps$-symmetric Poincar\'{e} pair\[((f\,\,f'):C\oplus C'\to D,(\delta\phi,\phi\oplus-\phi')).\]If such a pair exists we say $(C,\phi)$, $(C',\phi')$ are \textit{cobordant}.
\end{definition}

\begin{remark}Without assuming our chain complexes are at least homotopy equivalent to complexes in $\B(A)$ we run up against the Eilenberg swindle pathology. In $L$-theory this takes the form of \textit{every} $\eps$-symmetric Poincar\'{e} complex $(C,\phi)$ being nullcobordant! Specifically, take an algebraic glue of countably infinitely many copies of the cobordism $((1\,\,1):C\oplus C\to C,(0,\phi\oplus-\phi))$. We will keep track of the glue by labelling ends as $t^jC$ and glueing $t^kC$ to $t^lC$ if $k=l$ so that \[\bigcup_{j=0}^\infty(t^{j}C\oplus t^{j+1}C\to t^jC,(0,\phi\oplus -\phi))\simeq (C\to C[t],(0,\phi)),\] by the Eilenberg swindle.
\end{remark}

\medskip

\textbf{\textit{From this point onwards the use of the terms $\eps$-symmetric complex, $\eps$-symmetric pair and $\eps$-symmetric triad includes the assumption of all chain complexes being in $\B_+(A)$.}}

\medskip

\begin{theorem}[\cite{MR560997}]\label{thm:Lgroup}For $n\geq0$, cobordism is an equivalence relation on the set of homotopy equivalence classes of $n$-dimensional $\eps$-symmetric Poincar\'{e} complexes over $A$. Two homotopy equivalence classes are cobordant if and only if there exist representatives $(C,\phi)$ and $(C',\phi')$ respectively such that $(C',\phi)$ is the effect of surgery on $(C,\phi)$ with respect to some (connected) pair. The cobordism equivalence classes form a group $L^n(A,\eps)$, the \textit{$n$-dimensional $\eps$-symmetric $L$-group of $A$} with addition and inverses given by \[(C,\phi)+(C',\phi')=(C\oplus C',\phi\oplus\phi'),\qquad -(C,\phi)=(C,-\phi)\in L^{n}(A,\eps).\]

\end{theorem}

\subsection{$\Gamma$-groups}

\begin{definition}For $n\geq0$ let $(P,\theta)$ be an $n$-dimensional $\eps$-symmetric complex over $A$. If $S^{-1}A\otimes_A(P,\theta)$ is an Poincar\'{e} complex over $S^{-1}A$ then we say $(P,\theta)$ is \emph{$S$-Poincar\'{e} over $A$}. An \emph{$S^{-1}A$-cobordism} between two such complexes, $(P,\theta)$ and $(P',\theta')$, is an $(n+1)$-dimensional $\eps$-symmetric pair over $A$\[(f:P\oplus P'\to Q,(\delta\theta,\theta\oplus-\theta)),\]such that \[S^{-1}A\otimes_A(f:P\oplus P'\to Q,(\delta\theta,\theta\oplus-\theta))\]is an Poincar\'{e} pair over $S^{-1}A$. If such a pair exists we say that $(P,\theta)$ and $(P',\theta')$ are \textit{$S^{-1}A$-cobordant}.
\end{definition}

\begin{theorem}[\cite{MR620795}]For $n\geq0$, $S^{-1}A$-cobordism is an equivalence relation on the set of homotopy equivalence classes of $n$-dimensional $\eps$-symmetric $S$-Poincar\'{e} complexes over $A$. The cobordism equivalence classes form a group $\Gamma^n(A\to S^{-1}A,\eps)$, the \textit{$n$-dimensional $\eps$-symmetric $\Gamma$-group of $(A,S)$} with addition and inverses given by \[(P,\theta)+(P',\theta')=(P\oplus P',\theta\oplus\theta'),\qquad -(P,\theta)=(P,-\theta)\in \Gamma^{n}(A\to S^{-1}A,\eps).\]

\end{theorem}

\begin{theorem}[{\cite[Proof of Prop.\ $3.2.3$(i)]{MR620795}}]
For $n\geq 0$ there is an isomorphism \[\Gamma^n(A\to S^{-1}A,\eps)\cong L^n(S^{-1}A,\eps)\] given by sending an $n$-dimensional $\eps$-symmetric $S$-Poincar\'{e} complex $(C,\phi)$ over $A$ to the $n$-dimensional $\eps$-symmetric Poincar\'{e} complex $S^{-1}A\otimes_A(C,\phi)$ over $S^{-1}A$.
\end{theorem}

\subsection{Torsion $L$-groups}\label{subsec:Ltorsion}

As we are concerned in Chapter \ref{chap:DLtheory} with developing the $L$-theoretic analogue of the double Witt group of linking forms for a localisation $(A,S)$ we need to understand how torsion is handled in $L$-theory. A chain complex $C$ in $\B_+(A)$ is \textit{$S$-acyclic} with respect to a localisation $(A,S)$ if $S^{-1}A\otimes_AH^r(C)=0$ for all $r$. The subcategory of complexes $C$ in $\B_+(A)$ that are $S$-acyclic is denoted $\C_+(A,S)$. We say an $\eps$-symmetric complex or pair is \textit{in $\C_+(A,S)$} if the underlying complex or morphism (resp.) is in $\C_+(A,S)$.

\begin{remark}The triple\[\Gamma_+(A,S):=(\A(A),\B_+(A),\C_+(A,S))\]defines an (non-stable) algebraic bordism category, in the sense of \cite[\textsection 3]{MR1211640}.
\end{remark}

\begin{definition}An \textit{$(A,S)$-cobordism} between $n$-dimensional $\eps$-symmetric complexes $(C,\phi)$ and $(C',\phi')$ in $\C_+(A,S)$ is an $(n+1)$-dimensional $\eps$-symmetric Poincar\'{e} pair\[((f\,\,f'):C\oplus C'\to D,(\delta\phi,\phi\oplus-\phi'))\]in $\C_+(A,S)$. If such a pair exists we say $(C,\phi)$, $(C',\phi')$ are \textit{$(A,S)$-cobordant}.
\end{definition}

\begin{theorem}[\cite{MR620795}]For $n\geq0$, $(A,S)$-cobordism is an equivalence relation on the set of homotopy equivalence classes of $(n+1)$-dimensional $(-\eps)$-symmetric, Poincar\'{e} complexes in $\C_+(A,S)$. The equivalence classes form a group $L^n(A,S,\eps)$, the \emph{$n$-dimensional $\eps$-symmetric $L$-group of $(A,S)$}, with addition and inverses given by\[(C,\phi)+(C',\phi')=(C\oplus C',\phi\oplus\phi'),\qquad -(C,\phi)=(C,-\phi)\in L^n(A,S,\eps).\]
\end{theorem}

It is described in \cite[3.1]{MR620795} that, in a precise way we shall not develop here, an $(n+1)$-dimensional $\eps$-symmetric Poincar\'{e} complex in $C_+(A,S)$ can be thought of as a `resolution' of $(-\eps)$-symmetric Poincar\'{e} complex in the category $\H(A,S)$. This is an explanation of the dimension index convention of the group $L^n(A,S,\eps)$ and the sign convention for the $\eps$.

\subsection{$\widehat{L}$-groups of Seifert complexes}\label{subsec:Lseifert}

We now introduce a slightly different flavour of $L$-theory that will be significant for our later knot-theoretic applications. It is the category of chain complex with symmetric structure that corresponds to the Seifert forms of Chapter \ref{chap:laurent}. This description will be fairly reliant on the references from the original source \cite[pp814]{MR620795}.

\begin{definition}For $\eps^{-1}=\overline{\eps}$, an \textit{$n$-dimensional $\eps$-ultraquadratic complex} is a pair $(C,\hat{\psi})$, where $C$ is in $\B(A)$ and $\hat{\psi}\in (\Hom_A(C^{-*},C))_n$ is an $n$-cycle. An \textit{$(n+1)$-dimensional $\eps$-ultraquadratic pair} is $(f:C\to D,(\delta\hat{\psi},\hat{\psi}))$ where $f$ is a morphism in $\B(A)$ and $(\delta\hat{\psi},\hat{\psi})\in C(f\otimes f)_{n+1}$ is an $(n+1)$-cycle.
\end{definition}

By Proposition \ref{prop:ultra} there is the following morphism and relative morphism\[1+T_\eps:\Hom_A(C^{-*},C)\to W^\%C,\qquad 1+T_\eps:C(f\otimes f)_n\to C(f^\%).\]

\begin{definition}An $n$-dimensional $\eps$-ultraquadratic complex $(C,\hat{\psi})$ is \textit{Poincar\'{e}} if the associated $n$-dimensional $\eps$-symmetric complex $(C,(1+T_\eps)\hat{\psi})$ is Poincar\'{e}. Similarly for pairs. If $(C,\phi)$ is Poincar\'{e} then we call it an \textit{$n$-dimensional $\eps$-symmetric Seifert complex}.
\end{definition}

As we have a notion of pairs and Poincar\'{e} pairs in the $\eps$-ultraquadratic category we have an associated notion of cobordism.

\begin{definition}A \textit{Seifert cobordism} between $n$-dimensional $\eps$-ultraquadratic complexes $(C,\hat{\psi})$ and $(C',\hat{\psi}')$ is an $(n+1)$-dimensional $\eps$-ultraquadratic Poincar\'{e} pair\[((f\,\,f'):C\oplus C'\to D,(\delta\hat{\psi},\hat{\psi}\oplus-\hat{\psi})).\]If such a pair exists we say $(C,\hat{\psi})$ and $(C',\hat{\psi}')$ are \textit{Seifert cobordant}.
\end{definition}

\begin{theorem}[{\cite[p.814]{MR620795}}]For $n\geq0$, Seifert cobordism is an equivalence relation on the set of homotopy equivalence classes of $n$-dimensional $\eps$-symmetric Seifert complexes over $A$. Two homotopy equivalence classes are Seifert cobordant if and only if there exist representatives $(C,\hat{\psi})$ and $(C',\hat{\psi}')$ respectively such that $(C',\hat{\psi})$ is the effect of surgery on $(C,\hat{\psi})$ with respect to some (connected) $\eps$-ultraquadratic pair. The Seifert cobordism equivalence classes form a group $\widehat{L}^n(A,\eps)$, the \textit{$n$-dimensional $\eps$-symmetric $L$-group of Seifert forms over $A$} with addition and inverses given by \[(C,\hat{\psi})+(C',\hat{\psi}')=(C\oplus C',\hat{\psi}\oplus\hat{\psi}'),\qquad -(C,\hat{\psi})=(C,-\hat{\psi})\in \widehat{L}^{n}(A,\eps).\]There are obvious forgetful maps\[1+T_\eps:\widehat{L}^n(A,\eps)\to L^n(A,\eps).\]

\end{theorem}

\subsection{$L$-theory localisation exact sequence}\label{subsec:localisation}

We now state and prove a version of the Ranicki-Vogel $L$-theory localisation exact sequence. By assuming there is a half-unit in the ring we are able to make certain simplifications that make the proof easier. However, \ref{thm:LES} is true without this assumption. We will need the following theorem now and for later applications.

\begin{theorem}\label{thm:algsurgery}When there is a half-unit in the ring $A$, skew suspension defines isomorphisms for each $n\geq0$\[\begin{array}{rrcl}\overline{S}:&L^n(A,\eps)&\xrightarrow{\cong}& L^{n+2}(A,-\eps),\\
\overline{S}:&\Gamma^n(A\to S^{-1}A,\eps)&\xrightarrow{\cong}& \Gamma^{n+2}(A\to S^{-1}A,-\eps),\\
\overline{S}:&L^n(A,S,\eps)&\xrightarrow{\cong}& L^{n+2}(A,S,-\eps).\end{array}\]The inverse of the skew-suspension is given by \textit{algebraic surgery}, a process which takes an $\eps$-symmetric complex in some category and returns a cobordant $\eps$-symmetric complex in the same category, that is in the image of the skew-suspension morphism. (See \cite[\textsection 4]{MR560997} or \cite[\textsection 1.5]{MR620795} for a description of this.)
\end{theorem}

\begin{proof}See \cite[4.4]{MR560997} for the general case of skew-suspension and algebraic surgery in the projective symmetric $L$-groups when there is a half-unit in the ring. The other two isomorphism are proved analogously in the respective categories of chain complexes with structure.
\end{proof}

As $\C_+(A,S)$ is a subcategory of $\B_+(A)$ there is a forgetful homomorphism $L^{n}(A,S,\eps)\to L^{n+1}(A,-\eps)$. Now algebraic surgery defines an isomorphism $L^{n+1}(A,-\eps)\cong L^{n-1}(A,\eps)$ (with inverse given by skew-suspension). The forgetful map together with algebraic surgery result in a well-defined homomorphism\[i:L^n(A,S,\eps)\to L^{n-1}(A,\eps);\qquad(C,\phi)\mapsto (C',\phi'),\]such that $\overline{S}(C',\phi')$ is cobordant to $(C,\phi)$ in $\B_+(A)$.

If an $\eps$-symmetric complex or pair in $\B_+(A)$ is Poincar\'{e} then it is certainly $S$-Poincar\'{e}. Hence there is another well defined homomorphism:\[j:L^n(A,\eps)\to \Gamma^n(A\to S^{-1}A,\eps).\]

%
%
%
%
%

If $(P,\theta)$ is an $n$-dimensional $\eps$-symmetric $S$-Poincar\'{e} complex in $\B_+(A)$ then the boundary of the $n$-dimensional $\eps$-symmetric complex $S^{-1}A\otimes_A (P,\theta)$ is \textit{not} in general in $\B_+(A)$ as $(P,\theta)$ is not in general connected. However, the skew-suspension $\overline{S}\partial(P,\theta)$ of this boundary is easily seen to be in $\B_+(A)$. Moreover, $\partial(P,\theta)$, and hence $\overline{S}\partial(P,\theta)$, is $S$-acyclic as $(P,\theta)$ is $S$-Poincar\'{e}. Hence $\overline{S}(\partial P,\partial\theta)$ is an $(n+1)$-dimensional $\eps$-symmetric Poincar\'{e} complex in $\C_+(A,S)$. A relative version of this argument means there is moreover a well-defined homomorphism\[\partial:\Gamma^n(A\to S^{-1}A,\eps)\to L^{n}(A,S,\eps);\qquad (P,\theta)\mapsto \overline{S}\partial(P,\theta).\]

\begin{theorem}[{\cite{MR620795}}]\label{thm:LES}For $n\geq0$, the sequence of group homomorphisms\[...\to L^n(A,\eps)\xrightarrow{i} \Gamma^n(A\to S^{-1}A,\eps)\xrightarrow{\partial} L^n(A,S,\eps)\xrightarrow{j} L^{n-1}(A,\eps)\]is exact. When $n=0$, we interpret the term $L^{-1}(A,\eps)$ via the skew-suspension isomorphism as $L^1(A,-\eps)$.\end{theorem}

In fact, in the version we have stated, Theorem \ref{thm:LES} is not difficult to prove. A solid understanding of Theorem \ref{thm:LES} and the proof we offer will be helpful in understanding the technical work needed to extend this idea in Chapter \ref{chap:DLtheory}.

\begin{proof}In the proof we will assume there is a half unit in the ring $A$ and hence the skew-suspensions are isomorphisms by Theorem \ref{thm:algsurgery}, with inverses given by algebraic surgery. This assumption is not necessary but does simplify the proof slightly. For the full proof we refer the reader to \cite[3.2.3]{MR620795}.

We must show exactness in three places:

\medskip

\noindent\underline{$\Gamma^{n}(A\to S^{-1}A,\eps)$:} Clearly $\partial\circ i=0$ as $(\partial P,\partial\theta)=0$ for $(P,\theta)$ Poincar\'{e} over $A$.

Now suppose $[(P,\theta)]\in\ker(\partial)$ so that $\overline{S}(\partial P,\partial\theta)$ admits an $(A,S)$-nullcobordism which, we may assume via algebraic surgery if necessary, is a skew-suspension $\overline{S}(f:\partial P\to D,(\nu,\partial \theta)\in C(f^\%)_{n})$. The algebraic union $(D\cup_{\partial P} P^{n-*},\nu\cup_{\partial\theta}\overline{\theta})$ is Poincar\'{e} over $A$ and $S^{-1}A\otimes (D\cup_{\partial P} P^{n-*},\nu\cup_{\partial\theta}\overline{\theta})$ is chain equivalent over $S^{-1}A$ to $S^{-1}A\otimes(P,\theta)$.

\medskip

\noindent\underline{$L^{n}(A,S,\eps)$:} $j\circ \partial([P,\theta])=[(\partial P,\partial \theta)]\in L^{n-1}(A,\eps)$ and the algebraic thickening of $(P,\theta)$ defines a nullcobordism.

Now suppose that $j(C,\phi)$ is nullcobordant over $A$. The skew-suspension of this nullcobordism defines a nullcobordism $(f:C\to D,(\delta\phi,\phi))$ over $A$. By surgery on the pair in $\C_+(A,S)$, we may assume it is a skew-suspension $(f:C\to D,(\delta\phi,\phi))=\overline{S}(f':C'\to D',(\delta\phi',\phi'))$. Take the algebraic Thom complex $(D'/C',\delta\phi'/\phi')$. The boundary $(\partial (D'/C'),\partial(\delta\phi'/\phi'))\simeq(C',\phi')$ is $S^{-1}A$-acyclic and hence the Thom complex is $S$-Poincar\'{e}. So that $(D'/C',\delta\phi'/\phi')\in \Gamma^{n}(A\to S^{-1}A,\eps)$ and $\partial[(D'/C',\delta\phi'/\phi')]=[(C,\phi)]\in L^{n}(A,S,\eps)$.

\medskip

\noindent\underline{$L^n(A,\eps)$:} $i\circ j([C,\phi])=[(C',\phi')]\sim0\in \Gamma^{n}(A\to S^{-1}A,\eps)$ by observing the $S^{-1}A$-nullcobordism $(C'\to 0,(0,\phi'))$.

Now suppose $(P,\theta)\sim 0\in \Gamma^n(A\to S^{-1}A,\eps)$ and $(P,\theta)$ is Poincar\'{e} over $A$. Hence there exists an $S^{-1}A$-nullcobordism $(f:P\to Q,(\delta\theta,\theta))$. Write the effect of surgery on $(P,\theta)$ with data $(f:P\to Q,(\delta\theta,\theta))$ as $(P',\theta')$. As $\partial P\simeq0$, we have $P' \simeq C(\partial P\to \partial(Q,P))\simeq \partial(Q,P)$, so that $(P',\theta')$ is $S$-acyclic. Therefore the cobordism class of $(P,\theta)$ is in the image of $j$.

\end{proof}

\chapter{Double $L$-theory}\label{chap:DLtheory}

We now use the Algebraic Theory of Surgery to generalise the double Witt groups introduced in Chapter \ref{chap:linking} to the setting of chain complexes. Symmetric linking forms will be replaced by $S$-acyclic chain complexes with symmetric structure. Double Witt equivalence of forms will be replaced by a `double-cobordism' relation between $S$-acyclic chain complexes with structure. In $L$-theory, the cobordism groups of $S$-acyclic chain complexes with symmetric structure are isomorphic to their corresponding Witt groups of linking forms via algebraic surgery below the middle dimension. By assuming our localisation $(A,S)$ has dimension 0 we are able to show that there is `algebraic surgery above and below the middle dimension'. However, we also show that this is not the case for general rings and that generally the double $L$-groups capture more algebraic information than their double Witt group counterparts.

The techniques and concepts of algebraic $L$-theory necessary to understand Chapter \ref{chap:DLtheory} were developed in Chapter \ref{chap:algLtheory} in as self-contained an account as we were able to produce. The reader new to $L$-theory should hopefully find all necessary information there. For the more experienced reader we note also that Chapter \ref{chap:algLtheory} makes clear the chain complex conventions with which we work. Our assumptions on connectivity of the underlying complexes and the algebraic bordism category are also discussed.

\medskip

For this chapter, as usual, set $A$ and $R$ to be commutative Noetherian rings with involution and unit. Set $(A,S)$ be a localisation of $A$. Furthermore, assume throughout this chapter that $A$ contains a half unit $s$. This assumption is included as it is necessary for our definitions of the double $L$-groups. It has the additional consequence of dramatically simplifying the $L$-theory of Chapter \ref{chap:algLtheory} in the manner already described in \ref{subsec:chainmaps}.

\section{Double $L$-theory}

Recall the category of projective, finitely generated $A$-modules with $A$-module morphisms is written $\A(A)$. A chain complex $C$ over $\A(A)$ concentrated in finitely many dimensions is called \textit{finite}. If $H_r(C)=0$ for $r<0$, $C$ is called \textit{positive}. The category of finite, positive chain complexes $C$ over $\A(A)$ with morphisms (degree 0) chain maps is written $\B_+(A)$. The subcategory of $\B_+(A)$ consisting of chain complexes $C$ in $\B_+(A)$ such that $S^{-1}A\otimes_A H_r(C)=0$ for all $r$ is written $\C_+(A,S)$ (and n.b.\ $S^{-1}A\otimes_AH_r(C)\cong H_r(S^{-1}A\otimes_AC)$). The triple\[\Lambda_+(A,S):=(\A(A),\B_+(A),\C_+(A,S))\]  defines an algebraic bordism category (see Chapter \ref{sec:connective}) over which we will work. Much is known about the $L$-theory of this algebraic bordism category and in particular the relative $L$-groups of this algebraic bordism category are isomorphic to $L^n(A,S,\eps)$, the torsion $L$-groups of the localisation $(A,S)$, which are 4-periodic in $n$ with $L^{2k}(A,S,(-1)^k\eps)\cong W^{\eps}(A,S)$.

We will develop a deeper understanding of $S$-acyclic $\eps$-symmetric chain complexes than was previously available using the groups $L^n(A,S,\eps)$. The $(A,S)$-cobordism relationship from $L$-theory will be replaced with a new relationship called `$(A,S)$-double-cobordism'. This relation will be two cobordisms fitting together in a `complementary' way (of course if it were just two cobordisms with no complementary condition, we could use the same one twice and would not get a different group). This will form the $DL$-group $DL^n(A,S,\eps)$.

\begin{figure}[h]\[\def\picsingle{\resizebox{0.4\textwidth}{!}{ \includegraphics{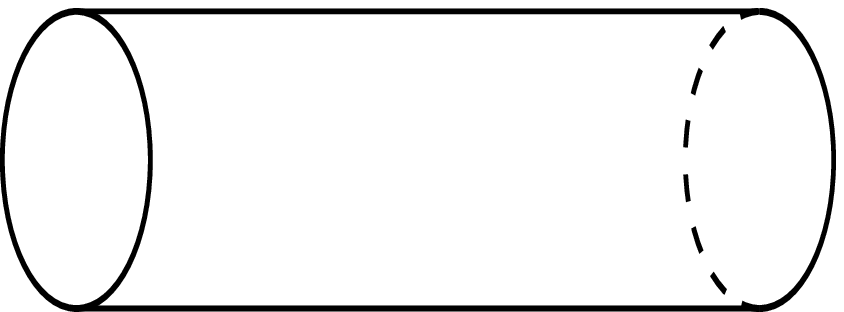}}}
\begin{xy} \xyimport(300,300){\picsingle}
,!+<7.2pc,2pc>*+!\txt{A cobordism from $C$ to $C'$.}
,(21,148)*!L{C}
,(145,148)*!L{D}
,(266,148)*!L{C'}
\end{xy}
\qquad\qquad
\def\picdouble{\resizebox{0.4\textwidth}{!}{ \includegraphics{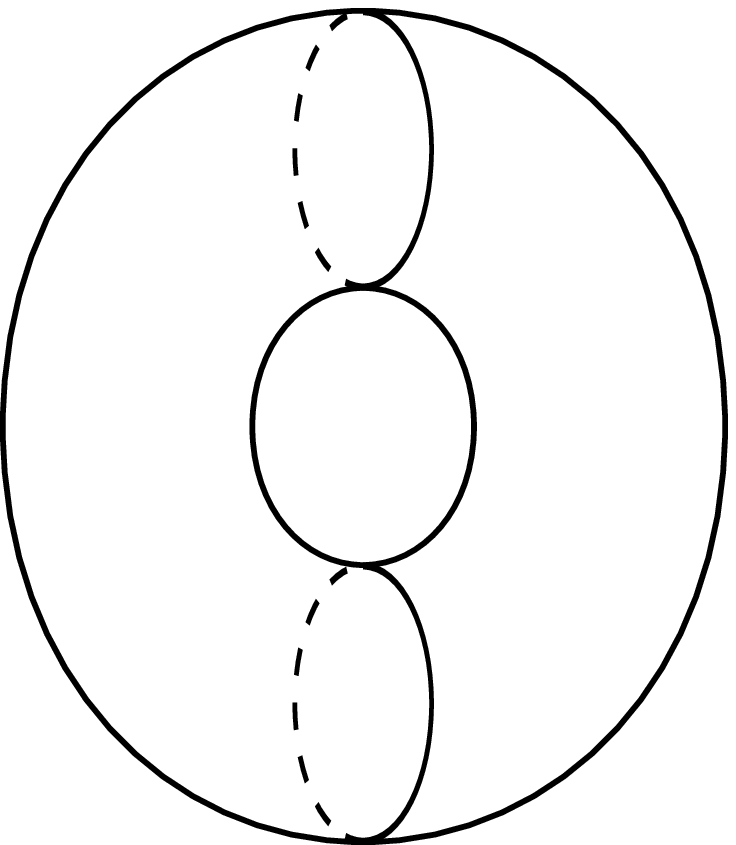}}}
\begin{xy} \xyimport(300,300){\picdouble}
,!CD+<0.3pc,-1pc>*+!CU\txt{A double-cobordism from $C$ to $C'$.}
,(145,250)*!L{C}
,(55,148)*!L{D_+}
,(225,148)*!L{D_-}
,(145,50)*!L{C'}\end{xy}\]
\caption{Schematic of cobordism vs.\ double-cobordism}
\label{pic:singvsdoub}
\end{figure}

To analyse the difference between the $DL$-groups of the localisation $(A,S)$ and the projective $L$-groups of $A$ we will define the `double $\Gamma$-group' $D\Gamma^n(A\to S^{-1}A,\eps)$ of $(A,S)$. The group $D\Gamma^n(A\to S^{-1}A,\eps)$ turns out to depend on the localisation $(A,S)$ in a stronger way than the corresponding group $\Gamma^n(A\to S^{-1}A,\eps)\cong L^n(S^{-1}A,\eps)$ in $L$-theory.

\subsection{The projective $DL$-groups}

\begin{definition}For $n\geq 0$, two cobordisms between $n$-dimensional $\eps$-symmetric Poincar\'{e} complexes $(C,\phi)$ and $(C',\phi')$ \[x_\pm:=(f_\pm:C\oplus C'\to D_\pm,(\delta_\pm\phi,\phi\oplus-\phi')\in C(f^\%_\pm)_{n+1})\](labelled ``$+$'' and ``$-$'') are \emph{complementary} if the chain map\[\left(\begin{smallmatrix}f_+\\f_-\end{smallmatrix}\right):C\oplus C'\to D_+\oplus D_-\]is a homotopy equivalence. In which case we say $(C,\phi)$ and $(C',\phi')$ are \emph{double-cobordant} and that the set $\{x_+,x_-\}$ is a \emph{double-cobordism} between them.
\end{definition}

The following is easily checked:

\begin{lemma}If $s+\overline{s}=1$ for $s\in A$ and $((f\,\,f'):C\oplus C'\to D,(\delta\phi,\phi\oplus -\phi'))$ is an $(n+1)$-dimensional $\eps$-symmetric Poincar\'{e} pair then so is $((\overline{s}f\,\,-sf'):C\oplus C'\to D,((s\overline{s})\delta\phi,\phi\oplus -\phi'))$.
\end{lemma}

\begin{lemma}\label{welldef}For $n\geq0$ let $(C,\phi)\simeq(C',\phi')$ be homotopy equivalent $n$-dimensional, $\eps$-symmetric Poincar\'{e} complexes over $A$, then $(C,\phi)$ and $(C',\phi')$ are double-cobordant.
\end{lemma}

\begin{proof}Here is where it is necessary to assume there exists $s\in A$ such that $s+\bar{s}=1$. 

Let $h:(C,\phi)\to(C',\phi')$ be given homotopy equivalence with homotopy inverse $g$. Then there are defined two cobordisms:\[\begin{array}{lrl}
((h\,\,&1):&C\oplus C'\to C',(0,\phi\oplus-\phi')),\\
((h\bar{s}\,&-s):&C\oplus C'\to C',(0,\phi\oplus-\phi')),
\end{array}\] and they are complementary as \[\begin{array}{rcl}\left(\begin{array}{cc}gs&g\\\bar{s}&-1\end{array}\right)\left(\begin{array}{cc}h&1\\h\bar{s}&-s\end{array}\right)\simeq\left(\begin{array}{cc}1&0\\0&1\end{array}\right)&:&C\oplus C'\to C\oplus C',\\\left(\begin{array}{cc}h&1\\h\bar{s}&-s\end{array}\right)\left(\begin{array}{cc}gs&g\\\bar{s}&-1\end{array}\right)\simeq \left(\begin{array}{cc}1&0\\0&1\end{array}\right)&:&C'\oplus C'\to C'\oplus C'.\end{array}\]
\end{proof}

\begin{proposition}\label{DLA}For $n\geq0$, double-cobordism is an equivalence relation on the set of homotopy equivalence classes of $n$-dimensional, $\eps$-symmetric, Poincar\'{e} complexes over $A$. The equivalence classes form a group $DL^{n}(A,\eps)$, the \emph{$n$-dimensional, $\eps$-symmetric $DL$-group of $A$}, with addition and inverses given by\[(C,\phi)+(C',\phi')=(C\oplus C',\phi\oplus\phi'),\qquad -(C,\phi)=(C,-\phi)\in DL^n(A,\eps).\]
\end{proposition}

\begin{proof}The previous lemma shows in particular that double-cobordism is well-defined and reflexive. It is clearly symmetric. To show transitivity consider two double-cobordisms\[\begin{array}{lcl}
	c_\pm&=&((f_\pm\,\,f'_\pm):C\oplus C'\to D_\pm,(\delta_\pm\phi,\phi\oplus-\phi')),\\
	c'_\pm&=&((\tilde{f}'_\pm\,\,f''_\pm):C'\oplus C''\to D_\pm,(\delta_\pm\phi',\phi'\oplus-\phi'')).\\
\end{array}\]We intend to re-glue the 4 cobordisms according to the schematic in Figure \ref{transitivity}.

\begin{figure}[h]
\def\pictransitivity{\resizebox{0.8\textwidth}{!}{ \includegraphics{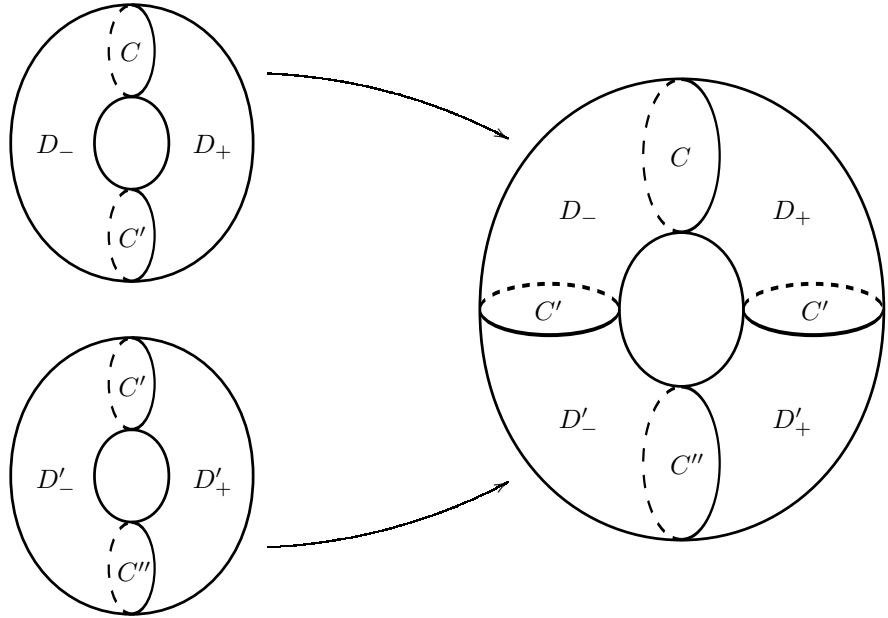}}}

\[\begin{xy} \xyimport(619.29,440){\pictransitivity}
,(82,405)*!L{C}
,(81,273)*!L{C'}
,(81,166)*!L{C'}
,(80,33)*!L{C''}
,(23,337)*!L{D_-}
,(133,337)*!L{D_+}
,(23,97)*!L{D_-'}
,(133,97)*!L{D_+'}
,(468,330)*!L{C}
,(468,110)*!L{C''}
,(373,220)*!L{C'}
,(560,220)*!L{C'}
,(390,290)*!L{D_-}
,(540,290)*!L{D_+}
,(390,140)*!L{D_-'}
,(540,140)*!L{D_+'}
,(180,390)*+{}="A";(360,340)*+{}="B"
,{"A"\ar@/^/"B"}
,(180,50)*+{}="C";(360,100)*+{}="D"
,{"C"\ar@/_/"D"}
\end{xy}\]
                \caption{Combining the double-nullcobordisms to show transitivity}
	       \label{transitivity}
\end{figure}

Recall the algebraic glueing operation of \ref{def:glue2}. As $c_+$ and $c'_+$ share a boundary component, likewise $c_-$ and $c'_-$, we form the two algebraic glues \[c_\pm\cup c'_\pm=(C\oplus C''\to D_\pm'',(\delta_\pm\phi'',\phi\oplus-\phi''))\]where $D''_\pm=D_\pm\cup_{C'}D'_\pm$ is the mapping cone\[C\left({\left(\begin{array}{c}f'_\pm\\\tilde{f}'_\pm\end{array}\right):C'\to D_\pm\oplus D'_\pm}\right).\]To see that these two new cobordisms are complementary, first note that as our initial two double-cobordisms were complementary we have\[\left(\begin{array}{cc}f_+&f'_+\\f_-&f'_-\end{array}\right)\oplus\left(\begin{array}{cc}\tilde{f}'_+&f''_+\\\tilde{f}'_-&f''_-\end{array}\right):(C\oplus C')\oplus (C'\oplus C'')\xrightarrow{\simeq} (D_+\oplus D_-)\oplus (D'_+\oplus D'_-).\]But as $C\oplus C''$ is homotopy equivalent to the cone on the obvious inclusion $i:C'\oplus C'\to C\oplus C'\oplus C'\oplus C''$ there is a homotopy commutative diagram with the map $g$ defined to make the left-hand square homotopy commute \[\xymatrix{
C'\oplus C' \ar[r]^-i\ar[d]^{=} & C\oplus C'\oplus C'\oplus C'' \ar[d]^{\simeq}\ar[r] &C(i)\simeq C\oplus C'' \ar[d]\\
C'\oplus C'\ar[r]^-g&D_+\oplus D_-\oplus D'_+\oplus D'_-\ar[r] &C(g)=D''_+\oplus D''_-}\] and hence the induced vertical map on the cones is a homotopy equivalence.
\end{proof}

\subsection{The torsion $DL$-groups}

Recall from \ref{subsec:Ltorsion} that for $n\geq0$, an $(A,S)$-cobordism of $n$-dimensional, $\eps$-symmetric, $S$-acyclic Poincar\'{e} complexes $(C,\phi)$, $(C',\phi')$ over $A$ is an $(n+1)$-dimensional, $\eps$-symmetric, $S^{-1}A$-acyclic Poincar\'{e} pair over $A$\[(f:C\oplus C'\to D,(\delta\phi,\phi\oplus -\phi')\in C(f^\%)_{n+1}).\]

\begin{definition}For $n\geq 0$, two $(A,S)$-cobordisms from $(C,\phi)$ to $(C',\phi')$ \[x_\pm:=(f_\pm:C\oplus C'\to D_\pm,(\delta_\pm\phi,\phi\oplus-\phi')\in C(f^\%_\pm)_{n+1})\] are \emph{complementary} if the chain map\[\left(\begin{smallmatrix}f_+\\f_-\end{smallmatrix}\right):C\oplus C'\to D_+\oplus D_-\]is a homotopy equivalence over $A$. In which case we say $(C,\phi)$ and $(C',\phi')$ are \emph{$(A,S)$-double-cobordant} and that the set $\{x_+,x_-\}$ is a \emph{$(A,S)$-double-cobordism} between them.
\end{definition}

Entirely analogously to Lemma \ref{welldef} and Proposition \ref{DLA} we obtain:

\begin{lemma}For $n\geq0$ let $(C,\phi)\simeq(C',\phi')$ be homotopy equivalent $n$-dimensional, $\eps$-symmetric $S^{-1}A$-acyclic Poincar\'{e} complexes over $A$, then $(C,\phi)$ and $(C',\phi')$ are $(A,S)$-double-cobordant.
\end{lemma}

\begin{proposition}For $n\geq0$, $(A,S)$-double-cobordism is an equivalence relation on the set of homotopy equivalence classes of $(n+1)$-dimensional, $(-\eps)$-symmetric, $S$-acyclic Poincar\'{e} complexes over $A$. The equivalence classes form a group $DL^{n}(A,S,\eps)$, the \emph{$n$-dimensional, $\eps$-symmetric $DL$-group of $(A,S)$}, with addition and inverses given by\[(C,\phi)+(C',\phi')=(C\oplus C',\phi\oplus\phi'),\qquad -(C,\phi)=(C,-\phi)\in DL^n(A,S,\eps).\]
\end{proposition}

\subsection{The $\widehat{DL}$-groups of Seifert forms}

Recall $R$ is a commutative Noetherian ring with involution and unit. For this subsection, assume $\eps^{-1}=\overline{\eps}$.

\medskip

Recall from \ref{subsec:Lseifert} that for $n\geq0$, a Seifert cobordism of $n$-dimensional, $\eps$-symmetric Seifert complexes $(C,\hat{\psi})$, $(C',\hat{\psi}')$ over $R$ is an $(n+1)$-dimensional, $\eps$-ultraquadratic Poincar\'{e} pair over $R$\[(f:C\oplus C'\to D,(\delta\hat{\psi},\hat{\psi}\oplus -\hat{\psi}')\in C(f\otimes f)_{n+1}).\]

\begin{definition}For $n\geq 0$, two Seifert cobordisms from $(C,\hat{\psi})$ to $(C',\hat{\psi}')$ \[x_\pm:=(f_\pm:C\oplus C'\to D_\pm,(\delta_\pm\hat{\psi},\hat{\psi}\oplus-\hat{\psi}')\in C(f_\pm\otimes f_\pm)_{n+1})\] are \emph{complementary} if the chain map\[\left(\begin{smallmatrix}f_+\\f_-\end{smallmatrix}\right):C\oplus C'\to D_+\oplus D_-\]is a chain homotopy equivalence over $R$. In which case we say $(C,\hat{\psi})$ and $(C',\hat{\psi}')$ are \emph{Seifert double-cobordant} and that the set $\{x_+,x_-\}$ is a \emph{ Seifert double-cobordism} between them.
\end{definition}

\begin{lemma}For $n\geq0$ let $(C,\hat{\psi})\simeq(C',\hat{\psi}')$ be homotopy equivalent $n$-dimensional, $\eps$-symmetric Seifert complexes over $R$, then $(C,\phi)$ and $(C',\phi')$ are Seifert double-cobordant.
\end{lemma}

\begin{proof}Let $h:(C,\hat{\psi})\to(C',\hat{\psi}')$ be a given chain homotopy equivalence with chain homotopy inverse $g$. Choose $(\hat{\psi}+\eps \hat{\psi}^*)^{-1}$ to be a chain homotopy inverse to $(\hat{\psi}+\eps \hat{\psi}^*):C^{n-*}\xrightarrow{\simeq}C$ and set\[e=\hat{\psi}(\hat{\psi}+\eps\hat{\psi}^*)^{-1}:C\to C.\]There there are complementary lagrangians\[\begin{array}{lccrl}
((&1&g&):&C\oplus C'\to C,(0,\hat{\psi}\oplus-\hat{\psi}')),\\
((&1-e&-eg&):&C\oplus C'\to C,(0,\hat{\psi}\oplus-\hat{\psi}')).
\end{array}\] (cf.\ Lemmas \ref{welldef} and \ref{doublyslice}).
\end{proof}

Note that for Seifert complexes we have not had to assume the underlying ring contains a half-unit.

\begin{proposition}For $n\geq0$, Seifert double-cobordism is an equivalence relation on the set of homotopy equivalence classes of $n$-dimensional, $\eps$-symmetric, Seifert complexes over $R$. The equivalence classes form a group $\widehat{DL}^n(R,\eps)$, the \emph{$n$-dimensional, $\eps$-symmetric $DL$-group of Seifert complexes over $R$}, with addition and inverses given by\[(C,\hat{\psi})+(C',\hat{\psi}')=(C\oplus C',\hat{\psi}\oplus\hat{\psi}'),\qquad -(C,\hat{\psi})=(C,-\hat{\psi})\in \widehat{DL}^{n}(R,\eps).\]
\end{proposition}

Recall that $P$ is the set of Alexander polynomials with coefficients in the ring $R$. We will need the double $L$ group of Seifert forms to eventually prove some results in Chapter \ref{chap:knots} concerning Seifert forms of knots. However, we do not intend to develop the theory much further as most information about them can be obtained another way, for the following (sketched) reason.

\begin{theorem}There is an isomorphism of abelian groups\[B:\widehat{DL}^n(R,\eps)\xrightarrow{\cong}DL^{n}(R[z,z^{-1}],P,\eps).\]
\end{theorem}

\begin{proof}[Proof (sketch)]The morphism $B$ is the $(\pm\eps)$-symmetric chain complex level \textit{covering} morphism of Ranicki (see \cite[p.237]{MR620795}). The `algebraic transversality' construction of Ranicki \cite[32.10]{MR1713074} shows that the covering is surjective on the level of chain homotopy equivalence classes. The covering construction sends Seifert (double-)nullcobordisms to $(R[z,z^{-1}],P)$-(double-)nullcobordisms and the kernel of the covering map on the level of chain homotopy equivalence classes consists of Seifert double-nullcobordant complexes in analogy to Theorem \ref{algtrans}.
\end{proof}

\subsection{The $D\Gamma$-groups}

Recall the following definition given in \ref{sec:connective}:

\begin{definition}For $n\geq0$ let $(P,\theta)$ be an $n$-dimensional $\eps$-symmetric complex over $A$. If $S^{-1}A\otimes(P,\theta)$ is a Poincar\'{e} complex over $S^{-1}A$ then we say $(P,\theta)$ is \emph{$S$-Poincar\'{e} over $A$}.

Let $(f:P\to Q,(\delta\theta,\theta))$ be an $(n+1)$-dimensional $\eps$-symmetric pair over $A$. If $S^{-1}A\otimes_A(f:P\to Q,(\delta\theta,\theta))$ is a Poincar\'{e} pair over $S^{-1}A$ then we say $(f:P\to Q,(\delta\theta,\theta))$ is \textit{$S$-Poincar\'{e} over $A$}.

An \emph{$S^{-1}A$-cobordism} of two $n$-dimensional $\eps$-symmetric $S$-Poincar\'{e} complexes, $(P,\theta)$ and $(P',\theta')$, is an $(n+1)$-dimensional, $\eps$-symmetric $S$-Poincar\'{e} pair \[(f:P\oplus P'\to Q,(\delta\theta,\theta\oplus-\theta')),\]over $A$.
\end{definition}

If $(P,\theta)$ is $S$-Poincar\'{e} over $A$ then the skew-suspension of the boundary $(\partial P,\partial\theta)$ is well-defined in $\C_+(A,S)$. If $(f:P\to Q,(\delta\theta,\theta))$ is $S$-Poincar\'{e} over $A$ then the skew-suspension of the surgery dual pair $(f':P'\to Q^{n+1-*},(\delta\theta',\theta'))$ is well-defined in $\C_+(A,S)$ and is connected.

\begin{definition}\label{def:relboundary}The \textit{relative boundary} of a connected $(n+1)$-dimensional $\eps$-symmetric pair $(f:P\to Q,(\delta\theta,\theta)$ is the chain complex\[\partial(Q,P):=\Sigma^{-1}C(\ev_l(\delta\theta,\theta):C(f)^{n+1-*}\to Q).\]Up to homotopy, it is given by $(P')^{n-*}\simeq \partial(Q,P)$, with $(P',\theta')$, as always, denoting the effect of surgery on $(f:P\to Q,(\delta\theta,\theta)$. As such, there is an $\eps$-symmetric Poincar\'{e} pair $(\partial f:\partial P\to \partial(Q,P),(\overline{\theta'},\partial\theta'))$ where $\partial f$  is the map $\partial P\to (P')^{n-*}\simeq \partial(Q,P)$ and the symmetric structure is the usual symmetric structure of the algebraic thickening $(P',\theta')$ modified by the homotopy equivalence $(P')^{n-*}\simeq \partial(Q,P)$. (It might be useful to consult diagram \ref{eq:thickening}.)
\end{definition}

\begin{figure}[h]
\def\picrelboundary{\resizebox{0.3\textwidth}{!}{ \includegraphics{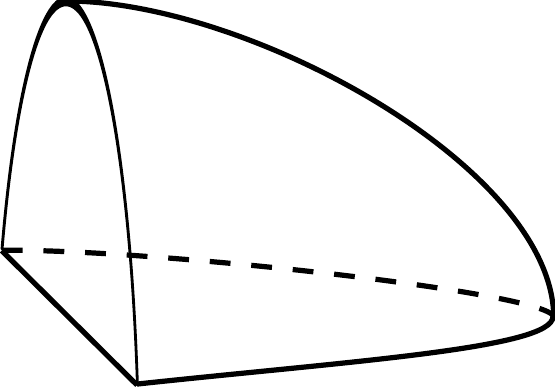}}}

\[\begin{xy} \xyimport(219.73,159.08){\picrelboundary}
,(25,85)*!L{P}
,(100,85)*!L{Q}
,(0,20)*!L{\partial P}
,(80,28)*!L{\partial (Q,P)}

\end{xy}\]
                \caption{A pair and relative boundary}
	       \label{transitivity}
\end{figure}

\subsection*{What is $\partial$-complementary?}
We intend to define a condition for two $S^{-1}A$-cobordisms to be `$\partial$-complementary'. To make sense of the subsequent definition, here is a discussion of the idea:

\medskip

Recall from Chapter \ref{chap:linking} that a suggested naive approach to the definition of the double Witt $\Gamma$-groups $D\Gamma^\eps(A\to S^{-1}A)$ was to take the monoid of $S$-non-singular $\eps$-symmetric forms over $A$ and to declare forms with $S$-complementary $S$-lagrangians to be the 0-objects. This approach is not the one ultimately used because the boundary linking form of such a 0-object is not necessarily hyperbolic, which is a flaw if we are interested in double Witt groups of linking forms and a well-defined localisation sequence. Our solution was instead to impose a joint condition on a pair of lagrangians that ensured the boundary linking form \textit{was} hyperbolic. What follows is the chain-level construction that is an extension of this idea.

\medskip

For this discussion let's restrict ourselves to thinking about $S^{-1}A$-cobordisms to 0 i.e. $S^{-1}A$-nullcobordisms. Suppose an $n$-dimensional, $\eps$-symmetric $S$-Poincar\'{e} complex over $A$ $(P,\theta)$ has two $S^{-1}A$-nullcobordisms $(f_\pm:P\to Q_\pm,(\delta_\pm\theta,\theta))$.

\[\begin{xy} 
\def\picSAnullcob{\resizebox{0.8\textwidth}{!}{ \includegraphics{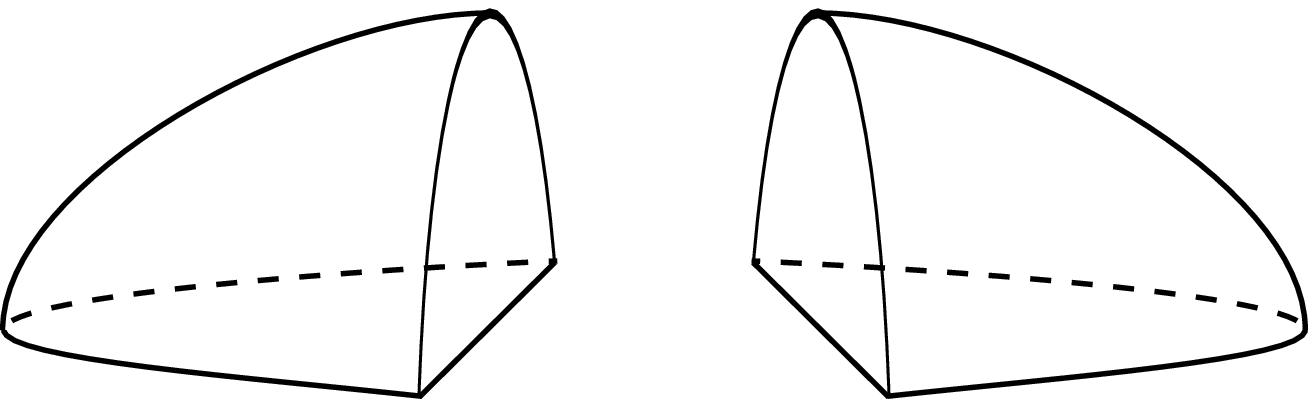}}}
\xyimport(219.73,159.08){\picSAnullcob}
,(32,38)*!L{\partial (Q_-,P)}
,(160,38)*!L{\partial (Q_+,P)}
,(172,90)*!L{Q_+}
,(43,90)*!L{Q_-}
,(85,30)*!L{\partial P}
,(128,30)*!L{\partial P}
,(80,90)*!L{P}
,(135,90)*!L{P}
\end{xy}\]

There is a homotopy commutative diagram, where the rows are cofibration sequences analogous to Meier-Vietoris sequences \[\xymatrix{
P^{n-*}\ar[dd]^{\theta_0}\ar[r]&C(f_+)^{n+1-*}\oplus C(f_-)^{n+1-*}\ar[r]\ar[dd]^-{\ev_l(\delta_+\theta,\theta)\oplus\ev_l(\delta_-\theta,\theta)}&(Q_+\cup_P Q_-)^{n+1-*}\ar[dd]^-{(\delta_+\theta\cup_\theta\delta_-\theta)_0}\\
&&\\
P\ar[r]^-{\lmat f_+\\-f_-\rmat}&Q_+\oplus Q_-\ar[r]&Q_+\cup_P Q_-}\]To measure the failure of $(P,\theta)$ to be Poincar\'{e} over $A$, we consider the desuspension of the mapping cone (also known as the boundary):\[\partial P:=\Sigma^{-1}C(\theta_0:P^{n-*}\to P)\]To measure the failure of $(f_\pm:P\to Q_\pm,(\delta_\pm\theta,\theta))$ to be Poincar\'{e} over $A$, we consider the desuspension of the mapping cone (also known as the relative boundary)\[\partial(Q_\pm,P):=\Sigma^{-1}C(\ev_l(\delta_\pm\theta,\theta):C(f_\pm)^{n+1-*}\to Q_\pm)\]

By definition, the algebraic union $(Q_+\cup_{P} Q_-,\delta_+\theta\cup_{\theta}\delta_-\theta)$ is Poincar\'{e} if and only if $(\delta_+\theta\cup_{\theta}\delta_-\theta)_0$ is a homotopy equivalence. We interpret this as a joint condition on the two $S^{-1}A$-nullcobordisms $(f_\pm:P\to Q_\pm,(\delta_\pm\theta,\theta))$. If it is met we say the two $S^{-1}A$-nullcobordisms $(f_\pm:P\to Q_\pm,(\delta_\pm\theta,\theta))$ are \emph{$\partial$-complementary} (read: `boundary' complementary). The following are clearly equivalent ways to state this condition:\begin{itemize}
\item $(Q_+\cup_{P} Q_-,\delta_+\theta\cup_{\theta}\delta_-\theta)$ is Poincar\'{e} over $A$,
\item $\partial(Q_+,P)\cup_{\partial P}\partial(Q_-,P)\simeq 0$,
\item The induced map $\partial P\to \partial(Q_+,P)\oplus\partial(Q_-,P)$ is a homotopy equivalence.
\end{itemize}

As stated, this condition only makes sense for $S^{-1}A$-nullcobordisms $(f_\pm:P\to Q_\pm,(\delta_\pm\theta,\theta))$ where $(P,\theta)$ is $n$-dimensional and $n>0$. In the special case $n=0$, we will say two such $S^{-1}A$-nullcobordisms are $\partial$-complementary if the $S^{-1}A$-nullcobordisms $\overline{S}(f_\pm:P\to Q_\pm,(\delta_\pm\theta,\theta))$ of the 2-dimensional $(-\eps)$-symmetric complex $\overline{S}(P,\theta)$ over $A$ are $\partial$-complementary.

\begin{definition}For $n\geq0$ let $(P,\theta)$, $(P',\theta')$ be $n$-dimensional, $\eps$-symmetric $S$-Poincar\'{e} complexes over $A$. We say $(P,\theta)$ and $(P',\theta')$ are \emph{$\partial$-double-cobordant} if there exist two $S^{-1}A$-cobordisms: \[(f_\pm:P\oplus P'\to Q_\pm,(\delta_\pm\theta,\theta\oplus-\theta')\in C(f^\%_\pm)_{n+1})\]that are $\partial$-complementary. In which case we call the two $S^{-1}A$-cobordisms $(f_\pm:P\oplus P'\to Q_\pm,(\delta_\pm\theta,\theta\oplus-\theta'))$ a \emph{$\partial$-double-cobordism} (Fig \ref{fig:bounddoub}).
\end{definition}

\begin{figure}[h]
\[\begin{xy} 
\def\picboundarydoubcob{\resizebox{0.7\textwidth}{!}{ \includegraphics{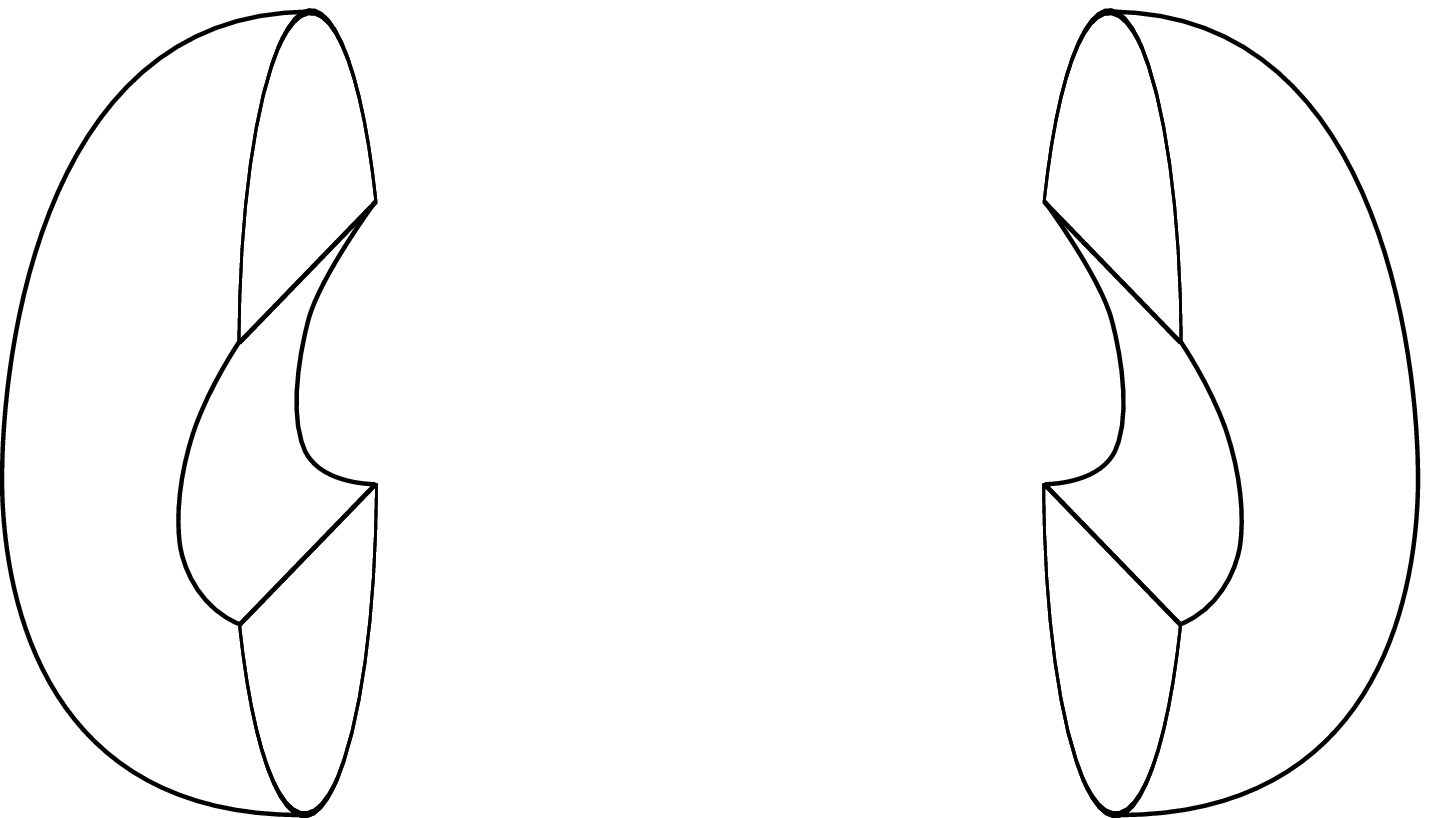}}}
\xyimport(510,290){\picboundarydoubcob}
,(107,235)*!L{P}
,(392,235)*!L{P}
,(25,140)*!L{Q_-}
,(120,160)*!L{\partial(Q_-,P\oplus P')}
,(145,225)*!L{\partial P}
,(345,225)*!L{\partial P}
,(107,55)*!L{P'}
,(392,55)*!L{P'}
,(462,140)*!L{Q_+}
,(280,160)*!L{\partial (Q_+,P\oplus P')}
,(145,110)*!L{\partial P'}
,(340,110)*!L{\partial P'}
\end{xy}\]
\caption{A $\partial$-double-cobordism between $P$ and $P'$.}
\label{fig:bounddoub}
\end{figure}

\begin{lemma}For $n\geq0$ let $(P,\theta)$, $(P',\theta')$ be $n$-dimensional, $\eps$-symmetric $S$-Poincar\'{e} complexes over $A$ and suppose $(P,\theta)\simeq (P',\theta')$ over $A$, then $(P,\theta)$ and $(P',\theta')$ are $\partial$-double-cobordant.
\end{lemma}

\begin{proof}Let $f:(P,\theta)\xrightarrow{\simeq} (P',\theta')$ be a homotopy equivalence over $A$ and take the $S^{-1}A$-cobordisms \[\begin{array}{lrl}
((f\,\,&1):&P\oplus P'\to P',(0,\theta\oplus-\theta')),\\
((f\bar{s}\,&-s):&P\oplus P'\to P',(0,\theta\oplus-\theta')).
\end{array}\]The induced map $\left(\begin{smallmatrix}f&1\\f\bar{s}&-s\end{smallmatrix}\right):P\oplus P'\to P'\oplus P'$ is a homotopy equivalence (cf. Lemma \ref{welldef}). Hence the cone $P'\cup_{P\oplus P} P'$ is contractible and thus trivially Poincar\'{e}.
\end{proof}

\begin{proposition}For $n\geq0$, $\partial$-double-cobordism is an equivalence relation on the set of $A$-homotopy equivalence classes of $n$-dimensional, $\eps$-symmetric $S$-Poincar\'{e} complexes over $A$. The equivalence classes form a group $D\Gamma^n(A\to S^{-1}A,\eps)$, the \emph{$n$-dimensional, $\eps$-symmetric $D\Gamma$-group of $(A,S)$}, with addition and inverses given by\[(P,\theta)+(P',\theta')=(P\oplus P',\theta\oplus\theta'),\qquad -(P,\theta)=(P,-\theta)\in D\Gamma^n(A\to S^{-1}A,\eps).\]
\end{proposition}

\begin{proof}The previous lemma shows in particular that $\partial$-double-cobordism is well defined and reflexive. It is clearly symmetric. For transitivity, re-glue the two $\partial$-double-cobordisms \[\begin{array}{rcl}
c_\pm&=&((f_\pm\,\,f'_\pm):P\oplus P'\to Q_\pm,(\delta_\pm\theta,\theta\oplus-\theta')),\\
	c'_\pm&=&((\tilde{f}'_\pm\,\,f''_\pm):P'\oplus P''\to Q_\pm,(\delta_\pm\theta',\theta'\oplus-\theta'')).\end{array}\]analogously to Figure \ref{transitivity} (cf. Proposition \ref{DLA}). Using the characterisation of the $\partial$-complementary condition in terms of the boundaries and relative boundaries: $\partial (P\oplus P')$, $\partial(P'\oplus P'')$, $\partial(Q_\pm,P\oplus P')$ and $\partial(Q'_\pm,P'\oplus P'')$, the transitivity argument used in \ref{DLA} can now be used again.
\end{proof}

\section{Double $L$-theory localisation exact sequence}\label{sec:localisation2}

\subsection{Some technical lemmas}

The following technical results are required to prepare us for the double $L$-theory localisation exact sequence of Theorem \ref{thm:DLLES}. One of the lemmas concerns so-called `algebraic Poincar\'{e} trinities' and is borrowed from {\cite[\textsection 4.5]{Borodzik:2012fk}}.

\begin{lemma}[``Folding the boundary'']\label{lem:technical}For $n\geq1$, if an $n$-dimensional $\eps$-symmetric Poincar\'{e} complex $(D'',\phi'')\simeq(D\cup_C D',\phi\cup_{\partial\phi}\phi')$ is the algebraic glue of two $n$-dimensional $\eps$-symmetric pairs and there is defined an $(n+1)$-dimensional $\eps$-symmetric Poincar\'{e} pair \[((f''\,\,\tilde{f}):D''\oplus \tilde{D}\to E,(\Phi,\phi\oplus-\tilde{\phi}))\]then there is an $(n+1)$-dimensional $\eps$-symmetric pair\[((D/C)\oplus \tilde{D}\to E/D',(\Phi/\phi',(\phi/\partial\phi)\oplus -\tilde{\phi})).\]
\end{lemma}

\[\begin{xy} 
\def\picfolding{\resizebox{0.3\textwidth}{!}{ \includegraphics{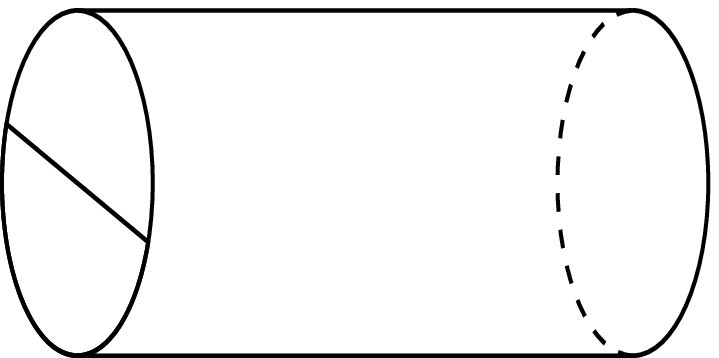}}}
\xyimport(254,139){\picfolding}
,(25,105)*!L{D}
,(18,52)*!L{D'}
,(-25,100)*!L{C}
,(124,77)*!L{E}
,(217,78)*!L{\tilde{D}}
\end{xy}
\begin{xy}
(0,0)*+{}
,(2,13)*+{}="A";(25,13)*+{}="B"
,{"A"\ar@/^/"B"}
\end{xy}
\begin{xy} 
\def\picfoldingtwo{\resizebox{0.3\textwidth}{!}{ \includegraphics{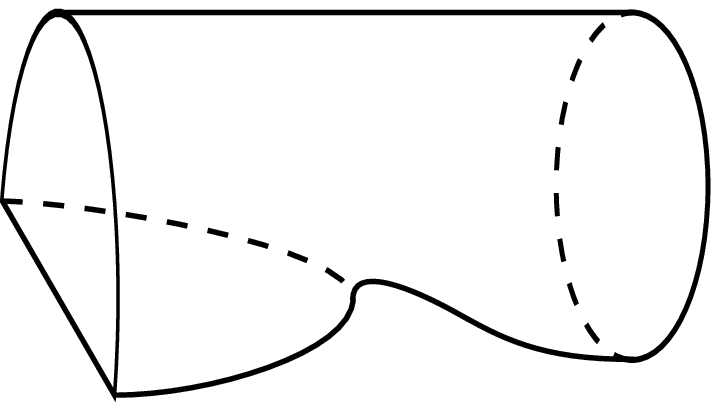}}}
\xyimport(291,139){\picfoldingtwo}
,(17,90)*!L{D}
,(68,40)*!L{D'}
,(-5,35)*!L{C}
,(150,79)*!L{E}
,(250,79)*!L{\tilde{D}}

\end{xy}\]

\begin{proof}
Apply the relative algebraic Thom construction to the $(n+1)$-dimensional $\eps$-symmetric triad given by the homotopy commuting square\[\xymatrix{C\ar[r]^{f'}\ar[d]_-{\lmat f\\0\rmat}&D'\ar[d]\\D\oplus \tilde{D}\ar[r]&E}\]with $(n+1)$-dimensional $\eps$-symmetric structure $(\Phi,\phi\oplus \tilde{\phi},\phi',\partial\phi)$.
\end{proof}

\begin{lemma}[``Algebraic trinities'' {\cite[\textsection 4.5]{Borodzik:2012fk}}]\label{lem:trinity}For $n\geq1$, suppose $(C,\phi)$ is an $(n-1)$-dimensional $\eps$-symmetric Poincar\'{e} complex over $A$ and there are three $n$-dimensional Poincar\'{e} pairs over $A$ \[(f_i:C\to D_i,(\delta_i\phi,\phi)\in C(f^\%_i)_n),\qquad i=1,2,3.\]Then there exists $(\partial E\to E,(\delta\nu,\nu))$, an $(n+1)$-dimensional, $\eps$-symmetric Poincar\'{e} pair over $A$ (cf. Figure \ref{fig:trinity}) such that \[\begin{array}{rcl}(\partial E,\nu)&\simeq& \bigoplus_{i\neq j}(D_i\cup_C D_j,\delta_i\phi\cup_{\phi}\delta_j\phi).\end{array}\]
\end{lemma}

\begin{figure}[h]
\[\begin{xy} 
\def\pictrinity{\resizebox{0.4\textwidth}{!}{ \includegraphics{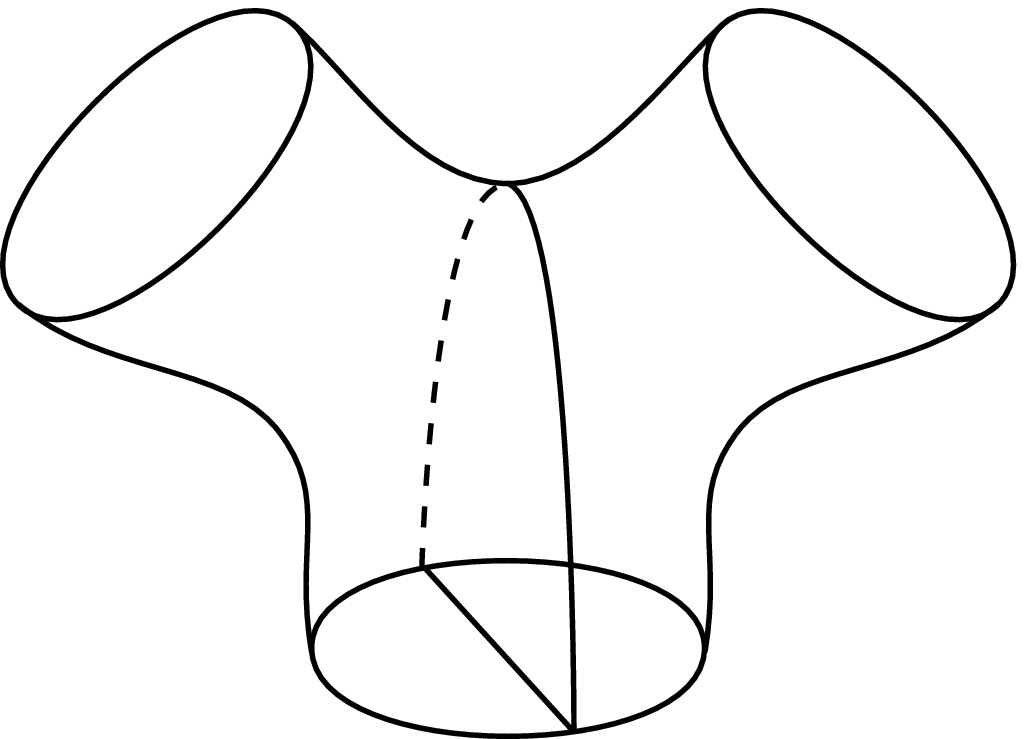}}}
\xyimport(366,264){\pictrinity}
,(169,105)*!LD{D_1}
,(140,23)*!LD{D_2}
,(215,23)*!LD{D_3}
,(360,210)*!LD{D_1\cup_C D_3}
,(-100,210)*!LD{D_1\cup_C D_2}
,(140,-40)*!LD{D_2\cup_C D_3}
,(280,90)*!L{\,\, C}="C";(182,43)*{}="line"
,,{"C"\ar@/_/"line"}
\end{xy}\]
        \caption{The three nullcobordisms of Lemma \ref{lem:trinity}, resulting in a Poincar\'{e} pair.}\label{fig:trinity}
\end{figure}

\begin{proof}[Proof (sketch)]Apply Lemma \ref{lem:technical} to each of the $(n+1)$-dimensional $\eps$-symmetric Poincar\'{e} pairs (these pairs should be thought of as `cylinder' or `trivial' cobordisms)
\[((1\,\,1):(D_1\cup_C D_2)\oplus (D_1\cup_C D_2)\to D_1\cup_C D_2,(0,(\delta_1\phi\cup_\phi\delta_2\phi)\oplus-(\delta_1\phi\cup_\phi\delta_2\phi))),\]
\[((1\,\,1):(D_1\cup_C D_3)\oplus (D_1\cup_C D_3)\to D_1\cup_C D_3,(0,(\delta_1\phi\cup_\phi\delta_3\phi)\oplus-(\delta_1\phi\cup_\phi\delta_3\phi))),\]
to obtain two $(n+1)$-dimensional $\eps$-symmetric pairs with a common boundary component $(D_1/C,\delta_1\phi/\phi)$, which we write as:\[\begin{array}{rcl}x_{12}&:=&((D_1/C)\oplus(D_1\cup_C D_2)\to (D_1\cup_C D_2)/D_2,(0/\delta_2\phi,(\delta_1\phi/\phi)\oplus -(\delta_1\cup_{\phi}\delta_2\phi))),\\
x_{13}&:=&((D_1/C)\oplus(D_1\cup_C D_3)\to (D_1\cup_C D_3)/D_3,(0/\delta_3\phi,(\delta_1\phi/\phi)\oplus -(\delta_1\phi\cup_{\phi}\delta_3\phi))).\end{array}\]Glue $x_{12}$ to $x_{13}$ along the common boundary component (cf.\ \ref{def:glue2}). The glued object\[x_{12}\cup x_{13}=((D_1\cup_C D_2)\oplus (D_1\cup_C D_3)\to E,(\delta\nu, (\delta_1\phi\cup_{\phi}\delta_2\phi)\oplus (\delta_1\phi\cup_{\phi}\delta_3\phi)))\](where $\delta\nu$ is defined by the glueing operation) is an $(n+1)$-dimensional $\eps$-symmetric pair, but it is not generally Poincar\'{e}. The failure to be Poincar\'{e} is \[\partial(E,(D_1\cup_C D_2)\oplus (D_1\cup_C D_3))\simeq D_2\cup_C D_3,\]and we can use this chain complex to improve $x_{12}\cup x_{13}$ to an $(n+1)$-dimensional $\eps$-symmetric Poincar\'{e} pair $(\partial E\to E,(\delta\nu,\nu))$. By analysing the symmetric structure of this improvement we see that \[\begin{array}{rcl}(\partial E,\nu)&\simeq& \bigoplus_{i\neq j}(D_i\cup_C D_j,\delta_i\phi\cup_{\phi}\delta_j\phi).\end{array}\]
\end{proof}

\begin{proposition}\label{prop:technical}For $n\geq 0$, suppose the boundary $(\partial P,\partial\theta)$ of an $n$-dimensional, $\eps$-symmetric $S$-Poincar\'{e} complex $(P,\theta)$ over $A$ admits two complementary $(A,S)$-nullcobordisms. Then $(P,\theta)$ is $\partial$-double-cobordant to an $n$-dimensional $\eps$-symmetric Poincar\'{e} complex.
\end{proposition}

\begin{proof}Write the two complementary $(A,S)$-nullcobordisms of $(\partial P,\partial \theta)$ as\[(f_\pm:\partial P\to D_\pm,(\nu_\pm,\partial \theta)\in C(f_\pm^\%)_{n}).\]Recall that as $(P,\theta)$ is a connected $n$-dimensional $\eps$-symmetric complex over $A$ then the algebraic thickening is the $n$-dimensional $\eps$-symmetric Poincar\'{e} pair $(i:\partial P\to P^{n-*},(\overline{\theta},\partial\theta))$ where $\partial P:=\Sigma^{-1}C(\theta_0)$ and $(\partial P,\partial\theta)$ is an $(n-1)$-dimensional $\eps$-symmetric Poincar\'{e} complex over $A$. It follows that both algebraic unions $(P^{n-*},\overline{\theta})\cup_{(\partial P,\partial\theta)}(D_\pm,\nu_\pm)$ are Poincar\'{e} complexes over $A$. 

\medskip

We will show that $(P,\theta)$ is $\partial$-double-cobordant to $(P^{n-*},\overline{\theta})\cup_{(\partial P,\partial\theta)}(D_+,\nu_+)$.

\medskip

Applying Lemma \ref{lem:trinity} to the three nullcobordisms\[(i:\partial P\to P^{n-*},(\overline{\theta},\partial\theta)),\quad(f_\pm:\partial P\to D_\pm,(\nu_\pm,\partial\theta)),\]we obtain an $(n+1)$-dimensional, $\eps$-symmetric Poincar\'{e} pair over $A$ which we write as \[x_Q:=(\partial Q\to Q,(\delta\Theta,\Theta))\] where \[\begin{array}{rcl}\partial Q&\simeq& (P^{n-*}\cup_\partial D_+)\oplus (P^{n-*}\cup_\partial D_-)\oplus (D_-\cup_\partial D_+)\\
&\simeq& (P^{n-*}\cup_\partial D_+)\oplus (P^{n-*}\cup_\partial D_-)\end{array} \]

Following the construction in the proof of Lemma \ref{lem:trinity}, by folding two `cylinder' cobordisms we also obtain two $(n+1)$-dimensional $\eps$-symmetric pairs:\[x_\pm:=((P^{n-*}/\partial P)\oplus (P^{n-*}\cup_\partial D_\pm)\to  (P^{n-*}\cup_\partial D_\pm)/D_\pm,(0/\nu_\pm,\overline{\theta}/\partial\theta\oplus -(\overline{\theta}\cup_{\partial\theta}\nu_\pm)))\]
\[\simeq(P\oplus (P^{n-*}\cup_\partial D_\pm)\to  (P^{n-*}\cup_\partial D_\pm)/D_\pm,(0/\nu_\pm,\theta\oplus -(\overline{\theta}\cup_{\partial\theta}\nu_\pm)))\]

Noting that $x_-$ and $x_Q$ have a common boundary component $(P^{n-*}\cup_{\partial P}D_-,(\overline{\theta}\cup_{\partial\theta}\nu_-))$, we will now glue pairs according to the schematic shown in Figure \ref{fig:fever}.

\begin{figure}[h]
\[\begin{xy} 
\def\picfever{\resizebox{\textwidth}{!}{ \includegraphics{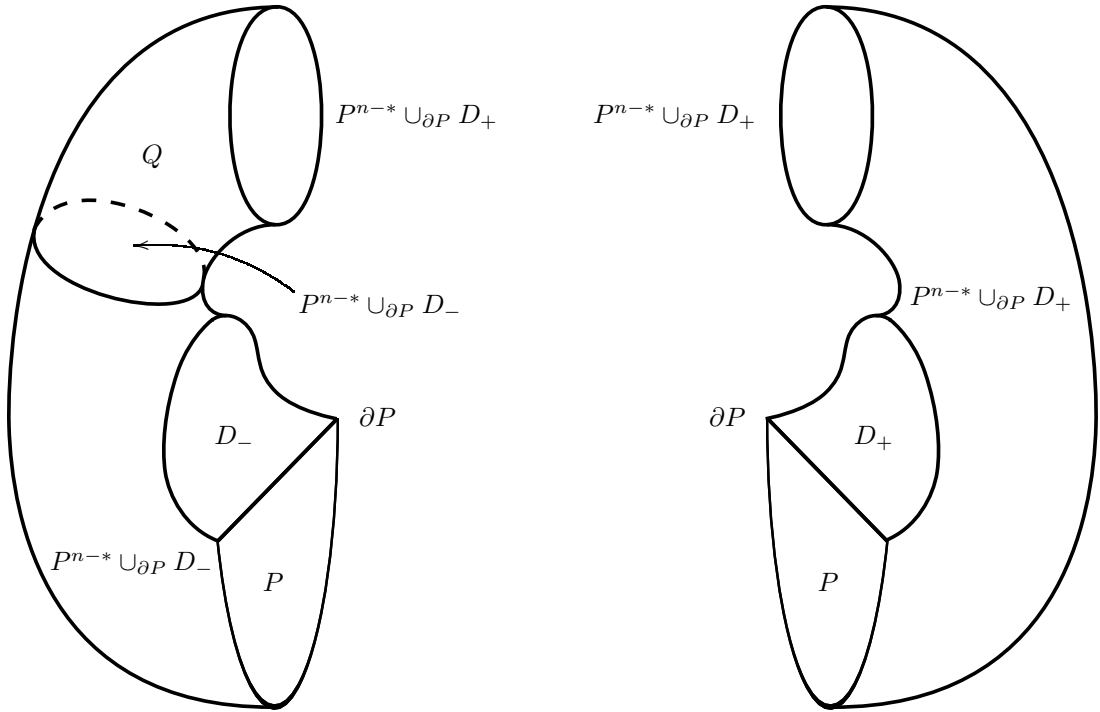}}}
\xyimport(456.86,291.11){\picfever}
,(110,50)*!LD{P}
,(340,50)*!LD{P}
,(355,108)*!LD{D_+}
,(90,108)*!LD{D_-}
,(60,224)*!LD{Q}
,(150,118)*!LD{\partial P}
,(295,118)*!LD{\partial P}
,(140,240)*!LD{P^{n-*}\cup_{\partial P}D_+}
,(247,240)*!LD{P^{n-*}\cup_{\partial P}D_+}
,(378,165)*!LD{P^{n-*}\cup_{\partial P}D_+}
,(22,56)*!LD{P^{n-*}\cup_{\partial P}D_-}
,(123,173)*!LU{\,P^{n-*}\cup_{\partial P}D_-}="A";(57,192)*{}="B"
,{"A"\ar@/_/"B"}
\end{xy}\]
                \caption{The $\partial$-double-cobordism from $(P,\theta)$ to $(P^{n+1-*},\overline{\theta})\cup_{(\partial P,\partial \theta)} (D_+,\nu_+)$}
	       \label{fig:fever}

\end{figure}

Hence there are defined two pairs $x_+$ and $x_-\cup x_Q$. For the sake of being explicit, we write these two pairs out fully (but for understanding, the pictures together with the yoga of glueing operations is more useful!):\[(P\oplus (P^{n-*}\cup_\partial D_+)\to  (P^{n-*}\cup_\partial D_+)/D_+,(0/\nu_+,\theta\oplus -(\overline{\theta}\cup_{\partial\theta}\nu_+))),\]
\[(P\oplus (P^{n-*}\cup_\partial D_+)\to  ((P^{n-*}\cup_\partial D_\pm)/D_-)\cup_{P^{n-*}\cup_{\partial P}D_-} Q,((0/\nu_-)\cup_{\overline{\theta}\cup_{\partial\theta}\nu_-}\delta\Theta,\theta\oplus -(\overline{\theta}\cup_{\partial\theta}\nu_\pm))).\]

Firstly, note these pairs are both $S$-Poincar\'{e} as the effect of surgery on $(P\oplus (P^{n-*}\cup_\partial D_+),\theta\oplus -(\overline{\theta}\cup_{\partial\theta}\nu_+))$ with data either pair is $S$-acyclic (in fact the effects are respectively homotopy equivalent to one of $(D_\pm/\partial P,\nu_\pm/\partial\theta)$, both of which are contractible over $S^{-1}A$). So each pair is an $S^{-1}A$-cobordism from $(P,\theta)$ to $(P^{n-*}\cup_\partial D_+,(\overline{\theta}\cup_{\partial\theta}\nu_+))$.

Next we claim that the $S^{-1}A$-cobordisms are $\partial$-complementary. To see this we observe\[\begin{array}{rcl}
D_+&\simeq&\partial((P^{n-*}\cup_\partial D_+)/D_+,P\oplus (P^{n-*}\cup_\partial D_+))\\
D_-&\simeq&\partial (((P^{n-*}\cup_\partial D_\pm)/D_-)\cup_{P^{n-*}\cup_{\partial P}D_-} Q,(P\oplus (P^{n-*}\cup_\partial D_+)))\\
\partial P&\simeq&\partial (P\oplus (P^{n-*}\cup_{\partial P} D_+))\end{array}\]and the induced map:\[\partial P\to D_+\oplus D_-\]was a homotopy equivalence by hypothesis.

Hence we have described two $\partial$-complementary $S^{-1}A$-cobordisms from $(P,\theta)$ to $(P^{n-*}\cup_\partial D_+,(\overline{\theta}\cup_{\partial\theta}\nu_+))$ as was required.
\end{proof}

\subsection{Double $L$-theory localisation exact sequence}

For $n\geq0$ we will define some group homomorphisms between the cobordism and double-cobordism groups we have constructed (checking along the way that these homomorphisms are well-defined). The morphisms are equal on the level of monoids of $\eps$-symmetric Poincar\'{e} complexes to those morphisms in the long exact sequence of Ranicki from Theorem \ref{thm:LES}.


If $(P,\theta)\sim0\in L^n(A,\eps)$ then by definition $(\partial P,\partial\theta)\simeq 0$ and there exists a nullcobordism $(f:P\to Q,(\delta\theta,\theta))$ with $\partial (Q, P)\simeq 0$ (as the pair is Poincar\'{e} over $A$). Hence two copies of this same nullcobordism are $\partial$-complementary (trivially). Therefore we define a homomorphism
\[Di:L^n(A,\eps)\to D\Gamma^n(A\to S^{-1}A,\eps);\quad Di([P,\theta]):=[(P,\theta)].\]

If $(P,\theta)\sim0\in D\Gamma^n(A\to S^{-1}A,\eps)$ then there exist $\partial$-complementary $S^{-1}A$-nullcobordisms $(f_\pm:P\to Q_\pm,(\delta\theta,\theta))$. Define $n$-dimensional $\eps$-symmetric pairs \[(\partial f_\pm:\partial P\to \partial(Q_\pm,P),(\overline{\theta},\partial\theta))\] using the algebraic surgery and algebraic thickening constructions defined in Chapter \ref{chap:algLtheory}.  These two pairs are complementary $(A,S)$-nullcobordisms as they come from $\partial$-complementary $S^{-1}A$-cobordisms. Hence we define a homomorphism

\[D\partial:D\Gamma^n(A\to S^{-1}A,\eps)\to DL^n(A,S,\eps);\quad D\partial([P,\theta]):=[\overline{S}(\partial P,\partial\theta)].\]

As a result of these definitions we can state our main technical result of double $L$-theory.

\begin{theorem}[Double $L$-theory localisation exact sequence]\label{thm:DLLES}For $n\geq0$, the sequence of group homomorphisms\[0\to L^{n}(A,\eps)\xrightarrow{Di} D\Gamma^{n}(A\to S^{-1}A,\eps)\xrightarrow{D\partial} DL^{n}(A,S,\eps)\]is exact.\end{theorem}

\begin{proof}
We must show exactness in two places:

\medskip

\noindent\underline{$D\Gamma^{n}(A\to S^{-1}A,\eps)$:} Clearly $D\partial\circ Di=0$ as $(\partial P,\partial\theta)=0$ for $(P,\theta)$ Poincar\'{e} over $A$.

Now suppose $[(P,\theta)]\in\ker(D\partial)$ so that $\overline{S}(\partial P,\partial\theta)$ has a pair of complementary nullcobordisms $\overline{S}(f_\pm:\partial P\to D_\pm,(\nu_\pm,\partial \theta)\in C(f^\%_\pm)_{n+1})$. Apply Proposition \ref{prop:technical} to show that $[(P,\theta)]\in\im(Di)$.

%
%

\medskip

\noindent\underline{$L^n(A,\eps)$:} Suppose $(P,\theta)\sim 0\in D\Gamma^n(A\to S^{-1}A,\eps)$ and $(P,\theta)$ is Poincar\'{e} over $A$. There exist two $\partial$-complementary $S^{-1}A$-nullcobordisms $(f_\pm:P\to Q_\pm,(\delta_\pm\theta,\theta))$. Write the effect of surgery on $(P,\theta)$ with data $(f_\pm:P\to Q_\pm,(\delta_\pm\theta,\theta))$ as $(P_\pm',\theta'_\pm)$. As $(P'_\pm)^{n-*}\simeq C(\partial P\to \partial(Q_\pm,P))\simeq \partial(Q_\pm,P)$, both $(P'_\pm,\theta'_\pm)$ are $S$-acyclic. Therefore the cobordism class of $(P,\theta)$ is in the image of $Dj$. In fact, more than this, as $\partial P\simeq \partial(Q_+,P)\oplus \partial(Q_-,P)$ by the $\partial$-complementary condition, both of $(f_\pm:P\to Q_\pm,(\delta_\pm\theta,\theta))$ are nullcobordisms over $A$.

\begin{figure}[h]
\[\begin{xy}
\def\pictwoend{\resizebox{0.35\textwidth}{!}{ \includegraphics{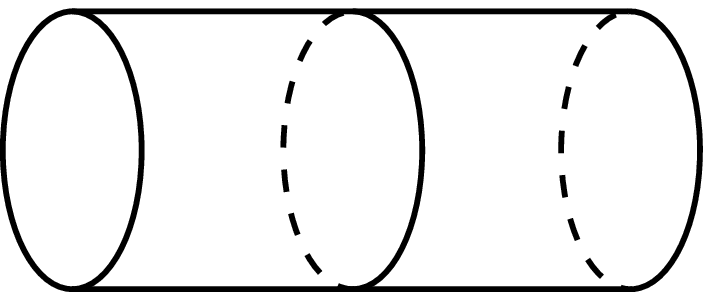}}}
\xyimport(253.02,102){\pictwoend}
,(170,43)*!LD{Q_+}
,(68,43)*!LD{Q_-}
,(265,43)*!LD{\partial(Q_+,P)\simeq 0}
,(-109,43)*!LD{0\simeq\partial(Q_-,P)}
,(120,46)*!LD{P}
\end{xy}\]
\end{figure}

\end{proof}

\begin{remark}Double $L$-theory can be developed by similar methods without the connectivity conditions on the underlying symmetric chain complexes - in the terminology of \cite[\textsection 3]{MR1211640} this is working over the \textit{stable} algebraic bordism category $(\A(A),\B(A),\C(A,S))$. In this version there is a double $L$-theory localisation exact sequence proved in exactly the same way, and in the stable version it is possible to extend the sequence:\[0\to L^{n}(A,\eps)\xrightarrow{Di} D\Gamma^{n}(A\to S^{-1}A,\eps)\xrightarrow{D\partial} DL^{n}(A,S,\eps)\xrightarrow{Dj} L^{n-1}(A,\eps).\]However, in the stable version of double $L$-theory, the connection between the double $L$-groups and the double Witt groups is not clear.
\end{remark}

\begin{remark}The reader may be wondering why there is no \textit{long} exact sequence for localisation in double $L$-theory as in the case of classical $L$-theory. Recall that for the classical $L$-theory localisation exact sequence of Theorem \ref{thm:LES}, the composition $j\circ i$ was shown to be 0 by observing that there was always an $S^{-1}A$-nullcobordism of an $n$-dimensional $\eps$-symmetric $S$-acyclic Poincar\'{e} complex $(C,\phi)$ given by the pair $(0:C\to 0,(0,\phi))$. There are no obvious candidates for a similar result showing that $(C,\phi)$ admits a $\partial$-double-nullcobordism.

Staying in the vein of desirable results which we lack, there is no obvious candidate for an exact sequence comparing the projective $DL$-groups and torsion $DL$-groups via the well-defined forgetful map\[DL^{n+1}(A,S,\eps)\to DL^n(A,\eps).\]The major obstacle here is the lack of something like Proposition \ref{prop:technical}, which is a non-obvious result that would be very challenging to extend.

We show in Chapter \ref{chap:knots} that a double knot-nullcobordism determines two complementary nullcobordisms of the Blanchfield complex of the knot, so that such a sequence would be desirable due to its geometric significance.
\end{remark}

\section{Surgery above and below the middle dimension}\label{sec:surgabove}

Recall that as there is assumed to be a half-unit in the ring, Theorem \ref{thm:algsurgery} tells us there are isomorphisms for $n\geq0$ \[\begin{array}{rrcl}\overline{S}:&L^n(A,\eps)&\xrightarrow{\cong}& L^{n+2}(A,-\eps),\\
\overline{S}:&\Gamma^n(A\to S^{-1}A,\eps)&\xrightarrow{\cong}& \Gamma^{n+2}(A\to S^{-1}A,-\eps),\\
\overline{S}:&L^n(A,S,\eps)&\xrightarrow{\cong}& L^{n+2}(A,S,-\eps).\end{array}\]Applying the skew-suspension twice thus establishes 4-periodicity in each of the above groups. This is a deep and extremely useful result. We would like to know to what extent this result is true in the double $L$- and $\Gamma$-groups.

For the standard $L$- and $\Gamma$-groups, surjectivity of the skew suspension map is shown by performing an appropriate surgery on an $(n+2)$-dimensional object to reduce its length to that of a skew-suspended $n$-dimensional object. Injectivity is shown by proving that if a skew-suspension is nullcobordant then we may pick the nullcobordism to be a skew suspension. Again, this is shown by surgery, but now in a relative setting.

We will now attempt to use the skew-suspension map to analyse the extent to which there exists periodicity in double $L$-theory. This will involve an investigation of how to shorten symmetric complexes within their double-cobordism classes in double $L$-theory using some kind of surgery. Under certain conditions on the ring we are able to achieve the required effect by `surgery above and below the middle dimension', however this process is difficult and the ring conditions we use are \textit{very} restrictive. The difficulty of performing these simplifications on a chain complex add algebraic ballast to the opinion that the `doubly-slice problem' considered in Chapters \ref{chap:blanchfield} and \ref{chap:knots} is very hard!

\subsection{Periodicity in double $L$-theory?}\label{subsec:periodicity}

\begin{proposition}\label{skewsusp}For $n\geq 0$, the skew-suspension gives well-defined injective homomorphisms\[\begin{array}{rrclll}\overline{S}:&DL^n(A,\eps)&\hookrightarrow &DL^{n+2}(A,-\eps);&\qquad&[(C,\phi)]\mapsto [\overline{S}(C,\phi)],\\
\overline{S}:&DL^{n}(A,S,\eps)&\hookrightarrow &DL^{n+2}(A,S,-\eps);&\qquad&[(C,\phi)]\mapsto [\overline{S}(C,\phi)].\end{array}\]
\end{proposition}

\begin{proof}We consider only the statement for the projective $DL$-groups $DL^n(A,\eps)$, the torsion $DL$-group case being entirely similar.

If $(C,\phi)\in DL^n(A,\eps)$ admits complementary nullcobordisms $(f_\pm:C\to D_\pm,(\delta_\pm\phi,\phi))$ then the skew-suspensions \[\overline{S}(f_\pm:C\to D_\pm,(\delta_\pm\phi,\phi)):=(\Sigma f_\pm:\Sigma C\to \Sigma D,\overline{S}(\delta\phi,\phi))\] are complementary nullcobordisms for $\overline{S}(C,\phi)\in DL^{n+2}(A,-\eps)$. Therefore the homomorphism is well defined.

To show injective, consider the general situation of a pair $x$ given by formal skew-desuspension an $(n+3)$-dimensional $(-\eps)$-symmetric pair \[x:=(\Sigma^{-1} f:\Sigma^{-1} C\to \Sigma^{-1} D,(\overline{S})^{-1}(\delta\phi,\phi)).\] Such a formal skew-desuspension is an $(n+1)$-dimensional $\eps$-symmetric pair if and only if the morphism $\Sigma^{-1}f$ is in $\B_+(A)$.

Now, more specifically, suppose that $(C,\phi)$ is an $n$-dimensional $\eps$-symmetric complex and that there are complementary nullcobordisms $(f_\pm:\Sigma C\to D_\pm,\overline{S}(\delta_\pm\phi,\phi))$. Then the condition of being complementary gives $H_r(\Sigma C)\cong H_r(D_+)\oplus H_r(D_-)$ so that the desuspensions $\Sigma^{-1}D_\pm$ are in $\B_+(A)$. Hence $(f_\pm:C\to \Sigma^{-1} D_\pm,(\delta_\pm\phi,\phi))$ are complementary nullcobordisms of $(C,\phi)$.

\end{proof}

\begin{corollary}For $n\geq1$, the skew suspension gives a well-defined injective homomorphism\[\overline{S}:D\Gamma^{n}(A\to S^{-1}A,\eps)\to D\Gamma^{n+2}(A\to S^{-1}A,-\eps);\qquad [(P,\theta)]\mapsto [\overline{S}(P,\theta)].\]
\end{corollary}

\begin{proof}As there is a half unit in the ring $A$, for each $k\geq 0$, the skew suspension is an isomorphism \[\overline{S}:L^{k}(A,\eps)\xrightarrow{\cong} L^{k+2}(A,-\eps).\]The result now follows from the double $L$-theory localisation exact sequence and a diagram chase of the commutative diagram with exact rows\[\xymatrix{0\ar[r]&L^{n}(A,\eps)\ar[r]^-{Dj}\ar[d]_-{\cong}^-{\overline{S}}&D\Gamma^n(A\to S^{-1}A,\eps)\ar[r]^-{D\partial}\ar[d]^-{\overline{S}}&DL^n(A,S,\eps)\ar@{^{(}->}[d]^-{\overline{S}}&\\
0\ar[r]&L^{n+2}(A,-\eps)\ar[r]^-{Dj}&D\Gamma^{n+2}(A\to S^{-1}A,-\eps)\ar[r]^-{D\partial}&DL^{n+2}(A,S,-\eps)&}\]
\end{proof}

Various different standard algebraic surgeries are defined in the literature for the purpose of giving an inverse to the skew-suspension map in $L$-theory, and if your intention is to merely make a chain complex shorter, then any of these surgeries (used in a valid context) will give the desired result. To define an inverse to the skew-suspension map in double $L$-theory we have the additional complication that we must precisely control the chain homotopy type of the effect of two (potentially different) surgeries in order to achieve the same effect with each. Moreover, the cobordisms resulting from the two surgeries will need to be complementary. In order to get the necessary control, we will surger with data given by projective resolutions of the top- and bottom-dimensional homology modules respectively so that the effect of surgery in each case will (up to homology) have both the top and bottom missing, by Poincar\'{e} duality. For a `top' surgery and a `bottom' surgery to have the same effect on the chain level, we will also need to impose fairly strict dimensionality conditions on the underlying ring with involution.

\begin{theorem}[Surgery above and below the middle dimension: projective case]\label{thm:doubproj} Suppose that $A$ has h.d.\ 0 and the element $s\in A$ with $s+\overline{s}=1$ is moreover a unit in $A$. Then for each $n\geq 0$, the skew-suspension defines an isomorphism\[\overline{S}:DL^n(A,\eps)\xrightarrow{\cong}DL^{n+2}(A,-\eps).\]
\end{theorem}

\begin{proof}Because $A$ has h.d.\ 0, all $A$-modules are projective, so without loss of generality we may assume an $(n+2)$-dimensional $(-\eps)$-symmetric Poincar\'{e} complex $(C,\phi)$ over $A$ is of the form\[0\to C_{n+2}\xrightarrow{0} C_{n+1}\to\dots\to C_1\xrightarrow{0} C_0\to 0.\] Consequently, $\phi\in W^\%C_n$ is described entirely by the chain equivalence $\phi_0:C^{n+2-*}\xrightarrow{\simeq}C$ which moreover defines an isomorphism of projective $A$-modules $\phi_0:C^{n+2-r}\xrightarrow{\cong}C_r$ for each $r$.

Define \[(D_+)_r=\left\{\begin{array}{lcl}C_{n+2}&\quad&r=n+2,\\0&&\text{otherwise,}\end{array}\right. \]and the obvious morphism $f_+:C\to D_+$. By inspecting the differential in $C(f_+^\%)$, the element $(0,\phi)\in C(f_+^\%)_{n+3}$ is clearly seen to be a cycle. Hence there is an $(n+3)$-dimensional $(-\eps)$-symmetric pair $x_+:=(f_+:C\to D_+,(0,\phi))$. 

Define \[(D_-)_r=\left\{\begin{array}{lcl}C_0&\quad&r=0,\\0&&\text{otherwise,}\end{array}\right. \]and the obvious morphism $f_-:C\to D_-$. There is then an $(n+3)$-dimensional $(-\eps)$-symmetric Poincar\'{e} pair $x_-:=(f_-:C\to D_-,(0,\phi))$.

Surgery on $(C,\phi)$ with data $x_\pm$ has effect $(C'_\pm,\phi'_\pm)$ which is given, up to homotopy, by \[(C'_\pm)_r=\left\{\begin{array}{lcl}C_r&\quad&1\leq r\leq n,\\0&&\text{otherwise,}\end{array}\right. \]with all differentials equal to 0, in which description $\phi'$ is determined entirely by $\phi_0'$ which is given (according to the choice of structure in \cite[p.145]{MR560997}) by $\phi_s$$(\phi_\pm')_0=\phi_0:C^{n+2-r}\to C$ for $1\leq r\leq n+1$ and is 0 otherwise. So there exists a homotopy equivalence $h:(C_-',\phi_-')\xrightarrow{\simeq}(C_+',\phi_+')$. We write the respective cobordisms associated to the two surgeries (see Corollary \ref{corr:cobordism} and the subsequent remark for definition) as \[((g_\pm\,\,g'_\pm):C\oplus C_\pm'\to \tilde{D}_\pm,(0,\phi\oplus-\phi'_\pm))\] whose underlying morphisms are given on the level of homology by \[
(g_+\,\,g'_+)_*=\left\{\begin{array}{ccl}(1\,\, 1):H_r(C)\oplus H_r(C)\to H_{r}(C)&\quad&1\leq r\leq n,\\ 1:H_r(C)\to H_{r}(C)&&r=0\\0&&\text{otherwise}\end{array}\right.
\]\[
(g_-\,\,g'_-)_*=\left\{\begin{array}{ccl}(1\,\, 1):H_r(C)\oplus H_r(C)\to H_{r}(C)&\quad&1\leq r\leq n,\\ 1:H_r(C)\to H_{r}(C)&&r=n+2\\0&&\text{otherwise}\end{array}\right.
\]Define two cobordisms from $(C,\phi)$ to $(C'_\pm,\phi_\pm')$ by\[\begin{array}{lccrlr}((&g_+&g_+'h&):&C\oplus C'_-\to \tilde{D}_+,&(0,\phi\oplus-\phi'))\\((&\overline{s}g_-&-sg_-'&):&C\oplus C'_-\to \tilde{D}_-,&(0,\phi\oplus-\phi')).\end{array}\]Using the facts that $s$ is a unit in $A$ and that $s+\overline{s}=1$, it is easily checked that \[\left(\begin{matrix}g_+&g_+'h\\\overline{s}g_-&-sg'_-\end{matrix}\right):C\oplus C_-'\to \tilde{D}_+\oplus \tilde{D}_-\]is an isomorphism on the level of homology. Hence we have shown that $(C,\phi)$ is double-cobordant to the skew-suspension of the $n$-dimensional $\eps$-symmetric Poincar\'{e} complex given by $(\Sigma^{-1}C_-',\overline{S}^{-1}\phi_-')$. In other words, the skew suspension map $\overline{S}:DL^n(A,\eps)\to DL^{n+2}(A,-\eps)$ is surjective and therefore an isomorphism as was required to be shown.
\end{proof}

\begin{corollary}Under the hypotheses of Theorem \ref{thm:doubproj} and for $k\geq 0$ \[\begin{array}{rcl}DL^{2k+1}(A,\eps)&=&0,\\DL^{2k}(A,\eps)&\cong&W^{(-1)^k\eps}(A).\end{array}\]
\end{corollary}
\begin{proof}For any 1-dimensional $(-1)^k\eps$-symmetric Poincar\'{e} complex $(C,\phi)$, the pairs $(f_\pm:C\to D_\pm,(0,\phi))$ defined in the proof of Theorem \ref{thm:doubproj} are complementary nullcobordisms, proving that $DL^1(A,(-1)^k\eps)=0$. Now apply Theorem \ref{thm:doubproj}.

A double-nullcobordism of a 0-dimensional $(-1)^k\eps$-symmetric Poincar\'{e} complex $(C,\phi)$ is given by $(f_\pm:C\to D_\pm,(\delta_\pm\phi,\phi))$ and $H_r(C)\cong H_r(D_+)\oplus H_r(D_-)$, which vanishes except when $r=0$. Hence up to chain homotopy, $D_\pm$ are strictly 0-dimensional and correspond precisely to complementary lagrangians of the non-singular $(-1)^k\eps$-symmetric intersection form associated to $(C,\phi)$. Now apply Theorem \ref{thm:doubproj} and Claim \ref{clm:DWvsW}.
\end{proof}

\begin{remark}This is the best result we could hope for in the projective $DL$-groups using these particular surgeries `above and below the middle dimension' that work by surgering resolutions of $H_{n+1}(C)$ and $H_0(C)$ respectively. If we hypothesise that $A$ has h.d.\ 1, for example, the two surgeries described in the proof will both still be valid but we no longer expect the respective effects of surgery to be homotopic to one another (different $\Ext$ terms appear in the respective effects of surgery). Of course, this does not preclude some other strategy from producing surgeries above and below the middle dimension! At present we have no counterexample to the statement that the skew-suspension maps are always isomorphisms in the projective double $L$-groups.
\end{remark}

The hypothesis in Theorem \ref{thm:doubproj} that the homological dimension is 0 is a strict one. However, there is a slogan that what can be done for projective complexes over a ring with h.d.\ $m$ can be done for $S$-acyclic complexes over a ring with h.d.\ $m+1$. In this spirit, we now obtain a similar result for $S$-acyclic complexes. Sadly, this result will not be strong enough to invert the skew-suspension maps for the corresponding double $L$-groups.

\begin{proposition}[Surgery above and below the middle dimension: torsion case up to homology]\label{prop:doubtors} Suppose the element $s\in A$ with $s+\overline{s}=1$ is a unit in $A$ and that $(A,S)$ has dimension 0. Then for each $n\geq 2$ and each $(n+1)$-dimensional $S$-acyclic $(-\eps)$-symmetric Poincar\'{e} complex $(C,\phi)$ over $A$, there exist two $(A,S)$ cobordisms \[((g_\pm\,\,g'_\pm):C\oplus C_\pm'\to \tilde{D}_\pm,(0,\phi\oplus-\phi_\pm'))\]such that $(C'_\pm,\phi'_\pm)$ are skew-suspensions, $H_*(C'_+)= H_*(C'_-)$ and for each $r=0,\dots,n+1$ \[\left(\begin{matrix}g_+&g_+'\\\overline{s}g_-&-sg'_-\end{matrix}\right):H_r(C)\oplus H_r(C'_\pm)\xrightarrow{\cong} H_r(\tilde{D}_+)\oplus H_r(\tilde{D}_-).\]
\end{proposition}

\begin{proof}Because $(A,S)$ has dimension 0, we may assume without loss of generality that an $(n+1)$-dimensional $\eps$-symmetric $S$-acyclic Poincar\'{e} complex $(C,\phi)$ over $A$ is of the form \[0\to C_{n+1}\xrightarrow{d} C_{n}\to\dots\to C_1\xrightarrow{d} C_0\to 0.\]Because $C$ is $S$-acyclic we have also that $H_{n+1}(C)=0$.

Define $f_+:C\to D_+$ to be a morphism in $\C_+(A,S)$ resolving the identity morphism $1:H_{n}(C)\to H_{n}(C)$ by a chain map  from $C$ to a complex of f.g.\ projective $A$-modules $D_+$ of the form\[0\to (D_+)_{n+1}\xrightarrow{d_+} (D_+)_{n}\to 0\to \dots\to 0\to 0.\]Define $f_-:C\to D_-$ to be a morphism in $\C_+(A,S)$ resolving the identity morphism $1:H_0(C)\to H_0(C)$ by a chain map from $C$ to a complex of f.g.\ projective $A$-modules $D_-$ of the form\[0\to0\to\dots\to 0\to (D_-)_1\xrightarrow{d_-} (D_-)_{0}\to 0.\]As in the proof of Theorem \ref{thm:doubproj} there are defined $(n+2)$-dimensional $\eps$-symmetric pairs $x_\pm:=(f_\pm:C\to D_\pm,(0,\phi))$. We write the effect of surgery on $(C,\phi)$ with data $x_\pm$ as $(C'_\pm,\phi'_\pm)$. We claim that the homology of the effect is given by\[H_r(C'_\pm)=\left\{\begin{array}{lcl}H_r(C)&\quad&1\leq r\leq n-2,\\0&&\text{otherwise,}\end{array}\right. \]We prove this claim for $C'_-$ (the case of $C_+'$ being entirely similar). Consider part of the long exact sequence in homology coming from the definition of the surgery\[ 0\to \underset{=0}{\underbrace{H_{n+1}(C_-')}}\to H_{n+1}(D_-^{n+3-*})\xrightarrow{\phi_0f_-^*} H_{n+1}(f_-)\to H_{n}(C'_-)\to 0\]But the Universal Coefficient Theorem states in this context that there is a short exact sequence\[0\to \underset{=H_{n+1}(D^{n+3-*}_-)}{\underbrace{\Ext^1_A(H_0(C),A)}}\overset{\beta}{\longrightarrow} H^1(C)\longrightarrow \underset{=0}{\underbrace{H_1(C)^*}}\to 0\]and we may identify $f_-^*$ with the isomorphism $\beta$ so that $H_{n-1}(C)=0$. The rest of the claim follows easily by considering the rest of the long exact homology sequence.

Exactly as in the proof of Theorem \ref{thm:doubproj}, we can now write the respective cobordisms associated to the two surgeries (according to Corollary \ref{corr:cobordism} and the subsequent remark) as \[((g_\pm\,\,g'_\pm):C\oplus C_\pm'\to \tilde{D}_\pm,(0,\phi\oplus-\phi'_\pm))\]and we see that the maps induced on homology are\[
(g_+\,\,g'_+)_*=\left\{\begin{array}{ccl}(1\,\, 1):H_r(C)\oplus H_r(C)\to H_{r}(C)&\quad&1\leq r\leq n-1,\\ 1:H_r(C)\to H_{r}(C)&&r=n\\0&&\text{otherwise}\end{array}\right.
\]\[
(g_-\,\,g'_-)_*=\left\{\begin{array}{ccl}(1\,\, 1):H_r(C)\oplus H_r(C)\to H_{r}(C)&\quad&1\leq r\leq n-1,\\ 1:H_r(C)\to H_{r}(C)&&r=0\\0&&\text{otherwise}\end{array}\right.
\]Proceeding as in Theorem \ref{thm:doubproj}, the rest of the proposition follows.
\end{proof}

Clearly, a desirable fact that is missing from the previous proposition is a proof that the effects $(C_\pm',\phi'_\pm)$ are homotopic to one another. As the complexes $C_-'$ and $C'_+$ were shown to be quasi-isomorphic, and the symmetric structures $\phi_\pm'$ are described by (degree 0) chain maps that induce the same map on homology, one might hope there exists a chain-level equivalence between these maps and hence a chain homotopy equivalence between $(C_+',\phi_+')$ and $(C_-',\phi_-')$. However, such a general result does not exist.

We stress that we have no counter examples to the statement that `all double $L$-groups are 4-periodic'. But we conjecture that it is not true in general and that for each $n\geq2$, there is an obstruction to inverting the skew-suspension map\[\overline{S}:DL^n(A,S,\eps)\hookrightarrow DL^{n+2}(A,S,-\eps).\]

\section{Comparison with double Witt groups}\label{sec:comparison}

Our general philosophy of $\eps$-symmetric forms is that they are best understood as short chain complexes with extra structure.

\subsection*{Double Witt groups of forms and Seifert forms}

The following is evidently true.

\begin{proposition}\label{prop:correspondenceproj}There are contravariant 1:1 correspondences that preserve the respective monoid structures.

\[\begin{array}{rcl}\left\{  \begin{array}{c}\text{0-dimensional, $\eps$-symmetric} \\ \text{(Poincar\'{e}) complexes over $R$}\end{array} \right\}_{\text{\big{/htpy.}}}&\longleftrightarrow &
\left\{  \begin{array}{c}\text{(non-singular) $\eps$-symmetric} \\ \text{ forms over $R$}\end{array} \right\}_{\text{\big{/iso}}} \\ &&\\ (C,\phi)&\mapsto& (H^0(C),\phi_0).\end{array}\]

and

\[\begin{array}{rcl}\left\{  \begin{array}{c}\text{0-dimensional, $\eps$-symmetric} \\ \text{Seifert complexes over $R$}\end{array} \right\}_{\text{\big{/htpy.}}}&\longleftrightarrow &
\left\{  \begin{array}{c}\text{non-singular $\eps$-symmetric} \\ \text{Seifert forms over $R$}\end{array} \right\}_{\text{\big{/iso}}} \\ &&\\ (C,\hat{\psi})&\mapsto& (H^0(C),\hat{\psi}).\end{array}\]
\end{proposition}

Here is a precise characterisation of metabolic Seifert forms considered as chain complexes.

\begin{proposition}An $\eps$-symmetric Seifert form over $R$ admits a lagrangian if and only if the associated $\eps$-symmetric homotopy equivalence class of 0-dimensional $\eps$-symmetric Seifert complexes over $R$ contains $(C,\hat{\psi})$ such that there is a Seifert nullcobordism $(f:C\to D,(\delta\hat{\psi},\hat{\psi}))$ with $H^1(D)=0$.
\end{proposition}

Hence a precise characterisation of hyperbolic Seifert forms considered as chain complexes

\begin{proposition}\label{hypcxseif}An $\eps$-symmetric Seifert form over $R$ is hyperbolic if and only if the associated $\eps$-symmetric homotopy equivalence class of 0-dimensional $\eps$-symmetric Seifert complexes over $R$ contains $(C,\hat{\psi})$ such that there are two complementary Seifert nullcobordisms $(f_\pm:C\to D_\pm,(\delta\hat{\psi},\hat{\psi}))$ with $H^1(D_\pm)=0$.
\end{proposition}

\begin{proposition}\label{prop:isogroups}There is an isomorphism of groups\[
\widehat{DW}^\eps(R)\xrightarrow{\cong} \widehat{DL}^0(R,\eps).\]
\end{proposition}

\begin{proof}Clearly well-defined and surjective by the previous 3 propositions. So show injective, suppose that if $(C,\hat{\psi})$ is a $\eps$-symmetric, 0-dimensional Poincar\'{e} complex associated to the non-singular, $\eps$-symmetric Seifert form $(K,\alpha)$. If there exists a pair of complementary Seifert nullcobordisms $(f_\pm:C\to D_\pm,(\delta_\pm\hat{\psi},\hat{\psi}))$ then in particular $0=H^1(C)=H^1(D_+)\oplus H^1(D_-)$ so that $H^1(D_\pm)=0$ but then $(K,\alpha)$ must be hyperbolic by Proposition \ref{hypcxseif}.
\end{proof}

\begin{definition}A non-singular $\eps$-symmetric Seifert form $(K,\alpha)$ over $R$ is called \textit{stably-hyperbolic} if there exist hyperbolic $\eps$-symmetric Seifert forms $H,H'$ such that $(K,\alpha)\oplus H\cong H'$.
\end{definition}

Note that \textit{a priori} the stably-hyperbolic Seifert forms are precisely the representatives of the 0 class in the double Witt group of Seifert forms. However, we obtain the following characterisation as a corollary of Proposition \ref{prop:isogroups}.

\begin{corollary}[`Stably hyperbolic $=$ hyperbolic']\label{stablyhypishypseif}A non-singular $\eps$-symmetric Seifert form $(K,\alpha)$ over $R$ is hyperbolic if and only if it is stably hyperbolic.
\end{corollary}

\begin{proof}``Only if'' is clear. Conversely, a stably hyperbolic $\eps$-symmetric Seifert form determines the 0 class in $\widehat{DW}^\eps(R)\cong \widehat{DL}^0(R,\eps)$. Therefore there is a Seifert double-nullcobordism of the corresponding 0-dimensional $\eps$-symmetric chain complex. But by the proof of the previous proposition, these nullcobordisms correspond to complementary lagrangians.
\end{proof}

\subsection*{Double Witt groups of linking forms}

\begin{proposition}[{\cite[3.4.1]{MR620795}}]\label{prop:correspondence}The following is a contravariant 1:1 correspondence that preserves the respective monoid structures.
\[\begin{array}{rcl}\left\{  \begin{array}{c}\text{1-dimensional, $(-\eps)$-symmetric} \\ \text{$S$-acyclic (Poincar\'{e}) complexes over $A$}\end{array} \right\}_{\text{\big{/htpy.}}}&\longleftrightarrow &
\left\{  \begin{array}{c}\text{(non-singular) $\eps$-symmetric} \\ \text{linking forms over $(A,S)$}\end{array} \right\}_{\text{\big{/iso}}} \\ &&\\ (C,\phi)&\mapsto& (H^1(C),\lambda_\phi), \end{array}\]where $\lambda_\phi([x],[y])=s^{-1}\phi_0(x,z)$ for $x,y\in C^1$, $z\in C^0$ and $s\in S$ such that $d^*z=sy$.
\end{proposition}

Here is a precise characterisation of metabolic forms considered as chain complexes:

\begin{proposition}[{\cite[3.4.5(ii)]{MR620795}}]\label{prop:3.4.5}A non-singular, $\eps$-symmetric linking form over $(A,S)$ admits a lagrangian if and only if the associated $(-\eps)$-symmetric homotopy equivalence class of 1-dimensional $(-\eps)$-symmetric Poincar\'{e} complexes over $A$ contains $(C,\phi)$ such that there is an $S$-acyclic 2-dimensional $(-\eps)$-symmetric Poincar\'{e} pair $(f:C\to D,(\delta\phi,\phi))$ with $H^2(D)=0$.
\end{proposition}

An easy consequence of this is a precise characterisation of hyperbolic forms considered as chain complexes:

\begin{proposition}\label{hypcx}A non-singular, $\eps$-symmetric linking form over $(A,S)$ is hyperbolic if and only if the associated $(-\eps)$-symmetric homotopy equivalence class of 1-dimensional $(-\eps)$-symmetric Poincar\'{e} complexes over $A$ contains $(C,\phi)$ such that there are two complementary $S$-acyclic 2-dimensional $(-\eps)$-symmetric pairs $(f_\pm:C\to D_\pm,(\delta_\pm\phi,\phi))$ with $H^2(D_\pm)=0$.
\end{proposition}

\begin{proposition}\label{iso}There is an isomorphism of groups\[DW^\eps(A,S)\xrightarrow{\cong} DL^0(A,S,\eps).\]
\end{proposition}
\begin{proof}As in Proposition \ref{prop:isogroups}.
\end{proof}

\begin{definition}A non-singular $\eps$-symmetric linking form $(T,\lambda)$ over $(A,S)$ is called \textit{stably-hyperbolic} if there exist hyperbolic $\eps$-symmetric linking forms $H,H'$ such that $(T,\lambda)\oplus H\cong H'$.
\end{definition}

\begin{corollary}[`Stably hyperbolic $=$ hyperbolic']\label{stablyhypishyp}Suppose there exists $s\in A$ such that $s+\bar{s}=1$. A non-singular $\eps$-symmetric linking form $(T,\lambda)$ over $(A,S)$ is hyperbolic if and only if it is stably hyperbolic.
\end{corollary}

\begin{proof}As in Corollary \ref{stablyhypishypseif}.
\end{proof}

\begin{remark}It is entertaining to see a proof of \ref{stablyhypishyp} entirely in the context of linking forms. This is just a formal exercise in tracing the $DL$ proof through the isomorphism of Claim \ref{iso} and we leave it to an interested reader.

\end{remark}

For any $(A,S)$, there is an isomorphism $W^\eps(A,S)\cong L^0(A,S,\eps)$ (\cite[3.4.7(ii)]{MR620795}), but it is not sufficient to prove that stably metabolic implies metabolic for linking forms in general. The reason is that an $(A,S)$-nullcobordism $(f:C\to D,(\delta\phi,\phi))$ of a 1-dimensional $(-\eps)$-symmetric $S$-acyclic Poincar\'{e} complex over $A$ might have $H^2(D)\neq0$, so that the corresponding $\eps$-symmetric linking form need not necessarily admit a lagrangian (cf.\ the torsion version of \cite[4.6]{MR560997}).

\subsection*{The linking form of a symmetric Poincar\'{e} complex}

In general just taking the middle-dimensional pairing that one hopes is the linking form of a symmetric Poincar\'{e} complex $(C,\phi)$ does not give a well-defined linking form as we saw in \ref{subsec:linkPD}. There are two issues: the cohomology modules not necessarily having length 1 resolutions and the linking form might not pair modules via the universal coefficient problem. We now make clear some circumstances in which taking the middle-dimensional pairing of a symmetric Poincar\'{e} complex $(C,\phi)$ is a valid operation.

\begin{proposition}\label{prop:welldeflagrang}Suppose $(A,S)$ has dimension 0 and $(C,\phi)$ is a $(2k+3)$-dimensional $\eps$-symmetric $S$-acyclic Poincar\'{e} complex over $A$. Then \[\lambda_\phi:H^{k+2}(C)\times H^{k+2}(C)\to S^{-1}A/A;\qquad ([x],[y])\mapsto s^{-1}\overline{\tilde{y}(\phi_0(x))},\]with $x,y\in C^{k+2}$, $\tilde{y}\in C^{k+1}$, and $s\in S$ such that $d^*\tilde{y}=sy$, is a well-defined, non-singular, $(-1)^{k}\eps$-symmetric linking form. Moreover:\begin{enumerate}[(i)]
\item If $(C,\phi)$ is $(A,S)$-nullcobordant then $(H^{k+2}(C),\lambda_\phi)$ is metabolic.
\item If $(C,\phi)$ is $(A,S)$-double-nullcobordant then $(H^{k+2}(C),\lambda_\phi)$ is hyperbolic.
\end{enumerate}
\end{proposition}

\begin{proof}The first part is standard; the chain complex $C$ has homology satisfying (H1) of Proposition \ref{prop:UCSS} so by the same argument the linking form $(H^{k+2}(C),\lambda_\phi)$ is well-defined and non-singular. The $(-1)^{k}$-symmetry must follow slightly differently than in Chapter \ref{chap:laurent} as we are no longer defining the pairing via cup-product so we cannot use its derivation properties. But the result can be recovered from a chain-level calculation requiring the higher chain homotopy $\phi_1$, see for instance the chain-level calculations in \cite[pp153]{MR2914616}.

For (i), suppose $(g:C\to D,(\delta\phi,\phi))$ is an $(A,S)$-nullcobordism of $(C,\phi)$. Write the functor $e^1(-)=\Ext^1_A(-,A)$ for brevity. Then there is a commutative diagram with exact rows \[\xymatrix{H^{k+2}(D;A)\ar[r]^-{g^*}&H^{k+2}(C;A)\ar[r]&H^{k+3}(D,C;A))\\
e^1(H_{k+1}(D;A))\ar[r]\ar[u]_-{\cong}&e^1(H_{k+1}(C;A))\ar[r]\ar[u]_-{\cong}&e^1(H_{k+2}(D,C;A))\ar[u]_-{\cong}\\ 
e^1(H^{k+3}(D,C;A))\ar[u]^-{\sm \ev_l(\delta\phi,\phi)}_-{\cong}\ar[r]&e^1(H^{k+2}(C;A))\ar[u]^-{\phi_0}_-{\cong}\ar[r]^-{e^1(g^*)}&e^1(H^{k+2}(D;A))\ar[u]^-{\sm \ev_r(\delta\phi,\phi)}_-{\cong}}\]The central column determines the linking form $(H^{k+2}(C;A),\lambda_\phi)$ adjointly. Hence the inclusion of the image $j:g^*(H^{k+2}(D;A))\hookrightarrow H^{k+2}(C;A)$ is a lagrangian submodule as the commutative diagram determines an exact sequence\[0\to g^*(H^{k+2}(D;A))\xrightarrow{j} H^{k+2}(C;A)\xrightarrow{j^\wedge \lambda_\phi}g^*(H^{k+2}(D;A))^\wedge \to 0.\]

For (ii), suppose $(f_\pm:C\to D_\pm,(\delta_\pm\phi,\phi))$ is an $(A,S)$-double-nullcobordism of $(C,\phi)$.  By the above we have that $(f_+^*\,\,\,\,f_-^*):H^{k+2}(D_+)\oplus H^{k+2}(D_+)\cong H^{k+2}(C)$ is now a direct sum decomposition by complementary lagrangians.
\end{proof}

We prove a more delicate version of the above for the knot-theoretic localisation $(A,S)=(\Z[z,z^{-1}],P)$ with $P$ the set of Alexander polynomials.

\begin{proposition}\label{prop:welldeflagrang2}Suppose $(A,S)=(\Z[z,z^{-1}],P)$ and $(C,\phi)$ is a $(2k+3)$-dimensional $\eps$-symmetric $S$-acyclic Poincar\'{e} complex over $A$. Then the Blanchfield form of Theorem \ref{thm:levblanch}:\[\lambda_\phi:f(H^{k+2}(C))\times f(H^{k+2}(C))\to P^{-1}\Z[z,z^{-1}]/\Z[z,z^{-1}];\qquad ([x],[y])\mapsto p^{-1}\overline{\tilde{y}(\phi_0(x))},\]with $x,y\in C^{k+2}$, $\tilde{y}\in C^{k+1}$, and $p\in P$ such that $d^*\tilde{y}=sy$, is a well-defined, non-singular, $(-1)^{k}\eps$-symmetric linking form. Moreover: \begin{enumerate}[(i)]
\item If $(C,\phi)$ is $(A,S)$-nullcobordant then $(f(H^{k+2}(C)),\lambda_\phi)$ is metabolic.
\item If $(C,\phi)$ is $(A,S)$-double-nullcobordant then $(f(H^{k+2}(C)),\lambda_\phi)$ is hyperbolic.
\end{enumerate}
\end{proposition}

\begin{proof}The Blanchfield form is well-defined and non-singular by Theorem \ref{thm:levblanch}. As the chain-level formula is identical to that of the linking form in \ref{prop:welldeflagrang}, the $(-1)^{k}$-symmetry follows from the same calculations as in that proof.

For (i), suppose $(g:C\to D,(\delta\phi,\phi))$ is an $(A,S)$-nullcobordism of $(C,\phi)$. Then we must appeal to the results of \cite{MR0461518} described in \ref{subsec:blanchdual}. Write $e^1(-)=\Ext^1_A(-,A)$ for brevity. As the chain complexes $C$, $D$ and $C(f)$ are all $S^{-1}A$-acyclic, we obtain the following commutative diagram with exact rows from the discussion in \ref{subsec:blanchdual} \[\xymatrix{f(H^{k+2}(D;A))\ar[r]\ar[d]^-{\cong}&f(H^{k+2}(C;A))\ar[r]\ar[d]^-{\cong}&f(H^{k+3}(D,C;A))\ar[d]^-{\cong}\\
e^1(f(H_{k+1}(D;A)))\ar[r]&e^1(f(H_{k+1}(C;A)))\ar[r]&e^1(f(H_{k+2}(D,C;A)))\\ 
e^1(f(H^{k+1}(D,C;A)))\ar[u]^-{\sm \ev_l(\delta\phi,\phi)}_-{\cong}\ar[r]&e^1(f(H^{k+2}(C;A)))\ar[u]^-{\phi_0}_-{\cong}\ar[r]&e^1(f(H^{k+2}(D;A)))\ar[u]^-{\sm \ev_r(\delta\phi,\phi)}_-{\cong}}\]As in Proposition \ref{prop:welldeflagrang}, the image of \[g^*:f(H^{k+2}(D;A))\to f(H^{k+2}(C;A))\]is a lagrangian submodule.

For (ii), suppose $(f_\pm:C\to D_\pm,(\delta_\pm\phi,\phi))$ is an $(A,S)$-double-nullcobordism of $(C,\phi)$.  By the above we have that $(f_+^*\,\,\,\,f_-^*):H^{k+2}(D_+)\oplus H^{k+2}(D_+)\cong H^{k+2}(C)$ is now a direct sum decomposition by complementary lagrangians.
\end{proof}

\begin{corollary}Suppose $(A,S)$ has dimension 0 (resp.\ $(A,S)=(\Z[z,z^{-1}],P)$) and that $(T,\lambda)$ is a non-singular, $\eps$-symmetric linking form over $(A,S)$. If $(T,\lambda)$ is stably metabolic then it is metabolic. If $(T,\lambda)$ is stably hyperbolic, then it is hyperbolic.
\end{corollary}

\begin{proof}Under the correspondence of Proposition \ref{prop:correspondence}, $(T,\lambda)$ goes to a 1-dimensional $(-\eps)$-symmetric $S$-acyclic Poincar\'{e} complex $(C,\phi)$ over $A$. If $(T,\lambda)$ is stably metabolic, there exists an $(A,S)$-nullcobordism $(f:C\to D,(\delta\phi,\phi))$ by the isomorphism $W^\eps(A,S)\cong L^2(A,S,-\eps)$. But then by Proposition \ref{prop:welldeflagrang} (resp. Proposition \ref{prop:welldeflagrang2}), $(H^1(C),\lambda_\phi)=(T,\lambda)$ is metabolic.

If $(T,\lambda)$ is stably hyperbolic we use the isomorphism $DW^\eps(A,S)\cong DL^0(A,S,\eps)$ and follow the same proof.
\end{proof}

The following corollary is also evident.

\begin{corollary}If $(A,S)$ has dimension 0 (resp.\ $(A,S)=(\Z[z,z^{-1}],P)$) then for $k\geq0$ the correspondence of Proposition \ref{prop:correspondence} defines a surjective homomorphism\[DL^{2k+2}(A,S,(-1)^{k+1}\eps)\twoheadrightarrow DW^\eps(A,S),\]with right inverse given by the isomorphism $DW^\eps(A,S)\cong DL^0(A,S,\eps)$ followed by the $(k+1)$-fold skew-suspension $\overline{S}^{k+1}$.
\end{corollary}

\subsection*{Double Witt $\Gamma$-groups}

We now describe how to reconcile the groups $D\Gamma^\eps(A\to S^{-1}A)$ of Chapter \ref{chap:linking} with our symmetric chain complex groups $D\Gamma^n(A\to S^{-1}A,\eps)$.

Suppose $(K,\alpha)$ is an $S$-non-singular $\eps$-symmetric form over $A$ and that $j:L\hookrightarrow K$ is the inclusion of a split $S$-lagrangian. Then there is a 3-dimensional $(-\eps)$-symmetric $S$-Poincar\'{e} pair $(f:P\to Q,(0,\theta))$ over $A$ given by \[P_r=\left\{\begin{array}{lcl}K^*&&r=1,\\0&&r\neq 1,\end{array}\right.\qquad Q_r=\left\{\begin{array}{lcl}L^*&&r=1,\\0&&r\neq 1\end{array}\right.\]with \[\theta_0:P^1=K\xrightarrow{\alpha}K^*=P_1,\qquad f:P_1=K^*\xrightarrow{j^*} L^*=Q_1.\]

\begin{figure}[h]
\def\picrelboundary{\resizebox{0.3\textwidth}{!}{ \includegraphics{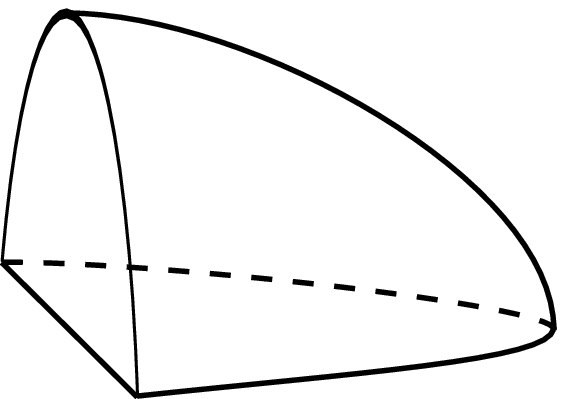}}}

\[\begin{xy} \xyimport(219.73,159.08){\picrelboundary}
,(25,85)*!L{P}
,(100,85)*!L{Q}
,(0,20)*!L{\partial P}
,(80,28)*!L{\partial (Q,P)}

\end{xy}\]
\end{figure}

The associated 2-dimensional  $(-\eps)$-symmetric Poincar\'{e} pair $(\partial f:\partial P\to \partial(Q,P),(\overline{\theta'},\partial\theta'))$ (see \ref{def:relboundary}) is given by\[\xymatrix{
\partial P\ar[d]^-{\partial f}    &    &0\ar[r]     &0\ar[d]\ar[r]     &K\ar[d]^-{1}\ar[r]^-{\alpha}    &K^*\ar[r]\ar[d]^-{j^*}    &0\\
\partial(Q,P)    &     &0\ar[r]     &L\ar[r]^-{j}    &K\ar[r]^-{j^*\alpha}     &L^*\ar[r]   &0}\]with $\partial\theta'$ entirely described by\[\xymatrix{
(\partial P)^{1-*}\ar[d]^-{\partial\theta'_0}&&0\ar[r]&K\ar[r]^-{\alpha^*}\ar[d]^-{\eps}&K^*\ar[r]\ar[d]^-{1}&0\\
\partial P&&0\ar[r]&K\ar[r]^-{\alpha}&K^*\ar[r]&0}\]and $\overline{\theta'}=0$.

\begin{remark} We now have an interesting explanation of why the proof of Proposition \ref{prop:boundarywelldef} was so complicated. It would have been altogether neater if the boundary of a split $S$-metabolic $S$-non-singular form $(K,\alpha)$ over $A$ was itself a metabolic linking form. As it was, we were able to show that it was \textit{stably} metabolic. As there is no guarantee that $H^2(\partial(Q,P))=L^*/j^*(K^*)$ vanishes, then by Proposition \ref{prop:3.4.5}, there is no guarantee that the obvious submodule $H^1(\partial(Q,P))\hookrightarrow H^1(\partial P)$ is a lagrangian. Hence for a general split $S$-metabolic $(K,\alpha)$, Proposition \ref{prop:boundarywelldef} is a sharp result.
\end{remark}

Now consider $\partial$-complementary split $S$-lagrangians $j_\pm:L_\pm\hookrightarrow K$. We wish to show that the associated $S^{-1}A$-nullcobordisms $(f_\pm:P_\pm\to Q_\pm,(0,\partial\theta'))$ are $\partial$-complementary. We must check there is a chain homotopy equivalence
\[\left(\begin{matrix}\partial f_+\\\partial f_-\end{matrix}\right):\partial P\to \partial(Q_+,P)\oplus \partial_-(Q_-,P)\]given by \[\xymatrix{
0\ar[r]     &0\ar[d]\ar[rr]     &&K\ar[d]^-{1}\ar[rr]^-{\alpha}    &&K^*\ar[r]\ar[d]^-{j^*}    &0\\
0\ar[r]     &L_+\oplus L_-\ar[rr]^-{j_+\oplus j_-}    &&K\oplus K\ar[rr]^-{(j_+^*\alpha)\oplus (j_-^*\alpha)}     &&L_+^*\oplus L_-^*\ar[r]   &0}\] But the cone $C\left(\lmat\partial f_+\\\partial f_-\rmat\right)$ is chain homotopic to\[\xymatrix{0\ar[r]&L_+\oplus L_-\ar[rr]^-{\left(\begin{smallmatrix}j_+&j_-\\0&\alpha j_-\end{smallmatrix}\right)}&&K\oplus K^*\ar[rr]^-{\left(\begin{smallmatrix}-j_+^*\alpha&j_+^*\\0&j_-^*\end{smallmatrix}\right)}&&L_+^*\oplus L_-^*\ar[r]&0}\]and this sequence is exact by the definition of $\partial$-complementary $S$-lagrangians. We have just shown:

\begin{proposition}There is a well-defined surjective group homomorphism\[G:D\Gamma^\eps(A\to S^{-1}A)\twoheadrightarrow D\Gamma^0(A\to S^{-1}A,\eps).\]
\end{proposition}

\begin{definition}\label{def:reduced}The \textit{reduced} $\eps$-symmetric double $\Gamma$ Witt group of a localisation $(A,S)$ is the group\[\widetilde{D\Gamma}^\eps(A\to S^{-1}A):=D\Gamma^\eps(A\to S^{-1}A)/\ker(G).\]
\end{definition}

\begin{remark}We have been unable to show in general whether the subgroup $\widetilde{D\Gamma}^\eps(A\to S^{-1}A)$ is a proper subgroup of $D\Gamma^\eps(A\to S^{-1}A)$.\end{remark}

\subsection{Double Witt group localisation exact sequence}\label{subsec:DWproof}

We can now finally return to the double Witt group localisation exact sequence of \ref{subsec:DWLES} and prove Theorem \ref{thm:DWLES}.

\begin{proof}[Proof (of Theorem \ref{thm:DWLES})]Consider the isomorphisms determined by the realisation of the forms and linking forms as chain complexes with symmetric structure\[\xymatrix{0\ar[r]&L^0(A,\eps)\ar[r]^-{Di}&D\Gamma^0(A\to S^{-1}A,-\eps)\ar[r]^-{D\partial}&DL^0(A,S,\eps)\\
&W^{\eps}(A)\ar[r]^-{Di}\ar[u]^-{\cong}&\widetilde{D\Gamma}^\eps(A\to S^{-1}A)\ar[r]^-{D\partial}\ar[u]^-{\cong}&DW^\eps(A,S)\ar[u]^-{\cong}}\]and it is straightforward to check that the maps on the bottom row defined in \ref{subsec:DWLES} make the diagram commute.
\end{proof}

\chapter{The Blanchfield complex}\label{chap:blanchfield}

In this chapter we prepare for our main geometric application of double $L$-theory by reviewing knot-cobordism from the perspective of the Algebraic Theory of Surgery. We first define the \textit{symmetric Blanchfield complex} $(C_K,\phi_K))$ of an $n$-knot $K$, an algebraic measure of the difference between $K$ and the unknot. As motivation for our approach to the `doubly-slice' problem in the subsequent chapter, we then recall Kervaire and Levine's knot-cobordism classification of knots for dimension $n\geq2$ using the language of algebraic $L$-theory.

\section{High-dimensional knots}

A \textit{topological $n$-knot}, hereafter called a \textit{knot} unless $n$ is to be specified, is an ambient isotopy class of oriented, locally flat embeddings $K:S^n\hookrightarrow S^{n+2}$ (where all spheres are considered to have a preferred orientation already). In a standard abuse of notation we will also use the word knot to mean a particular $K$ in an ambient isotopy class and the image of $K$ in $S^{n+2}$. The \textit{unknot} is the ambient isotopy class of $U:S^n\hookrightarrow S^{n+2}$, the standard unknotted $n$-sphere in the unit sphere $S^{n+2}\subset\R^{n+3}$ given by setting the last two co-ordinates to 0. The \textit{inverse knot} $-K$ of a knot $K$ is given by reversing the orientation on a mirror image of $K$ in $S^{n+2}$. Any embedding $K:S^{n}\hookrightarrow S^{n+2}$ has trivial normal bundle (2-plane bundles over spheres are trivial for $n>1$, and when $n=1$ consider that the normal bundle is oriented) and hence, by choosing a framing, we may excise a small, trivial tubular neighbourhood of the knot from $S^{n+2}$. Thus, the \textit{knot exterior} is the manifold with boundary \[(X_K,\partial X_K):=(\cl(S^{n+2}\sm (K(S^n)\times D^2)),S^n\times S^1)\]which has a preferred orientation coming from the ambient $S^{n+2}$. The knot exterior $X_K$ is homotopy equivalent to the \textit{knot complement} $S^{n+2}\sm K$ and hence has the homology of a circle $H_*(X_K)= H_*(S^1)$ by Alexander duality. The following is a key justification for studying the knot exterior rather than the knot itself:

\begin{theorem}For $n=1$ the ambient isotopy class of an $n$-knot $K$ is uniquely determined by specifying a knot exterior $X_K$ up to homeomorphism (\cite{MR972070}). For $n\geq 3$ there can be at most two inequivalent $n$-knots $K$ and $K'$ such that $X_K\cong X_{K'}$ (\cite{MR0242175}).
\end{theorem}

\begin{remark}These so-called `non-reflexive' pairs of inequivalent knots with homeomorphic exterior are known to exist in dimensions $n=$ 2 \cite{zbMATH03540077}, 3, 4, 5 \cite{MR0413117} and $n\equiv$ 3, 4 modulo 8 \cite{zbMATH00270255}. However, as we shall see in Section \ref{sec:knotcob}, non-reflexive pairs belong to the same `knot-cobordism' class.
\end{remark}

If there is a locally flat embedding of the manifold with boundary $(F^{n+1},S^n)\hookrightarrow S^{n+2}$ then we say the embedded $F$ is a \textit{Seifert surface} for the boundary knot. It was shown by many authors independently that every knot $K$ admits a Seifert surface $F^{n+1}$ (e.g. \cite{MR0189052}, \cite{MR0160218}). For $n\neq 2$, the unknot is characterised as the only knot which admits $D^{n+1}$ as a Seifert surface (this is due to Levine \cite{MR0179803} \cite{MR0246314}).

It is always possible to `push' a Seifert surface into the standard $D^{n+3}$ that cobounds the ambient sphere $S^{n+2}$. That is, we may modify a locally flat embedding $(F^{n+1},S^{n})\hookrightarrow S^{n+2}$ to a locally flat embedding of pairs $(F^{n+1},S^{n})\hookrightarrow (D^{n+3},S^{n+2})$, without changing the ambient isotopy class of the bounding knot $K$, and so that the embedded $F$ intersects $S^{n+2}$ in the knot $K$.

\begin{proposition}[{\cite[22.1]{MR1713074}}]If $(M,\partial M)$ is an $n$-dimensional manifold with boundary, the following sets are in natural 1:1 correspondence:
\begin{enumerate}[(i)]
\item the cohomology group $H^2(M)$,
\item the homology group $H_{n-2}(M,\partial M)$,
\item the set $[M,BSO(2)]$ of isomorphism classes of real orientable 2-plane bundles over $M$,
\item the ambient cobordism classes of codimension 2 submanifold pairs $(N,\partial N)\subset (M,\partial M)$.
\end{enumerate}
\end{proposition}

As a result of this proposition, any codimension 2 submanifold pair $(N,\partial N)\subset (D^{n+3},S^{n+2})$ has trivial normal bundle, and hence by choosing a  framing we may embed $(N,\partial N)\times D^2\subset (D^{n+3},S^{n+2})$. Define the \textit{exterior} of such a submanifold pair (with respect to a choice of framing) as the compact, oriented manifold triad\[(Y_N;X_{\partial N},\partial_+Y_N;\partial X_{\partial N}):=(\cl(D^{n+3}\sm (N\times D^2));\cl(S^{n+2}\sm (\partial N\times D^2)),N\times S^1;\partial N\times S^1).\]

\begin{figure}[h]\[\def\picextYone{\resizebox{0.4\textwidth}{!}{ \includegraphics{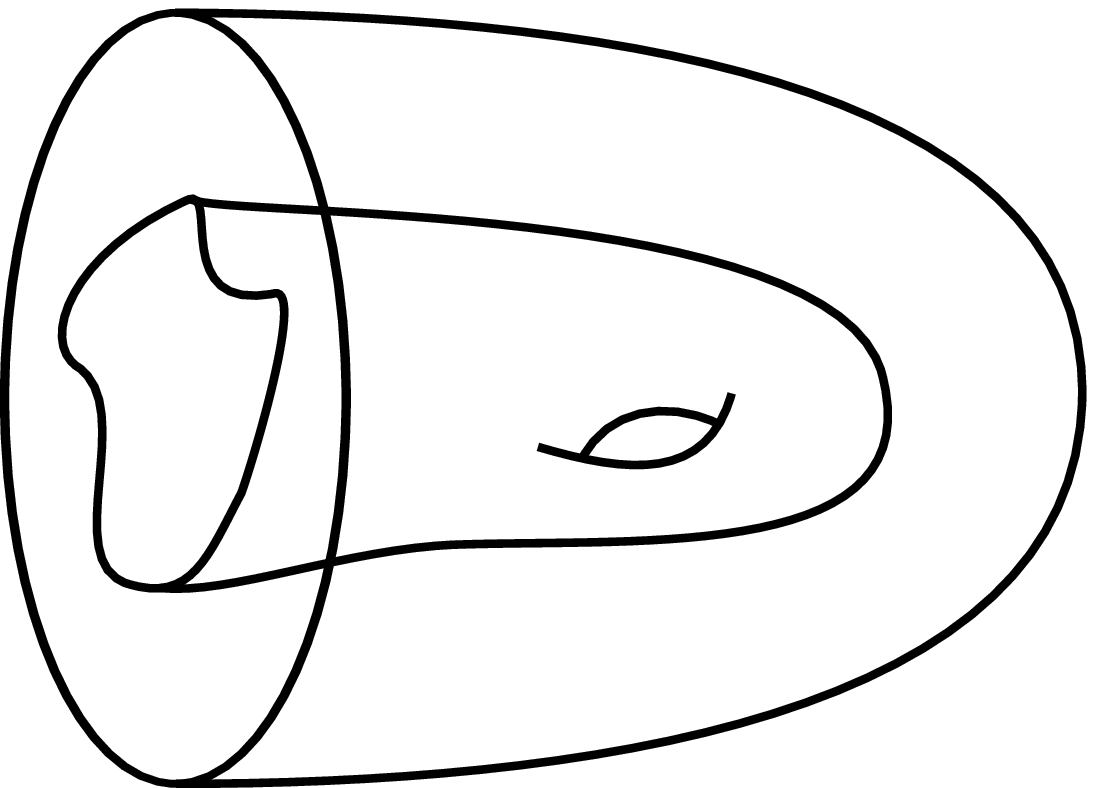}}}
\begin{xy} \xyimport(300,300){\picextYone}
,!+<7pc,-2.2pc>*+!\txt{Schematic of the framed \\codimension 2 submanifold pair.}
,(35,330)*!L{S^{n+2}}
,(190,320)*!L{D^{n+3}}
,(33,255)*!L{X_{\partial N}}
,(-70,240)*!L{\partial N\times S^1}
,(-40,225)*+{}="A";(30,200)*+{}="B"
,{"A"\ar"B"}
,(-70,70)*!L{\partial N\times D^2}
,(-30,80)*+{}="C";(55,110)*+{}="D"
,{"C"\ar"D"}
,(160,250)*!L{Y_N}
,(110,180)*!L{N\times D^2}
,(120,40)*!L{N\times S^1}
,(145,47)*+{}="A";(150,97)*+{}="B"
,{"A"\ar"B"}
\end{xy}
\qquad\quad
\def\picrelboundary{\resizebox{0.4\textwidth}{!}{ \includegraphics{pic_relboundary}}}
\begin{xy} \xyimport(300,300){\picrelboundary}
,!CD+<0.3pc,-1pc>*+!CU\txt{Schematic of the exterior \\as a triad.}
,(23,180)*!L{X_{\partial N}}
,(0,30)*!L{\partial X_{\partial N}}
,(140,170)*!L{Y_N}
,(120,50)*!L{\partial_+Y_N}
\end{xy}\]
\end{figure}

\subsubsection*{Connected sum}

Writing $x=(x_1,\dots,x_{m+1})\in S^m\subseteq\R^{m+1}$, the standard unit sphere, define the upper and lower hemispheres $D_\pm^m=\{x\in S^m\,|\,\pm x_1\geq 0\}$ so that we have the standard hemispherical decomposition\[S^m=D_+^m\cup_{S^{m-1}}D_-^m\subset \R^{m+1}.\]Define the standard orientation reversing homeomorphism of the unit sphere $S^m\subset\R^{m+1}$ by \[r_m:S^m\to S^m;\qquad (x_1,\dots,x_{m+1})\mapsto (-x_1,x_2\dots,x_m).\]

Up to ambient isotopy, we may assume a knot $K:S^n\hookrightarrow S^{n+2}$ has image \[K(S^n)=\Delta\cup D^n\subset S^{n+2}\] where $\Delta=K(S^n)\cap D_-^{n+2}$, and $D^n=K(S^n)\cap D_+^{n+2}$ is the standard unknotted $n$-disc given by $D^n=\{x_1\geq0, x_{n+2}=x_{n+3}=0\}\subset D^{n+2}_+\subset \R^{n+3}$. Consequently, $\partial\Delta=\partial D^n=S^{n-1}\subset S^{n+1}$, a standard unknotted sphere. Moreover, we may assume up to ambient isotopy that $D^n=K(D_+^n)\subset S^{n+2}$ so that $\Delta=K(D_-^n)$. The locally flat embedding of pairs\[\Delta_K=K_{D_-^n}:(D_-^n,S^{n-1})\hookrightarrow(\Delta,U(S^{n-1}))\subset(D^{n+2},S^{n+1})\] is called a \textit{disc knot} for $K$. The \textit{connected sum} $K\# K':S^n\hookrightarrow S^{n+2}$ of two knots $K, K'$ is the oriented, locally flat embedding given by glueing disc knots along the common unknotted boundary via the standard orientation reversing homeomorphism of a sphere\[K\# K':S^n\hookrightarrow S^{n+2};\qquad x\mapsto\left\{\begin{array}{ccl}\Delta_K(x)&&x\in D_-^n,\\r_{n+2}\Delta_{K'}r_n(x)&&x\in D_+^n.\end{array}\right.\]

\subsubsection*{Slice knots}

If there is a locally flat embedding of pairs $(D,K):(D^{n+1},S^n)\hookrightarrow (D^{n+3},S^{n+2})$ then we say the knot $K$ is \textit{slice} and the locally flat embedding $D$ is a \textit{slice disc} for $K$. In contrast to Levine's characterisation of unknottedness, a knot may admit a slice disc without that knot being the unknot. The terminology `slice knot' was introduced in \cite[p135]{MR0140099} (for $n=1$), where examples of non-trivial, non-slice knots are given. For instance, the `Stevedore's knot' is not slice.

\begin{lemma}If $K$ and $K'$ are slice, then $K\# K'$ is slice.
\end{lemma}
\begin{proof}[Proof (sketch)]We may cut the standard unit disc $D^{m}\subset \R^{m}$ into upper and lower \textit{hemidiscs} by setting $u(D^{m})=\{x\in D^{m}\,|\,x_1\geq0\}$ and $l(D^{m})=\{x\in D^{m}\,|\,x_1\leq 0\}$ so that we have a decomposition of pairs\[(D^{n+3},S^{n+2})=(u(D^{n+3}),D^{n+2}_+)\cup(l(D^{n+3}),D^{n+2}_-).\]The standard unknotted disc $D^{n+1}\subset D^{n+3}$ has a standard \textit{half slice disc} $u(D^{n+1})\subset u(D^{n+3})$.

By an ambient isotopy of the $D^{n+3}$, we may assume the respective slice discs for $K$ and $K'$ coincide with the standard half slice disc on $u(D^{n+3})$. We now remove the respective upper hemidiscs $u(D^{n+3})$ from $D^{n+3}$, and perform a relative version of the connected sum construction for knots. The result is the connected sum of knots $K\#K'$ bounding a connected sum of slice discs, which is again a slice disc.

\end{proof}

The study of slice knots can be rephrased in the language of cobordism.

\begin{definition}Two knots $K,K':S^n\hookrightarrow S^{n+2}$ are \textit{knot-cobordant} if there exists a locally flat embedding $f:S^{n}\times[0,1]\hookrightarrow S^{n+2}\times[0,1]$ such that $f(x,0)=(K(x),0)$ and $f(x,1)=(-K'(x),1)$. $K$ is \textit{knot-nullcobordant} if $K$ is knot-cobordant to the unknot $U$.
\end{definition}

If $K$ is slice, we may cut off a cap from a slice disc pair $(D^{n+3},D)$ to form a knot-nullcobordism of $K$. Conversely, if there is a knot-nullcobordism of $K$, we may glue a cap to the end of the knot-cobordism that contains $U$ to improve the embedded $S^n\times[0,1]$ to a slice disc. The following is a standard expansion on this idea.

\begin{proposition}\label{prop:kminusk}$K$ and $K'$ are knot-cobordant if and only if $K\#(-K')$ is slice.
\end{proposition}

\begin{proof}[Proof (sketch)]Starting with a knot-cobordism $f:S^{n}\times[0,1]\hookrightarrow S^{n+2}\times[0,1]$ between $K$ and $K'$, take an embedded arc $\alpha:[0,1]\hookrightarrow f(S^{n}\times[0,1])$ such that $\alpha(0)\in K$ and $\alpha(1)\in K'$. Excise an open tubular neighbourhood of $\alpha\subset S^{n+2}\times[0,1]$ to obtain a cobordism from $\Delta_K$ to $\Delta_{K'}$ embedded in $D^{n+2}\times[0,1]$, all relative to the unknot $(S^{n+1},U(S^{n-1})$. Contracting the relative part results in the cobordism of disc knots becoming a slice disc for the connected sum $K\# (-K')$.

To complete the proof, note this process is reversible.
\end{proof}

\section{The symmetric Blanchfield complex of a knot}

The \textit{Blanchfield complex} of a knot will be our central object of study and is the bridge between the algebraic $L$-theory of the previous chapters and our knot-theoretic applications. The Blanchfield complex is an invariant of an ambient isotopy class of embeddings $K:S^n\hookrightarrow S^{n+2}$ that is defined for both odd- and even-dimensional knots. It is the symmetric chain complex generalisation of the classical knot invariant called the \textit{Blanchfield form}, which is defined only for odd-dimensional knots. We will define the Blanchfield form of an odd-dimensional knot below and show how it derives from the Blanchfield complex.

First we spell out the details of the construction of the Blanchfield complex of an $n$-knot $K$, originally defined in \cite[pp822]{MR620795}. This will be the payoff for the careful description of the chain-diagonal approximation of maps of pairs that led up to Proposition \ref{prop:pairdiag}. We begin with a well-known proposition.

\begin{proposition}Suppose $f:(F^{n+1},S^n)\hookrightarrow (D^{n+3},S^{n+2})$ is a locally flat embedding of pairs, and write $f|_{S^n}=K$. There is a \textit{meridian map}, that is a map\[\psi:Y_F\to S^1,\]inducing an isomorphism $\psi_*:H_*(X_K)\cong H_*(S^1)$ and restricting to projection to the second factor\[\psi|_{ \partial_+ Y}=\text{pr}_2: F\times S^1\to S^1.\]The meridian map is uniquely defined up to homotopy by the fact that it restricts to projection on $\partial_+Y$. If $F=D$ is a slice disc then the meridian map induces an isomorphism $\psi_*:H_*(Y_D)\cong H_*(S^1)$.
\end{proposition}

\begin{proof}Firstly, any map $\psi:X_K\to S^1$ that restricts to the specified projection on the boundary $\partial X_K$ will be a homology isomorphism. This is because the boundary of a fibre of the embedded normal bundle of $K$ is a representative of the Alexander dual to $K$.

Write $Y=Y_F$. To construct $\psi$, first write the projection to the second factor as $\hat{\psi}:\partial_+ Y=F\times S^1\to S^1$. We identify the functor $[-,S^1]$, returning the set  of homotopy classes of maps from a manifold to $S^1$, with the functor $H^1(-;\Z)$ (considered as a set-valued functor). Writing the inclusion $i_+:\partial_+Y\hookrightarrow Y$, we are therefore looking to lift the cohomology class corresponding to $\hat{\psi}$ along $i_+^*:H^1(Y)\to H^1(\partial_+ Y)=H^1(F\times S^1)\cong H^0(F)\oplus H^1(F)$. But the class of $\hat{\psi}$ lies entirely in $H^0(F)$, and is a generator, as $\hat{\psi}$ collapses $F$ to a point. Thus, by Poincar\'{e} duality, our lifting problem is equivalent to lifting the fundamental class $[F]\in H_{n+1}(F,\partial F)$ to $H_{n+2}(Y,\partial Y)$.

To identify the lifting obstruction, consider that the cone $\mathcal{C}=C(C_*(F\times D^2,\partial F\times D^2)\to C_*(D^{n+3},S^{n+2}))$ is chain homotopy equivalent to $C_*(Y,\partial Y)$. To see this, take this iterated cone the other way so that $\mathcal{C}=C(C_*(S^{n+2},K\times D^2)\to C_*(D^{n+3},F\times D^2))$. But now it is clear that $\mathcal{C}=C(C_*(X_K,\partial X_K)\to C_*(Y,\partial_+ Y))$ by excision. So finally our lifting problem is to lift $[F]$ along the connecting morphism $\delta$ in the long exact sequence \[\dots\to H_{n+2}(Y, \partial Y)\xrightarrow{\delta} H_{n+1}(F,\partial F)\to H_{n+1}(D^{n+3},S^{n+2})\to\dots\]But there is no obstruction to lifting along $\delta$ - indeed for $r\leq n+2$ the map $\delta:H_r(Y,\partial Y)\to H_{r-1}(F,\partial F)$ is an isomorphism, and hence there is a unique lift.

Finally, when $F$ is a slice disc $(F,\partial F)=(D^{n+1},S^n)$, it is straightforward to combine the isomorphism $\delta$, the long exact sequence of the pair $(Y,\partial Y)$, the Meier-Vietoris sequence for the boundary $\partial Y=\partial_+ Y\cup_{S^1\times S^n}X_K$ and the Alexander duality of $X_K$ to yield the fact that $Y_D$ is a homology $S^1$. The class of the meridian map in $H^1(Y)\cong \Z$ is a generator and hence $\psi_*:H_*(Y_D)\to H_*(S^1)$ exhibits this isomorphism.
\end{proof}

\begin{remark}It is possible to construct the meridian map in a more explicitly obstruction theoretic way. For example, to construct the meridian for just a knot exterior (i.e.\ the presence of a cobounding surface not assumed) proceed as follows:

Write $\hat{\psi}:S^n\times S^1\to S^1$, the projection to the second factor. Standard obstruction theory (e.g. \cite[\textsection 7]{MR1841974}) tells us that the obstructions to extending $\hat{\psi}$ cell-by-cell to all of $X_K$ lie in the groups $H^{q+1}(X_K,\partial X_K;\pi_q(S^1))$. But these coefficients vanish for $q>1$, so there is no obstruction here. For $q=1$, $n>1$ we have $H^2(X_K,\partial X_K)\cong H_n(X_K)=0$. We need only check the case $q=1$, $n=1$, which is the obstruction to extending $\hat{\psi}$ over the 1-skeleton of $X_K$ when $K$ is a classical 1-knot. In this case one may instead use the classical algorithm of Seifert for constructing a Seifert surface.
\end{remark}

\begin{corollary}\label{cor:meridian}There is a \textit{meridian map} on the knot exterior\[\psi:X_K\to S^1\]uniquely defined up to homotopy by the property that $\psi|_{\partial X_K}:S^n\times S^1\to S^1$ is projection to the second factor.
\end{corollary}

The homotopy class of the meridian map $\psi\in[X_K,S^1]=[X_K,D^{n+1}\times S^1]$ may be represented by a (degree 1) map of compact, oriented, $(n+2)$-dimensional manifolds with boundary\[(f,\partial f):(X_K,\partial X_K)\to (D^{n+1}\times S^1,S^n\times S^1),\]with $\partial f$ the identity map.

\begin{figure}[h]\[\def\picextkernel{\resizebox{0.8\textwidth}{!}{ \includegraphics{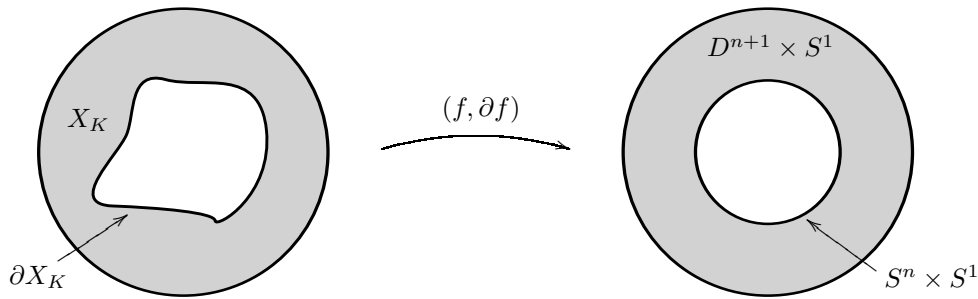}}}
\begin{xy} \xyimport(259,139){\picextkernel}
,!+<7.2pc,2pc>*+!\txt{}
,(10,85)*!L{X_K}
,(-7,10)*!L{\partial X_K}
,(5,15)*+{}="A";(30,42)*+{}="B"
,{"A"\ar"B"}
,(250,10)*!L{S^n\times S^1}
,(250,10)*+{}="A";(225,40)*+{}="B"
,{"A"\ar"B"}
,(198,120)*!L{D^{n+1}\times S^1}
,(100,70)*+{}="A";(160,70)*+{}="B"
,{"A"\ar@/^/"B"}
,(120,90)*!L{(f,\partial f)}
\end{xy}\]
  \caption{Schematic of the surgery problem determining the kernel pair.}
\end{figure}

Using the standard infinite cyclic cover $D^{n+1}\times \R\to D^{n+1}\times S^1$ with group of covering translations $\Z\cong\langle z\rangle$, we may now apply Proposition \ref{prop:pairdiag} to obtain the associated kernel pair $\sigma^*(\overline{f},\overline{\partial f})$, which is an $(n+2)$-dimensional symmetric Poincar\'{e} pair over the Laurent polynomial ring $\Z[\Z]\cong\Z[z,z^{-1}]$ with the involution $\overline{z}= z^{-1}$.

\begin{definition}The \textit{Blanchfield complex} of an $n$-knot $K:S^n\hookrightarrow S^{n+2}$ is the $(n+2)$-dimensional symmetric complex $(C_K,\phi_K)$ over $\Z[z,z^{-1}]$ defined as the algebraic Thom construction (see \ref{subsec:cxpair}) of the kernel pair $\sigma^*(\overline{f},\overline{\partial f})$.
\end{definition}

It is easy to identify $(D^{n+1}\times S^1,S^n\times S^1)\cong (X_U,\partial X_U)$ so that the Blanchfield complex of $K$ can be thought of as a measure of the difference between $K$ and the unknot $U$. In other words, we can think of the Blanchfield complex as a surgery problem trying to improve the knot exterior to an unknot exterior via codimension 2 surgery (cf.\ \cite[\textsection 7.8]{MR620795}).

The Blanchfield complex is an invariant of the ambient isotopy class of $K$ that is well-defined up to homotopy equivalence of $(n+2)$-dimensional symmetric complexes over $\Z[z,z^{-1}]$. Recall that the symmetric construction assumes we have made a choice of finite $CW$ complex structure on the compact manifolds with boundary $(X_K, \partial X_K)$, $(D^{n+1}\times S^1,S^n\times S^1)$ so that $C_K$ is an object of $\B_+(\Z[z,z^{-1}])$. In fact it is a chain complex of f.g.\ free modules.

\begin{claim}\label{clm:auto}The Blanchfield complex $(C_K,\phi_K)$ of an $n$-knot $K$ is Poincar\'{e} and such that\[C_K\oplus C_*(\overline{D^{n+1}\times S^1})\simeq C_*(\overline{X_K}).\]So in particular there is an isomorphism in reduced homology $\widetilde{H}_*(C_K)\cong \widetilde{H}_*(X_K)$. Furthermore,\[1-z:C_K\to C_K\]is an automorphism of $C_K$.
\end{claim}

\begin{proof}By Proposition \ref{prop:thomthick1}, a symmetric complex is Poincar\'{e} if and only if it is the Thom construction of a pair that is homotopy equivalent to a pair of the form\[(0:0\to D,(\phi,0)).\]But indeed,  $\partial f=\text{id}:\partial X_K\to \partial X_U$ implies a chain homotopy equivalence $\overline{\partial f}^!$ in $\B_+(\Z[z,z^{-1}])$ so that $C(\overline{\partial f}^!)$ is contractible and the kernel pair $\sigma^*(\overline{f}^!,\overline{\partial f}^!)$ is of the required form. The direct sum decomposition of Proposition \ref{prop:pairdiag} reduces to the claimed decomposition under the Thom construction.

The augmentation $\eps:\Z[z,z^{-1}]\to \Z$ sending $z\mapsto 1$ fits into the free $\Z[z,z^{-1}]$-module resolution\[0\to \Z[z,z^{-1}]\xrightarrow{1-z}\Z[z,z^{-1}]\xrightarrow{\eps} \Z\to 0.\]Applying this coefficient sequence to $C_K$ shows that the statement that $1-z$ acts as an automorphism of $C_K$ is equivalent to saying that $\Z\otimes_{\Z[z,z^{-1}]} C_K$ is acyclic. But it is easy to see that $\Z\otimes_{\Z[z,z^{-1}]} C_K$ is acyclic as the original map $(f,\partial f)$ was a $\Z$-homology equivalence (by Alexander duality, as already noted). Taking the cone $C(\overline{f}^!)$ and forgetting the action of the covering translations results in an acyclic complex.
\end{proof}

Recall that $P$ denotes the set of Alexander polynomials.

\begin{lemma}\label{lem:auto}If $C$ is a chain complex in $\B(\Z[z,z^{-1}])$ then $(1-z):C\to C$ is an automorphism if and only if there exists $p\in P$ such that $pH_*(C)=0$.
\end{lemma}

\begin{proof}
Suppose $H_*(C)$ is $P$-torsion. Then, as localisation is exact, $H_*(P^{-1}\Z[z,z^{-1}]\otimes_{\Z[z,z^{-1}]}C)=0$. The augmentation map $\eps:\Z[z,z^{-1}]\to \Z$ from above factors as\[\eps:\Z[z,z^{-1}]\to P^{-1}\Z[z,z^{-1}]\to \Z\]because $p(1)\in\Z$ is a unit for all $p\in P$. Hence $H_*(\Z\otimes_{\Z[z,z^{-1}]}C)=0$, which has already been observed to be equivalent to saying that $(1-z):C\to C$ is an automorphism.

Conversely assume $(1-z):C\to C$ is an automorphism. We will borrow the proof of Levine \cite[Corollary 1.3]{MR0461518}. The homology $H=\bigoplus_rH_r(C)$ is a f.g.\ $\Z[z,z^{-1}]$-module. Choose generators $x_1,\dots,x_m$ and write the coefficients of the automorphism $(1-z)$ on homology as $\lambda_{ij}\in\Z[z,z^{-1}]$ so that $(1-z)x_i=\sum_{j=1}^m\lambda_{ij}x_j$. Rearranging, we have $\sum_{j=1}^m\mu_{ij}x_j=0$ with\[\mu_{ij}=\left\{\begin{array}{lcl}\lambda_{ij}&&i\neq j,\\\lambda_{ii}-(1-z)&&i=j.\end{array}\right.\]The $m\times m$ matrix with $(i,j)$th entry $\mu_{ij}$ has determinant $p(z)\in \Z[z,z^{-1}]$, and $px_i=0$ for $i=1,\dots,m$. Hence $pH=0$. Moreover, as $(1-z)$ is an automorphism of $H$, the determinant $q$ of the $m\times m$ matrix with $(i,j)$th entry $\lambda_{ij}$ is a unit $q\in\{\pm1, \pm z, \pm z^{-1}\}\subset\Z[z,z^{-1}]$. But as $\eps(\mu_{ij})=\eps(\lambda_{ij})$ we have $\eps(p)=\eps(q)=\pm1$. Therefore $p\in P$ as required.
\end{proof}

\begin{corollary}\label{cor:Pacyclic}The Blanchfield complex $(C_K,\phi_K)$ of an $n$-knot $K$ is an $(n+2)$-dimensional $P$-acyclic symmetric Poincar\'{e} complex over $\Z[z,z^{-1}]$.
\end{corollary}

\subsubsection*{Seifert and Blanchfield forms of an $(2k+1)$-knot}\label{subsec:blanseif}

Suppose $n=2k+1$. An $n$-dimensional knot has two very tractable homological invariants, called the Blanchfield and Seifert forms of the knot, which we now reconcile with the Blanchfield complex as described above.

We turn first to the Blanchfield form of an $n$-knot $K$, originally considered in \cite{MR0085512}. Recall the Blanchfield complex $(C_K,\phi_K)$, and note there is a chain homotopy equivalence \[(\phi_K)_0:C_K^{n+2-*}\xrightarrow{\simeq} (C_K)_*\] so that we may apply Theorem \ref{thm:levblanch}.

\begin{definition}The \textit{Blanchfield form for $K$} is the non-singular $(-1)^{k}$-symmetric linking form over $(\Z[z,z^{-1}],P)$ defined by Theorem \ref{thm:levblanch}\[Bl:f(H^{k+2}(C_K))\times f(H^{k+2}(C_K))\to P^{-1}\Z[z,z^{-1}]/\Z[z,z^{-1}].\]By Claim \ref{clm:auto}, this can also be stated as\[Bl:f(H^{k+2}(\overline{X_K}))\times f(H^{k+2}(\overline{X_K}))\to P^{-1}\Z[z,z^{-1}]/\Z[z,z^{-1}].\]
\end{definition}

\begin{remark}Our definition is Poincar\'{e} dual to the usual definition of a Blanchfield form as in e.g.\ \cite{MR0461518}. It is also Poincar\'{e} dual to what is called the `modified' Blanchfield pairing in \cite[p154]{MR0324706}.
\end{remark}

We look now at the Seifert form for an $n$-knot $K$. Suppose we have made a choice of Seifert surface $j:F\hookrightarrow S^{n+2}$ for the knot $K$. The cohomology linking form on $H^{k+2}(S^{n+2})$ is clearly trivial - for instance, given $a,b\in C_{k+1}(S^{n+2})$, we can move these chains within a homology class so that there is no geometric linking. But the underlying pairing given by linking is non-trivial on the chain level \[l:C^{k+2}(S^{n+2})\times C^{k+2}(S^{n+2})\to \Z;\qquad (x,y)\mapsto x(a),\]where $da=y\cap[S^{n+2}]\in C_{k+1}$.

Choose a chain homotopy inverse $D:C_{*}(S^{n+2})\to C^{n+2-*}(S^{n+2})$ to the chain-level cap-product $-\cap[S^{n+2}]$.

\begin{definition}\label{def:seifert}The \textit{Seifert form of $(F,K)$} is the (well-defined) non-singular $(-1)^{k+1}$-symmetric Seifert form $(H^{k+1}(F),\psi)$ over $\Z$ given by\[\psi:H^{k+1}(F)\times H^{k+1}(F)\to \Z;\qquad([u],[v])\mapsto l(x,y),\]where $u,v\in C^{k+1}(F)$ and $x=D(j_*(u\cap[F]))$, $\partial y=D(j_*(v\cap[F]))$. It has the property that $(H^{k+1}(F),\psi+(-1)^{k+1}\psi^*)$ is the non-singular, $(-1)^{k+1}$-symmetric middle-dimensional cohomology intersection pairing of $F$.
\end{definition}

Fixing an orientation of an embedded normal bundle to $F$, denote $i^+,i^-:F\to S^{n+2}\sm F$ small displacements in the positive and negative normal directions. A Meier-Vietoris argument shows that the reduced homology morphism \[i^+_*-i^-_*:\widetilde{H}_r(F)\to \widetilde{H}_r(S^{n+2}\sm F)\] is an isomorphism for all $r$. It is then possible (see \cite[1.1]{MR718824}) to interpret the morphism \[e:=(\psi+(-1)^k\psi^*)^{-1}\psi: H^{k+1}(F)\to H^{k+1}(F)\] as the Poincar\'{e} dual of the morphism \[(i^+_*-i^-_*)^{-1}i^+_*:H_{k+1}(F)\to H_{k+1}(F).\]

\begin{remark}The Seifert form is often defined in terms of these displacements $i^\pm$. The geometric definition of linking (cf.\ \ref{subsec:motivation}) requires that the chains involved are disjoint, and the use of $i^\pm$ is one way to ensure this. Hence setting $a=j_*(u\cap[F])$ and $b=j_*(v\cap[F])$ in Definition \ref{def:seifert}, we have\[\psi([u],[v])=\text{link}(i^+_*a,b)=\text{link}(a,i^-_*b)=\text{link}(i^+_*a,i^-_*b).\]
\end{remark}

In \cite[pp43.]{MR0461518}, the connection between the Seifert and Blanchfield forms for a knot is made clear. Suppose we are given a Seifert surface $F$ for $K$ and we remove an open normal neighbourhood of $F$ to perform a cut-and paste construction of the infinite cyclic cover of $X_K$ (this is the construction of \ref{subsec:infinite}, but relative to $\partial F=K$). Exactly as in \ref{subsec:infinite} there is a Meier-Vietoris sequence and it is shown in \cite[p43.]{MR0461518} that this breaks into short exact sequences in reduced homology\[0\to i_!\widetilde{H}_r(F)\xrightarrow{(i^+_*z-i^-_*)}i_!\widetilde{H}_r(S^{n+2}\sm F;\Z)\to \widetilde{H}_r(\overline{X_K})\to 0,\]so we may write\[0\to i_!\widetilde{H}_r(F)\xrightarrow{(i^+_*-i^-_*)^{-1}(i^+_*z-i^-_*)}i_!\widetilde{H}_r(F)\to \widetilde{H}_r(\overline{X_K})\to 0,\]which, when $r=k+1$, is Poincar\'{e} dual to the sequence\[0\to i_!H^{k+1}(F)\xrightarrow{-((1-e)+ez)}i_!H^{k+1}(F)\to \underset{\cong H^{k+2}(\overline{X_K})}{\underbrace{H^{k+2}(\overline{X_K},\overline{\partial X_K})}}\to 0,\]so that the algebraic covering of the Seifert module is indeed the Blanchfield module in the sense of Chapter \ref{chap:laurent}. What is more, the covering of the Seifert form of an $n$-knot is the Blanchfield form of an $n$-knot (in the sense of Chapter \ref{chap:laurent}), we refer the reader to \cite[14.3]{MR0461518} for the geometric argument that this is the case.

\section{Knot-cobordism}\label{sec:knotcob}

For $n\neq 1$, the classification of $n$-knots up to knot-cobordism began with Kervaire's application (\cite{MR0189052}) of the surgery techniques of Kervaire-Milnor (\cite{MR0148075}) to the study of even-dimensional knots. Kervaire used codimension 2 surgery to improve an arbitrary Seifert surface to a slice disc, showing that all even-dimensional knots are slice. Levine (\cite{MR0246314}) continued the programme by analysing the odd-dimensional case with the same techniques. In this case it is generically only possible to do surgery `below the middle dimension' of the Seifert surface.

\begin{definition}A \textit{simple knot} is an $n$-knot $K$ that admits a Seifert surface $F^{n+1}$ with $\pi_q(F)=0$ for $2q<n$. By \cite[Theorem 2]{MR0179803}, this definition is equivalent to saying that $\pi_q(X_K)\cong \pi_q(S^1)$ for $2q< n$.
\end{definition}

Using surgery `below the middle dimension' (and an application of Hirsch's Engulfing Theorem) Levine showed that all knots are knot-cobordant to simple knots. Next Levine classified simple knots up to knot-cobordism using an $L$-theoretic obstruction described by the Witt class of a Seifert form. The knot classification programme was completed by Stoltzfus (\cite{MR0467764}) who finally computed the Witt group of Seifert forms over $\Z$.

We will now recap this story in more detail, using the Blanchfield complex as our $L$-theoretic obstruction.
 
\begin{definition}The set of ambient isotopy classes of $n$-knots, equipped with the operation of connected sum of knots $K_1\# K_2:S^n\hookrightarrow K_1(S^n)\# K_2(S^n)$ is a commutative monoid called $Knots_n$, with unit given by the unknot $U$.
\end{definition}

\begin{lemma}For any $n$-knot $K$, the knot $K\#(-K)$ is slice. If $K$ and $K'$ are slice knots then $K\# K'$ is a slice knot.
\end{lemma}

\begin{proof}By proposition \ref{prop:kminusk}, this is equivalent to $K$ being knot-cobordant to $K$. But we just take the cylinder on $(S^{n+2},K(S^n))$.
\end{proof}

\begin{definition}The \textit{$n$-dimensional (topological) knot-cobordism group} is the group given by the monoid construction\[\mathcal{C}_n:=Knots_n/\{\text{slice knots}\}.\]\text{A priori}, $K=0\in\mathcal{C}_n$ if and only if there exists a slice $n$-knot $J$ such that $J\# K$ is slice. However, it is easy to check that  if $K=0\in\mathcal{C}_n$ then in fact $K$ itself is slice.

The simple knots form a submonoid $Knots^{simp}_n\subseteq Knots_n$ and it is straightforward to check that the slice simple knots are a closed submonoid of $Knots^{simp}_n$. We therefore define the \textit{n-dimensional (topological) simple knot cobordism group}\[\mathcal{C}^{simp}_n=Knots^{simp}_n/\{\text{simple slice knots}\}.\]
\end{definition}

To state the main surgery classification theorem we will first need to define the Rochlin invariant. The \textit{signature} of a non-singular, symmetric Seifert form $(L,\psi)$ over $\Z$ is defined to be the signature (number of positive eigenvalues subtract number of negative eigenvalues) of the symmetric bilinear form $(L,\psi+\psi^*)\otimes_\Z\R$. This signature is divisible by 8 (essentially because the non-singular symmetric form $(L,\psi+\psi^*)$ over $\Z$ admits a quadratic extension, see \cite[\textsection 4]{MR0246314}). As the signature of a metabolic Seifert form is 0, signature is an invariant of the Witt class $[(L,\psi)]\in \widehat{W}(\Z)$.

\begin{definition}The \textit{Rochlin invariant of $[(L,\psi)]\in \widehat{W}(\Z)$} is given by dividing the signature of $(L,\psi)$ by 8 and taking the residue modulo 2.

Combining algebraic surgery below the middle dimension (Theorem \ref{thm:algsurgery}) and \cite[3.4.7(ii)]{MR620795}, there is an isomorphism $L^{4l}(\Z[z,z^{-1}],P,-1)\cong \widehat{W}(\Z)$ for $l\geq0$. The \textit{Rochlin invariant of $[(C,\phi)]\in L^{4l}(\Z[z,z^{-1}],P,-1)$} is the Rochlin invariant of a corresponding Witt class of Seifert forms.
\end{definition}

The surgery classification of knots is as follows:

\begin{theorem}[{\cite{MR0189052}, \cite{MR0246314}}]\label{thm:concordance}$\,$\begin{itemize}
\item For $n\geq 4$, \[\mathcal{C}_{n}\cong L^{n+1}(\Z[z,z^{-1}],P,-1)\cong\left\{\begin{array}{lcl}0&&n=2k,\\ \widehat{W}^{(-1)^k}(\Z)&&n=2k+1.\end{array}\right.\]

\item For $n=3$, there is a short exact sequence\[0\to \mathcal{C}_3\to L^4(\Z[z,z^{-1}],P,-1)\to \Z/2\Z\to 0\]with the map to $\Z/2\Z$ given by the Rochlin invariant.

\item For $n=2$, $\mathcal{C}_2=0$.

\item For $n=1$, there is a surjective morphism $\mathcal{C}_1\twoheadrightarrow L^2(\Z[z,z^{-1}],P,-1)$.
\end{itemize}
\end{theorem}

\begin{remark}The calculations of the Witt groups $\widehat{W}^{(-1)^k}(\Z)$ were finally made by Stoltzfus \cite{MR0467764} to be, for both $k$ odd and even, of the type\[\widehat{W}^{(-1)^k}(\Z)\cong\bigoplus_\infty\Z\oplus\bigoplus_\infty(\Z/2\Z)\oplus\bigoplus_\infty(\Z/4\Z).\]Some explanation of where these factors come from is also available in \cite[\textsection 42]{MR1713074}.
\end{remark}

We will now show that the assignment\[\sigma^L:\mathcal{C}_n\to L^{n+1}(\Z[z,z^{-1}],P,-1);\qquad K\mapsto (C_K,\phi_K)\]is a well-defined group homomorphism.

\begin{lemma}\label{lem:blanchslice}The algebraic cobordism class $(C_K,\phi_K)\in L^{n+1}(\Z[z,z^{-1}],P,-1)$ of the Blanchfield complex of an $n$-knot is a well-defined invariant of the knot-cobordism class.
\end{lemma}

\begin{proof}It suffices to show a slice disc $(D,K)$ will induce a nullcobordism of the Blanchfield complex of $K$. This is now shown by the relative version of the Blanchfield complex construction.

Represent the homotopy class of the meridian map $\psi\in[Y_D,S^1]=[Y,D^{n+3}\times S^1]$ by a (degree 1) map of compact oriented manifold triads\[F=(f,\partial f, \partial'f,\partial\partial f):(Y_D;X_K,\partial_+ Y_D;\partial X_K)\to(D^{n+3}\times S^1; D^{n+1}\times S^1,D^{n+1}\times S^1;S^n\times S^1),\]where both $\partial'f$ and $\partial\partial f$ are identity maps. (We can think of $F$ as a map from the slice disc exterior to the `trivial slice disc' exterior.)

\begin{figure}[h]\[\def\picslicetriad{\resizebox{0.85\textwidth}{!}{ \includegraphics{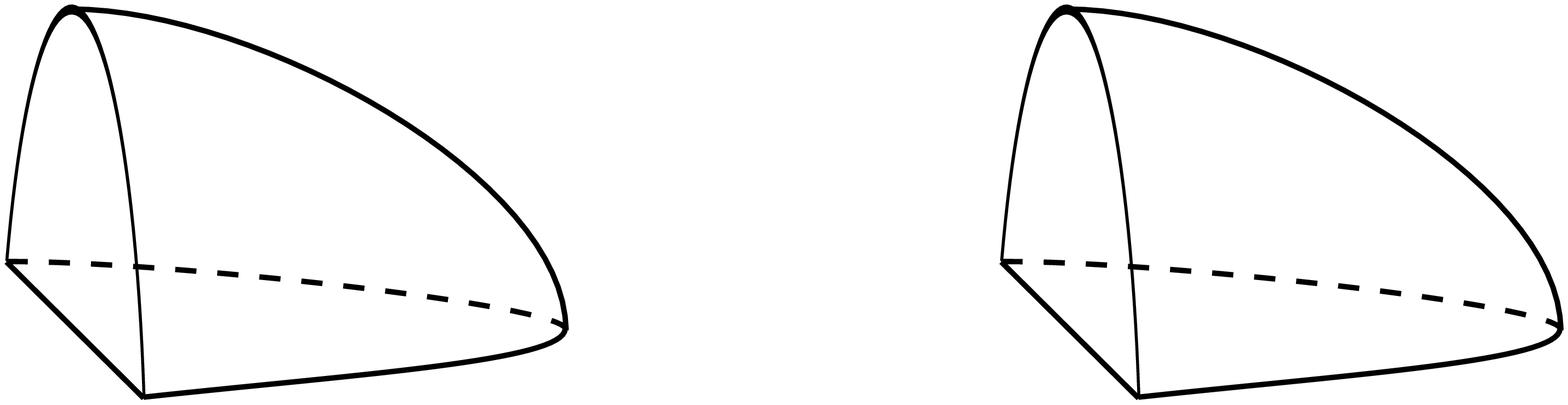}}}
\begin{xy} \xyimport(259,139){\picslicetriad}
,!+<7.2pc,2pc>*+!\txt{}
,(6,80)*!L{X_K}
,(-5,20)*!L{\partial X_K}
,(30,30)*!L{\partial_+ Y_D}
,(40,80)*!L{Y_D}
,(129,70)*!L{D^{n+1}\times S^1}
,(145,25)*!L{S^{n}\times S^1}
,(192,30)*!L{D^{n+1}\times S^1}
,(195,80)*!L{D^{n+2}\times S^1}
,(85,100)*+{}="A";(160,100)*+{}="B"
,{"A"\ar@/^/"B"}
,(120,123)*!L{F}
\end{xy}\]
  \caption{A schematic of the surgery problem determining the kernel triad.}
\end{figure}

Using the standard $\Z$-cover $D^{n+3}\times\R\to D^{n+3}\times S^1$ we obtain the kernel triad $\sigma^*(\overline{F}^!)=(\Gamma,(\Phi,\delta\phi,\delta'\phi,\phi))$, some $(n+3)$-dimensional symmetric Poincar\'{e} triad in $\B(\Z[z,z^{-1}])$. Now the relative algebraic Thom construction on this triad results in a surgery dual set $\{x,x'\}$. But as $C(\partial\partial f^!),C(\partial' f^!)\simeq 0$, we have\[x=(C(\overline{\partial f}^!)\to C(\overline{f}^!),(\Phi/0,\delta\phi/0)),\qquad x'=(0\to C(C(\overline{\partial f}^!)\to C(\overline{f}^!)),(\Phi/\delta\phi,0/0)).\]Hence we have the pair $x=(C_K\to C(\overline{f}^!),(\Phi,\phi_K))$ for $(C_K,\phi_K)$ the Blanchfield complex of $K$. We need to show $x$ is an algebraic nullcobordism, i.e.\ that $x$ is a Poincar\'{e} pair. But the kernel triad $\sigma^*(\overline{F}^!)$ is Poincar\'{e}, so by definition of a Poincar\'{e} triad $(0\cup_0 C_K\to C(\overline{f}^!),(\Phi,0\cup_0\phi_K))=x$ is Poincar\'{e}.

It remains to show the nullcobordism $x$ is in the correct category. But as the meridian map $\psi$ is a homology equivalence on $Y_D$ and on $X_K$, by the same arguments as in Claim \ref{clm:auto} and Lemma \ref{lem:auto}, the Poincar\'{e} pair $x$ is moreover in $\C_+(\Z[z,z^{-1}],P)$, the $P$-acyclic category. Hence $(C_K,\phi_K)=0\in L^{n+1}(\Z[z,z^{-1}],P,-1)$ as required.
\end{proof}

We will henceforth write the function \[\sigma^L:\mathcal{C}_n\to L^{n+1}(\Z[z,z^{-1}],P,-1);\qquad K\mapsto (C_K,\phi_K).\]

\begin{lemma}$\sigma^L$ is a homomorphism of groups.
\end{lemma}

\begin{proof}We must show $\sigma^L(K\# K')\cong \sigma^L(K)\oplus \sigma^L(K')$. To do this we use the notation and constructions used in the definition of the connected sum  above\[K\# K':S^n\hookrightarrow S^{n+2};\qquad x\mapsto\left\{\begin{array}{ccl}\Delta_K(x)&&x\in D_-^n,\\r_{n+2}\Delta_{K'}r_n(x)&&x\in D_+^n.\end{array}\right.\]

Modify the beginning of that construction by picking framings of the normal bundles to $K$ and $K'$ so that we extend to normal neighbourhood embeddings $K\times D^2$ and $K'\times D^2$. Now, as in the construction of the disc knots $\Delta_K$, $\Delta_{K'}$, assume we have \[K(S^n)=\Delta_K(D_-^n)\cup D^n\qquad K'(S^n)=\Delta_{K'}(D_-^n)\cup D^n.\]

We may assume (by an ambient isotopy) that $K$ and $K'$ meet $\partial D_+^{n+2}=S^{n+1}$ transversely. Further, we may assume that on $\partial D_+^{n+2}$, the normal neighbourhoods of $K$ and $K'$ agree with the standard normal neighbourhood (of fixed radius $\eps$) of the unknotted $S^{n-1}\subset S^{n+1}$. Now when we form the connected sum $K\# K'$ the boundaries of the normal neighbourhoods are precisely identified, and hence so are the copies of $S^1$ that will define the meridian maps.

The \textit{disc-knot exterior} for $K$ is the compact connected oriented manifold triad \[\begin{array}{rcl}X_{\Delta_K}&=&(\cl(D^{n+2}\sm(\Delta_K(D^{n})\times D^2));\Delta_K\times \partial D^2,\partial \Delta_K\times D^2;\partial \Delta_K\times \partial D^2)\\&=&(\cl(D^{n+2}\sm(\Delta_K(D^{n})\times D^2));D^n\times S^1,S^{n-1}\times D^2;S^{n-1}\times S^1).\end{array}\]The restriction of the map $(f,\partial f):(X_K,\partial X_K)\to (X_U,\partial X_U)$ coming from the meridian map for $X_K$ is a degree 1 map of triads\[F_{K}:X_{\Delta_K}\to X_{\Delta_U}\]to the disc knot exterior of the unknot. It is homotopy equivalent to a map of triads \[F_K\simeq (f_K;\text{id},\text{id};\text{id}).\]

\begin{figure}[h]\[\def\pichomo{\resizebox{0.85\textwidth}{!}{ \includegraphics{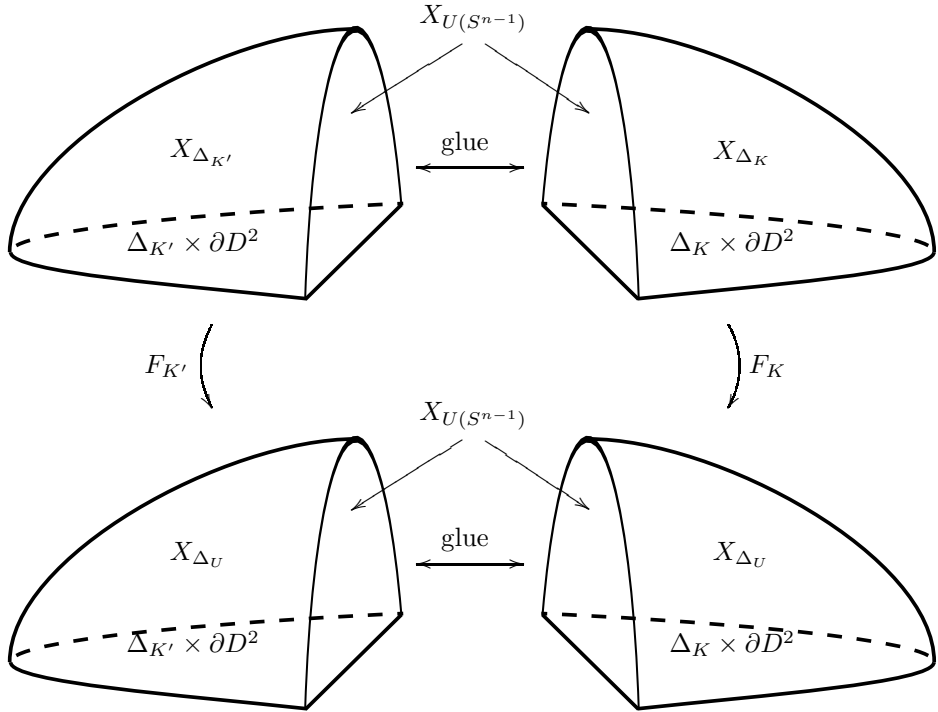}}}
\begin{xy} \xyimport(259,139){\pichomo}
,!+<7.2pc,2pc>*+!\txt{}
,(116,140)*!L{X_{U(S^{n-1})}}
,(132,136)*+{}="A";(165,120)*+{}="B"
,{"A"\ar"B"}
,(130,136)*+{}="A";(95,120)*+{}="B"
,{"A"\ar"B"}
,(185,95)*!L{\Delta_K\times\partial D^2}
,(35,95)*!L{\Delta_{K'}\times\partial D^2}
,(197,113)*!L{X_{\Delta_{K}}}
,(47,113)*!L{X_{\Delta_{K'}}}
,(113,110)*+{}="A";(147,110)*+{}="B"
,{"A"\ar"B"}
,(147,110)*+{}="A";(113,110)*+{}="B"
,{"A"\ar"B"}
,(122,115)*!L{\text{glue}}
,(200,80)*+{}="A";(200,60)*+{}="B"
,{"A"\ar@/^/"B"}
,(207,70)*!L{F_K}
,(60,80)*+{}="A";(60,60)*+{}="B"
,{"A"\ar@/_/"B"}
,(40,70)*!L{F_{K'}}
,(116,60)*!L{X_{U(S^{n-1})}}
,(132,56)*+{}="A";(165,40)*+{}="B"
,{"A"\ar"B"}
,(130,56)*+{}="A";(95,40)*+{}="B"
,{"A"\ar"B"}
,(185,14)*!L{\Delta_K\times\partial D^2}
,(35,14)*!L{\Delta_{K'}\times\partial D^2}
,(197,32)*!L{X_{\Delta_{U}}}
,(47,32)*!L{X_{\Delta_{U}}}
,(113,30)*+{}="A";(147,30)*+{}="B"
,{"A"\ar"B"}
,(147,30)*+{}="A";(113,30)*+{}="B"
,{"A"\ar"B"}
,(122,35)*!L{\text{glue}}
\end{xy}\]
  \caption{Instructions for glueing.}
  \label{pic:homo}
\end{figure}

$\sigma^L(K\# K')$ is given by the algebraic Thom complex of the kernel pair (with respect to the infinite cyclic cover coming from the meridian map) of the degree 1 map of pairs\[F_K\cup F_{K'}:(X_{K\# K'},\partial X_{K\# K'})\to (X_U,\partial X_U),\] where we have glued along the common part of the triads' boundaries $X_{U(S^{n-1})}$ as indicated in Figure \ref{pic:homo}. Algebraically, glueing is taking a cone, and building the kernel of a pair/triad is also taking a cone. The order in which we take cones does not affect the result, so we will take kernel triads first and glue afterwards. But as $F_K\simeq (f_K,\text{id},\text{id},\text{id})$, the kernel triad $\sigma^*(\overline{F}^!)$ is homotopy equivalent to \[\xymatrix{0\ar[r]\ar[d]&0\ar[d]\\0\ar[r] &C(\overline{f_K}^!)}\]and similarly with $K'$. The glue of these triads indicated in Figure \ref{pic:homo} is the cone\[C(0\to C(\overline{f_K}^!)\oplus C(\overline{f_K}^!))= C(\overline{f_K}^!)\oplus C(\overline{f_K}^!),\]with symmetric structure also given by direct sum. So \[(C_{K\# K'},\phi_{K\# K'})\simeq  (C(\overline{f_K}^!)\oplus C(\overline{f_{K'}}^!),\psi_K\oplus \psi_{K'})\] for some symmetric structures $\psi_K$, $\psi_{K'}$.

But now consider making this entire construction above with $K'=U$, the unknot. Then $f_U\simeq \text{id}$ and we have that $(C_{K},\phi_{K})=(C_{K\# U},\phi_{K\# U})\simeq (C(\overline{f_K}^!),\psi_K)$. So in fact we have that \[\sigma^L(K\# K')\cong \sigma^L(K)\oplus \sigma^L(K'),\]as required.
\end{proof}

\begin{proposition}\label{prop:realise}If $K:S^{2k+1}\hookrightarrow S^{2k+3}$ is simple then $(C_K,\phi_K)$ is chain homotopic to the $(k+1)$-fold skew-suspension of a 1-dimensional $(-1)^{k+1}$-symmetric $P$-acyclic Poincar\'{e} complex $(C,\phi)$, such that the associated $(-1)^k$-symmetric linking form $(H^1(C),\lambda_\phi)$ (see Proposition \ref{prop:correspondence}) is the Blanchfield form of the knot.
\end{proposition}

\begin{proof}By Claim \ref{clm:auto}, it is enough show that there is a chain homotopy equivalence from such a $C$ to the reduced chain complex $\widetilde{C}_*(X_K)$. But by \cite[Theorem 2]{MR0179803}, there is a Seifert surface $F$ with $\pi_i(F)=0$ for $i\neq0, k+1$. In \ref{subsec:blanseif} we showed the existence of a chain homotopy equivalence of reduced chain complexes\[C\left((i^+_*-i^-_*)^{-1}(i_*^+z-i_*^-):i_!\widetilde{C}_*(F)\to i_!\widetilde{C}_*(F)\right)\simeq \widetilde{C}_*(X_K).\]The Seifert form of this $(F,K)$ is $(C^{k+1}(F)=L,\psi)$, determining a 1-dimensional $(-1)^{k+1}$-symmetric $P$-acyclic complex $(C,\phi)$:\[\xymatrixcolsep{4pc}\xymatrix{
C^{1-*}:&0\ar[r]&L[z,z^{-1}]\ar[rr]^{(1-e)+ez}\ar[d]^{(1-z)(\psi+\eps\psi^*)}&&L[z,z^{-1}]\ar[r]\ar[d]^-{-(1-z^{-1})(\psi+\eps\psi^*)}&0\\
C_*:&0\ar[r]&L^*[z,z^{-1}]\ar[rr]^{((1-e)+ez)^*}&&L^*[z,z^{-1}]\ar[r]&0}\]where the vertical arrows indicate $\phi_0:C^{1-*}\xrightarrow{\simeq}C_*$, and the higher chain map $\phi_1:C^1\to C_1$ is determined by the $(-1)^{k+1}$-symmetry of $\psi$.
\end{proof}

\begin{remark}In \cite[5.1, 10.1]{MR0358795} Kearton obtains a handlebody decomposition of a simple $(2k+1)$-knot exterior with only 0-, $(k+1)$-, and $(k+2)$-handles. Combining this with the geometric arguments of \cite[\textsection 8]{MR0324706} and \cite[2.6]{MR0200922}, it seems likely that one can construct a handle decomposition for the knot exterior such that, by lifting this handle decomposition to the infinite cyclic cover, the handle chain complex exhibits $(C,\phi)$ above. \end{remark}

\begin{lemma}$\sigma^L$ is surjective for $n\neq 3$. For $n=3$, the Rochlin invariant of $(C,\phi)\in L^{4}(\Z[z,z^{-1}],P,-1)$ defines an isomorphism $\coker(\sigma^L)\cong\Z/2\Z$.
\end{lemma}

\begin{proof}Algebraic surgery below the middle dimension (\cite[3.23(ii)]{MR620795}) identifies \[L^{n+1}(\Z[z,z^{-1}],P,-1)\cong\left\{\begin{array}{lcl}L^0(\Z[z,z^{-1}],P,(-1)^{k})&&n=2k+1,\\ L^1(\Z[z,z^{-1}],P,(-1)^{k+1})&&n=2k.\end{array}\right.\]But $L^1(\Z[z,z^{-1}],P,\pm)=0$ (as it consists of 0-dimensional $P$-acyclic complexes of projective modules) so the statement of the lemma is evidently true for even-dimensional knots.

For $n=2k+1$, we identify the $L$-group with Witt groups of linking forms (\cite[3.4.7(ii)]{MR620795}), then apply the inverse of the covering isomorphism of Theorem \ref{algtrans}:\[L^0(\Z[z,z^{-1}],P,(-1)^{k})\cong W^\eps(\Z[z,z^{-1}],P,(-1)^{k})\cong\widehat{W}_{(-1)^{k+1}}(\Z).\] So the problem now is to realise any given Witt class of $(-1)^{k+1}$-symmetric Seifert forms over $\Z$ by an $n$-knot. For a constructive solution to this problem, we refer the reader to \cite[Lemma 3]{MR0246314}, where it is shown that given a Witt class of non-singular $(-1)^{k+1}$-symmetric Seifert forms over $\Z$ and $n=2k+1>3$, one may build a simple $n$-knot $K:S^n\hookrightarrow S^{n+2}$ with Seifert surface $F^{n+1}\subset S^{n+2}$ so that the Seifert form of the knot realises a representative of this Witt class. When $n=3$ this is possible if and only if the Rochlin invariant of the Seifert form vanishes.

To complete the proof we must check that our knot $K$ realising the Witt class of Seifert forms does indeed realise the corresponding $(n+2)$-dimensional symmetric $P$-acyclic Poincar\'{e} complex. But we may appeal to Prop \ref{prop:realise} and the proof is complete.
\end{proof}

\begin{theorem}[Surgery below the middle dimension]\label{thm:surgbelow}For $n=2k$, all $n$-knots are slice (\cite{MR0189052}). For $n=2k+1$, $k>0$, all $n$-knots are knot-cobordant to a simple $n$-knot (\cite{MR0246314}).
\end{theorem}

\begin{proof}[Proof (idea)]Suppose $K$ is an $n$-knot  ($n$ odd or even) and that $F\subset S^{n+2}$ is any Seifert surface. In both the even and odd cases the strategy is to perform surgeries on $F$ in such a way that the trace of the surgeries embeds in the $D^{n+3}$ cobounding the ambient $S^{n+2}$. The surgeries are performed to successively kill $\pi_1(F),\pi_2(F),\dots$ until there is an obstruction to surgery. (One method is to perform the surgeries entirely abstractly and then use a general immersion theorem, followed by a Whitney trick argument to embed the trace, as in \cite[p85]{MR0283786}. Another method is to show that each successive surgery can be performed in an embedded fashion, so that the trace embeds as you go along, as in \cite[III.6]{MR0189052}.)

More precisely, fix $q$ with $2q<n+1$. Suppose there is an embedded manifold triad $(W';F\sqcup F',K\times I;K\sqcup K)$ where $W'\hookrightarrow D^{n+3}$ such that $W'\cap S^{n+2}=F$, that $(W',F)$ only has handles of index $i\leq q$ and such that $\pi_1(F')=\pi_2(F')=\dots=\pi_{q-1}(F')=0$. (To begin with, we may take $q=0$ and $W=F\times I$.) Kervaire \cite[III.6]{MR0189052} details how to to kill $\pi_q(F')$ by performing embedded $(n+3)$-dimensional handle attachments to $F'$ in $D^{n+3}$ with attaching spheres given by embedded representatives of generators $[S^q\hookrightarrow F']\in \pi_q(F')$.

This process is repeated for increasing $q$ until at some $q$ there is an obstruction to continuing. This results is an embedded manifold triad $(W;F\sqcup F_0,K\times I;K\sqcup K)$ where $W\hookrightarrow D^{n+3}$ such that $W\cap S^{n+2}=F$.

\begin{figure}[h]\[\def\picextW{\resizebox{0.4\textwidth}{!}{ \includegraphics{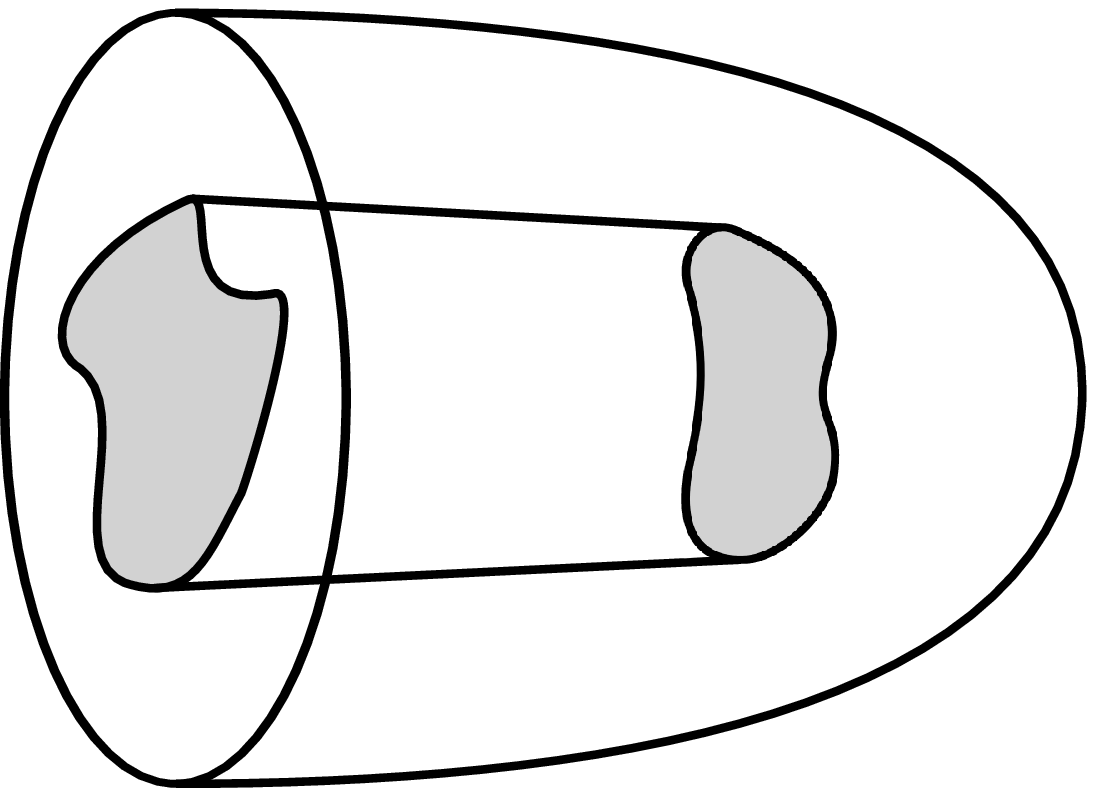}}}
\begin{xy} \xyimport(259,139){\picextW}
,!+<7.2pc,2pc>*+!\txt{}
,(35,149)*!L{S^{n+2}}
,(190,139)*!L{D^{n+3}}
,(35,75)*!L{F}
,(110,70)*!L{W}
,(173,73)*!L{F_0}
,(145,25)*!L{K\times I}
,(145,23)*+{}="A";(110,37)*+{}="B"
,{"A"\ar"B"}
\end{xy}\]
  \caption{Trace of embedded surgeries.}
  \label{pic:trace}
\end{figure}

When $n\neq 2$ is even, there is never an obstruction to attaching handles to kill all homotopy of $F$, so that $F_0$ is a disc and $K$ is seen to be slice. The case $n=2$ must be treated slightly differently (see \cite[p.~265]{MR0189052}) but the result is the same.

When $n=2k+1$, the embedded surgeries of \cite[III.6]{MR0189052} result in an $F_0$ having $\pi_i(F_0)=0$ for $i\neq 0,k+1$. There are now dimensional obstructions to killing $\pi_{k+1}(F_0)$ via embedded surgery. Instead Levine (\cite[Lemma 4]{MR0246314}) uses an engulfing theorem of Hirsch to embed a smaller $\hat{D}^{n+3}$ in the interior of $D^{n+3}$ so that $W\cap \hat{D}^{n+3}=F_0$. Excising the interior of this smaller disc results in an $h$-cobordism between $S^{n+2}$ and $\partial \hat{D}^{n+3}$ and by the $h$-cobordism theorem we then have a homeomorphism \[f:\cl(D^{n+3}\sm\hat{D}^{n+3})\xrightarrow{\cong} S^{n+2}\times I, \quad\text{with $f|_{S^{n+2}}=\text{id}$}.\]But this homeomorphism defines a knot-cobordism between $K$ and the simple knot $f(\partial F_0)=\partial f(F_0)$.
\end{proof}

\begin{corollary}The forgetful map is an isomorphism\[\mathcal{C}^{simp}_n\xrightarrow{\cong} \mathcal{C}_n.\]
\end{corollary}

We now furnish the necessary final details of the proof that $\sigma^L$ is injective for $n\geq 3$.

\begin{lemma}[`Stably metabolic = metabolic' ({\cite[1.6]{MR0467764}})] For $R$ a Dedekind domain, a non-singular $\eps$-symmetric Seifert form $(K,\psi)$ over $R$ vanishes in $\widehat{W}_\eps(R)$ if and only if it is metabolic. (Cf.\ Corollary \ref{stablyhypishyp} and the subsequent remark.)
\end{lemma}

\begin{lemma}[{\cite[Lemma 5]{MR0246314}}]For $n=2k+1>1$, suppose an $n$-knot $K$ has Seifert surface $F$ such that $\pi_i(F)=0$ for $i\neq 0,k+1$ and that the associated Seifert form is metabolic. Then $K$ is slice.
\end{lemma}

\begin{proof}[Proof (idea)]When $n>3$, a basis for the Poincar\'{e} dual $L\subset H_{k+1}(F)$ of a lagrangian for the Seifert form can be realised by framed, embedded $(k+1)$-spheres $S_i\times D^{k+1}\hookrightarrow F\subset S^{n+2}$. The vanishing of the Seifert form on (the Poincar\'{e} dual of) these homology classes, and the dimensions involved, ensure that we can use the Whitney trick to remove any geometric intersections. Now, the framed embeddings can be extended to framed embeddings of $(k+2)$-discs $D_i\times D^{k+1}\hookrightarrow D^{n+3}$ in the cobounding $D^{n+3}$. Again, by the Whitney trick these framed discs may be assumed to be disjoint. These framed discs can be thought of as the cores of handle attachments, that determine surgeries killing the basis of $L$. The effect of these surgeries can then be shown to be a slice disc for $K$.

When $n=3$, the Seifert surface is a 4-manifold and the Whitney trick no longer applies to make the embedded 2-spheres disjoint. We refer the reader to \cite[\textsection 13]{MR0246314}for the more delicate 4-manifold techniques required for this case. Essentially the mechanism is that the Whitney trick works \textit{stably} in dimension 4 (this is Wall's Stabilisation Theorem \cite{MR0163324}), so we may stabilise the Seifert surface by adding connected sums $F\#(\#_mS^2\times S^2)$ until we can use the Whitney trick. The connected sum changes the Seifert surface and hence the Seifert form, however the \textit{Witt class} of the Seifert form is not affected as the addition is by $m$ copies of a standard hyperbolic matrix.
\end{proof}

\chapter{Doubly-slice knots}\label{chap:knots}

We now apply the double $L$-theory of Chapter \ref{chap:DLtheory} to the study of high-dimensional knots. We prove that the algebraic double-cobordism class of the Blanchfield complex is an obstruction to a knot being doubly-slice (Definition \ref{def:doublyslice}). Our obstruction subsumes the previously known obstructions for odd-dimensions of the Blanchfield form and the Seifert form. Our algebraic results of previous chapters, particularly Chapter \ref{chap:DLtheory}, are then interpreted in this context. In particular we show the new result that any Seifert form for an odd-dimensional knot is hyperbolic (previously it was only known to be stably hyperbolic). We also give a new (and shorter) proof of the result of Stoltzfus and Bayer-Fluckiger (\cite{MR833015}) that for $n\neq 1$ a simple odd-dimensional $n$-knot is stably doubly-slice if and only if it is doubly slice. We discuss future applications of double $L$-theory (and extensions of it) to the doubly-slice problem.

\section{The doubly-slice problem}

If we slice an $(n+1)$-knot $J$ into two pieces using a plane $\R^{n+2}\cup\{\infty\}\subset S^{n+3}$ and the intersection $K=J\cap \R^{n+2}\subset S^{n+2}$ is an $n$-knot then necessarily it is slice with two (possibly different) slice discs. On the other hand, given an $n$-knot $K$, a slice disc $(D,K)\subset (D^{n+3},S^{n+2})$ can be doubled to form an $(n+1)$-knot\[ (D^{n+3},D)\cup_{(S^{n+2},K)}-(D^{n+3},D)=(S^{n+3},D\cup_K-D).\]So slice $n$-knots are precisely the $n$-knots which are cross sections of $(n+1)$-knots.

It is not so surprising that cross sections of a knotted sphere can be knotted spheres. Perhaps more surprising is that cross sections of an \textit{unknotted} $(n+1)$-sphere can be knotted spheres. The problem of detecting whether an $n$-knot is a cross section of the $(n+1)$-unknot was introduced by Fox \cite[pp138]{MR0140099} (the first non-trivial example of this type was constructed by John Stallings but never published, again see \cite[pp138]{MR0140099}). An equivalent way of describing this type of knot is the following:

\begin{definition}\label{def:doublyslice}A pair of slice discs $(D_\pm,K)$ for an $n$-knot $K$ are \textit{complementary} if the result of glueing the slice discs together along $K$ is an $(n+1)$-dimensional unknot\[(S^{n+3},D_+\cup_K-D_-)=(S^{n+3},U).\]An $n$-knot $K$ is \textit{doubly-slice} if it admits a pair of complementary slice discs.
\end{definition}

The first major investigation into the problem of detecting doubly-slice knots was Sumners' paper \cite{MR0290351}. In this paper, Sumners notes the following corollary to a theorem of Zeeman (\cite[\textsection 6]{MR0195085}).

\begin{proposition}[{\cite[2.9]{MR0290351}}]For $K$ any $n$-knot, $K\#(-K)$ is doubly-slice.
\end{proposition}

\begin{proof}We give a quick reproof of Zeeman's theorem and mention this as a consequence in Appendix \ref{chap:twistspin}.
\end{proof}

\begin{lemma}If $K, K'$ are doubly-slice $n$-knots then $K\#K'$ is doubly slice. Hence the doubly-slice knots form a closed submonoid of $Knots_n$.
\end{lemma}
\begin{proof}Suppose $(D_\pm,K)$ are complementary slice discs for $K$ and $(D_\pm',K')$ are complementary slice discs for $K'$. Then the boundary connected sums $(D_+,K)\#(D'_+,K')$ and $(D_-,K)\#(D_-',K')$ glue along $K\#K'$ to form the $(n+1)$-knot $(D_+\cup_KD_-)\#(D_+'\cup_K'D_-')$. But this is the $(n+1)$-unknot as it is the connected sum of unknots.
\end{proof}

This affords the following definition:

\begin{definition}The \textit{$n$-dimensional (topological) double knot-cobordism group} is the group given by the monoid construction\[\mathcal{DC}_n:=Knots_n/\{\text{doubly-slice knots}\}.\]Simple doubly-slice knots are a submonoid of $Knots_n^\text{simp}$ and so we also define\[\mathcal{DC}_n^\text{simp}:=Knots_n^\text{simp}/\{\text{simple doubly-slice knots}\}.\]
\end{definition}

In contrast to the knot-cobordism group, it is unknown whether `stably doubly-slice' implies doubly slice. This is perhaps the most important unsolved question for doubly-slice knots:

\begin{question}\label{q:stab}Does there exist $K$ such that $K=0\in\mathcal{DC}_n$ but $K$ is not doubly-slice?
\end{question}

At present, it is only known that the double knot-cobordism class of a knot is an obstruction to the knot being doubly-slice. A negative answer to Question \ref{q:stab} would reduce the problem of detecting doubly-slice knots to the problem of determining their double knot-cobordism class.

\section{The Blanchfield complex as a doubly-slice obstruction}

We now show that the algebraic double-cobordism class of the Blanchfield complex of an $n$-knot is obstruction to the knot being doubly-slice, thereby building a new invariant of the double knot-cobordism class of a knot.

\begin{proposition}\label{prop:invariant}If $K$ is doubly-slice then the Blanchfield complex $(C_K,\phi_K)$ is algebraically double-nullcobordant.
\end{proposition}
\begin{proof}We must check that Blanchfield complex of a doubly-slice knot $K$ admits complementary $(\Z[z,z^{-1}],P)$-nullcobordisms. By the proof of Lemma \ref{lem:blanchslice}, we know that a pair of complementary slice discs $(D_\pm, K)$ results in a pair of morphisms of compact oriented manifold triads\[F_\pm\simeq (f_\pm;\partial f_\pm,\text{id};\text{id}):(Y_{D_\pm};\partial X_K,\partial_0Y_{D_\pm};\partial X_{\partial K})\to (Y_{U};\partial X_U,\partial_0Y_{U};\partial X_{U}),\]which result in a pair of algebraic nullcobordisms of the Blanchfield complex \[x_\pm\simeq(C_K\to C(\overline{f}_\pm^!),(\Phi_\pm,\phi_K)).\]We wish to check that the algebraic glue $x_+\cup x_-\simeq 0$. But algebraic glueing is a mapping cone on the chain level. So is the algebraic Thom construction on a pair, and the construction of the kernel triads. We perform these mapping cones in any order and receive the same result, hence the underlying chain complex of $x_+\cup x_-$ is the result of performing these operations in the following order: glue the maps of algebraic triads $\overline{F_+}^!\cup \overline{F_-}^!$ (by glueing the triads along the knot exteriors $C(X_K)$), form the kernel triad $\sigma(\overline{F_+}^!\cup \overline{F_-}^!)$, then perform the algebraic Thom construction on this triad. But as the slice discs $(D,K)$ were complementary, we have that the glue $D_+\cup_K D_-$ is unknotted in $S^{n+3}$ and hence $\overline{F_+}^!\cup \overline{F_-}^!$ is chain homotopic to the identity. Therefore the triad $\sigma(\overline{F_+}^!\cup \overline{F_-}^!)$ is contractible and hence $x_+\cup x_-\simeq 0$ as required.
\end{proof}

\begin{remark}The transitivity of the double $L$-groups mean that we have just given a partial affirmative answer to an algebraic question of Levine \cite[3(2)]{MR718271}. There is no `product structure' in algebraic $L$-theory, so our answer is not complete. However we conjecture that the techniques of double $L$-theory could be modified to include product structure and answer this question affirmatively.
\end{remark}

\begin{corollary}\label{cor:blanchfield}The Blanchfield form of an odd-dimensional doubly-slice knot is hyperbolic.
\end{corollary}

\begin{proof}Apply Proposition \ref{prop:welldeflagrang2}.
\end{proof}

\begin{corollary}For $n\geq 1$, there is a well-defined homomorphism\[\sigma^{DL}:\mathcal{DC}_n\to DL^{n+1}(\Z[z,z^{-1}],P,-1);\qquad [K]\mapsto (C_K,\phi_K).\]When $n=2k+1$ there is a well-defined homomorphism\[\sigma^{DW}:\mathcal{DC}_n\to DW^{(-1)^k}(\Z[z,z^{-1}],P);\qquad [K]\mapsto (H^{k+2}(C),\lambda_{\phi_K}),\]and for any choice of Seifert surface $F$ there is a well-defined homomorphism\[\sigma^{\widehat{DW}}:\mathcal{DC}_n\to \widehat{DW}^{(-1)^{k+1}}(\Z);\qquad [K]\mapsto (H^{k+1}(F),\psi).\]
\end{corollary}

We will now combine some of the algebraic results in Chapter \ref{chap:DLtheory} to prove a new result about Seifert forms for knots.

\begin{theorem}\label{thm:theorem}Every Seifert form for an odd-dimensional doubly-slice knot $K$ is hyperbolic.
\end{theorem}

\begin{proof}Any covering of a Seifert form for $K$ is hyperbolic by Corollary \ref{cor:blanchfield}, and covering is an isomorphism $\widehat{DW}_\eps(\Z)\cong DW^{-\eps}(\Z[z,z^{-1}],P)$, so now every Seifert form for $K$ is \textit{stably} hyperbolic. But by Proposition \ref{stablyhypishypseif} stably hyperbolic Seifert forms are moreover hyperbolic.
\end{proof}

We now have several algebraic responses to Question \ref{q:stab}.

\begin{theorem}Suppose for $n\geq1$ that $K=0\in \mathcal{DC}_n$. If $K$ is not doubly slice, this cannot be detected by the Blanchfield complex. If $n=2k+1$ and $K$ is not doubly slice, this cannot be detected by the Blanchfield form or any choice of Seifert form.
\end{theorem}

\begin{remark}We have worked primarily with odd-dimensional knots in this thesis as these are the knots for which there is some sort of linking form. However, the Blanchfield complex yields middle-dimensional information about even-dimensional knots as well (which will correspond to the \textit{linking formations} of  \cite[3.5]{MR620795}). Much has been written about the middle-dimensional invariants of even-dimensional knots and indeed the `Farber-Levine pairing' can be derived from the Blanchfield complex (this is described in great detail in \cite{MR0461518}). But this pairing is not the full story (see Kearton's $F$-form \cite{MR741655} and Farber's $L$-quintuple \cite{MR718824}) and we hope in future work to investigate the extent to which this can be reconciled with, and derived from, the Blanchfield complex data.
\end{remark}

\subsection{Simple doubly-slice knots}

Recall the Kervaire-Levine surgery classification of knots hinged on constructing a knot-cobordism from any knot to a simple knot via surgery below the middle-dimension (see \ref{thm:surgbelow}). We discuss such a procedure for double knot-cobordism (or lack thereof) in the next subsection. If we assume we are dealing with simple knots to begin with, the algebraic results of Section \ref{sec:comparison} yield the following partial answer to Question \ref{q:stab}.

\begin{theorem}\label{thm:stoltzfus}For odd $n=2k+1>1$, a simple $n$-knot $K$ has $[K]=0\in \mathcal{DC}^{simp}_n$ if and only if $K$ is doubly slice.
\end{theorem}

\begin{proof}`If' is clear. Conversely, if $\sigma^{DL}(K)=0$, we have that the Blanchfield form $(T,\lambda)$ for $K$ has $(T,\lambda)=0\in DW^{(-1)^k}(\Z[z,z^{-1}],P)$. But by Corollary \ref{stablyhypishyp}, this means $(T,\lambda)$ is hyperbolic. Hence any Seifert surface $F$ for $K$ has hyperbolic Seifert form by Proposition \ref{stablyhypishypseif}. Take a basis of $H^{k+1}(F;\Z)$ with respect to which the matrix of the Seifert form is hyperbolic. The Poincar\'{e} dual basis to this can be realised by framed, embedded $(k+1)$-spheres which can be used as instructions for surgery on $F$ to realise two complementary slice disks as in \cite[Theorem 3.1]{MR0290351} (case $k>1$) and \cite{MR0380817} (case $k=1$).
\end{proof}

\begin{remark}This is not the first proof of Theorem \ref{thm:stoltzfus}. In \cite{MR833015}, a less general form of Corollary \ref{stablyhypishyp} is obtained by very different methods to our own. The authors derive Theorem \ref{thm:stoltzfus} from this.
\end{remark}

\subsection{There is no `double surgery below the middle dimension'}

The doubly-slice problem at first appears very similar to the slice problem, and the main approach to it (indeed the approach taken in this thesis) has been to treat it similarly - as a high-dimensional cobordism problem. We now briefly describe the issues related to treating the doubly-slice problem in this way, and in particular how these issues are reflected in double $L$-theory. We begin with a striking theorem of Ruberman.

\begin{theorem}[{\cite[4.17]{MR709569}, \cite[3.3]{MR933307}}]\label{thm:ruberman}In every even dimension, there exist knots with hyperbolic Farber-Levine pairing (see \cite[\textsection 6]{MR0461518} and \cite[3.16]{MR521738}) but which are not doubly slice.

In every odd dimension, there exists an infinite family of knots $K$ with hyperbolic Blanchfield form but which are not doubly slice. When $n\neq 1$ all knots in the family have exteriors which are homotopy equivalent (rel.\ boundary preserving meridians) to one another and to a doubly-slice knot. When $n=1$ the exteriors have the same $\Z[\Z]$-homology type.
\end{theorem}

One consequence of this theorem is that there can be no general procedure (cf.\ \ref{thm:surgbelow}) that modifies a knot within its double knot-cobordism class to be simple. \textit{There is no double surgery below the middle dimension.} That is\[\mathcal{DC}_n\not\cong\mathcal{DC}^{simp}_n.\]

\begin{remark}The knot families of Theorem \ref{thm:ruberman} may not be typical. So the theorem does not preclude the existence of an odd-dimensional knot $K$ with hyperbolic Blanchfield form but non-vanishing $\sigma^{DL}(K)$. We believe Theorem \ref{thm:framespun} could be a fruitful source of future examples.

The mechanism for detecting non-doubly-slice knots in Theorem \ref{thm:ruberman} is a high-dimensional application of the Casson-Gordon invariants (see \cite{MR709569}, \cite{MR933307}). Necessarily, the existence of these invariants requires interesting cyclic representations of the fundamental group $\pi_1(S^{n+2}\sm K)$. If, for example, $\pi_1(S^{n+2}\sm K)\cong\Z$, then perhaps the issues arising from Ruberman's examples cannot occur. No Ruberman-type results are currently known for knots with $\pi_1(S^{n+2}\sm K)\cong\Z$.
\end{remark}

Theorem \ref{thm:ruberman} does not preclude double \textit{algebraic} surgery below the middle dimension. On the level of double $L$-theory this is intimately related to the question of whether there is periodicity in the double $L$-groups (cf.\ \ref{subsec:periodicity}).

By Proposition \ref{iso}, one consequence of the statement that `all double $L$-groups are 4-periodic' would be that all even-dimensional double $L$-groups were isomorphic to double Witt groups. If we furthermore assume the localisation $(A,S)$ has homological dimension 0 (resp.\ that we are working over $(A,S)=(\Z[z,z^{-1}],P)$) then by Proposition \ref{prop:welldeflagrang} (resp.\ \ref{prop:welldeflagrang2}) this is equivalent  to saying that the middle-dimensional linking pairing (resp.\ the Blanchfield form) of an odd-dimensional $S$-acyclic symmetric Poincar\'{e} complex contains all the double-cobordism invariants. As such, Theorem \ref{thm:ruberman} seems to lend weight to the idea that there is no periodicity in double $L$-theory.

\begin{remark}Conceivably, the `secondary obstructions' beyond the Blanchfield pairing could all be occurring at a level that homological algebraic invariants, such as the Blanchfield complex, cannot see. But in fact, \textit{homology-level} secondary obstructions are identified in \cite[Proposition p.252]{MR718271}. These homology-level obstructions involve the ring structure in cohomology. Product structures are not well accounted for in $L$-theory and are not seen by a class in double $L$-theory. Even though the Alexander-Whitney chain diagonal approximation incorporates a chain-level cup product, it is lost under equivalence of symmetric Poincar\'{e} complexes. Building this extra data into the invariant is an interesting avenue of research that we intend to pursue in the future.\end{remark}

\appendix

\chapter{A trace function}\label{chap:trace}

For this appendix, let $R$ be a commutative Noetherian ring with involution and $A=R[z,z^{-1}]$ be the ring with involution given by $\overline{z}=z^{-1}$. Let $S$ be the set of characteristic polynomials\[S=\left\{p(z)=\sum_{M}^Na_kz^k\,|\,a_{M},a_N\in R^\times\right\}\subset R[z,z^{-1}].\]We will define and explain an $R$-module morphism\[\chi:S^{-1}A/A\to R\] such that for every $T$ in $\H(A,S)$ there is induced a natural isomorphism of $A$-modules\begin{equation}\label{eq:cond1}\chi_*:\Hom_A(T,S^{-1}A/A)\to \Hom_R(i_!T,R)\end{equation}(the $A$-module structure on the target is given by setting the action of $z^{-1}$ as the action of $\zeta(T)^*$). Such a function is called a \textit{trace function} for $(A,S)$ if it additionally satisfies\begin{equation}\label{eq:cond2}\chi(\overline{x})=-\overline{\chi(x)}.\end{equation}The concept is a common tool, coming originally from algebraic number theory, akin to the `trace' of a field extension (cf.\ \cite[IV. \textsection5]{MR1878556}). This account is based on the version of the trace function considered by Trotter \cite{MR0645546} and later by Litherland \cite[A3]{MR780587}. In fact, Litherland shows that for $R$ a field, the properties \ref{eq:cond1} and \ref{eq:cond2} uniquely determine the function $\chi$. Although apparently well-known (see also \cite{MR0249519} and \cite{MR1713074}), it seems hard to find an explanation of this trace function so we hope this appendix is useful to the reader.

\begin{remark}In contrast to previous accounts (e.g.\ \cite{MR780587} or \cite{MR1713074}), we have not called $\chi$ a `universal trace function' for the following reason. Recall that if a contravariant functor $\mathcal{F}:\mathcal{C}\to \text{Set}$ is \textit{representable} then there exists a \textit{universal element} $(a,x)$ where $a$ is in $\mathcal{C}$ and $x$ is in $\mathcal{F}(a)$ with the property that for every pair $(b,y)$ with $b$ in $\mathcal{C}$ and $y$ in $\mathcal{F}(b)$ there exists a unique morphism $f:b\to a$ such that $\mathcal{F}(f)x=y$. In this language, the contravariant functor we wish to represent is \[\mathcal{F}=\Hom_R(-,R):\H(A,S)\to \text{Set};\qquad T\mapsto \Hom_R(i_!T,R),\]and our trace function is written $(S^{-1}A/A,\chi)$. But we see this is \textit{not} strictly speaking a universal element for $\mathcal{F}$ because the $A$-module $S^{-1}A/A$ is not $S$-torsion.

On the other hand, $\mathcal{F}$ is an $R$-module valued presheaf (in the sense of category theory) and hence, by the Yoneda Lemma, can always be expressed as the colimit of representable functors. In some sense the definition of the trace map is given by such a colimit.
\end{remark}

We define some $A$-modules:\[\begin{array}{ccrcl}
\text{the Novikov rings}&\qquad&A_+&=&R((z))=\left\{\sum_{r=N}^{\infty}a_rz^{ r}\,|\,a_r\in R,\,N\in\Z\right\},\\
&&&&\\
&\qquad&A_-&=&R((z^{-1}))=\left\{\sum_{r=-\infty}^{N}a_rz^{ r}\,|\,a_r\in R,\, N\in\Z\right\},\\
&&&&\\
\text{the Laurent series}&&A_\infty&=&R[[z,z^{-1}]]=\left\{\sum_{-\infty}^{\infty}a_rz^{r}\,|\,a_r\in R\right\}.\end{array}\]The Novikov ring $A_+=(Z_+)^{-1}R[[z]]$ is the localisation of the power series ring $R[[z]]$ with respect to the multiplicative subset $Z_+=\{z^k\,|\,k\geq 0\}$. If $R$ is a field, the Novikov ring $A_+$ is the fraction field of $R[[z]]$. Of course, the analogous statements are true for $R[[z^{-1}]]$, $Z_-$ and $A_-$. The Laurent series $A_\infty$ does not generally carry a ring structure at all.
\begin{lemma}\label{lem:inverse}If $p\in S$ then $p\in (A_+)^\times\cap(A_-)^\times$.
\end{lemma}
\begin{proof}First we show that if $p\in R[z]$ is a polynomial with constant term a unit then it is invertible in $R[[z]]$. This result follows from the stronger claim\[(R[[z]])^\times = \left\{p(z)\in R[[z]]\,|\,p(0)\in\R^\times\right\}.\]To see this claim, first note that the augmentation $R[[z]]\to R$ with $z\mapsto 1$ is a ring morphism  so that a unit $p(z)\in R[[z]]^\times$ has $p(0)\in R^\times$. Conversely, if $p(z)=\sum_0^\infty a_kz^k$ with $a_0\in R^\times$, then \[p(z)^{-1}=\left(1+\sum_{j=1}^\infty\left((-(a_0)^{-1}\sum_{k=1}^\infty a_kz^k)^j\right)\right)(a_0)^{-1}\in R[[z]].\]

Now as suppose $p\in S$, then for some $M\in\Z$, $z^Mp\in R[z]$ is a bionic polynomial. Hence it is invertible in $R[[z]]$ and moreover in $(Z_+)^{-1}R[[z]]=A_+$.

A similar argument with $A_-$ completes the proof.
\end{proof}
There are natural injective $A$-module morphisms\[j_\pm:A\hookrightarrow A_\pm,\qquad k_\pm:A_\pm\hookrightarrow A_\infty,\qquad l_\pm:S^{-1}A\hookrightarrow A_\pm\]($l_\pm$ are well-defined by Lemma \ref{lem:inverse}) such that $k_+l_+=k_-l_-:S^{-1}A\to A_\infty$. Hence there is a commutative diagram of $A$-modules where the rows are exact\begin{equation}\label{eq:injectiveA}\xymatrix{0\ar[r]&A\ar[r]^-{j}\ar[d]_-{=}&S^{-1}A\ar[r]\ar[d]^-{\lmat l_+\\l_-\rmat}&S^{-1}A/A\ar[r]\ar[d]^-{l}&0\\
0\ar[r]&A\ar[r]^-{\lmat j_+\\j_-\rmat}&A_+\oplus A_-\ar[r]^-{\lmat k_+ &-k_-\rmat}&A_\infty\ar[r]&0}\end{equation}and the map $l$ is induced by the diagram. Then $\chi$ is defined as the composite:\[\xymatrix{\chi:S^{-1}A/A\ar[rr]^-{l}&&A_\infty\ar[rr]^{-\text{const.}\circ z} &&R,}\]where ``\text{const.}'' denotes taking the coefficient of $z^0$ in a Laurent series.

\medskip

To get more familiar with $\chi$, we will consider what this means generally, and then we will illustrate with an example. Let $p,q\in A$ and \[f=p/q=\sum_{-M}^N a_rz^r/\sum_{-L}^K b_rz^r\in S^{-1}A\]be a quotient of finite Laurent series, which represents a class $[f]\in S^{-1}A/A$. The map $l$ begins by using the maps $l_\pm$, which construct inverses $q^{-1}$ in the respective Novikov rings:\[l_+(q^{-1}p)=p\sum_{r=N_+}^\infty b_r^+z^r\in A_+,\qquad l_-(q^{-1}p)=p\sum_{r=-\infty}^{N_-} b_r^-z^r\in A_-.\]There is then an isomorphism of $A$-modules\[(k_+\,\,-k_-):(A_+\oplus A_-)/A\xrightarrow{\cong} A_\infty,\]so that \[l([f])=p\left(\left(\sum_{r=N_+}^\infty b_r^+z^r\right)-\left(\sum_{r=-\infty}^{N_-} b_r^-z^r\right)\right)\in A_\infty.\]So finally\[\chi([f])=\sum_{r+s=0}a_r(b_s^+-b_s^-)\in R.\]

\begin{example}Suppose $f=(a-z)^{-1}$ for some $a\in R^\times$. Then
\begin{eqnarray*}l_+(f)&=&(a^{-1}+a^{-2}z+a^{-3}z^2+\dots),\\
l_-(f)&=&-(z^{-1}+az^{-2}+a^2z^{-3}+\dots),\end{eqnarray*}and\[\chi([f])=a^{-1}.\]
\end{example}

\subsection*{Why is this trace function correct?}

The definition of the trace function is very neat and if we assume $R$ is a field as in Litherland it is not too difficult to prove that it has the properties we claim.

But why does it exist? We derive a natural function below that is equal to the trace $\chi_*$ up multiplication by $z$. This $z$ is accounted for in Chapter \ref{chap:laurent} as a slight discrepancy between the dualities of torsion modules and autometric modules under the `covering' operation.

The following is a straightforward claim we will need, easily confirmed:

\begin{claim}[{\cite[p24]{MR1048238}}]\label{clm:funnyF} If $M$ is a f.g.\ projective $R$-module then there is a natural isomorphism of $A$-modules\[\begin{array}{rcl}F:i_!\Hom_R(M,R)&\to&\Hom_A(i_!M,A);\\
z^rf&\mapsto& (z^sx\mapsto f(x)z^{s-r}).\end{array}\]with inverse given by \[\begin{array}{rcl}F^{-1}:\Hom_A(i_!M,A)&\to&i_!\Hom_R(M,R);\\
f&\mapsto& \sum_{-\infty}^\infty z^rf_r,\qquad f_r(x)=a_{0,r}\text{ where }f(z^{r}x)=\sum_{-\infty}^\infty a_{s,r}z^s.\end{array}\]
\end{claim}

Now, given an $A$-module $T$ in $\H(A,S)$, we can pick a particularly nice projective $A$-module resolution. By Proposition \ref{prop:covmonod1} we may take the resolution of $T$ by f.g.\ projective $A$-modules \[\xymatrix{0\ar[r] &i_!i^!T\ar[rr]^-{z-\zeta} &&i_!i^!T\ar[r]& T\ar[r] &0}\]and there is hence the following commuting diagram of $A$-modules with exact rows:

\[\xymatrix{0\ar[r]& (i_!i^!T)^*\ar[d]_-{\cong}^-{F^{-1}}\ar[rr]^-{(z-\zeta)^*}&& (i_!i^!T)^*\ar[d]_-{\cong}^-{F^{-1}}\ar[r]& \Ext^1_A(T,A)\ar[d]^-{\cong}\ar[r]& 0\\
0\ar[r]&i_!(i^!T)^*\ar[rr]^-{z^{-1}-\zeta^*}&&i_!(i^!T)^*\ar[r]&\Hom_R(i^!T,R)\ar[r]&0}\] Here, $z^{-1}-\zeta^*=F^{-1}\circ(z-\zeta)^*\circ F$, by direct calculation.

We have shown:

\begin{proposition}\label{prop:natural}For each $T$ in $\H(A,S)$, there exists a natural isomorphism of $A$-modules\[\Hom_A(T,S^{-1}A/A)\xrightarrow{\cong}\Hom_R(i^!T,R).\]
\end{proposition}

Suppose $T$ is in $\H(A,S)$, and take a length 1 resolution by f.g.\ projective $A$-modules\[0\to P_1\xrightarrow{d} P_0\to T\to 0.\]The connecting morphism used in the Ext long exact sequence of Lemma \ref{lem:ext2} is the natural isomorphism of $A$-modules\[\xymatrix{\Hom_A(T,S^{-1}A/A)=\ker(\Hom_A(d,S^{-1}A/A))\ar[r]^-{\delta}_-{\cong}&\coker(\Hom_A(d,A))=:\Ext_A^1(T,A).}\]The injective morphism of exact sequences in diagram \ref{eq:injectiveA} tells us in particular that the morphism $\delta$ will factor through the connecting morphism in the long exact Ext sequence associated to the coefficient sequence\[0\to A\to A_+\oplus A_-\to A_\infty\to 0.\]Namely, $\delta=\delta'\circ l_*$, where $\delta'$ is the natural $A$-module morphism\[\xymatrix{\Hom_A(T,A_\infty)=\ker(\Hom_A(d,A_\infty))\ar[r]^-{\delta'}&\coker(\Hom_A(d,A))=:\Ext_A^1(T,A).}\]

And now we identify the morphism $\delta$ with the morphism $\chi_*\circ(-z)$.

\begin{proposition}\label{prop:badtrace}For our choice of resolution and a fixed $g\in\Hom_A(T,S^{-1}A/A)$ we have \[F^{-1}\circ\delta'\circ l_*(g):i^!T\to R;\qquad x\mapsto -l(g(x))_{-1},\]the negative of the coefficient of $z^{-1}$ in $l(g(x))\in A_\infty$.
\end{proposition}
\begin{proof}We prefer to work with $(A_+\oplus A_-)/A$ via the isomorphism\[(k_+\,\,-k_-):\frac{A_+\oplus A_-}{A}\xrightarrow{\cong} A_\infty.\]Suppose $f\in\Hom_A(T,(A_+\oplus A_-)/A)$, so that for any $x\in P_0$\[f:P_0/P_1\to (A_+\oplus A_-)/A;\qquad x+P_1\mapsto \left(\sum_{N_+(x)}^\infty a_r^+(x)z^r,\sum_{-\infty}^{N_-(x)}a_r^-(x)z^r\right)+A.\]We must chase $f$ around the change of coefficients diagram\[\xymatrix{
0\ar[r]&\Hom_A(P_0,A)\ar[r]\ar[d]&\Hom(P_1,A)\ar[r]\ar[d]&0\\
0\ar[r]&\Hom_A(P_0,A_+\oplus A_-)\ar[r]\ar[d]&\Hom(P_1,A_+\oplus A_-)\ar[r]\ar[d]&0\\
0\ar[r]&\Hom_A(P_0,(A_+\oplus A_-)/A)\ar[r]&\Hom(P_1,(A_+\oplus A_-)/A)\ar[r]&0}\]

First, choose a lift to $\tilde{f}\in \Hom_A(P_0,A_+\oplus A_-)$. This results in fixing for each $x\in P_0$ a choice of $\tilde{f}(x)=\left(\sum_{N_+(x)}^\infty a_r^+(x)z^r,\sum_{-\infty}^{N_-(x)}a_r^-(x)z^r\right)$. Now fix the projective resolution to be \[\xymatrix{0\ar[r] &i_!i^!T\ar[rr]^-{z-\zeta} &&i_!i^!T\ar[r]& T\ar[r] &0}\] and apply the differential on the middle row of the change of coefficients diagram. As the original $f$ was in the kernel of the differential on the bottom row, the morphism $(z-\zeta)^*\tilde{f}$ lifts to $\tilde{\tilde{f}}\in\Hom_A(P_1,A)$ and each $(z-\zeta)^*\tilde{f}(x)$ is in the image of $\lmat j_+\\j_-\rmat:A\hookrightarrow A_+\oplus A_-$. In particular, \[a^+_{r-1}(x)-a^+_r(\zeta x)=a^-_{r-1}(x)-a^-_r(\zeta x) \qquad\text{for all $r$},\]and there exist integers $M_+ <M_-$ such that this quantity vanishes when $r>M_-$ and $r<M_+$. We have identified the connecting morphism $\delta'$, and by inspection of the morphism $F^{-1}$, we have for any $x\in P_1$:\[F^{-1}\circ\delta'(f)(x)=\sum_{M_+}^{M_-}(a_{-(r+1)}^{\pm}(x)-a_{-r}^\pm(\zeta x))z^r.\] Finally we pass to the cokernel of $z^{-1}-\zeta^*:i_!(i^!T)^*\to i_!(i^!T)^*$, so that $z^{-1}$ acts as $\zeta^*$:\[F^{-1}\circ \delta'(f)(x)=\sum_{M_+}^{M_-}a_{-(r+1)}^{\pm}(\zeta^{-r}x)-a_{-r}^\pm(\zeta^{1-r} x).\]To evaluate this sum, observe that it is part of a larger sum that can be separated into two telescoping sums\[\begin{array}{rcl}&&\sum_{r=1}^\infty(a_{-(r+1)}^+(\zeta^{-r} x)-a_{-r}^+(\zeta^{1-r} x))+\sum_{-\infty}^{r=0}(a_{-(r+1)}^-(\zeta^{-r} x)-a_{-r}^-(\zeta^{1-r} x))\\
&&\\
&=&(\lim_{r\to\infty}a_{-(r+1)}^+(\zeta^{-r} x)-a_{-1}^+(x))+(a_{-1}^-(x)-\lim_{r\to-\infty}a_{-r}^-(\zeta^{1-r} x)).\\
&&\\
&=&-(a_{-1}^+(x)-a_{-1}^-(x)).\end{array}\]To see this last equality, take a fixed generating set $\{x_1,\dots, x_n\}$ for $i^!T$ as an $R$-module. But for $x_s$, a given generator, $a_{-r}^+(x_s)\to 0$ and $a_r^-(x_s)\to 0$ as $r\to\infty$ by definition of $A_+$ and $A_-$. Now for each $r$ we have $\zeta^r x$ in the span of the generating set, and the $a_r^\pm$ are $R$-linear.

Finally, recalling our isomorphism \[(k_+\,\,-k_-):\frac{A_+\oplus A_-}{A}\xrightarrow{\cong} A_\infty,\]we see that the final quantity is indeed the negative of the coefficient of $z^{-1}$ in $(k_+\,\,-k_-)(f(x))\in A_\infty$ as required.
\end{proof}

\begin{corollary}For each $T$ in $\H(A,S)$, the trace function is a natural isomorphism of $A$-modules\[\chi_*:\Hom_A(T,S^{-1}A/A)\xrightarrow{'cong} \Hom_R(i^!T,R).\]
\end{corollary}

Why not take the natural morphism $\chi_*\circ(-z)$ to be our isomorphism in the definition of monodromy (\ref{def:monodromy})? The reason comes from the way duality in a non-singular $\eps$-symmetric autometric form $(K,\theta,h)$, matches up to duality in the corresponding linking form via the complex\[0\to i_!i^!K\xrightarrow{z-h} i_!i^!K\to 0\] and its dual complex.

\begin{proposition}\label{prop:traceisgood}If $(K,\theta,h)$ is a non-singular $\eps$-symmetric autometric form and $B(K,\theta,h)=(T,\lambda)$ is the covering then $\HH (B(K,\theta,h))\cong B(K,\theta,h)$.
\end{proposition}

\begin{proof}On the level of $(K,h)$ and $T$, this was already shown in \ref{prop:covmonod1}. We must show the definitions of the forms agree. But it is enough to show that $\chi_*\lambda([x],[y])=\theta(x,y)$ for $x,y\in K$ and $[x],[y]\in \coker(z-h:i_!i^!K\to i_!i^!K)=T$. By definition\[\lambda([x],[y])=z^{-1}\theta(x,(z-h)^{-1}y)\in S^{-1}A/A\cong A_\infty.\]But as we are working over $A_\infty$, we can express the inverse $(z-h)^{-1}=f_+-f_-$ where $f_\pm$ are the expansions\[\begin{array}{rcll}
f_+&=&h^{-1}+h^{-2}z+h^{-3}z^2+\dots&\in \Hom_A(i_!i^!K,A_+),\\
f_-&=&-(z^{-1}+hz^{-2}+h^2z^{-3}+\dots)&\in \Hom_A(i_!i^!K,A_-).\end{array}\]So in fact, recalling that $\lambda$ is conjugate linear in the second variable, we can write\[\lambda([x],[y])=z^{-1}\left(\left(\sum_1^\infty\theta(x,h^{-k}y)z^{k-1}\right)+\left(\sum_1^\infty\theta(x,h^{k-1}y)z^k\right)\right).\]So that the trace of $\lambda([x],[y])$ is indeed $\theta(x,y)$ as required.

\end{proof}

\chapter{Twist and frame spinning}\label{chap:twistspin}

Throughout this appendix, given $k<l$ we view $S^k$ as the subset of $S^l\subset\R^{l+1}$ given by setting the first $l-k$ coordinates to zero.  
Given $k$ and $l$ we furthermore pick an identification of $D^k\times D^l$ with $D^{k+l}$. We view $D^2$ also as a subset of $\C$. If $U$ is a submanifold of a manifold $V$ we use the notation $\nu U$ for an open tubular neighbourhood of $U$ in $V$.

Let $K:S^n\hookrightarrow S^{n+2}$ be an $n$-knot and recall we denote the associated disc knot by $\Delta_K$ and the respective exteriors $X_K=S^n\sm \nu K$ and $X_{\Delta_K}=D^{n+2}\sm \nu \Delta_K$. We will furthermore assume that $\Delta_K\subset D^{n+2}=D^2\times D^n$ in such a way that $\Delta_K=0\times D^n$.

\section{Twist spinning}\label{sec:twistspinning}

This section of the appendix was joint work with Stefan Friedl and is based on the preprint \cite{Friedl:2013fk}. We will give a quick reproof of Zeeman's twist-spinning result \cite[\textsection 6]{MR0195085}. As a corollary, we obtain the observation of Sumners {\cite[2.9]{MR0290351}.

\medskip

Given $z\in S^1$ we now denote by 
\[ \begin{array}{rcl} \rho_z\colon D^{n+2}=D^2\times D^n&\to & D^{n+2}=D^2\times D^n\\
(w,x)&\mapsto & (zw,x)\end{array}\]
the rotation by $z$ in the $D^2$-factor. Note that $\rho_z$ restricts to the identity on $\Delta_K\cap S^{n+1}$. 
Also note that we can and will assume that the decomposition $D^{n+2}=D^2\times D^n$ is oriented in such a way that for any $x\in S^{n+1}\sm \nu \Delta_K$ the closed curve 
\[ \begin{array}{rcl} S^1&\to & X_{\Delta_K} \\
z&\mapsto  & \rho_z(x)\end{array}\]
gives the oriented  meridian of $K$. 

\begin{figure}[h]\[\def\picdisc{\resizebox{0.3\textwidth}{!}{ \includegraphics{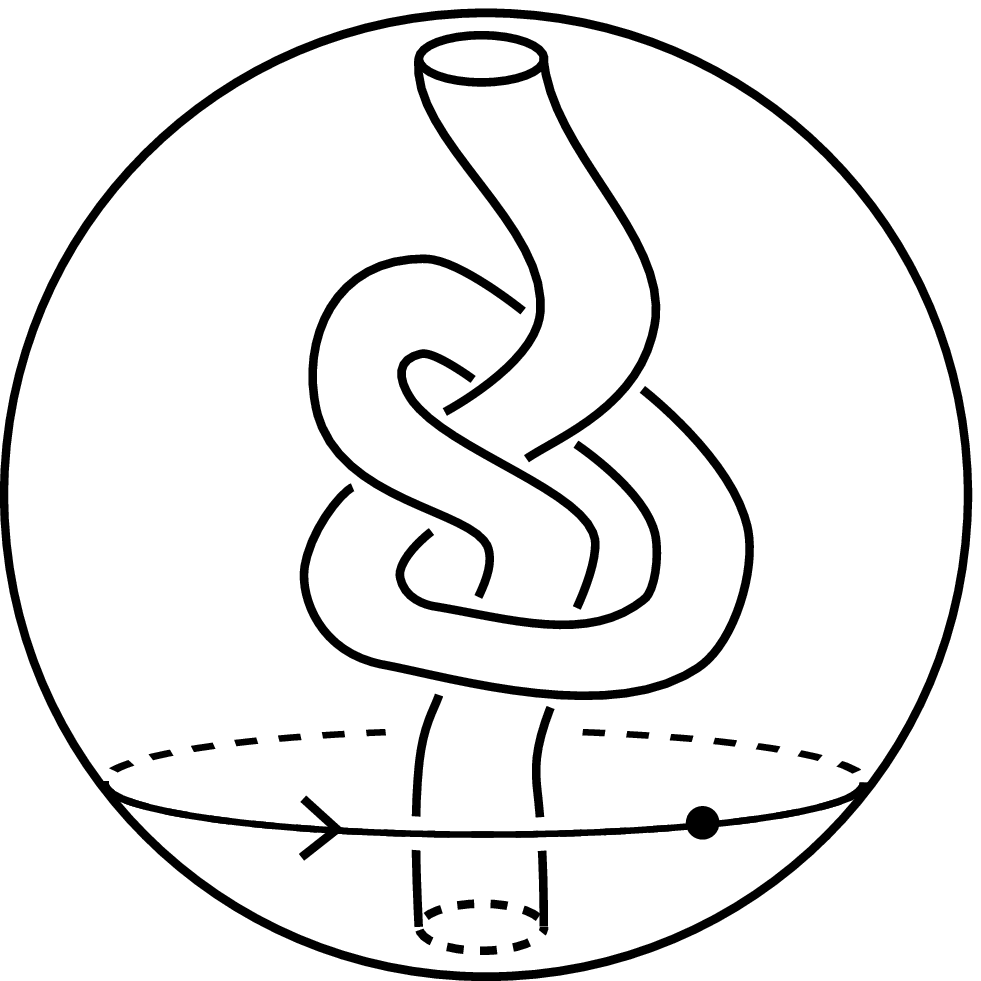}}}
\begin{xy} \xyimport(259,139){\picdisc}
,!+<7.2pc,2pc>*+!\txt{}
,(-150,100)*!L{S^{n+1}\sm\nu\Delta_{K}}
,(-35,100)*+{}="A";(20,100)*+{}="B"
,{"A"\ar"B"}
,(250,130)*!L{X_{\Delta_K}=D^{n+2}\sm \nu \Delta_K}
,(250,130)*+{}="A";(190,105)*+{}="B"
,{"A"\ar"B"}
,(25,75)*!L{\nu\Delta_K}
,(50,80)*+{}="A";(105,95)*+{}="B"
,{"A"\ar"B"}
,(173,15)*!L{x}
,(-25,25)*!L{\rho_z(x)}
\end{xy}\]
  \label{pic:disc}
\end{figure}

Now let $k\in \Z$. In order to define the $k$-twist spin of $K$, we use the following decomposition \[ S^{n+3}=S^1\times D^{n+2}\,\cup \, D^2\times S^{n+1}.\] Denote by $\Phi_k$ the homeomorphism \[ \begin{array}{rcl} \Phi_k\colon S^1\times D^{n+2}&\to & S^1\times D^{n+2}\\
(z,x)&\mapsto & (z,\rho_{z^k}(x)).\end{array}\] The \textit{$k$-twist spin $S_k(K)$} is then defined as 
\[ S_k(K):=
\underset{\subset S^1\times D^{n+2}}{\underbrace{ \Phi_k(S^1\times J)}}\,\,\cup\,\,
\underset{\subset D^2\times S^{n+1}}{\underbrace{ 
D^2\times  S^{n-1}}}.\] More informally, $S_k(K)$ is given by spinning the disc knot $\Delta_K$ around the $S^1$ direction, performing $k$ twists around $\Delta_K$ as you go, and then capping off the result by $D^2\times S^{n-1}$.

\begin{theorem}\label{mainthm}
If $k\neq 0$, then $X_{S_k(K)}$ fibres over $S^1$, where the fibre is the result of removing an open ball from the $k$-fold branched cover of $K$.
\end{theorem}

\begin{proof}
The proof consists of two parts. We will first describe $X=X_{S_k(K)}$ in a different, more convenient, way. We will then use this new description to write down the promised fibre bundle over $S^1$. The first part is well-known, in fact this description of $X$ is also given in \cite[p.~201]{MR2179262}.

We  write $Y=X_{\Delta_{K}}$. Note that $Y\cap \partial D^{n+2}=S^{n+1}\sm \nu S^{n-1}$. 
As usual we can identify $S^{n+1}\sm \nu S^{n-1}$ with $S^1 \times D^n$.
 We now see that 
\[ \begin{array}{rcl} X&=&S^1\times D^{n+2}\,\sm\, \Phi_k(S^1\times \nu J)\,\,\,\cup \,\,\,D^2\times  (S^{n+1}\sm \nu S^{n-1})\\
&=& \Phi_k(S^1\times Y)\,\,\,\cup \,\,\,D^2\times S^1 \times D^n.\end{array}\]
Note that  ${\Phi_k}$ restricts to an automorphism of $\partial (S^1\times Y)=S^1\times S^1 \times D^n$.
We can thus glue $S^1\times Y$ and $D^2\times S^1 \times D^n$ together via the restriction of  ${\Phi_k}$ to $S^1\times S^1 \times D^n$.
The map
\[  S^1\times Y\,\,\cup_{{\Phi_k}} \,\,D^2\times S^1 \times D^n\to
\Phi_k(S^1\times Y)\,\,\cup \,\,D^2\times S^1 \times D^n,\]
which is given by $\Phi_k$ on the first subset and by the identity on the second subset, is then a well-defined homeomorphism.

We will now use this new description of $X$ to write down the fibre bundle structure over $S^1$. 
Write $\psi:X\to S^1$ for the meridian map of $X$ (see Corollary \ref{cor:meridian}).
It is straightforward to verify that the assumption that $k\ne 0$ implies that the map
\begin{equation} \label{equ:fibre} \begin{array}{rcl} p\colon S^1\times Y&\to& S^1\\
(z,x)&\mapsto & z^{-k}\psi( x)\end{array}\end{equation}
defines a fibre bundle.
It now follows from the definitions that the map 
\[  S^1\times Y\,\,\cup_{{\Phi_k}} \,\,D^2\times S^1 \times D^n\to S^1\]
which is given by $p$ on the first subset and by projection on the $S^1$-factor in the second subset is the projection of a fibre bundle. 

It remains to identify the fibre of the fibration.
The fibre `on the right' (of the decomposition) is $D^{2}\times \{1\}\times D^n$ whereas the fibre `on the left' is given by 
\[ Y_k=\{ (y,z)\in S^1\,|\, \psi(y)=z^{k}\}\]
which is just the $k$-fold cyclic cover of $Y$ corresponding to the surjective morphism
$\pi_1(Y)\to H_1(Y;\Z)\xrightarrow{\cong} \Z\to \Z/k\Z$. 

Note that $Y$ is in fact homeomorphic to the knot exterior $X$, and that hence $Y_k$ is just the $k$-fold cyclic cover of $S^{n+2}\sm \nu K$. 
It is now straightforward
 to see that the fibre
\[ Y_k\cup_{S^1\times \{1\}\times D^n} D^{2}\times \{1\}\times D^n\]
is the result of attaching a 2-handle to $Y_k$ along the preimage of a meridian under the covering map $Y_k\to Y$. Put differently, the fibre is obtained 
by removing an open ball from the $k$-fold branched cover of $K$.
\end{proof}

\begin{corollary}\label{cor:sk1trivial}
If $K\subset S^{n+2}$ is a knot, then $S_{\pm 1}(K)\subset S^{n+3}$ is a trivial knot.
\end{corollary}

\begin{proof}
First note that  the $\pm 1$-fold branched cover of $S^{n+2}$ along $K$ is just $S^{n+2}$ again. It thus follows  from Theorem \ref{mainthm} that $S_{\pm 1}(K)$ bounds an $(n+2)$-ball in $S^{n+3}$, which means that  $S_{\pm 1}(K)\subset S^{n+3}$ is a trivial knot.
\end{proof}

We also make following observation concerning twist spins.

\begin{lemma}\label{lem:kplus-k}
If $K:S^n\hookrightarrow S^{n+2}$ is a knot, then for any $k\in\Z$ the knot $ S_{k}(K)\cap S^{n+2}$ is isotopic to $K\# -K$.
\end{lemma}

\begin{proof}
We denote by $\Delta'$ the disc knot which is defined by $\Phi(-1\times \Delta_K)=-1\times \Delta'$. 
Put differently, $\Delta'$ is  the result of rotating $\Delta_K\subset D^2\times D^n=D^{n+2}$  by $k\pi$.  Note that $\Delta$ is isotopic in $D^{n+2}$ to $\Delta_K$ rel the boundary.
We write 
\[ S^{n+3}=S^1\times  D^{n+2}\,\,\cup \,\, D^2\times S^{n+1}\]
with equator sphere
\[ S^{n+2}=\{ \pm 1\} \times D^{n+2} \,\,\cup \,\, D^1\times S^{n+1}.\]
The above decomposition of $S^{n+3}$ gives rise to an  orientation preserving map 
\[ \Psi\colon S^1\times D^{n+2} \to S^{n+3}\]
such that 
\[ \Psi(S^1\times D^{n+2})\cap S^{n+2} = \{-1\} \times D^{n+2}\,\,\cup \{1\}\times D^{n+2}.\]
Note that the restriction of $\Psi$ to $ \{-1\} \times D^{n+2}$
is \textit{orientation reversing} and that the restriction of $\Psi$ to $ \{1\} \times D^{n+2}$
is \textit{orientation preserving}. In particular
$\Phi_k(S^1\times \Delta_K)\cap S^{n+2}$ is the union of $\Delta_K$ with the mirror image of $\Delta'$. 

Since $\Delta_K$ and $\Delta'$ are isotopic rel the boundary it now follows easily that  $S_k(K)\cap S^{n+2}$ is isotopic to the connected sum of $K$ and $-K$.
\end{proof} 

The following corollary is now an immediate consequence of 
Corollary \ref{cor:sk1trivial} and Lemma \ref{lem:kplus-k}.

\begin{corollary}
If $K\subset S^{n+2}$ is an oriented knot, then $K\#(-K)$ is doubly slice.
\end{corollary}

\section{Frame spinning}

This section of the appendix was joint work (unpublished) with Mark Powell.

\medskip

Frame spinning is a process that takes an $n$-knot $K:S^n\hookrightarrow S^{n+2}$ and a closed embedded submanifold $M^m\subset S^{m+n}$ with framing $\varphi$ of the normal bundle, and returns an $(n+m)$-knot $T_M^\varphi(K)$. In \cite{MR2063133} the author shows that all frame spun knots are slice. In this section we show the following extension to that theorem.

\begin{theorem}\label{thm:framespun}All odd-dimensional frame spun knots have hyperbolic Seifert forms.
\end{theorem}

We begin by describing the frame spin $T_M^\varphi(K)$ with input $(K,M,\varphi)$ as above. It is proved in {\cite{MR2063133} that the existence of the embedding $M\subset S^{m+n}$ is enough to show $M$ is stably parallelisable. Hence it has vanishing Pontryagin classes and Stiefel Whitney classes. In particular it is orientable.

Now take the standard framed unknotted $S^{m+n}\times D^2\subset S^{m+n+2}$. Combined with the framing $\varphi$, this determines an embedding of normal disc bundles $(M\times D^{n+2},M\times D^n)\subset (S^{m+n+2},S^{m+n})$. We now excise the interiors of the disc bundles and glue back in the product of $M$ and the disc knot $\Delta_K\subset D^{n+2}$. More precisely $T_M^\varphi(K)$ is defined as \[\cl(S^{m+n}\sm(M\times D^n)\cup_{M\times S^{n-1}}M\times \Delta_K,\]embedded as\[\cl(S^{m+n+2}\sm(M\times D^{n+2})\cup_{M\times S^{n+1}}M\times D^{n+2}.\](We remark that when $M$ is an $m$-dimensional unknot, we have $T_M^\varphi(K)=S_0(K)$, the 0-twist-spin.)

We will use the following main result.

\begin{theorem}[{\cite[Proof of 4.4]{MR2063133}}]Suppose $m+n$ is odd and we are given an $n$-knot $K$ and $(M^m,\varphi)$. Then the frame spin $T_M^\varphi(K)$ admits a Seifert surface $F$ such that\begin{enumerate}[(i)]
\item If $n$ is even then the Seifert form of $F$ is hyperbolic.
\item If $n$ is odd then the Seifert form of $F$ is double Witt equivalent to a form $(K,\psi)$ where $\psi$ has matrix \[A\otimes \tau\] for $A$ a Seifert matrix of $K$ and $\tau$ an integer matrix of the middle-dimensional intersection form on $H^{m/2}(M;\Z)/\{\text{$\Z$-torsion}\}$.
\end{enumerate}
\end{theorem}

\begin{proof}[Proof (of Theorem \ref{thm:framespun})]If $m\equiv 2$ modulo 4 then there exists a basis with respect to which $\tau$ is the standard symplectic matrix\[\tau=\left(\begin{matrix}0&I\\-I& 0\end{matrix}\right).\]Hence any matrix $A\otimes \tau$ is hyperbolic.

If $m\equiv 0$ modulo 4 then we use the vanishing of the Pontryagin classes of $M$ and the Hirzebruch signature formula to conclude that the signature of $M$ is 0. Furthermore, the vanishing of the Stiefel-Whitney classes of $M$, combined with the Wu formula for the Steenrod squares implies that the diagonal entries of any matrix for $\tau$ are even integers. As the signature is 0, any unimodular matrix representing the form is either trivial or is indefinite. But by the classification of unimodular matrices, a unimodular, indefinite matrix with signature 0 and even entries on the diagonal is hyperbolic. So as above, $A\otimes \tau$ is hyperbolic.
\end{proof}

We remark finally that the homotopy type of frame spun knots and `frame spun Seifert surfaces' is well-studied in \cite{zbMATH00270255}, and cast in terms of the homotopy type of $M$, making these frame-spun knots a particularly rich source of future examples of knots with hyperbolic Seifert forms where there is some control over the connectivity of the Seifert surface.

\bibliographystyle{amsalpha}
\providecommand{\bysame}{\leavevmode\hbox to3em{\hrulefill}\thinspace}
\providecommand{\MR}{\relax\ifhmode\unskip\space\fi MR }
\providecommand{\MRhref}[2]{%
  \href{http://www.ams.org/mathscinet-getitem?mr=#1}{#2}
}
\providecommand{\href}[2]{#2}

\end{document}